\documentclass[11pt, reqno]{amsart}
\usepackage{amsmath, amssymb, latexsym, enumerate,  graphicx,tikz, stmaryrd,pifont}
\usetikzlibrary{arrows,decorations,decorations.pathmorphing,snakes,positioning}
\usepackage{mdframed}
\usepackage{mdwlist}
\usepackage[pdftex]{hyperref}
\usepackage{amscd,mathrsfs,epic,empheq,float}
\usepackage{bbm,bm}
\usepackage[all]{xy}
 \usepackage{todonotes}
\usepackage{tikz}
\usepackage{graphicx}
\usepackage{caption}
\usepackage[vcentermath]{youngtab}
\newcommand{\renc}{\renewcommand}

 \newlength{\baseunit}               
 \newcount{\numlines}                
\newtheorem{theorem}{Theorem}[section]
\newtheorem{lemma}[theorem]{Lemma}
\newtheorem{prop}[theorem]{Proposition}
\newtheorem{corollary}[theorem]{Corollary}

\theoremstyle{definition}
\newtheorem{definition}[theorem]{Definition}

\newtheorem{remark}[theorem]{Remark}

\newtheorem{example}[theorem]{Example}
\newtheorem{ex}[theorem]{Example}

\newcommand{\bu}{{\scriptstyle \blacklozenge}}

\newcommand{\cM}{{\mathcal M}}

\newcommand{\cO}{{\mathcal O}}

\newcommand{\cT}{{\mathcal T}}

\newcommand{\ABr}{\mathcal{A}({\rm{D}_k})}

\newcommand{\ov}{\overline}

\newcommand{\B}{\mathbb{B}_{\Lambda}}

\newcommand{\D}{\mathbb{D}_{\Lambda}}

\def\down{\vee}
\def\up{\wedge}

\newcommand{\MOD}{\text{-}\op{mod}}
\newcommand{\gMOD}{\text{-}\op{gmod}}
\newcommand{\mg}{\mathfrak{g}}

\newcommand{\mR}{\mathbb{R}}

\newcommand{\mC}{\mathbb{C}}

\newcommand{\mZ}{\mathbb{Z}}
\newcommand{\mX}{\mathbb{X}}
\newcommand{\mD}{\mathbb{D}}

\newcommand{\pos}{\operatorname{p}}
\newcommand{\cupg}{\operatorname{cup}(\gamma)} 
\newcommand{\capg}{\operatorname{cap}(\gamma)}

\newcommand{\la}{\lambda}
\newcommand{\La}{\Lambda}
\newcommand{\Lap}{\Lambda_k^{\ov{p}}}

\newcommand{\p}{\mathfrak{p}}

\newcommand{\op}{\operatorname}

\renewcommand{\S}{\mathbb{S}_k}

\DeclareMathOperator{\End}{End}   
  
\DeclareMathOperator{\Mod}{\text{-}mod}

\DeclareMathOperator{\Perv}{Perv}

\newcommand{\CupConnectGraph}[1]{%
\begin{tikzpicture}[#1]%
\draw[thick] (0,0.6ex) -- (2.75ex,0.6ex);%
\draw[opacity=0] (0,0) -- (3ex,0);%
\end{tikzpicture}%
}

\newcommand{\DCupConnectGraph}[1]{%
\begin{tikzpicture}[#1]%
\draw[thick] (0,0.6ex) -- (2.75ex,0.6ex);%
\fill (1.375ex,0.6ex) circle(0.4ex);%
\draw[opacity=0] (0,0) -- (3ex,0);%
\end{tikzpicture}%
}

\newcommand{\RayConnectGraph}[1]{%
\begin{tikzpicture}[#1]%
\draw[opacity=0] (0,0) -- (3ex,0);%
\draw[thick] (0,0.6ex) -- (2.75ex,0.6ex);%
\draw[thick] (2.75ex,1.2ex) -- (2.75ex,0);%
\end{tikzpicture}%
}

\newcommand{\DRayConnectGraph}[1]{%
\begin{tikzpicture}[#1]%
\draw[opacity=0] (0,0) -- (3ex,0);%
\draw[thick] (0,0.6ex) -- (2.75ex,0.6ex);%
\fill (1.375ex,0.6ex) circle(0.4ex);%
\draw[thick] (2.75ex,1.2ex) -- (2.75ex,0);%
\end{tikzpicture}%
}

\newcommand*\circled[1]{\tikz[baseline=(char.base)]{\node[shape=circle,draw,inner sep=1pt] (char) {$\scriptstyle #1 $};}}
            
\newcommand{\CupConnect}{\CupConnectGraph{}}
\newcommand{\DCupConnect}{\DCupConnectGraph{}}
\newcommand{\RayConnect}{\RayConnectGraph{}}
\newcommand{\DRayConnect}{\DRayConnectGraph{}}

\newcommand{\un}{\underline}

\begin{document}
\title[Diagrams for perverse sheaves on isotropic Grassmannians]{Diagrammatic description for the categories of perverse sheaves on isotropic Grassmannians}
\author{Michael Ehrig}
\address{M.E.: Mathematisches Institut, Universit\"at Bonn, Endenicher Allee 60, Room 1.003, 53115 Bonn, Germany}
\email{mehrig@math.uni-bonn.de}

\author{Catharina Stroppel}
\address{C.S.: Mathematisches Institut, Universit\"at Bonn, Endenicher Allee 60, Room 4.007, 53115 Bonn, Germany}
\email{stroppel@math.uni-bonn.de}

\thanks{M.E. was financed by the DFG Priority program 1388. This material is based on work supported by the National Science Foundation under Grant No. 0932078 000, while the authors were in residence at the MSRI in Berkeley, California.}

\begin{abstract}
 For each integer $k\geq 4$ we describe diagrammatically a positively graded Koszul algebra $\mathbb{D}_k$ such that the category of finite dimensional
$\mathbb{D}_k$-modules is equivalent to the category of perverse sheaves on the isotropic Grassmannian of type ${\rm D}_k$ or ${\rm B}_{k-1}$, constructible with
respect to the Schubert stratification. The algebra is obtained by a (non-trivial) ``folding'' procedure from a generalized Khovanov arc algebra. Properties like graded cellularity and explicit closed formulas for graded decomposition numbers are established by elementary tools.
\end{abstract}

\maketitle

\tableofcontents
\section{Introduction}
\renc{\thetheorem}{\Alph{theorem}}
This is the first of four articles
studying some generalizations $\mathbb{D}_\Lambda$
of (generalized)  Khovanov arc algebras from \cite{Khovtangles} and \cite{BS1}  to types $\rm D$ and $\rm B$.

In this article we introduce these algebras and prove that for the special choice of a principal block, $\Lambda=\Lambda^{\ov 0}_k$, the category of finite dimensional modules over $\mathbb{D}_k:=\mathbb{D}_\Lambda$ is equivalent to the category $\op{Perv}_k$ of perverse sheaves on the isotropic Grassmannian $X$ of type ${\rm D}_k$ or  ${\rm B}_{k-1}$ constructible with respect to the Schubert stratification. The equivalence in type  $\rm D$ will be established by giving an explicit isomorphism to the endomorphism algebra of a projective
generator described in \cite{Braden}. Then we observe that type  ${\rm B}_{n+1}$ can in fact be deduced immediately by an isomorphism of varieties. Hence we stick for now to type ${\rm D}_n$. 
Since we have $X=G/P$ for $G=SO(2k,\mC)$ and $P$ the parabolic subgroup of type ${\rm A}_{k-1}$, these categories are then also equivalent to the Bernstein-Gelfand-Gelfand parabolic category $\cO_0^\p(\mathfrak{so}(2k,\mC))$ of type ${\rm D}_k$ with parabolic of type ${\rm A}_{k-1}$, see \cite{Humphreys} for the definition.

The algebras $\mathbb{D}_\Lambda$ are constructed in an elementary way, purely in terms of diagrams; they come naturally equipped with a grading. As a vector space they have an explicit homogeneous basis given by certain oriented circle diagrams similar to \cite{BS1}. Under our  isomorphism, Braden's algebra inherits a grading from $\mathbb{D}_k$ which agrees with the geometric Koszul grading from \cite{BGS}.\\

To be more specific, recall that Schubert varieties, and hence the simple objects in $\op{Perv}_k$ are labelled by symmetric partitions fitting into a box of size $k\times k$, or equivalently by the representatives of minimal length for the cosets $S_k\backslash W({\rm D}_k)$ of the Weyl group $W({\rm D}_k)$ of type ${\rm D}_k$ modulo a parabolic subgroup isomorphic to the symmetric group $S_k$. We first identify in Section~\ref{section21} these cosets with {\it diagrammatic weights} $\la$ (i.e. with $\{\up, \down\}$-sequences of length $k$ with an even number of $\up$'s) and then associate to each such diagrammatic weight a {\it (decorated) cup diagram} $\underline{\la}$ on $k$ points, see Definition~\ref{decoratedcups}. For instance, the eight possible decorated cup diagrams for $k=4$ are displayed in Example~\ref{exB} below. Such a cup diagram $\underline{\la}$ can be paired with a second cup diagram  $\underline{\mu}$ by putting $\underline\mu$ upside down on top of $\underline{\la}$ to obtain a {\it (decorated) circle diagram} $\underline{\la}\overline{\mu}$, see Example~\ref{ps} for such circle diagrams. Adding additionally a compatible weight diagram in between gives us an {\it oriented circle diagram}, see for instance Example~\ref{triples}. Let $\mathbb{D}_k$ be the vector space spanned by these decorated circle diagrams for fixed $k$.\\

The following theorem collects some of our main results.

\begin{theorem}
\label{intro}
\begin{enumerate}[1.)]
\item The vector space  $\mathbb{D}_k$ can be equipped with a diagrammatically defined associative algebra structure such that  $\mathbb{D}_k$ is isomorphic to the endomorphism algebra of a minimal projective generator of $\op{Perv}_k$ (see Theorem~\ref{thm:main}).
\item The multiplication is compatible with the grading (Proposition~\ref{forgotgraded}), such that $\mathbb{D}_k$ becomes a graded algebra (see Theorem~\ref{algebra_structure}).
\item The underlying graph of the quiver $\mathcal{Q}_k$ of $\op{Perv}_k$ is a finite graph with vertices labelled by the cup diagrams on $k$ points. Each cup $C$ in a cup diagram $\un\la$ defines an outgoing and an ingoing arrow to respectively from a cup diagram $\un\mu$, where $\mu$ is obtained from $\la$ by swapping $\up$ with $\down$ at the points connected by $C$
(see Lemma~\ref{lapairscup} and Corollary~\ref{samegraph}). 

\item The graded Cartan matrix $C_k$ of $\mathbb{D}_k$ is indexed by the cup diagrams $\underline{\la}$ and the entries count the number of possible orientations of the circle diagram $\un{\la}\ov\mu$ with their gradings (see Lemma~\ref{Cartanmatrix}).
\end{enumerate}
\end{theorem}

In case $k=4$ we obtain the following:

\subsection{An explicit example} 
\label{exB}
The quiver $\mathcal{Q}_4$ looks as follows:
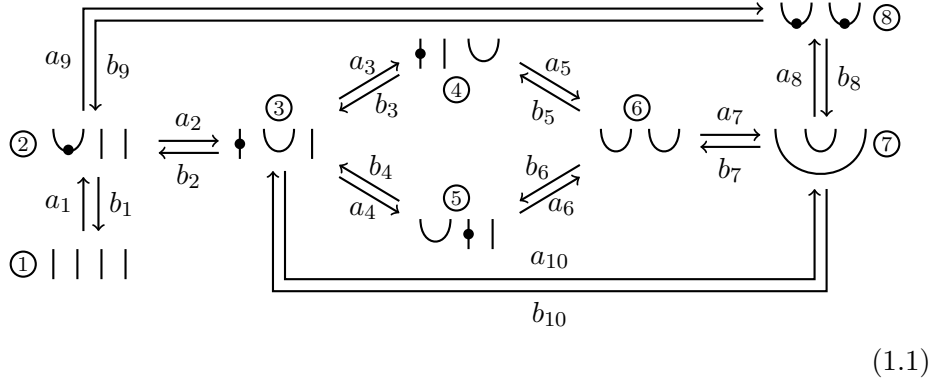
\begin{equation} \label{fig:quiv}
\begin{tikzpicture}[thick,scale=.8]
\begin{scope}[xshift=3.1cm,yshift=-2cm]
\draw (.4,0) to +(0,-.5);
\draw (.8,0) to +(0,-.5);
\draw (1.2,0) to +(0,-.5);
\draw (1.6,0) to +(0,-.5);
\end{scope}
\draw[->] (4.25,-.8) to node[right]{$b_1$} +(0,-.9);
\draw[<-] (4,-.8) to node[left]{$a_1$} +(0,-.9);
\node[shape=circle,draw,inner sep=1pt] at (3,-2.25) {$\scriptstyle 1 $};

\begin{scope}[xshift=3.1cm]
\draw (.4,0) .. controls +(0,-.5) and +(0,-.5) .. +(.5,0);
\fill (.65,-.35) circle(2.5pt);
\draw (1.2,0) to +(0,-.5);
\draw (1.6,0) to +(0,-.5);
\end{scope}
\draw[->] (5.25,-.2) to node[above]{$a_2$} +(1,0);
\draw[<-] (5.25,-.4) to node[below]{$b_2$} +(1,0);
\node[shape=circle,draw,inner sep=1pt] at (3,-.25) {$\scriptstyle 2 $};

\begin{scope}[xshift=6.2cm]
\draw (.4,0) to +(0,-.5);
\fill (.4,-.25) circle(2.5pt);
\draw (.8,0) .. controls +(0,-.5) and +(0,-.5) .. +(.5,0);
\draw (1.6,0) to +(0,-.5);
\end{scope}
\draw[->] (8.25,.5) to node[above,pos=0.4]{$a_3$} +(1,.6);
\draw[<-] (8.25,.3) to node[below,pos=0.8]{$b_3$} +(1,.6);
\draw[<-] (8.25,-.6) to node[above,pos=0.7]{$b_4$} +(1,-.6);
\draw[->] (8.25,-.8) to node[below,pos=0.4]{$a_4$} +(1,-.6);
\node[shape=circle,draw,inner sep=1pt] at (7.25,.35) {$\scriptstyle 3 $};

\begin{scope}[xshift=9.2cm,yshift=1.5cm]
\draw (.4,0) to +(0,-.5);
\fill (.4,-.25) circle(2.5pt);
\draw (.8,0) to +(0,-.5);
\draw (1.2,0) .. controls +(0,-.5) and +(0,-.5) .. +(.5,0);
\end{scope}
\node[shape=circle,draw,inner sep=1pt] at (10.2,0.65) {$\scriptstyle 4 $};

\begin{scope}[xshift=9.2cm,yshift=-1.5cm]
\draw (.4,0) .. controls +(0,-.5) and +(0,-.5) .. +(.5,0);
\draw (1.2,0) to +(0,-.5);
\fill (1.2,-.25) circle(2.5pt);
\draw (1.6,0) to +(0,-.5);
\end{scope}
\node[shape=circle,draw,inner sep=1pt] at (10.2,-1.15) {$\scriptstyle 5 $};

\draw[->] (11.25,1.1) to node[above,pos=.65]{$a_5$} +(1,-.6);
\draw[<-] (11.25,.9) to node[below,pos=.4]{$b_5$} +(1,-.6);
\draw[<-] (11.25,-1.2) to node[above,pos=.3]{$b_6$} +(1,.6);
\draw[->] (11.25,-1.4) to node[below,pos=.7]{$a_6$} +(1,.6);

\begin{scope}[xshift=12.2cm]
\draw (.4,0) .. controls +(0,-.5) and +(0,-.5) .. +(.5,0);
\draw (1.2,0) .. controls +(0,-.5) and +(0,-.5) .. +(.5,0);
\end{scope}
\draw[->] (14.25,-.1) to node[above]{$a_7$} +(1,0);
\draw[<-] (14.25,-.3) to node[below]{$b_7$} +(1,0);
\node[shape=circle,draw,inner sep=1pt] at (13.2,0.35) {$\scriptstyle 6 $};

\begin{scope}[xshift=15cm]
\draw (.5,0) .. controls +(0,-1) and +(0,-1) .. +(1.5,0);
\draw (1,0) .. controls +(0,-.5) and +(0,-.5) .. +(.5,0);
\end{scope}
\draw[->] (16.15,.2) to node[left]{$a_8$} +(0,1.3);
\draw[<-] (16.35,.2) to node[right]{$b_8$} +(0,1.3);
\draw[<-] (16.15,-1) -- +(0,-1.5) -- node[above]{$a_{10}$} +(-8.8,-1.5) -- +(-8.8,.3);
\draw[->] (16.35,-1) -- +(0,-1.7) -- node[below]{$b_{10}$} +(-9.2,-1.7) -- +(-9.2,.3);
\node[shape=circle,draw,inner sep=1pt] at (17.35,-.25) {$\scriptstyle 7 $};

\begin{scope}[xshift=15.2cm,yshift=2.1cm]
\draw (.4,0) .. controls +(0,-.5) and +(0,-.5) .. +(.5,0);
\fill (.65,-.35) circle(2.5pt);
\draw (1.2,0) .. controls +(0,-.5) and +(0,-.5) .. +(.5,0);
\fill (1.45,-.35) circle(2.5pt);
\end{scope}
\draw[<-] (15.3,2) -- +(-11.3,0) -- node[left]{$a_9$} +(-11.3,-1.7);
\draw[->] (15.3,1.8) -- +(-11.1,0) -- node[right]{$b_9$} +(-11.1,-1.5);
\node[shape=circle,draw,inner sep=1pt] at (17.35,1.85) {$\scriptstyle 8 $};

\end{tikzpicture}
\end{equation}
The algebra $\mathbb{D}_4$ is the path algebra of this quiver $\mathcal{Q}_4$ modulo the following defining relations:
\begin{itemize}
\item[$\blacktriangleright$] \textit{Diamond relations:} Let $x$, $y$ be two distinct vertices from one of the  following four sets, called \emph{diamonds} in Definition \ref{Bradenalgebra}:
$$(\circled{3},\circled{4},\circled{6},\circled{5}),\: (\circled{3},\circled{4},\circled{6},\circled{7}),\: (\circled{3},\circled{5},\circled{6},\circled{7}),\: (\circled{2},\circled{3},\circled{7},\circled{8}).$$ 
 Then any two paths of length two from $x$ to $y$ involving only vertices inside the same diamond are equal.
\item[$\blacktriangleright$] \textit{Non-extendibility relations:} The following compositions are zero
\begin{equation*}
\begin{array}{cccccccc}
a_1a_2, & b_2b_1, & a_2a_3, & b_3b_2, & a_2a_4, & b_4b_2, & a_7a_8, & b_8b_7.
\end{array}
\end{equation*}
\item[$\blacktriangleright$] \textit{Loop relations:} The following elements are zero
\begin{equation*}
\begin{array}{ccccc}
a_1b_1, & b_1a_1+a_2b_2, & a_9b_9, & b_2a_2 - a_3b_3, & b_5a_5+b_6a_6+a_7b_7, \\
a_{10}b_{10}, & a_5b_5, & a_6b_6, & a_3b_3-a_4b_4, & b_7a_7+a_8b_8+2b_{10}a_{10}.
\end{array}
\end{equation*}
\end{itemize}
The grading of the algebra is just given by the length of the paths. (Note that all relations are homogeneous). Observe moreover that the algebra is quadratic, i.e. generated in degrees zero and one with relations in degree two, see Theorem~\ref{prop:generatedindegree1} for a direct proof of the fact that $\mathbb{D}_k$ is generated in degrees zero and one.

With the idempotents numbered as in \eqref{fig:quiv}, the graded Cartan matrix is
\begin{center}
\resizebox{\linewidth}{!}{
$\begin{pmatrix}
1&q&0&0&0&0&0&q^2\\
q&1+q^2&q&0&0&0&q^2&q+q^3\\
0&q&1+q^2&q&q&q^2&q+q^3&q^2\\
0&0&q&1+q^2&q^2&q+q^3&q^2&0\\
0&0&q&q^2&1+q^2&q+q^3&q^2&0\\
0&0&q^2&q+q^3&q+q^3&1+2q^2+q^4&q+q^3&0\\
0&q^2&q+q^3&q^2&q^2&q+q^3&1+2q^2+q^4&q+q^3\\
q^2&q+q^3&q^2&0&0&0&q+q^3&1+2q^2+q^4
\end{pmatrix}$
}
\end{center}

\subsection{Definition of the algebra}
The definition of the multiplication on $\mathbb{D}_\Lambda$ is similar to the type $\rm A$ version, i.e. to Khovanov's original definition, \cite{Khovtangles} with its generalization in \cite{Str09} and \cite{BS1}, which used the fact that the Frobenius algebra $R=\mC[x]/(x^2)$ defines a 2-dimensional TQFT $\cT$, hence assigns in a functorial way to a union $(S^1)^n$ of $n$ circles the tensor product $\cT((S^1)^n)=R^{\otimes n}$ of $n$ copies of $R$ and to a cobordisms $(S^1)^n\rightarrow (S^1)^m$ between two finite union of circles a linear map $\cT((S^1)^n)\rightarrow \cT((S^1)^m)$.  A basis vector in our algebra $\mathbb{D}_\Lambda$ corresponds to a fixed orientation of some circle diagram with, say $n$, circles and so can be identified canonically with a standard basis vector of $\cT((S^1)^n)$, see Proposition~\ref{coho}. This allows to use the linear maps attached to the cobordisms to define the multiplication of $\mathbb{D}_\Lambda$. In contrast to the type $\rm A$ case, this construction needs however some additional signs incorporated in a subtle way (encoded by the decorations on the diagrams), see Section~\ref{sec:explicitmult}. Depending on the viewpoint these signs destroy either the locality or force to consider the circles to be embedded in the plane. Therefore, general topological arguments using the TQFT-structure cannot be applied (for instance to deduce the associativity). 

 We give in Section~\ref{sec:idea} the main idea of the multiplication, but then have to work hard  to show the associativity by algebro-combinatorial methods, see Theorem~\ref{thm:surgeries_commute}. For the understanding of the multiplication itself, it is enough to have a look at the algebraic definition of the vector space underlying $\D$ in Section~\ref{sec:spaceofdiagrams} and at the explicit rule for the multiplication in the algebraic context in  Section \ref{sec:explicitmult} with Examples in Section~\ref{sec:examples}. We recommend to skip Section~\ref{sec:surgery} completely for the first reading.

The framework here is set up slightly more general than actually needed, but allows to work with more involved diagrams, similar to \cite{BS1}, \cite{BS3}. In this framework it is straight forward to define also bimodules over our algebras which correspond to projective functors on the category $\cO$ side. A detailed study of these functors will appear in a subsequent paper.
 
 \subsection{Origin of the decorated cup diagrams} The resemblance of our construction here with \cite{BS1}, but also its technical subtleties appear in fact from a (non-trivial) {\it folding} of the type $\rm A$ setup.  Our type $\rm D$ weights of length $k$ could be seen as antisymmetric type $\rm A$ weights of length $2k$ and our cup diagrams as certain symmetric  (with respect to  a middle vertical axis) type $\rm A$ cup diagrams.  It was already observed in \cite{LS} that such symmetric cup diagrams can be used to describe the parabolic Kazdhan-Lusztig polynomials of type  $({\rm D}_k, {\rm A}_{k-1})$. To obtain our decorated cup diagrams from a symmetric cup diagram of type $\rm A$ one could follow  \cite[5.2]{LS} and fold the diagram along its mirror axis to obtain a cup diagram of half the size. The folding identifies each cup/ray with its mirror image if it does not cross the vertical reflection line, and folds it into itself otherwise. More precisely, given a symmetric cup diagram with $2k$ endpoints of cups/rays (called vertices), numbered by $1,2,\ldots, 2k$ from left to right, then folding removes first all vertices labelled $1,2,\ldots, k$ and all cups/rays with at least one endpoint amongst these points. If $i_1<i_2 < \ldots < i_r$  are the  number of the vertices not attached to a cup/ray anymore, then we reconnect these vertices by putting a cup decorated with a $\bullet$ connecting $i_1$ with $i_2$, $i_3$ with $i_4$ etc. and finally put  a ray decorated with a $\bullet$ at vertex $i_r$ in case $r$ is odd. 
 For instance, 

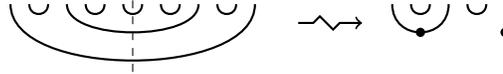
\begin{eqnarray*}
\begin{tikzpicture}[thick,scale=0.50,snake=zigzag, line before snake = 2mm, line after snake = 2mm]
\draw (0,0) .. controls +(0,-2) and +(0,-2) .. +(6.5,0);
\draw (.5,0) .. controls +(0,-.35) and +(0,-.35) .. +(.5,0);
\draw (2,0) .. controls +(0,-.35) and +(0,-.35) .. +(.5,0);
\draw (3,0) .. controls +(0,-.35) and +(0,-.35) .. +(.5,0);
\draw (4,0) .. controls +(0,-.35) and +(0,-.35) .. +(.5,0);
\draw (5.5,0) .. controls +(0,-.35) and +(0,-.35) .. +(.5,0);
\draw (1.5,0) .. controls +(0,-1) and +(0,-1) .. +(3.5,0);
\draw[dashed,thin] (3.25,.1) -- (3.25,-1.8);
\draw[->, snake] (7.65,-.5) -- +(1.7,0);

\begin{scope}[xshift=6.65cm]
\draw (4,0) .. controls +(0,-.35) and +(0,-.35) .. +(.5,0);
\draw (3.5,0) .. controls +(0,-1) and +(0,-1) .. +(1.5,0);
\draw (5.5,0) .. controls +(0,-.35) and +(0,-.35) .. +(.5,0);
\draw (6.5,0) -- (6.5,-1.5);
\fill (4.25,-.75) circle(3.5pt);
\fill (6.5,-.75) circle(3.5pt);
\end{scope}
\end{tikzpicture}
\end{eqnarray*}
In the following, we just record the folded diagram. To make this unambiguous we mark the cups (or rays) which got folded into itself by a dot, $\bullet$. By construction, dots can only occur on cups which are not nested and always to the left of possible rays, see Definition~\ref{def:decorated} for a precise definition. Such types of decorations are well-known tools in the theory of diagram algebras, see e.g. \cite{Green}, \cite{Martinblob}. In \cite{LS} it was also shown that on the space of such cup diagrams we have a natural action of the type $\rm D$ Temperley-Lieb algebra as defined in \cite{Green}. This algebra contains the type $\rm A$ Temperley-Lieb algebra as a subalgebra, plus one more generator which is usually displayed by a cap-cup diagram with a $\bullet$ on its cap and on its cup. The above mentioned bimodules can be used to categorify the type $\rm D$ Temperley-Lieb algebra. The underlying combinatorics was developed in \cite{LS}.

\subsection{Important properties of the algebras}
In Section~\ref{sec:cell} we study in detail the structure of the associative algebra $\mathbb{D}_\Lambda$. We establish its cellularity in the sense of Graham and Lehrer \cite{GL} in the graded version of \cite{MH} and determine explicitly the decomposition matrices $M_k$ by an easy counting formula in terms of parabolic Kazhdan-Lusztig polynomials of type $({\rm D}_k, {\rm A}_{k-1})$, see Lemma~\ref{lem:KLpolys}. It allows us to identify the Cartan matrix $C_k=M^t_kM_k$ with the Cartan matrix of Braden's algebra $\ABr$. This is then used in Section ~\ref{sec:iso_theorem} to show that the explicit assignment in Theorem~\ref{mainthm} defines indeed an isomorphism $\Phi_k$. Note that our generators are obtained from Braden's by some formal logarithm. This fact is responsible for two major advantages of our presentation in contrast to Braden's which will play an important role in the subsequent papers:
 \begin{itemize}
 \item[$\blacktriangleright$] the nice positive (Koszul) grading of $\mathbb{D}_k$  (which becomes invisible under the exponential map $\Phi_k^{-1}$); and
\item[$\blacktriangleright$] the Cartan matrix and decomposition numbers are totally explicit in our setup (but not really computable in terms of Braden's generators).
\end{itemize}

To emphasize that our results should not just be seen as straightforward generalizations of known results, let us indicate some  applications and connections which will appear in detail in the subsequent papers of the series. For the reader interested mainly in this application we recommend  to skip Sections~\ref{sec:surgery}--\ref{sdecomposition} and  pass to Section~\ref{sec:iso_theorem}.

\subsection{Connections to cyclotomic VW-algebras and Brauer algebras}
In part II of this series, \cite{ES2}, we develop in detail the graded versions of {level 2} cyclotomic Nazarov-Wenzl algebras and show their blocks are isomorphic to certain idempotent truncations of the algebras $\mathbb{D}_\Lambda$. In particular, they inherit a positive Koszul grading and a geometric interpretation in terms of perverse sheaves on isotropic Grassmannians. In the same paper together with \cite{ESKoszul}, we also relate them to non-semisimple versions of Brauer algebras \cite{CMD1}, \cite{CMD2} proving that Brauer algebras are Morita equivalent to idempotent truncations of sums of our algebras $\mathbb{D}_\Lambda$. 
This gives in particular a conceptual explanation of why {\it all} the Brauer algebras are determined by the combinatorics of Weyl groups of type $\rm D$, an amazing phenomenon first observed by Martin, Cox and DeVisscher, \cite{CoxdVi}. As an application we show in \cite{ESKoszul} that Brauer algebras for non-zero parameter $\delta$ are always Koszul. 

\subsection{Categorified coideal subalgebras}
In the same paper  \cite{ES2} we also study the action of translation functors in some detail. Here a completely new phenomenon appears. Besides their natural categorifications of Hecke algebras, translation functors were usually used in type $\rm A$ to categorify  actions of (super) quantum groups, see for instance \cite{BS3}, \cite{FKS}, \cite{Antonio} for specific examples in this context, or \cite{Mazcat} for an overview. In particular, they give standard examples of $2$-categorifications of Kac-Moody algebras in the sense of \cite{Rou2KM} and \cite{KL} and the combinatorics of the quantum group and the Hecke algebras are directly related, see e.g. \cite{FKK}, \cite{BS3} and \cite{Antonio}. This coincidence fails outside type $\rm A$. Instead we obtain the action of a quantized fixed point Lie algebra $\mathfrak{gl}(n)^\sigma$ for an involution $\sigma$. Although these algebras were studied in quite some detail by algebraic methods under the name {\it coideal subalgebras}, \cite{Letzter}, \cite{Kolb} it seems a geometric approach is  so far missing.  Our main Theorem~\ref{intro} will finally provide such a geometric approach and also gives, together with \cite{ES2},  an instance of a categorification of modules for these algebras. For these coideal subalgebras the graded versions of cyclotomic VW-algebras play the analogous role to cyclotomic {\it quiver Hecke algebra} or {\it Khovanov-Lauda algebras} attached to the universal enveloping algebra $\mathcal{U}(\mathfrak{g})$ in \cite{Rou2KM}, \cite{KL}. On the uncategorified level, the precise connection between these coideal subalgebras and the geometry of type $\rm D$ flag varieties was developed in \cite{Li}.

\subsection{Hermitian symmetric pairs}
Our focus on the type $\rm D$ case might look artificial. Instead one could consider all $G/P$ for classical pseudo hermitian symmetric cases, \cite{BHKostant}, \cite{BoeCollingwood}. In fact, the natural inclusions of algebraic groups $\op{SO}(2n+1,\mC)\hookrightarrow \op{SO}(2n+2,\mC)$ and
$\op{Sp}(2n,\mC)\hookrightarrow \op{SL}(2n,\mC)$ induce isomorphisms of the partial flag varieties for the pairs $(G,P)$ of type $({\rm B}_n, {\rm A}_{n-1})$ and $({\rm D}_{n+1}, {\rm A}_{n})$ respectively $({\rm C}_n,{\rm C}_{n-1})$ and $({\rm A}_{2n-1}, {\rm A}_{2n-2})$ which are compatible with the Schubert stratification, \cite[3.1]{BrionPolo}. Hence the corresponding categories of perverse sheaves are naturally isomorphic. 

Therefore, our algebras $\mathbb{D}_k$ together with the generalized Khovanov algebras from \cite{BS1} govern these cases as well,  see Section~ \ref{typeB}. Moreover it provides an easy closed formula for the type  $({\rm B}_n, {\rm A}_{n-1})$ Kazhdan-Lusztig polynomials, \eqref{B}. 
In particular, the non-simply laced case can be ``unfolded'' to the simply laced case for larger rank. For the remaining classical symmetric pairs, apart from type $({\rm D}_n,{\rm D}_{n-1})$ one might use the fact that the principal blocks of category $\cO$ for type ${\rm B}_n$ and ${\rm C}_n$ are equivalent via Soergel's combinatorial description, \cite{Sperv}. 
There is also a, in some sense, Langlands' dual picture of this unfolding relating different Springer fibres and Slodowy slices. To connect it with our algebras we want to stress that $\mathbb{D}_k$ has an alternative description as a convolution algebra using the cohomology of pairwise intersections of components of Springer fibres of type ${\rm D}_k$, see \cite{ES1}. In particular, the centre of our algebras  $\mathbb{D}_k$  can be identified with the cohomology ring of some type ${\rm D}$ Springer fibre, see \cite{BrundanSpringer}, \cite{Str09} for an analogous result in type $A$. For a detailed analysis of the corresponding cup diagram combinatorics and the folding procedure in terms of Springer theory we refer to \cite{Wilbert}.   

\subsection*{Acknowledgements} {We like to thank Jon Brundan, Ngau Lam, Antonio Sartori and Vera Serganova for many helpful discussions and Daniel Tubbenhauer for comments on the manuscript. We are in particular grateful to the referee for his/her extremely careful reading of the manuscript and for several helpful suggestions.}

\renc{\thetheorem}{\arabic{section}.\arabic{theorem}}
\section{The isotropic Grassmannian, weight and cup diagrams}
For the whole paper we fix a natural number $k$ and set $n=2k$. 
\subsection{The isotropic Grassmannian} \label{section21}
We start by recalling the isotropic Grassmannian and some of its combinatorics. 

We consider the group $\op{SL}(n,\mC)$ with standard Borel $B$ given by all upper triangular
matrices. Let $P_{k,k}$ be the standard parabolic subgroup given by all matrices $A=(a_{i,j})_{1\leq i,j\leq n}$ where $a_{i,j}=0$ for $i>k, j\leq
k$, and let $\mathcal{X}_{k,k}=\op{Gr}_k(\mC^n)=\op{SL}(n,\mC)/P_{k,k}$ be the corresponding {\it Grassmann variety} of $k$-dimensional subspaces in $\mC^n$.\\

Now fix on $\mC^n$ the non-degenerate quadratic form
$$Q(x_1, \ldots, x_{2k}) = x_1x_{n} + x_2x_{n-1} + \cdots + x_kx_{k+1}.$$
 Then the  space $\{V \in \mathcal{X}_{k,k} \mid Q_{|V} = 0\}$ has two connected components (namely the equivalence classes for the
relation $V\sim V'$ if $\dim{(V\cap V')}-k$ is even). The 
(isotropic) {\it Grassmannian of type ${\rm D}_k$}
is  the component  $Y_k$ containing the point $\{x_{k+1}=\cdots =x_n = 0\}$. The subgroup $\op{SO}(n,\mC)$ of
$\op{SL}(n,\mC)$ of transformations preserving $Q$ acts on $Y_k$. Moreover we have $$Y_k \cong\op{SO}(n,\mC)/(\op{SO}(n,\mC)\cap P_{k,k}).$$ The group $B_D=\op{SO}(n,\mC)\cap B$ is a Borel subgroup of $\op{SO}(n,\mC)$ and we fix the
stratification of $Y_k$ by $B_D$-orbits. Note that the latter are just the intersections $Y_\la=\Omega_\la\cap Y_k$ of the {\it Schubert cells}
$\Omega_\la$ of $\mathcal{X}_{k,k}$ with $Y_k$.
Fix the common labelling of Schubert cells of $\mathcal{X}_{k,k}$ by partitions whose Young diagrams fit into a $k\times k$-box (such that the dimension of the cell equals the number of boxes).\footnote{see e.g. \cite[9.4]{Fulton} with the notation $\tilde{\Omega}_\la$ there.} Then
$\Omega_\la\not=\emptyset$ precisely if its Young diagram is {\it symmetric} in the sense that it is fixed under reflection about the
diagonal and the number of boxes on the diagonal is even. Let $\Omega_k$ be the set of symmetric Young diagrams fitting into a $k\times k$-box.
For instance, the set $\Omega_4$ labelling the strata of $Y_4$ is the following.
\begin{eqnarray}
\label{Young}
{\tiny\Yvcentermath1\yng(4,4,4,4)\quad\quad\yng(4,4,2,2)\quad\quad\yng(4,3,2,1)\quad\quad\yng(4,2,1,1)\quad\quad\yng(3,3,2)\quad\quad\yng(3,2,1)
\quad\quad\yng(2,2)\quad\quad\emptyset}
\end{eqnarray}
\\
We encode $\la\in\Omega_k$ by an $\{\up,\down\}$-sequence $\mathbf{s}(\la)$ of length $k$ as follows: starting at the top right corner of the $k\times k$-square walk
along the right boundary of the diagram to the bottom left corner of the square - each step either downwards, encoded by a $\down$, or to the left, encoded
by an $\up$. We write the sequence from the right to the left. For instance, the leftmost and rightmost Young diagram in \eqref{Young} give rise to the sequences
\begin{eqnarray*}
\up\up\up\up\down\down\down\down&\text{respectively}&\down\down\down\down\up\up\up\up.
\end{eqnarray*}
If the diagram $\la$ labels a $B_D$-orbit $Y_\la$ then the sequence obtained in this way is automatically antisymmetric of length $2k$. In particular, it is enough to omit the second half and remember only the first (right) half which is then our $\mathbf{s}(\la)$. We denote by $\S=\{\mathbf{s}(\la)\mid \la\in\Omega_k\}$ the set of such sequences of length $k$.
%
For instance the diagrams \eqref{Young} correspond to the elements in $\mathbb{S}_4$ as follows: 
\small
\begin{equation}
\label{oje}
{\down\down\down\down}\;,\;{\up\up\down\down}\;,\;{\up\down\up\down}\;,\;{\down\up\up\down}\;,\;
{\up\down\down\up}\;,\;{\down\up\down\up}\;,\;{\down\down\up\up}\;,\;
{\up\up\up\up}
\end{equation}
\normalsize
We now put these $\mathbf{s}(\la)$ onto a general framework of diagrammatic weights and then connect them with the usual combinatorics of the Weyl group.

\subsection{Weights and linkage}
\label{linkage}
In the following we identify $\mZ_{> 0}$ with integral points, called {\it vertices}, on the positive numberline. A {\it (diagrammatic) weight} $\la$ is a diagram obtained by labelling each of the vertices by $\times$ (cross), $\circ$ (nought), $\down$ (down) or
$\up$ (up) such that the vertex $i$ is labelled $\circ$ for $i\gg 0$. For instance,
\begin{eqnarray}
\label{lla}
\la&=&\overset{1}{\makebox[1em]{$\down$}}
\overset{2}{\makebox[1em]{$\up$}}
\overset{3}{\makebox[1em]{$\up$}}
\overset{4}{\makebox[1em]{$\up$}}
\overset{5}{\makebox[1em]{$\up$}}
\overset{6}{\makebox[1em]{$\down$}}
\overset{7}{\makebox[1em]{$\down$}}
\overset{8}{\makebox[1em]{$\down$}}
\overset{9}{\makebox[1em]{$\down$}}
\overset{10}{\makebox[1em]{$\up$}}
\overset{\raisebox{1.25pt}{$\scriptstyle 11$}}{\makebox[1em]{$\circ$}}
\overset{\raisebox{1.25pt}{$\scriptstyle 12$}}{\makebox[1em]{$\circ$}} \, \cdots,\\
\label{lla2}
\la&=&{\makebox[1em]{$\up$}}
{\makebox[1em]{$\times$}}
{\makebox[1em]{$\times$}}
{\makebox[1em]{$\up$}}
{\makebox[1em]{$\circ$}}
{\makebox[1em]{$\down$}}
{\makebox[1em]{$\down$}}
{\makebox[1em]{$\circ$}}
{\makebox[1em]{$\up$}}
{\makebox[1em]{$\up$}}
{\makebox[1em]{$\circ$}}
{\makebox[1em]{$\circ$}} \, \cdots,
\end{eqnarray}
are examples of weights. The $\cdots$'s indicate that all further entries of the sequence equal $\circ$ and the label above the symbols indicates the positions. As in \eqref{lla2} we will usually omit these numbers. We denote by $\mX$ the set of diagrammatic weights. To $\lambda \in \mX$ we attach the sets
\begin{equation}
\label{Pquestion}
P_?(\lambda) = \{ i \in \mZ_{> 0} \mid \lambda_i = ? \}, \; \text{where   $? \in \{\up, \down, \times, \circ \}$ }.
\end{equation}

For the weights $\eqref{lla}$ respectively \eqref{lla2} these subsets of  $\mZ_{> 0}$ are
\small
\begin{eqnarray*}
&P_\up(\lambda)=\{2,3,4,5,10\}, P_\down(\lambda)=\{1,6,7,8,9\}, P_\times(\lambda) = \emptyset,  P_\circ(\lambda) = \{i\mid i\geq 11 \};&\\
&P_\up(\lambda)=\{1,4,9,10\}, P_\down(\lambda)=\{6,7\}, P_\times(\lambda) = \{2,3\},  P_\circ(\lambda) = \{5,8, i\mid i \geq 11 \}.&
\end{eqnarray*}
\normalsize
Two diagrammatic weights $\la,\mu\in\mX$ are {\it linked} if $\mu$ is obtained from $\la$ by finitely many combinations of {\it basic linkage moves} of swapping neighboured labels $\up$ and $\down$ (ignoring symbols $\times$ and $\circ$) or swapping the pairs  $\up\up$ and $\down\down$ at the first two positions (again ignoring symbols $\times$ and $\circ$) in the sequence. The equivalence classes of the linkage relation are called {\it blocks}.

\begin{ex}
The weight $\la$ from \eqref{lla}
is linked to
$\mu={\down}\down\up\down\up\down\down\down\down\up\circ\circ \, \cdots$
but not to
$\mu'={\down}\down\up\up\up\down\down\down\down\up\circ\circ \, \cdots,$ since the parities of $\up$'s don't agree.
\end{ex}

A block $\Lambda$ is hence given by fixing the positions of all $\times$ and $\circ$ and the parity $\ov{|P_\up(\lambda)|}$ of the number of $\up$'s. Formally a block $\Lambda$ is the data of a parity, either $\ov{0}$ or $\ov{1}$, and a \emph{block diagram} $\Theta$, that is a sequence $\Theta_i$ of symbols $\times$, $\circ$, $\bu$, indexed by $i\in\mZ_{>0}$ with the same conditions as for diagrammatic weights, i.e. all but finitely many $\Theta_i$ are equal to $\circ$. Let $P_?(\Theta)$ for $?\in\{\bu,\times,\circ\}$ be defined analogously to \eqref{Pquestion}. Then 
the block $\Lambda$ is the equivalence class of weights
\begin{eqnarray}
\label{defblock}
\Lambda=\Lambda_\Theta^{\ov{\epsilon}} = \left\lbrace \begin{array}{l|c}
&\; P_\up(\lambda) \cup P_\down(\lambda) = P_\bu(\Theta),\\
\lambda \in \mX \;\; &\; P_\times(\lambda) = P_\times(\Theta)\text{, } P_\circ(\lambda) = P_\circ(\Theta),\\
&\; \ov{|P_\up(\lambda)|} = \ov{\epsilon}.
\end{array} \right\rbrace
\end{eqnarray}
For a block  $\Lambda$ with its block diagram $\Theta$ define $P_\times(\Lambda)=P_\times(\Theta)$, $P_\circ(\Lambda)=P_\circ(\Theta)$, and $P_\bu(\Lambda)=P_\bu(\Theta)$.

 A block $\Lambda$ is called {\it regular} if $P_\times(\Lambda) = \emptyset$ and {\it principal (of type ${\rm D}_k$)} if additionally $P_\bu(\Lambda) = \{1,\ldots,k\}$. A weight $\lambda$ is called {\it regular} if it belongs to a regular block. For instance $\eqref{lla}$ is a principal weight, and thus also a regular weight. For fixed $k$ the set of principal weights of type ${\rm D}_k$ decomposes into exactly two blocks, the {\it even} block $\Lambda_k^{\ov{0}}$, where each weight has an even number of $\up$'s, and the {\it odd} block $\Lambda_k^{\ov{1}}$, where each weight has an odd number of $\up$'s. Both blocks correspond to the same block diagram
 \begin{eqnarray*}
\Theta &=& \underset{k}{\underbrace{{\bu}\, \cdots \, \bu}}\circ\circ \, \cdots.
\end{eqnarray*}
From now on we will identify the set  $\S=\{\mathbf{s}(\la)\mid \la\in\Omega_k\}$  with a subset of principal weights by adding $\circ$'s at the vertices $i>k$. Then the  following example shows the two regular blocks for $k=4$.

\begin{ex}
\label{exn4}
For $k=4$ we get the following set of even regular weights with the set  $\op{C}_k^{\ov{0}}$ of cup diagrams
associated via Definition~\ref{decoratedcups} below,
\begin{eqnarray} \label{sequ1}
\begin{tikzpicture}[thick,scale=0.5]
\begin{scope}[xshift=-3cm]
\node at (.8,.8) {\small $\down\down\down\down$};
\draw (0,0) -- +(0,-.8);
\draw (.5,0) -- +(0,-.8);
\draw (1,0) -- +(0,-.8);
\draw (1.5,0) -- +(0,-.8);

\node at (3.8,.8) {\small $\up\up\down\down$};
\draw (3,0) .. controls +(0,-.5) and +(0,-.5) .. +(.5,0);
\fill (3.25,-.35) circle(2.5pt);
\draw (4,0) -- +(0,-.8);
\draw (4.5,0) -- +(0,-.8);

\node at (6.8,.8) {\small $\up\down\up\down$};
\draw (6,0) -- +(0,-.8);
\fill (6,-.4) circle(2.5pt);
\draw (6.5,0) .. controls +(0,-.5) and +(0,-.5) .. +(.5,0);
\draw (7.5,0) -- +(0,-.8);

\node at (9.8,.8) {\small $\down\up\up\down$};
\draw (9,0) .. controls +(0,-.5) and +(0,-.5) .. +(.5,0);
\draw (10,0) -- +(0,-.8);
\fill (10,-.4) circle(2.5pt);
\draw (10.5,0) -- +(0,-.8);
\end{scope}

\begin{scope}[xshift=9cm, yshift=2cm]
\node at (.8,-1.2) {\small $\up\down\down\up$};
\draw (0,-2) -- +(0,-.8);
\fill (0,-2.4) circle(2.5pt);
\draw (0.5,-2) -- +(0,-.8);
\draw (1,-2) .. controls +(0,-.5) and +(0,-.5) .. +(.5,0);

\node at (3.8,-1.2) {\small $\down\up\down\up$};
\draw (3,-2) .. controls +(0,-.5) and +(0,-.5) .. +(.5,0);
\draw (4,-2) .. controls +(0,-.5) and +(0,-.5) .. +(.5,0);

\node at (6.8,-1.2) {\small $\down\down\up\up$};
\draw (6,-2) .. controls +(0,-1) and +(0,-1) .. +(1.5,0);
\draw (6.5,-2) .. controls +(0,-.5) and +(0,-.5) .. +(.5,0);

\node at (9.8,-1.2) {\small $\up\up\up\up$};
\draw (9,-2) .. controls +(0,-.5) and +(0,-.5) .. +(.5,0);
\fill (9.25,-2.35) circle(2.5pt);
\draw (10,-2) .. controls +(0,-.5) and +(0,-.5) .. +(.5,0);
\fill (10.25,-2.35) circle(2.5pt);
\end{scope}
\end{tikzpicture}
\end{eqnarray}
and odd regular weights with the set  $\op{C}_k^{\ov{1}}$ of corresponding cup diagrams:
\begin{eqnarray} \label{sequ2}
\begin{tikzpicture}[thick,scale=.5]
\begin{scope}[xshift=-3cm]
\node at (.8,.8) {\small $\up\down\down\down$};
\draw (0,0) -- +(0,-.8);
\fill (0,-.4) circle(2.5pt);
\draw (.5,0) -- +(0,-.8);
\draw (1,0) -- +(0,-.8);
\draw (1.5,0) -- +(0,-.8);

\node at (3.8,.8) {\small $\down\up\down\down$};
\draw (3,0) .. controls +(0,-.5) and +(0,-.5) .. +(.5,0);
\draw (4,0) -- +(0,-.8);
\draw (4.5,0) -- +(0,-.8);

\node at (6.8,.8) {\small $\down\down\up\down$};
\draw (6,0) -- +(0,-.8);
\draw (6.5,0) .. controls +(0,-.5) and +(0,-.5) .. +(.5,0);
\draw (7.5,0) -- +(0,-.8);

\node at (9.8,.8) {\small $\up\up\up\down$};
\draw (9,0) .. controls +(0,-.5) and +(0,-.5) .. +(.5,0);
\fill (9.25,-0.35) circle(2.5pt);
\draw (10,0) -- +(0,-.8);
\fill (10,-.4) circle(2.5pt);
\draw (10.5,0) -- +(0,-.8);
\end{scope}

\begin{scope}[yshift=2cm, xshift=9cm]
\node at (.8,-1.2) {\small $\down\down\down\up$};
\draw (0,-2) -- +(0,-.8);
\draw (0.5,-2) -- +(0,-.8);
\draw (1,-2) .. controls +(0,-.5) and +(0,-.5) .. +(.5,0);

\node at (3.8,-1.2) {\small $\up\up\down\up$};
\draw (3,-2) .. controls +(0,-.5) and +(0,-.5) .. +(.5,0);
\fill (3.25,-2.35) circle(2.5pt);
\draw (4,-2) .. controls +(0,-.5) and +(0,-.5) .. +(.5,0);

\node at (6.8,-1.2) {\small $\up\down\up\up$};
\draw (6,-2) .. controls +(0,-1) and +(0,-1) .. +(1.5,0);
\draw (6.5,-2) .. controls +(0,-.5) and +(0,-.5) .. +(.5,0);
\fill (6.75,-2.75) circle(2.5pt);

\node at (9.8,-1.2) {\small $\down\up\up\up$};
\draw (9,-2) .. controls +(0,-.5) and +(0,-.5) .. +(.5,0);
\draw (10,-2) .. controls +(0,-.5) and +(0,-.5) .. +(.5,0);
\fill (10.25,-2.35) circle(2.5pt);
\end{scope}
\end{tikzpicture}
\end{eqnarray}
\end{ex}

Note, $\op{C}_k^{\ov{p}}$ is precisely the set of diagrams, where the number of dotted rays plus undotted cups is even respectively odd, depending if $p=0$ or $p=1$.

\begin{lemma}
\label{newlemma}
 The map $\la\mapsto \mathbf{s}(\la)$ defines a bijection between  $\Omega_k$ and  $\Lambda_k^{\ov{k}}$.
\end{lemma}
\begin{proof}
This follows directly from the definitions.
\end{proof}

The basic linkage moves induce a partial ordering on each block, the {\it Bruhat order}, by declaring that changing a pair of labels $\up\down$ to $\down\up$ or a pair
$\down\down$ to $\up\up$, ignoring any $\times$ or $\circ$ in between, makes a weight smaller in this Bruhat order. Repeatedly applying the basic moves implies:

\begin{lemma}
\label{lem:Bruhat}
Changing a (not necessarily neighboured) pair of labels $\up\down$ to $\down\up$ or a pair
$\down\down$ to $\up\up$ makes a weight smaller in the Bruhat order.
\end{lemma}

If we denote the weights from \eqref{oje} by $\la_1, \ldots, \la_8$ then $\la_1>\la_2>\la_3>\la_4,\la_5>\la_6>\la_7>\la_8$ in this Bruhat order.

\subsection{The Weyl group of type ${\rm D}_k$}
\label{typeD}
To make the connection with the Bruhat order on Coxeter groups, let $W=W(k)\cong (\mZ/2\mZ)^{k-1} \rtimes S_k$ be the Weyl group of type ${\rm D}_k$. It is generated by simple reflections $s_i$,  $0\leq i\leq k-1$ subject to the relations $s_i^2=e$ for
all $i$, and for $i,j\not=0$  $s_is_j=s_js_i$ if $|i-j|>1$ and $s_is_js_i=s_js_is_j$ if $|i-j|=1$, and additionally $s_0s_2s_0=s_2s_0s_2$
and $s_0s_j=s_js_0$ for $j\not=2$.

 It has two parabolic subgroups isomorphic to the symmetric group $S_k$, namely $W_{\ov{0}}$ generated by $s_i$, $i\not=0$ and $W_{\ov{1}}$ generated
 by $s_i$, $i\not=1$. Let $W^{\ov 0}$ and $W^{\ov 1}$ be the set of shortest coset representatives for $W_{\ov 0}\backslash W$ and $W_{\ov 1}\backslash W$
 respectively.

Let $p\in\{0,1\}$. The group $W$ acts from the right on the sets $\Lambda_k^{\ov{p}}$ as follows: $s_i$ for $i>1$ swaps the $i$-th and $(i+1)$-st label (from the left), and 
$s_1$ swaps the first two in case they are $\up\down$ or $\down\up$ and is the identity otherwise, whereas $s_0$ turns $\up\up$ into $\down\down$ and vice versa at the first two positions and is the identity otherwise. The sequence consisting of $k$ $\down$'s is
fixed by the parabolic $W_{\ov 0}$, whereas the sequence consisting of one $\up$ followed by $k-1$ $\down$'s is fixed by $W_{\ov 1}$. Sending the identity element $e\in W$ to one of these
sequences $s$ and then $w$ to $s.w$  defines canonical bijections $\Phi^{\ov p}: W^{\ov p}\longrightarrow\Lambda_k^{\ov{p}}$.  The elements from $W^{\ov p}$ corresponding to the weights in \eqref{sequ1} and \eqref{sequ2} are respectively
\begin{eqnarray}
{e}\quad {s_0}\quad {s_0s_2}\quad {s_0s_2s_1}\quad {s_0s_2s_3}\quad {s_0s_2s_3s_1}\quad {s_0s_2s_1s_3s_2}\quad {s_0s_2s_1s_3s_2s_0}\\
{e}\quad {s_1}\quad {s_1s_2}\quad {s_1s_2s_0}\quad {s_1s_2s_3}\quad {s_1s_2s_3s_0}\quad {s_1s_2s_3s_0s_2}\quad {s_1s_2s_3s_0s_2s_1}
\end{eqnarray}
More generally, blocks $\Lambda$ with $|P_\bu(\Lambda)|=k$ can be identified with $W$-orbits, where $W$ acts on the vertices not labelled  $\times$ or $\circ$. 

Consider the set $\Omega_k$ as poset equipped with the partial order given by containment of the corresponding Young diagrams, with the smaller diagram corresponding to the smaller element, and the set $\Lambda_k^{\ov{k}}$ equipped with the Bruhat order and the sets $W^{\ov{p}}$ with the reversed Bruhat order. Then the definitions directly imply:

\begin{lemma}
The assignment $\la\mapsto \mathbf{s} (\la)$ and the isomorphism $\Phi^{\overline{k}}$ define canonical bijections of posets between  $\Omega_k$ and $\Lambda_k^{\ov{k}}$, and $W^{\ov  k}$.
\end{lemma}

\begin{remark}
{\rm Note that flipping the first symbol in $ \mathbf{s}(\la)$ composed with the map from Lemma~\ref{newlemma}
would alternatively give us an isomorphism of posets between $\Omega_k$ and $\Lambda_k^{\ov{k+1}}$ and then also $W^{\ov  {k+1}}$. One should think of the second choice as in fact passing to the second natural choice of parabolic of Dynkin type ${\rm A}_{k-1}$ in the complex semisimple Lie algebra $\mg$ of type ${\rm D}_k$. It corresponds to different, but {\it canonically equivalent} parabolic category $\cO(\mg)$'s. For the main results of the paper it is therefore not relevant which parity we choose. The interplay of these two different equivalent parabolic categories with other Zuckerman functors plays an important role in \cite{ES2}. There,  the above equivalences are used to categorify a {\it skew Howe duality for quantum symmetric pairs} using special cases of {\it Letzter's coideal subalgebras} \cite{Letzter}. In that setup this canonical equivalence corresponds (after applying  Koszul duality) to the extra generator $B_0$ in the coideal subalgebra defined there. The grading of our algebras $\mathbb{D}_k$ is hereby crucial, since a shift in the grading up by one corresponds to the action of the parameter $q$ from the coideal subalgebra.}
\end{remark}

\section{Cup diagrams, $\la$-pairs  and the quiver}
The goal of this section is to introduce the required combinatorics to define the type $\rm D$ generalized Khovanov algebra and describe the graph of the quiver (Definition \ref{Extquivdef} and Corollary~\ref{CorCartan}).
Throughout this section we fix a block $\Lambda$ with its block diagram $\Theta$. We abbreviate
$$P_\times = P_\times(\Lambda), \qquad P_\circ = P_\circ(\Lambda),\qquad  P_\bu = P_\bu(\Lambda).$$
\subsection{Cup diagrams}
Consider the semi-infinite strip
\begin{eqnarray*}
R = \left\{(x,y) \in \mR_{\geq 0} \times \mR \, \middle| \, \frac{1}{2} \geq y \geq -\frac{1}{2}\right \} \subset \mR_{\geq 0} \times \mR
\end{eqnarray*}
as the union of the two smaller strips $R^- = \{ (x,y) \in R \mid 0 \geq y \geq - \frac{1}{2}\}$ and $R^+ = \{ (x,y) \in R \mid  \frac{1}{2} \geq y \geq 0 \}$.
For $? \in \{ \times , \circ, \bu\}$  consider  $P_?\subset R$ by sending $x\in P_? \subset\mZ_{>0}$ to $(x,0)$. 

Later on we will need the following statistic for $i \in P_\bu$
\begin{eqnarray}
\label{pso}
\pos(i)=\#\{ j \in P_\bu \mid j \leq i\} \in \mZ_{>0}.
\end{eqnarray}
Note, for principal blocks $\pos(i)=i$, the position on the number line.

\begin{definition}
\label{defarcs}
A subset $\gamma \subset R^-$ is called an \emph{arc} if there exists an injective continuous map $\tilde{\gamma}:[0,1] \rightarrow R$ with image $\gamma$ such that
\begin{itemize}
\item[$\blacktriangleright$] $\widetilde{\gamma}(0) \in P_\bu$,
\item[$\blacktriangleright$] $\widetilde{\gamma}(1) \in P_\bu$ ({\it cup condition}) or $\widetilde{\gamma}(1) \in \mR_{\geq 0} \times \{- \frac{1}{2}\}$  ({\it ray condition}),
\item[$\blacktriangleright$] $\widetilde{\gamma}( (0,1) ) \subset (R^-)^\circ$, the interior of $R^-$.
\end{itemize}
It is called a \emph{cup} if the cup condition is satisfied and is called a \emph{ray} otherwise.
\end{definition}

\begin{definition} \label{cup_diagram}
An \emph{undecorated cup diagram} $c$ (for the block diagram $\Theta$) is a finite union of arcs $\{ \gamma_1,\ldots,\gamma_r \}$, such that
\begin{itemize}
\item[$\blacktriangleright$] $\gamma_i \cap \gamma_j = \emptyset$ for $i \neq j$,
\item[$\blacktriangleright$] for every $(p,0) \in P_\bu$ there exists an arc $\gamma_{i(p)}$ such that $(p,0) \in \gamma_{i(p)}$.
\end{itemize}
\end{definition}

In the following we will always consider two diagrams as the same if they differ only by an isotopy of $R$ fixing $\mR_{\geq 0} \times \{0\}$ pointwise and $\mR_{\geq 0} \times \{-\frac{1}{2}\}$ as a set. Hence we usually draw cup diagrams by drawing cups as half circles and rays as straight lines.

We now introduce the notion of a decorated cup diagram, a generalisation of the setup from \cite{BS1}:

\begin{definition}
Let $c$ be an undecorated cup diagram with set of arcs $\{\gamma_1,\ldots, \gamma_r\}$. Then a \emph{decoration} of $c$ is a map
$${\rm deco}_c:\quad\{\gamma_1,\ldots, \gamma_r\} \rightarrow \{0,1\}.$$
If the value for $\gamma_i$ is $1$ we call the arc \emph{dotted}, otherwise we call it \emph{undotted}.

An undecorated cup diagram together with a decoration is called a \emph{decorated cup diagram}. We visualize this by putting a ``$\bullet$'' at an arbitrary interior point of the arc in the underlying undecorated cup diagram.
\end{definition}

\begin{ex}
\label{excup}
The following shows two examples of possible decorations for an undecorated cup diagram with the block diagram listed on the top.

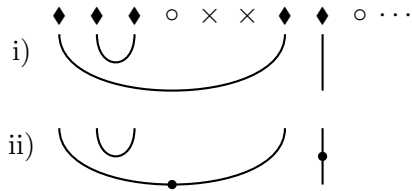
\begin{eqnarray*}
\begin{tikzpicture}[thick,scale=0.50]
\node at (0,.5) {$\bu$};
\node at (1,.5) {$\bu$};
\node at (2,.5) {$\bu$};
\node at (3,.5) {$\circ$};
\node at (4,.5) {$\times$};
\node at (5,.5) {$\times$};
\node at (6,.5) {$\bu$};
\node at (7,.5) {$\bu$};
\node at (8,.5) {$\circ$};
\node at (9,.5) {$\cdots$};

\node at (-1,-.5) {i)};
\draw (0,0) .. controls +(0,-2) and +(0,-2) .. +(6,0);
\draw (1,0) .. controls +(0,-1) and +(0,-1) .. +(1,0);
\draw (7,0) -- (7,-1.5);

\begin{scope}[yshift=-2.5cm]
\node at (-1,-.5) {ii)};
\draw (0,0) .. controls +(0,-2) and +(0,-2) .. +(6,0);
\fill (3,-1.5) circle(3.5pt);
\draw (1,0) .. controls +(0,-1) and +(0,-1) .. +(1,0);
\draw (7,0) -- (7,-1.5);
\fill (7,-.75) circle(3.5pt);
\end{scope}

\end{tikzpicture}
\end{eqnarray*}
\end{ex}

\begin{definition}
\label{def:decorated}
A decorated cup diagram $c$ with set of arcs $\{\gamma_1,\ldots,\gamma_r\}$ and decoration ${\rm deco}_c$ is called \emph{admissible} if
${\rm deco}_c(\gamma_i)=1$ implies that an arbitrary point in the interior of $\gamma_i$ can be connected with the left boundary of $R$ by a path not intersecting any other arc \footnote{This condition arises also in the context of blob algebras as studied for instance in \cite{Martinblob}, and of generalized Temperley-Lieb algebras \cite{Green}. It is derived naturally from the theory in \cite{BS1} by a `folding' procedure, see \cite[5.2]{LS}.}.

For instance putting a dot on the interior cup in i) or ii) in Example~\ref{excup} would create decorated cup diagrams which are not admissible.

In the following we will refer to an admissible decorated cup diagram simply as a \textit{cup diagram}.
\end{definition}

\begin{remark}
\label{decorated}
By definition, the following does not occur in a cup diagram:
\begin{itemize}
\item[$\blacktriangleright$] a dotted cup nested inside another cup, or
\item[$\blacktriangleright$] a dotted cup or dotted ray to the right of a (dotted or undotted) ray.
\end{itemize}
\end{remark}

\begin{definition}
\label{decoratedcups}
For a weight $\la \in \Lambda$, the associated \emph{decorated cup diagram}, denoted $\underline \la$ is obtained by the following procedure.
\begin{enumerate}[(Cup 1)]
\item First connect neighboured vertices in $P_\bu$ labelled $\down\up$ successively by a cup (ignoring already joint vertices and vertices not in $P_\bu$) as long as possible. (Note that the result is independent of the order the connections are made).
\item Attach to each remaining $\down$ a vertical ray.
\item Connect from left to right pairs of two neighboured $\up$'s by cups.
\item If a single $\up$ remains, attach a vertical ray.
\item Finally define the decoration ${\rm deco}_c$ such that all arcs created in step (Cup 1) and (Cup 2) are undotted whereas the ones created in steps (Cup 3) and (Cup 4) are dotted.
\end{enumerate}
Note that, by construction, this is an admissible decorated cup diagram.
\end{definition}

We denote by $\mathcal{C}_\Lambda$ the set of cup diagrams associated with a block $\Lambda$. We will see in Lemma~\ref{lem:orient} that each such cup diagram corresponds to exactly one weight in $\Lambda$; for instance the cup diagram i) from Example~\ref{excup} corresponds to  $\down\down\up\circ\times\times\up\down\circ\cdots$, and the diagram ii) to  $\up\down\up\circ\times\times\up\up\circ\cdots$.

\subsection{$\la$-pairs and the arrows in the quiver}
Recall that we want to describe the quiver $\mathcal{Q}_k$ of the category $\op{Perv}_k$ of perverse sheaves on isotropic Grassmannians. Its vertices are labelled by Schubert varieties, respectively their Young diagrams $\la$ or associated weight $\mathbf{s}(\la)$, see Section~\ref{section21}. Hence we can also label them by the cup diagrams $\underline{\mathbf{s}(\la)}$ which we denote, by abuse of notation, by $\underline{\la}$ as well. The arrows attached to a vertex $\la$ will correspond to so-called $\la$-pairs, a notion introduced in \cite[1.6]{Braden} which we now adapt to our setting.

\begin{definition}
 \label{lapair}
Fix a block $\Lambda$. Let $\la \in \Lambda$ be a weight and $\gamma$ a cup in $\underline{\la}$ connecting the vertices $l_\gamma < r_\gamma$ in $P_\bu$.
With $\gamma$ we associate a pair of integers $(\alpha,\beta)\in\mathbb{Z} \times \mathbb{N}$ called a $\la$-\emph{pair} defined as
\begin{eqnarray}
 \label{lapair2}
(\alpha,\beta)&=&
\begin{cases}
(\pos(l_\gamma),\pos(r_\gamma))&\text{if $\gamma$ is undotted},\\
(-\pos(l_\gamma),\pos(r_\gamma))&\text{if $\gamma$ is dotted}.
\end{cases}
\end{eqnarray}
Given a $\la$-pair let $\la'$ be the weight obtained by changing $\down$ into $\up$ and $\up$ into $\down$ at the entries $l_\gamma$ and $r_\gamma$ of $\lambda$. In this case we write $\la\stackrel{(\alpha,\beta)}{\longrightarrow}\la'$ or short $\la\rightarrow\la'$ or equivalently $\la'\leftarrow \la$. Note that this implies $\la'>\la$. We abbreviate $\la\leftrightarrow\mu$ in case $\la\rightarrow\mu$ or $\la\leftarrow\mu$.
\end{definition}

The relation $\mu\leftarrow \la$ has a nice interpretation in terms of the associated cup diagrams $\underline\la, \underline\mu$: 

\begin{lemma}
\label{lapairscup}
Given weights $\la,\mu$. We have $\mu\leftarrow \la$ if and only if the corresponding cup diagrams $\underline\la$ and $\underline\mu$ differ precisely by one of the following local moves  $\underline\mu\leftarrow\underline\la$.

\begin{equation} \label{eqn:lambdapairlist}
\begin{tikzpicture}[thick,scale=0.65]
\draw (0,0) node[above]{$\down$} .. controls +(0,-.5) and +(0,-.5) .. +(.5,0) node[above]{$\up$};
\draw (1,0) node[above]{$\down$} .. controls +(0,-.5) and +(0,-.5) .. +(.5,0) node[above]{$\up$};
\draw[<-] (1.75,-.2) -- +(1,0);
\draw (3,0) node[above]{$\down$} .. controls +(0,-1) and +(0,-1) .. +(1.5,0) node[above]{$\up$};
\draw (3.5,0) node[above]{$\down$} .. controls +(0,-.5) and +(0,-.5) .. +(.5,0) node[above]{$\up$};

\begin{scope}[xshift=12cm]
\draw (0,0) node[above]{$\up$} .. controls +(0,-.5) and +(0,-.5) .. +(.5,0) node[above]{$\up$};
\fill (.25,-.36) circle(2.5pt);
\draw (1,0) node[above]{$\down$} .. controls +(0,-.5) and +(0,-.5) .. +(.5,0) node[above]{$\up$};
\draw[<-] (1.75,-.2) -- +(1,0);
\draw (3,0) node[above]{$\up$} .. controls +(0,-1) and +(0,-1) .. +(1.5,0) node[above]{$\up$};
\fill (3.75,-.74) circle(2.5pt);
\draw (3.5,0) node[above]{$\down$} .. controls +(0,-.5) and +(0,-.5) .. +(.5,0) node[above]{$\up$};
\end{scope}

\begin{scope}[xshift=6cm]
\draw (0,0) node[above]{$\up$} .. controls +(0,-1) and +(0,-1) .. +(1.5,0) node[above]{$\up$};
\draw (.5,0) node[above]{$\down$} .. controls +(0,-.5) and +(0,-.5) .. +(.5,0) node[above]{$\up$};
\fill (0.75,-.74) circle(2.5pt);
\draw[<-] (1.75,-.2) -- +(1,0);
\draw (3,0) node[above]{$\down$} .. controls +(0,-.5) and +(0,-.5) .. +(.5,0) node[above]{$\up$};
\draw (4,0) node[above]{$\up$} .. controls +(0,-.5) and +(0,-.5) .. +(.5,0) node[above]{$\up$};
\fill (4.25,-.36) circle(2.5pt);
\end{scope}

\begin{scope}[yshift=-2cm,xshift=6cm]
\draw (0,0) node[above]{$\down$} .. controls +(0,-1) and +(0,-1) .. +(1.5,0) node[above]{$\up$};
\draw (.5,0) node[above]{$\down$} .. controls +(0,-.5) and +(0,-.5) .. +(.5,0) node[above]{$\up$};
\draw[<-] (1.75,-.2) -- +(1,0);
\draw (3,0) node[above]{$\up$} .. controls +(0,-.5) and +(0,-.5) .. +(.5,0) node[above]{$\up$};
\fill (3.25,-.36) circle(2.5pt);
\draw (4,0) node[above]{$\up$} .. controls +(0,-.5) and +(0,-.5) .. +(.5,0) node[above]{$\up$};
\fill (4.25,-.36) circle(2.5pt);
\end{scope}

\begin{scope}[yshift=-4cm,xshift=-1cm]
\draw (0,0) node[above]{$\up$} .. controls +(0,-.5) and +(0,-.5) .. +(.5,0) node[above]{$\up$};
\fill (.25,-.36) circle(2.5pt);
\draw (1,0) node[above]{$\down$} -- +(0,-.8);
\draw[<-] (1.25,-.2) -- +(1,0);
\draw (2.5,0) node[above]{$\up$} -- +(0,-.8);
\fill (2.5,-.4) circle(2.5pt);
\draw (3,0) node[above]{$\down$} .. controls +(0,-.5) and +(0,-.5) .. +(.5,0) node[above]{$\up$};
\end{scope}

\begin{scope}[yshift=-4cm,xshift=4cm]
\draw (0,0) node[above]{$\up$} -- +(0,-.8);
\fill (0,-.4) circle(2.5pt);
\draw (.5,0) node[above]{$\down$} .. controls +(0,-.5) and +(0,-.5) .. +(.5,0) node[above]{$\up$};
\draw[<-] (1.25,-.2) -- +(1,0);
\draw (2.5,0) node[above]{$\down$} .. controls +(0,-.5) and +(0,-.5) .. +(.5,0) node[above]{$\up$};
\draw (3.5,0) node[above]{$\up$} -- +(0,-.8);
\fill (3.5,-.4) circle(2.5pt);
\end{scope}

\begin{scope}[yshift=-4cm,xshift=9cm]
\draw (0,0) node[above]{$\down$} .. controls +(0,-.5) and +(0,-.5) .. +(.5,0) node[above]{$\up$};
\draw (1,0) node[above]{$\down$} -- +(0,-.8);
\draw[<-] (1.25,-.2) -- +(1,0);
\draw (2.5,0) node[above]{$\down$} -- +(0,-.8);
\draw (3,0) node[above]{$\down$} .. controls +(0,-.5) and +(0,-.5) .. +(.5,0) node[above]{$\up$};
\end{scope}

\begin{scope}[yshift=-4cm,xshift=14cm]
\draw (0,0) node[above]{$\down$} -- +(0,-.8);
\draw (.5,0) node[above]{$\down$} .. controls +(0,-.5) and +(0,-.5) .. +(.5,0) node[above]{$\up$};
\draw[<-] (1.25,-.2) -- +(1,0);
\draw (2.5,0) node[above]{$\up$} .. controls +(0,-.5) and +(0,-.5) .. +(.5,0) node[above]{$\up$};
\fill (2.75,-.36) circle(2.5pt);
\draw (3.5,0) node[above]{$\up$} -- +(0,-.8);
\fill (3.5,-.4) circle(2.5pt);
\end{scope}

\begin{scope}[yshift=-6cm, xshift=4.5cm]
\draw (0,0) node[above]{$\up$} -- +(0,-.8);
\fill (0,-.4) circle(2.5pt);
\draw (.5,0) node[above]{$\down$} -- +(0,-.8);
\draw[<-] (.75,-.2) -- +(1,0);
\draw (2,0) node[above]{$\down$} .. controls +(0,-.5) and +(0,-.5) .. +(.5,0) node[above]{$\up$};
\end{scope}

\begin{scope}[yshift=-6cm, xshift=9.5cm]
\draw (0,0) node[above]{$\down$} -- +(0,-.8);
\draw (.5,0) node[above]{$\down$} -- +(0,-.8);
\draw[<-] (.75,-.2) -- +(1,0);
\draw (2,0) node[above]{$\up$} .. controls +(0,-.5) and +(0,-.5) .. +(.5,0) node[above]{$\up$};
\fill (2.25,-.36) circle(2.5pt);
\end{scope}
\end{tikzpicture}
\end{equation}
In particular $\la$ and $\mu$ are from the same block.
\end{lemma}
\begin{proof}
This follows directly from the definitions.
\end{proof}

\begin{definition}
\label{Extquivdef}
The \emph{quiver associated with a block $\Lambda$} is the oriented graph $\mathcal{Q}(\Lambda)$ with vertex set the weights $\la\in\Lambda$ and arrows $\la\rightarrow \mu$ whenever $\la\leftrightarrow\mu$ in the sense of Definition~\ref{lapair}. (For the  example of a principal block for $k=4$ we refer to \eqref{fig:quiv}, where the arrows labelled $b$ correspond to the $\la$-pairs.)
\end{definition}

\begin{definition}
An \emph{undecorated cap diagram} (respectively \emph{decorated cap diagram}) $c'$ is obtained by reflecting an undecorated (respectively decorated) cup diagram $c$ along the horizontal line $\mR_{\geq 0} \times \{0\}$ inside $R$. The image of a cup in such a diagram is called a \emph{cap}. It is \emph{dotted} if the cup was dotted.

A decorated cap diagram is called \emph{admissible} if $c$ is admissible. In which case we just call it a {\it cap diagram}.
\end{definition}

We will denote the reflection at the horizontal line by ${}^*$ and use the notation $c^*$ to denote the cap diagram that we obtain from a given cup diagram $c$. In particular, we obtain a cap diagram associated with $\la$, denoted $\ov{\la} = \underline{\la}^*$ for any $\la\in\Lambda$. 

\begin{definition}
A \emph{decorated circle diagram} $d=c c'$ is the union of the arcs contained in a decorated cup diagram $c$ and the ones in a decorated cap diagram $c'$, i.e., visually we stack the cap diagram $c'$ on top of the cup diagram $c$.
The decorations of $c$ and $c'$ give the decoration for $d$.

It is called \emph{admissible} if both $c$ and $c'$ are admissible. Again we refer to admissible decorated circle diagrams simply as circle diagrams.
\end{definition}

A \emph{ connected component} of a decorated circle diagram $d$ with set of arcs $\{\gamma_1,\ldots,\gamma_r\}$ is a connected component of the subset $\bigcup_{1 \leq i \leq r} \gamma_i \subset R$. It is always a circle or a line.

\begin{definition}
A \emph{line} in a decorated circle diagram $d$ is a connected component of $d$ that intersects the boundary of the strip $R$. A \emph{circle} in a decorated circle diagram $d$ is a connected component of $d$ that does not intersect the boundary of the strip $R$. A line is \emph{propagating}, if one endpoint is at the top and the other at the bottom of the diagram.
\end{definition}

\subsection{Orientations, gradings and the space $\D$}

An \emph{oriented cup diagram}, denoted $c\mu$ is obtained from a cup diagram $c \in \mathcal{C}_\Lambda$ by putting a weight $\mu \in \Lambda$ on top of $c$ such that all arcs of $c$ are `oriented' in one of the ways displayed in \eqref{oriented}, and additionally we require that the vertices in $\mu$ labelled $\circ$ agree with $P_\circ$ and the vertices labelled $\times$ agree with $P_\times$.

Similarly, an \emph{oriented cap diagram} $\mu c'$ is obtained by putting such a weight $\mu$ below a cap diagram with $k$ vertices such that all arcs of $c'$ become `oriented' in one of the following ways:
\begin{eqnarray}
\label{oriented}
\begin{tikzpicture}[thick,>=angle 60,scale=0.8]
\node at (-1,-1.2) {degree};
\draw [>->] (0,0) .. controls +(0,-1) and +(0,-1) .. +(1,0) node at
+(0.5,-1.2) {0};
\draw [<-<] (2,0) .. controls +(0,-1) and +(0,-1) .. +(1,0) node at
+(0.5,-1.2) {1};
\draw [<-<] (4,-0.7) .. controls +(0,1) and +(0,1) .. +(1,0) node at
+(0.5,-0.5) {0};
\draw [>->] (6,-0.7) .. controls +(0,1) and +(0,1) .. +(1,0)  node at
+(0.5,-0.5) {1};
\draw [>-] (8,0) -- +(0,-0.7) node at +(0,-1.2) {0};
\draw [->] (9,0) -- +(0,-0.7) node at +(0,-1.2) {0};

\node at (-1,-3.2) {degree};
\draw [<->] (0,-2) .. controls +(0,-1) and +(0,-1) .. +(1,0) node at
+(0.5,-1.2) {0};
\fill (0.5,-2.75) circle(2.5pt);
\draw [>-<] (2,-2) .. controls +(0,-1) and +(0,-1) .. +(1,0) node at
+(0.5,-1.2) {1};
\fill (2.5,-2.76) circle(2.5pt);
\draw [>-<] (4,-2.7) .. controls +(0,1) and +(0,1) .. +(1,0) node at
+(0.5,-0.5) {0};
\fill (4.5,-1.93) circle(2.5pt);
\draw [<->] (6,-2.7) .. controls +(0,1) and +(0,1) .. +(1,0)  node at
+(0.5,-0.5) {1};
\fill (6.5,-1.95) circle(2.5pt);
\draw [-<] (8,-2) -- +(0,-0.7) node at +(0,-1.2) {0};
\fill (8,-2.35) circle(2.5pt);
\draw [<-] (9,-2) -- +(0,-0.7) node at +(0,-1.2) {0};
\fill (9,-2.35) circle(2.5pt);
\end{tikzpicture}
\end{eqnarray}

Additionally, as displayed in \eqref{oriented}, we assign a degree to each arc of an oriented cup/cap diagram. The cups/caps of degree 1 are called {\it clockwise} and of degree 0 are called {\it anticlockwise}. To make sense of the word orientation the $\bullet$ on a cup/cap should be seen as an orientation reversing point.

\begin{definition}
The \emph{degree} of an oriented cup diagram $c\mu $ (respectively cap diagram $c'\mu$) is the sum of the degrees of all arcs contained in $\mu c$ (respectively $c'\mu$), i.e. the number of clockwise oriented cups in $\mu c$ (respectively clockwise oriented caps in $c'\mu$).
\end{definition}

\begin{remark}
The definition of ``oriented diagram'' and ``degree''  also applies to non-admissible cup/cap diagrams as well as to all types of circle diagrams.
\end{remark}

Note that $\underline{\la}$ (respectively  $\ov{\la}$) is precisely the unique cup (respectively cap) diagram such that $\underline{\la}\la$ (respectively $\la\ov{\la}$) is an oriented cup/cap diagram of degree $0$, i.e. all its cups/caps are anticlockwise. Conversely we have the following:

\begin{lemma} \label{lem:orient}
Given a cup diagram $c$ with block sequence $\Theta$ one can find a unique weight $\la$ such that $c=\underline{\la}$. The corresponding parity is unique; and if $\mu$ is a weight with the same block sequence such that $\underline{\la}\mu$ is oriented then $\mu$ has the same parity, hence $\la$ and $\mu$ are in the same block $\Lambda$.

The parity of $\Lambda$ equals the number of dotted rays plus undotted cups in $c$ modulo two.
\end{lemma}
\begin{proof}
We choose $\la$ to be the unique weight which induces the orientation on $c$ of degree zero. If $\underline{\la}\mu$ is oriented, then $\mu$ differs from $\la$ by swapping pairs of $\down\up$ to $\up\down$ or $\up\up$ to $\down\down$. In any case, $\mu$ has the same  parity as $\la$, thus they are in the same block	. The number of $\up$'s, hence the parity equals the number of dotted rays plus undotted cups in $c$ modulo two.
\end{proof}

Write $\la\subset\mu$ if $\underline{\la}\mu$ is an oriented cup diagram, then our combinatorics is compatible with the Bruhat order in the following sense:
\begin{lemma}\label{lem:Bruhator} Fix a block $\Lambda$ and $\lambda,\mu \in \Lambda$.
\begin{enumerate}[1.)]
\item If $\la\subset\mu$ then $\la\leq \mu$ in the Bruhat order.
\item If $a \la b$ is an oriented circle diagram
then $a = \underline{\alpha}$ and $b = \overline{\beta}$ for unique weights
$\alpha,\beta$ with $\alpha \subset \la \supset \beta$.
\end{enumerate}
\end{lemma}

\begin{proof}
 For part 2.) we more precisely claim: if $c\la$ is an oriented cup diagram then $c = \underline{\alpha}$ for a unique weight $\alpha$ with $\alpha \subset\la$; if $\la c'$ is an oriented cap diagram then $c'= \overline{\beta}$
for a unique weight $\beta$ with $\la \subset \beta$. Indeed, $\alpha$ is the unique weight obtained by reversing the labels at each clockwise cup into an anticlockwise cup. Clearly, $\underline{\alpha}=c$ and hence $\alpha\subset\la$. Similarly for $\beta$.
If $\la\subset\mu$ then $\underline{\la}\mu$ is oriented, hence $\mu$ is obtained from $\la$ by (possibly) changing some of the anticlockwise cups into clockwise. This however means we either change a pair $\down\up$ into $\up\down$ or $\up\up$ into $\down\down$. In either case the weight gets bigger in the Bruhat order by Lemma~\ref{lem:Bruhat} and 1.) holds.
\end{proof}

As in \cite{BS1} we call the number of cups in $\underline{\la}$ the {\it defect} or {\it degree of atypicality} of $\la$, denoted $\op{def}(\la)$, and the maximum of all the defects
of elements in a block the defect of a block, (which equals $\frac{k}{2}$ or $\frac{k-1}{2}$ depending whether $k = | P_\bu(\Lambda)|$ is even or odd).
Note that in contrast to \cite{BS1} we work with finite blocks only, since infinite blocks would have automatically infinite defect. Lemma~\ref{lem:orient} implies that the defect counts the number of possible orientations of a cup diagram $\underline{\la}$, since each cup has precisely two orientations:
\begin{eqnarray}
\label{defectcount}
|\{\mu\mid \la\subset\mu\}|&=&2^{\op{def}(\la)}.
\end{eqnarray}

The following connects the notion of $\la$-pairs with the degree:

\begin{prop}
\label{degone}
If $\la\rightarrow \mu$ then $\la$ and $\mu$ are in the same block. Moreover, $\la\rightarrow \mu$ if and only if $\underline{\la}\mu$ has degree $1$
and $\la<\mu$ in the Bruhat order.
\end{prop}

\begin{proof}
Assume $\la\rightarrow \mu$. Then by definition of a $\la$-pair we have a cup $C$ connecting vertices $i$ and $j$ ($i<j$) such that if $C$ is
undotted (respectively dotted) the weight $\mu$ differs from $\la$ only at these places with an $\up$ at position $i$ and $\down$ at position
$j$ (respectively a $\down$ at position $i$ and $\down$ at position $j$). Hence, in $\underline{\la}\mu$, the cup $C$ becomes
clockwise, while all other cups stay anticlockwise and thus it has degree $1$ and $\mu>\la$.

Assuming $\underline{\la}\mu$ has degree $1$ and $\la<\mu$. If $C$ is the unique clockwise cup, then $\mu$ equals $\la$ except at the vertices of $C$ where $\down\up$ gets replaced by $\up\down$ or $\up\up$ gets turned into $\down\down$. Hence the diagrams $\underline{\la}$ and $\underline{\mu}$ look locally as in Lemma~\ref{lapairscup}.
\end{proof}

\section{The graded vector space $\D$ with homogeneous basis $\B$}
Our goal is to introduce a type $\rm D$ analogue of the algebras studied in \cite{Str09}, \cite{BS1}.  It is a diagrammatic algebra $\D$ which,  in the special choice of a regular block $\Lambda$ , will be isomorphic to the endomorphism ring of a minimal projective generator in the category of perverse sheaves on the isotropic Grassmannian $Y_k$ equipped with the stratification by $B_D$-orbits. In this section we define and explore the underlying graded vector space.

\subsection{The space of circle diagrams} \label{sec:spaceofdiagrams}
Let $\Lambda$ be a fixed block.

\begin{definition}
Denote by
\begin{eqnarray}
\label{DefB}
\B&=&\left\{ \underline{\la}\nu\ov{\mu}\mid \la,\mu,\nu\in\Lambda,\quad\underline{\la}\nu \text{ and $\nu\ov{\mu}$ are
oriented}\right\},
\end{eqnarray}
the set of {\it oriented circle diagrams} associated with $\Lambda$. The {\it degree} of an element in $\B$
is the sum of the degrees of the cups and caps as displayed in \eqref{oriented}. The degree of a connected component in an oriented
circle diagram is the sum of the degrees of its cups and caps. We denote by $\D$ the graded complex vector space on the basis $\B$.
\end{definition}
Note that the $\underline{\la}\la\ov{\la}$ for $\la\in\Lambda$ span the degree zero part of $\D$.  For $\la,\mu\in \Lambda$ let ${}_\la(\D)_\mu{}$ be the subspace spanned by all $\underline{\la}\nu\ov{\mu}$, $\nu\in\Lambda$.

The most important special case will be the graded vector spaces $\mathbb{D}_k=\D$ associated with the principal blocks $\Lambda=\La_k^{\overline{0}}$ equipped with its distinct homogeneous basis $\mathbb{B}_k=\mathbb{B}_{\La_k^{\overline{0}}}$.

\begin{ex}
For $k=4$, the graded vector space $\mathbb{D}_k$ is already of  dimension $66$ with graded Poincare polynomial $p_{\mD_4}(q)=8+20q+24q^2+12q^3+2q^4$, and Cartan matrix given in Example~\ref{exB}.
\end{ex}

The vector space  $\D$ has the following alternative description which will be useful to describe the multiplication and to make the connection to geometry:

\begin{definition}
\label{relations}
For a cup or cap diagram $c$ and $i,j \in P_\bu(\Lambda)$ we write
\begin{itemize}
\item[$\blacktriangleright$] $i \CupConnect j$ if $c$ contains a cup/cap connecting $i$ and $j$,
\item[$\blacktriangleright$] $i \DCupConnect j$ if $c$ contains a dotted cup/dotted cap connecting $i$ and $j$,
\item[$\blacktriangleright$] $i \RayConnect$ if $c$ contains a ray that ends at $i$,
\item[$\blacktriangleright$] $i \DRayConnect$ if $c$ contains a dotted ray that ends at $i$.
\end{itemize}
\end{definition}

\begin{definition}
Let $\la,\mu\in\Lambda$ and assume the circle diagram $\un\la\ov{\mu}$ can be oriented. Let $I=I_ {\un\la\ov{\mu}}$ be the ideal in  $\mC [X_i \mid i \in P_\bu(\Lambda)]$ generated by 
\begin{center}
$X_a^2$, $X_a+X_b$ if $a \CupConnect b$, and $X_a-X_b$ if $a \DCupConnect b$, and finally $X_a$ if $a \RayConnect$ or $a \DRayConnect$. 
\end{center}
(Here the relations refer to both, $\ov\mu$ and $\underline{\la}$, and  $a,b \in P_\bu(\Lambda)$.) \\
\end{definition}

In the following we fix the grading on the rings  $\mC [X_i \mid i \in P_\bu(\Lambda)]\,/\,I_ {\un\la\ov{\mu}}$  induced by putting the $X_i$'s in degree $2$. 
\begin{prop}
\label{coho}
Let $\la,\mu\in\Lambda$ and consider the circle diagram $\un\la\ov{\mu}$. If $\un{\lambda}\ov{\mu}$ is not orientable we set $\cM(\un{\lambda}\ov{\mu})=\{0\}$, and otherwise
\begin{eqnarray*}
\cM(\un\la\ov{\mu})&=&\mC [X_i \mid i \in P_\bu(\Lambda)] \,/\,I_{\un\la\ov{\mu}} \;\langle {\rm mdeg}(\un\la\ov{\mu})\rangle
\end{eqnarray*}
with a shift $\langle {\rm mdeg}(\un\la\ov{\mu})\rangle$ of grading by ${\rm mdeg}(\un\la\ov{\mu})$ given in Definition \ref{def:minimaldegree} (using the conventions on graded modules as in Section~\ref{graded}).
\begin{enumerate}[1.)]
\item There is an isomorphism of graded vector space
\begin{eqnarray*}
\Psi_{\un\la \ov{\mu}}:\quad {}_\mu(\D)_\la{}& \longrightarrow &\cM(\un\la\ov{\mu}), \\
\underline{\la}\nu\ov{\mu}&\longmapsto&\prod_{C\in \mathscr{C}_{\rm clock}(\underline{\la}\nu\ov{\mu})}X_{t(C)}
\end{eqnarray*}
where $\mathscr{C}_{\rm clock}(\underline{\la}\nu\ov{\mu})$ is the set of clockwise oriented circle in $\underline{\la}\nu\ov{\mu}$ and $\op{t}(C)$ is the rightmost vertex belonging to a circle $C$. 
\item If $\un{\lambda}\ov{\mu}$ is orientable, the monomials of the form $X_{i_1} \cdots X_{i_s} \in \cM(\underline{\lambda}\ov{\mu})$, with $i_j \in P_\bu(\Lambda)$ distinct and $i_j=\op{t}(C_j)$ for some circle $C_j$, correspond under $\Psi_{\un\la \ov{\mu}}$ to the standard basis vectors from $\B$ in ${}_\mu(\D)_\la{}$.
\end{enumerate}
\end{prop}

The proof will be given at the end of the section.

\begin{ex}
\label{ps}
Consider the principal blocks $\Lap$ for $k=4$. The following circle diagrams $\un\la\ov{\mu}$ have exactly two possible orientations, anticlockwise or clockwise which are sent respectively to the elements $1, X_4\in \cM(\un\la\ov{\mu})$ under the isomorphism $\Psi_{\un\la \ov{\mu}}$. Similar for $\un\mu\ov{\la}$ with $\cM(\un\mu\ov{\la})=\cM(\un\la\ov{\mu})$.
\small
\begin{eqnarray*}
\label{1}
\begin{tikzpicture}[thick, scale=0.5]
\draw (0,-.1) .. controls +(0,1) and +(0,1) .. +(1,0);
\draw (2,-.1) .. controls +(0,1) and +(0,1) .. +(1,0);
\draw (1,-.1) .. controls +(0,-1) and +(0,-1) .. +(1,0);
\draw (0,-.1) .. controls +(0,-2.2) and +(0,-2.2) .. +(3,0);
\node at (4,-.5) {$\rightsquigarrow$};
\node at (13.8,0.2) {$\cM(\un\la\ov{\mu})=\mC[X_1,X_2,X_3,X_4]/I\cong \mC[X_4]/(X_4^2)$};
\node at (14.5,-1) {$I=(X_1^2,X_2^2,X_3^2,X_4^2,X_1+X_2,X_2+X_3,X_3+X_4,X_1+X_4)$};
\end{tikzpicture}\\
\label{2}
\begin{tikzpicture}[thick, scale=0.5]
\draw (0,-.1) .. controls +(0,1) and +(0,1) .. +(1,0);
\fill (0.5,.63) circle(4pt);
\draw (2,-.1) .. controls +(0,1) and +(0,1) .. +(1,0);
\fill (2.5,.63) circle(4pt);
\draw (1,-.1) .. controls +(0,-1) and +(0,-1) .. +(1,0);
\draw (0,-.1) .. controls +(0,-2.2) and +(0,-2.2) .. +(3,0);
\node at (4,-.5) {$\rightsquigarrow$};
\node at (13.8,0.2) {$\cM(\un\la\ov{\mu})=\mC[X_1,X_2,X_3,X_4]/I\cong \mC[X_4]/(X_4^2)$};
\node at (14.5,-1) {$I=(X_1^2,X_2^2,X_3^2,X_4^2,X_1-X_2,X_2+X_3,X_3-X_4,X_1+X_4)$};
\end{tikzpicture}\\
\label{3}
\begin{tikzpicture}[thick, scale=0.5]
\draw (0,-.1) .. controls +(0,1) and +(0,1) .. +(1,0);
\fill (0.5,.63) circle(4pt);
\draw (2,-.1) .. controls +(0,1) and +(0,1) .. +(1,0);
\draw (1,-.1) .. controls +(0,-1) and +(0,-1) .. +(1,0);
\draw (0,-.1) .. controls +(0,-2.2) and +(0,-2.2) .. +(3,0);
\fill (1.5,-1.75) circle(4pt);
\node at (4,-.5) {$\rightsquigarrow$};
\node at (13.8,0.2) {$\cM(\un\la\ov{\mu})=\mC[X_1,X_2,X_3,X_4]/I\cong \mC[X_4]/(X_4^2)$};
\node at (14.5,-1) {$I=(X_1^2,X_2^2,X_3^2,X_4^2,X_1-X_2,X_2+X_3,X_3+X_4,X_1-X_4)$};
\end{tikzpicture}\\
\label{4}
\begin{tikzpicture}[thick, scale=0.5]
\draw (0,-.1) .. controls +(0,1) and +(0,1) .. +(1,0);
\draw (2,-.1) .. controls +(0,1) and +(0,1) .. +(1,0);
\fill (2.5,.63) circle(4pt);
\draw (1,-.1) .. controls +(0,-1) and +(0,-1) .. +(1,0);
\draw (0,-.1) .. controls +(0,-2.2) and +(0,-2.2) .. +(3,0);
\fill (1.5,-1.75) circle(4pt);
\node at (4,-.5) {$\rightsquigarrow$};
\node at (13.8,0.2) {$\cM(\un\la\ov{\mu})=\mC[X_1,X_2,X_3,X_4]/I\cong \mC[X_4]/(X_4^2)$};
\node at (14.5,-1) {$I=(X_1^2,X_2^2,X_3^2,X_4^2,X_1+X_2,X_2+X_3,X_3-X_4,X_1-X_4)$};
\end{tikzpicture}
\end{eqnarray*}
\end{ex}
\normalsize

\begin{remark}
{\rm
In \cite{ES1}, see also \cite{Wilbert}, the case $\Lambda=\Lambda_k^{\ov{p}}$ is studied from a geometric point of view. Weights $\la$, $\mu$ are identified with fixed point of the natural $\mathbb{C}^*$-action on the (topological) Springer fibre of type ${\rm D}_k$ for the principal nilpotent of type ${\rm A}_{k-1}$, and the cup diagrams $\un{\la}$, $\un\mu$ are canonically identified with the cohomology of the closure of the corresponding attracting cells $\mathcal{A}_\la$, $\mathcal{A}_\mu$. The vector space $\cM(\un\mu\ov{\la})$ is then the cohomology of the intersection $\mathcal{A}_\la\cap \mathcal{A}_\mu$.
}
\end{remark}

\subsection{Properties of oriented circle diagrams}
The following is a crucial result linking orientability with decorations. 
\begin{lemma}
\label{lem:stupidlemmasub}
Fix a block $\Lambda$ and $\la,\mu\in\Lambda$. Then the following holds:
\begin{enumerate}[1.)]
\item The circle diagram $\un\la\ov\mu$ is orientable if and only if the number of dots is even on each of its circles and its propagating lines, and odd on each of its non-propagating lines.
\item 
In this case there are exactly $2^c$ possible orientations, where $c$ is the number of circles. They are obtained by choosing for each of the circles one out of its two possible orientations and for each ray the unique possible orientation. 
\end{enumerate}
\end{lemma}

\begin{proof}
We first prove part 1.) assuming that the diagram has an orientation.  Then the number of dots on each circle and each propagating line has to be even, since each dot can be interpreted as an orientation reversing point.  For a non-propagating line we consider first the case where the labels at the two rays are the same. Then, again by interpreting dots as orientation reversing points the number of dots between them has to be odd. If they are opposite, i.e. $\up$ (attached to a ray with a dot) and $\down$ (attached to an undotted  ray), the number of points between them has to be even, hence altogether (with the dot attached to the ray) the total number is odd again. 

To show the converse assume that the conditions on the number of dots hold. Note that if we find an orientation by some weight $\nu$, then $\nu\in\Lambda$ by Lemma~\ref{lem:orient}. First note that a possible orientation of a circle is determined by the label ($\up$ or $\down$) at one of its vertices. Hence there are at most two orientations. That each choice defines in fact an orientation is clear if we interpret the dots as orientation reversing points, since by assumption their number is even. For lines we have at most one orientation defined uniquely by the label on a vertex attached to a ray.  Then the parity condition on the dots ensures again that this extends to a orientation on the line with the label at the second vertex attached to a ray the unique allowed one.  

We implicitly already proved the second statement.
\end{proof}

\begin{prop}
\label{lem:stupidlemma}
Fix a block $\Lambda$ and $\un\la\nu\ov\mu\in\B$. For a circle $C$ contained in  $\un\la\ov\mu$ its induced degree equals:
\begin{enumerate}[1.)]
\item  the number of its caps plus $1$, if its rightmost label is $\down$; and  
\item  the number of its caps minus $1$, if its rightmost label is $\up$.
\end{enumerate}
\end{prop}

\begin{definition}
\label{clockanti}
We call $C$ with the induced orientation \emph{clockwise} respectively  \emph{anticlockwise} depending whether the degree is maximal respectively minimal possible \footnote{In case $C$ has no dots this agrees with the obvious notion of (anti)clockwise.}.
\end{definition}
\begin{proof}[Proof of Proposition~\ref{lem:stupidlemma}]
In case the circle $C$ has no dots, the  statement follows by the
same arguments as in \cite[Lemma 2.1]{BS1}: namely either $C$ consists out of precisely one cup and one cap, i.e.

\begin{center}
\begin{tikzpicture}[scale=.5,thick,>=angle 60]
\node at (-1,0) {$\deg\bigg($};
\draw (-.25,0) to +(2,0);
\draw (0,0) .. controls +(0,-1) and +(0,-1) .. +(1.5,0);
\draw (0,0) .. controls +(0,1) and +(0,1) .. +(1.5,0);
\node at (0.005,.1) {$\down$};
\node at (1.505,-.1) {$\up$};
\node at (2.7,0) {$\bigg) = 0$,};
\end{tikzpicture}
\hspace{1cm}
\begin{tikzpicture}[scale=.5,thick,>=angle 60]
\node at (-1,0) {$\deg\bigg($};
\draw (-.25,0) to +(2,0);
\draw (0,0) .. controls +(0,-1) and +(0,-1) .. +(1.5,0);
\draw (0,0) .. controls +(0,1) and +(0,1) .. +(1.5,0);
\node at (0.005,-.1) {$\up$};
\node at (1.505,.1) {$\down$};
\node at (2.7,0) {$\bigg) = 2$,};
\end{tikzpicture}
\end{center}
where the statement is clear, or it contains a kink which
can be removed using one of the following straightening rules:
\begin{equation}
\label{kink}
\begin{tikzpicture}[thick, snake=zigzag, line before snake = 2mm, line
after snake = 2mm]
\draw[thin] (-.25,0) -- +(1.5,0);
\begin{scope}[yscale=-1]
\draw (0,0) -- +(0,.3);
\draw[dotted, snake] (0,.3) -- +(0,.8);
\draw (0,0) .. controls +(0,-.5) and +(0,-.5) .. +(.5,0);
\draw (.5,0) .. controls +(0,.5) and +(0,.5) .. +(.5,0);
\draw (1,0) -- +(0,-.3);
\draw[dotted, snake] (1,-.3) -- +(0,-.8);
\end{scope}
\node at (0,-.1) {$\up$};
\node at (1.01,-.1) {$\up$};
\node at (.51,.08) {$\down$};

\draw [->,line join=round,decorate, decoration={zigzag,segment
length=4,amplitude=.9,post=lineto,post length=2pt}]  (1.5,0) -- +(.5,0);

\draw[thin] (2.25,0) -- +(.5,0);
\draw (2.5,0) -- +(0,.3);
\draw[dotted, snake] (2.5,.3) -- +(0,.8);
\draw (2.5,0) -- +(0,-.3);
\draw[dotted, snake] (2.5,-.3) -- +(0,-.8);
\node at (2.51,-.1) {$\up$};

\begin{scope}[xshift=5cm]
\draw[thin] (-.25,0) -- +(1.5,0);
\begin{scope}[yscale=-1]
\draw (0,0) -- +(0,.3);
\draw[dotted, snake] (0,.3) -- +(0,.8);
\draw (0,0) .. controls +(0,-.5) and +(0,-.5) .. +(.5,0);
\draw (.5,0) .. controls +(0,.5) and +(0,.5) .. +(.5,0);
\draw (1,0) -- +(0,-.3);
\draw[dotted, snake] (1,-.3) -- +(0,-.8);
\end{scope}
\node at (0,.1) {$\down$};
\node at (1.01,.1) {$\down$};
\node at (.51,-.1) {$\up$};

\draw [->,line join=round,decorate, decoration={zigzag,segment
length=4,amplitude=.9,post=lineto,post length=2pt}]  (1.5,0) -- +(.5,0);

\draw[thin] (2.25,0) -- +(.5,0);
\draw (2.5,0) -- +(0,.3);
\draw[dotted, snake] (2.5,.3) -- +(0,.8);
\draw (2.5,0) -- +(0,-.3);
\draw[dotted, snake] (2.5,-.3) -- +(0,-.8);
\node at (2.51,.1) {$\down$};
\end{scope}
\end{tikzpicture}
\end{equation}

The result is a new oriented circle diagram with two
fewer vertices than before. It is oriented in the same way as before, except that the circle $C$ 
has either one clockwise cup and one anticlockwise
cap less, or one anticlockwise cup and one clockwise
cap less. Hence both the degree as well as the number of caps decreases by $1$. Moreover we never change the label at the rightmost vertex. The claim follows therefore by induction.

Assume now $C$ has dots. The claim is obviously true for the
circles including two vertices:
\begin{center}
\begin{tikzpicture}[scale=.5,thick,>=angle 60]
\node at (-1,0) {$\deg\bigg($};
\draw (-.25,0) to +(2,0);
\draw (0,0) .. controls +(0,-1) and +(0,-1) .. +(1.5,0);
\node at (0.005,-.1) {$\up$};
\node at (1.505,-.1) {$\up$};
\fill (.75,-.76) circle(4pt);
\draw (0,0) .. controls +(0,1) and +(0,1) .. +(1.5,0);
\fill (.75,.76) circle(4pt);
\node at (2.7,0) {$\bigg) = 0$,};
\end{tikzpicture}
\hspace{1cm}
\begin{tikzpicture}[scale=.5,thick,>=angle 60]
\node at (-1,0) {$\deg\bigg($};
\draw (-.25,0) to +(2,0);
\draw (0,0) .. controls +(0,-1) and +(0,-1) .. +(1.5,0);
\fill (.75,-.76) circle(4pt);
\draw (0,0) .. controls +(0,1) and +(0,1) .. +(1.5,0);
\node at (0.005,.1) {$\down$};
\node at (1.505,.1) {$\down$};
\fill (.75,.76) circle(4pt);
\node at (2.7,0) {$\bigg) = 2$,};
\end{tikzpicture}
\end{center}

If there is a kink without dots as above then we remove it and argue again by induction. We therefore assume there are no such kinks,
i.e. each occurring kink has precisely one dot (either on the cap or the cup), since having a dot on both contradicts Remark
~\ref{decorated} if the component is a closed circle. Using Remark~\ref{decorated} again one easily verifies that a circle with more than
$2$ vertices and
no undotted kink contains at least one of the following diagrams as subdiagram
\begin{eqnarray}
\label{snake}
\usetikzlibrary{arrows}
\begin{tikzpicture}[thick,>=angle 60, scale=0.7,snake=zigzag, line before snake = 2mm, line after snake = 1mm]
\draw[dotted,snake] (0,0.1) to +(0,-1.4);
\draw (0,0) .. controls +(0,1) and +(0,1) .. +(1,0);
\node at (0.005,.1) {$\down$};
\node at (1.005,.1) {$\down$};
\fill (0.5,0.75) circle(2.5pt);
\draw (1,0) .. controls +(0,-1) and +(0,-1) .. +(1,0.1);
\draw (2,0) .. controls +(0,1) and +(0,1) .. +(1,0);
\node at (2.005,-.1) {$\up$};
\node at (3.005,-.1) {$\up$};
\fill (2.5,0.77) circle(2.5pt);
\draw[dotted,snake] (3,0.1) to +(0,-1.4);

\begin{scope}[xshift=4cm,y=-1cm]
\draw[dotted,snake] (0,0.1) to +(0,-1.4);;
\draw (0,0) .. controls +(0,1) and +(0,1) .. +(1,0);
\node at (0.005,-.1) {$\down$};
\node at (1.005,-.1) {$\down$};
\fill (0.5,0.75) circle(2.5pt);
\draw (1,0.1) .. controls +(0,-1) and +(0,-1) .. +(1,0.1);
\draw (2,0) .. controls +(0,1) and +(0,1) .. +(1,0);
\node at (2.005,.1) {$\up$};
\node at (3.005,.1) {$\up$};
\fill (2.5,0.77) circle(2.5pt);
\draw[dotted,snake] (3,0.1) to +(0,-1.4);
\end{scope}

\begin{scope}[xshift=8cm]
\draw[dotted,snake] (0,0.1) to +(0,-1.4);
\draw (0,0) .. controls +(0,1) and +(0,1) .. +(1,0);
\node at (0.005,-.1) {$\up$};
\node at (1.005,-.1) {$\up$};
\fill (0.5,0.75) circle(2.5pt);
\draw (1,0.1) .. controls +(0,-1) and +(0,-1) .. +(1,0.1);
\draw (2,0) .. controls +(0,1) and +(0,1) .. +(1,0);
\node at (2.005,.1) {$\down$};
\node at (3.005,.1) {$\down$};
\fill (2.5,0.77) circle(2.5pt);
\draw[dotted,snake] (3,0.1) to +(0,-1.4);
\end{scope}

\begin{scope}[xshift=12cm,y=-1cm]
\draw[dotted,snake] (0,0.1) to +(0,-1.4);
\draw (0,0) .. controls +(0,1) and +(0,1) .. +(1,0);
\node at (0.005,.1) {$\up$};
\node at (1.005,.1) {$\up$};
\fill (0.5,0.75) circle(2.5pt);
\draw (1,0) .. controls +(0,-1) and +(0,-1) .. +(1,0.1);
\draw (2,0) .. controls +(0,1) and +(0,1) .. +(1,0);
\node at (2.005,-.1) {$\down$};
\node at (3.005,-.1) {$\down$};
\fill (2.5,0.77) circle(2.5pt);
\draw[dotted,snake] (3,0.1) to +(0,-1.4);
\end{scope}
\end{tikzpicture}
\end{eqnarray}
We claim that the degree equals the degree of the circle obtained by removing the pair of dots and adjusting the orientations between them fixing the outer labels. Then by induction, we transferred our statements to the undotted case and we are done, since we again do not change the rightmost label. 
The possible orientations for \eqref{snake} are as follows (where $a$ stands for anticlockwise and $c$ for clockwise, cups/caps read from left to right):
\begin{itemize}
\item[$\blacktriangleright$] in the first two diagrams: with dots $(c,a,a)$, without dots  $(a,c,a)$, hence the total degree is $1$;
\item[$\blacktriangleright$]
in the last two diagram: with dots $(a,c,c)$, without dots $(c,a,c)$, hence the total degree is $2$.
\end{itemize}
Then the claim follows.
\end{proof}

\begin{remark}
\label{stupidcount}
By removing successively kinks as shown in \eqref{kink} we see that any circle has an even total number of cups plus caps. Together with Lemma~\ref{lem:stupidlemmasub} we obtain that every circle in an oriented circle diagram must have an even total number of undotted cups plus undotted caps.
\end{remark}

We are now also able to define the shift of grading in Proposition~\ref{coho}.

\begin{definition} \label{def:minimaldegree}
Let $\un\la\ov{\mu}$ be a circle diagram that can be equipped with an orientation. The degree of the diagram $\un\la \nu \ov{\mu}$ where the orientation is chosen such that all circles are oriented anticlockwise is called the {\it minimal degree} of $\un\la\ov{\mu}$ and denoted by ${\rm mdeg}(\un\la\ov{\mu})$.
\end{definition}

\begin{proof}[Proof of Proposition~\ref{coho}]
If $\un\la\ov{\mu}$ is not orientable, the statement is trivial. Assume otherwise. The map $\Psi_{\un\la \ov{\mu}}$ is a well-defined map of vector spaces. Moreover the degree of $\underline{\la}\nu\ov{\mu}$ equals  ${\rm mdeg}(\un\la\ov{\mu})+ 2|\mathscr{C}_{\rm clock}(\underline{\la}\nu\ov{\mu})|$ by Proposition~\ref{lem:stupidlemma} which is by definition the same degree as its image under $\Psi_{\un\la \ov{\mu}}$, hence $\Psi_{\un\la \ov{\mu}}$ is homogeneous of degree zero. It only remains to show that the basis $\B \cap {}_\mu(\D)_\la{}$ of ${}_\mu(\D)_\la{}$ is mapped to a basis. 
The relations from Definition~\ref{relations} imply that the monomials in the $X_i$'s, where $i$ runs through the set $J$ of rightmost points of the circles in   $\un\la\ov{\mu}$ generate $\cM(\un{\lambda}\ov{\mu})$ as a vector space. By Remark~\ref{stupidcount} the number of undotted cups plus undotted caps is even,  hence the relations imply that $X_i$ for $i\in J$ is non-zero in $\cM(\un{\lambda}\ov{\mu})$. Since the relations only couple indices in the same component of the diagram, the monomials in the $X_i$'s with $i\in J$ are linearly independent. 
\end{proof}

\subsection{The diagrammatic multiplication rule}
\label{sec:idea}
We fix a block $\Lambda$. To define the product of two basis elements $(a \la b)$ and $(c \mu d)$ in $\D$ we describe now what is called the {\it decorated generalized surgery procedure}. 

Take the two circle diagrams $(a \la b)$ and $(b^* \mu d)\in\D$ and put $(b^* \mu d)$  on top of $(a \la b)$ stitching all the lower vertical rays of  $(b^*\mu d)$ together with their mirror images in 
the top part of $(a \la b)$. In case the rays are dotted, eliminate these dots.
Since dots will always be removed in pairs, the new diagram will be orientable (in the obvious sense), see Definition \ref{def:stacked_diagram_orientation}. In fact it just  inherits an orientation from the original orientations.

Pick a symmetric pair $\gamma$ of a cup and a cap (possibly dotted) in the middle section of $( a \la b )( b^* \mu d )$ that can be connected without crossing rays or arcs and such that  to the right of $\gamma$ there are no dotted arcs. That means if we replace $\gamma$ by two vertical lines all involved diagrams remain admissible. 

Now perform this replacement, i.e. forgetting orientations for a while, cut open the cup and the cap in $\gamma$ and  stitch the loose ends together to form a pair of vertical line segments (always without dots even if there were dots on $\gamma$ before):
\begin{eqnarray} \label{pic:surgery}
\begin{tikzpicture}[thick, snake=zigzag, line before snake = 2mm, line after snake = 2mm]
\draw[thin] (-.25,0) -- +(1,0);
\draw[thin] (-.25,1.5) -- +(1,0);

\draw[thin, dashed] (.25,.38) -- +(0,.74);

\draw (0,1.5) .. controls +(0,-.5) and +(0,-.5) .. +(.5,0);
\draw (0,0) .. controls +(0,.5) and +(0,.5) .. +(.5,0);
\draw[->, snake] (1,.75) -- +(1,0);
\draw[thin] (2.25,0) -- +(1,0);
\draw (2.5,0) -- +(0,1.5);
\draw (3,0) -- +(0,1.5);
\draw[thin] (2.25,1.5) -- +(1,0);

\begin{scope}[xshift=5.5cm]
\draw[thin] (-.25,0) -- +(1,0);
\draw[thin] (-.25,1.5) -- +(1,0);

\draw[thin, dashed] (.25,.38) -- +(0,.74);

\draw (0,1.5) .. controls +(0,-.5) and +(0,-.5) .. +(.5,0);
\fill (0.25,1.125) circle(2.5pt);
\draw (0,0) .. controls +(0,.5) and +(0,.5) .. +(.5,0);
\fill (0.25,.375) circle(2.5pt);
\draw[->, snake] (1,.75) -- +(1,0);
\draw[thin] (2.25,0) -- +(1,0);
\draw (2.5,0) -- +(0,1.5);
\draw (3,0) -- +(0,1.5);
\draw[thin] (2.25,1.5) -- +(1,0);
\end{scope}
\end{tikzpicture}
\end{eqnarray}

Each time this procedure is performed, we define a linear map, ${\rm surg}_{D,D'}$, from the vector space with basis all orientations of the diagram $D$ we have before the procedure to the  the vector space with basis all orientations of the diagram  $D'$ we have after the procedure. Composing these maps until we reach a situation without any cup-cap-pairs left in the middle part will give the multiplication. 
That is the result of the multiplication of  $(a \la b)$ and $(b^* \mu d)$ will be a linear combination of oriented diagrams where the underlying shape is just obtained by applying \eqref{pic:surgery} to all cup-cap pairs.

The definition of the map ${\rm surg}_{D,D'}$ depends on whether the number of components increases, decreases or stays the same when passing from $D$ to $D'$. There are three possible scenarios:
\begin{itemize}
\item[$\blacktriangleright$] \textbf{Merge:} two components are replaced 
by one component, or
\item[$\blacktriangleright$] \textbf{Split:} one component is replaced 
by two components, or
\item[$\blacktriangleright$] \textbf{Reconnect:} two components are 
replaced by two new components.
\end{itemize}
To define the surgery map ${\rm surg}_{D,D'}$ attached to $\gamma$ denote by $i$ the left and by $j$ the right vertex of the cup (or equivalently the cap) in $\gamma$. Depending on these numbers $i<j$ we define ${\rm surg}_{D,D'}$ on the basis of orientations of $D$:\\[0.05cm]
\begin{mdframed}
\noindent
\textbf{Merge:} Assume components  
$C_1$ and $C_2$ are merged into component $C$.
\begin{itemize}
\item[$\blacktriangleright$]  If both are clockwise circles or one is a clockwise circle and the 
other a line, the result is zero.
\item[$\blacktriangleright$] If both circles are anticlockwise, then apply \eqref{pic:surgery} and 
orient the result anticlockwise.
\item[$\blacktriangleright$] If one component is anticlockwise and one is clockwise, apply 
\eqref{pic:surgery}, orient the result clockwise and also multiply with 
$(-1)^a$, where $a$ is defined in \eqref{mac}.
\end{itemize}
\noindent
\textbf{Split:} Assume component $C$ 
splits into $C_i$ and $C_j$ (containing the vertices at 
$i$ respectively $j$). If, after applying \eqref{pic:surgery}, the diagram is not orientable, the map is just zero. Hence assume that the result is orientable, then we have:
\begin{itemize}
\item[$\blacktriangleright$] If $C$ is clockwise, then apply \eqref{pic:surgery} and orient $C_i$ and 
$C_j$ clockwise. Finally multiply with $(-1)^{\pos(i)}(-1)^{a_r}$, where 
$a_r$ is defined in \eqref{splitc}.
\item[$\blacktriangleright$] If $C$ is anticlockwise, then apply \eqref{pic:surgery} and take two 
copies of the result. In one copy orient $C_j$ anticlockwise and $C_i$ 
clockwise and moreover multiply with $(-1)^{\pos(i)}(-1)^{a_i}$; in the other 
copy orient $C_i$ anticlockwise and $C_j$ clockwise and multiply with 
$(-1)^{\pos(i)}(-1)^{a_j}$, where $a_i$ and $a_j$ are defined in 
\eqref{splita}.
\end{itemize}

\noindent \textbf{Reconnect:} In this case two lines are transformed into two lines by applying \eqref{pic:surgery}. If the new diagram is orientable, necessarily with the same orientation as before, do so, otherwise the result is zero.
\end{mdframed}
\hfill\\[0.05cm]

For a component $C$ in an orientable diagram $D$ obtained by the above procedure we 
denote by $\op{t}(C)$ the rightmost vertex on $C$. For two vertices 
$r,s$ in $D$ connected by a sequence of arcs let $\op{a}_D(r,s)$ be the 
total number of undotted cups plus undotted caps on such a connection. 
(By Remark \ref{rmk:signonsequence} this is independent of the choice of the connection.)
\\[0.05cm]
\begin{mdframed}
Then the signs are defined as follows: in the merge case in question
\begin{eqnarray} \label{mac} 
a=\op{a}_D(\op{t}(C_r),\op{t}(C)),
\end{eqnarray}
where $C_r$ ($r\in\{1,2\}$) is the clockwise component (necessarily a circle); and for the split cases
\begin{eqnarray} \label{splitc} 
a_r=\op{a}_{D'}(r,\op{t}(C_r))+u,
\end{eqnarray}
where $r \in \{i,j\}$ such that $C_r$ does not contain $\op{t}(C)$ and $u=1$ if $\gamma$ is undotted and $u=0$ if it is dotted. Finally
\begin{eqnarray} \label{splita} 
a_j=\op{a}_{D'}(j,\op{t}(C_j)) \text{ and } a_i=\op{a}_{D'}(i,\op{t}(C_i))+u,
\end{eqnarray}
where $u=1$ if $\gamma$ is undotted and $u=0$ if $\gamma$ is dotted.
\end{mdframed}
\hfill\\[0.05cm]
For explicit examples see Section~\ref{sec:examples}. Note that the involved signs depend on the whole diagram and are responsible for the {\it non-locality} of the multiplication rule. In particular, the {\it associativity} of the multiplication becomes a non-obvious fact. Moreover it needs to be shown that the result of the multiplication {\it does not depend on the chosen order} of the cup-cap-pairs. For instance in the Examples ~\ref{ex:mergesplitexample}-\ref{ex:splitnotorientable} we also compare different orderings. Both facts follow for the type ${\rm A}$ version, from \cite{Khovtangles} and \cite{BS1}, easily from topological arguments, which are not applicable here. The next two section will deal with these two issues.

\section{Diagrammatic and algebraic surgery}
\label{sec:surgery}
For the formal definition of the multiplication in the Khovanov algebra $\D$ we first introduce the notion of a stacked circle diagram, which is slightly more general than the notion of circle diagrams, since several of them can be stacked on top of each other. Afterwards we define the combinatorial and algebraic surgery maps and prove the commutativity theorem for surgeries as in Theorem~\ref{thm:surgeries_commute} which finally implies the associativity of the multiplication.

\subsection{Stacked circle diagrams}
We fix a block $\Lambda$ with block diagram $\Theta$. 
\begin{definition}
Let $h\in\mZ_{\geq 0}$. An \emph{(admissible) stacked circle diagram} \emph{of height $h$}  for $\Lambda$  is a sequence, $\un{\lambda}(\mathbf{a})\ov{\mu}$, of circle diagrams of the form 
\begin{center}
$\un{\lambda}a_1^*$, $a_1a_2^*$, $\ldots$ , $a_{h-1}a_h^*, a_{h}\ov{\mu}$, 
\end{center}
such that $\lambda,\mu \in \Lambda$ and all $a_s$, for $1 \leq s \leq h$, are cup diagrams for  $\Theta$.  In case $h=0$ we just have an ordinary circle diagram $\un{\lambda}(\mathbf{a})\ov{\mu}=\un{\lambda}\ov{\mu}$. 
\end{definition}

We represent stacked circle diagrams diagrammatically by vertically stacking $h+1$ decorated circle diagrams on top of each other starting with $\un\la$ in the bottom and ending with $\ov\mu$ at the top such that the vertices for the $l$-th diagram have coordinates $(x,l-1)$ with $x\in\mZ_{>0}$ as in the circle diagram. Note that the middle sections consists of symmetric diagrams $a_s^*a_s$, $1\leq s\leq h$. We call the cap and cup diagrams  involved in these $\it{internal}$.

\begin{remark}
Note that the internal cup diagrams in a stacked circle diagram need not be diagrams obtained from weights in the block $\Lambda$, but only ones that are compatible with the block diagram $\Theta$. 
\end{remark}

The following generalizes the notion of oriented circle diagrams.

\begin{definition} \label{def:stacked_diagram_orientation}
Let $D=\un{\lambda}(\mathbf{a})\ov{\mu}$ be a stacked circle diagram of height $h$ for $\Lambda$. By an {\it orientation} of $D$ we mean a sequence $\bm{\nu} =(\nu_0, \ldots, \nu_h)$ of weights in $\Lambda$, such that if we label the vertices of $a_{s}a_{s+1}^*$ by $\nu_s$ (with $a_{0} = \underline{\lambda}$ and $a_{h+1}^*=\overline{\mu}$) all connected components in the vertically stacked diagram are oriented, where again the $\bullet$'s should be seen as orientation reversing points. \emph{An oriented stacked  circle diagram} is a stacked circle diagram $\un{\lambda}(\mathbf{a})\ov{\mu}$ together with an orientation $\bm{\nu}$, denoted by $\un{\lambda}(\mathbf{a},\bm{\nu})\ov{\mu}$.
We call $D$ \emph{orientable} if there exists an orientation of $D$. 
\end{definition}

To define the surgery procedure we introduce the notion of a surgery:

\begin{definition} \label{def:admissible_surgery}
Let $D$ and $D'$ be stacked circle diagrams of height $h>0$, say  $D=\un{\lambda}(\mathbf{a})\ov{\mu}$ and $D'=\un{\lambda}(\mathbf{b})\ov{\mu}$. Then we say that $D'$ \emph{is obtained from $D$ by a surgery} if there exists exactly one $1 \leq l \leq {h}$ such that $b_l \neq a_l$, and this $b_l$ is obtained from $a_l$ by replacing exactly one cup connecting positions $i < j$ in $a_l$ by two (undotted) rays. More precisely, in this case $D'$ is obtained from $D$ by \emph{a surgery at  positions $i < j$ on level $l$}. 
\end{definition}
Note that in this case $D$ and $D'$ differ only locally as in \eqref{pic:surgery}.

\begin{definition}
\label{MSR}
If the number of connected components decreases in a surgery from $D$ to $D'$ we call the surgery a \emph{merge}, if the number increases we call it a \emph{split}, and if the number stays the same we call it a \emph{reconnect}.
\end{definition}

As in the case of circle diagrams many of the formulas for stacked circle diagrams depend on a chosen rightmost vertex in each component, but in contrast to the circle diagram case such vertices are not unique. The restriction of only working with admissible diagrams ensures that at least the rightmost vertices in all occurring circles have the same label.  This will allow us to define (anti)clockwise circles in oriented stacked circle diagrams.  We first introduce so-called tags:

\begin{definition}
Let $D=\un{\lambda}(\mathbf{a})\ov{\mu}$ be a stacked circle diagram of height $h$ for $\Lambda$. A \emph{tag} $\op{t}$ of $D$ is an assignment 
\begin{eqnarray*}
\op{t} :\quad \mathscr{C}(D) &\longrightarrow& P_\bu(\Lambda)\times  \{0,\ldots,h\},
\end{eqnarray*}
where $\mathscr{C}(D)$ is the set of all circles in $D$, such that for each $C \in \mathscr{C}(D)$ we have $\op{t}(C) = (i,l) \in C$ with  $i$ maximal amongst all vertices contained in $C$. We call $\op{t}(C)$ the \emph{tag of $C$} (induced from $\op{t}$).
\end{definition}

The following sign is independent of the choice of tags.

\begin{lemma} \label{lem:sign_of_pos}
Let $D=\un{\lambda}(\mathbf{a})\ov{\mu}$ be an orientable stacked circle
diagram of height $h$ for $\Lambda$ with tag $\op{t}$. Let $0\leq l \leq
h$ and $i \in P_\bu(\Lambda)$ such that the connected component $C$
of $D$ containing $(i,l)$ is a circle.
Define
\begin{eqnarray*}
{\rm sign}_D(i,l) &=& (-1)^{\# \{ j \;\mid\; {\rm deco}(\gamma_j) = 0 \text{
and $\gamma_j$ is a cup/cap} \}}
\end{eqnarray*}
where $\gamma_1,\ldots,\gamma_t$ is a sequence of arcs in $D$ such that
their concatenation is a path from $(i,l)$ to $ \op{t}(C)$. Then ${\rm
sign}_D(i,l)$ is independent of $ \op{t}$ and the chosen sequence of arcs.
\end{lemma}

\begin{proof}
Since the circle $C$ is orientable, ${\rm sign}_D(i,l)$ does not depend on which sequence of arcs one chooses from $(l,i)$ to $\op{t}(C)$ by Remark~\ref{stupidcount}.

To show that it is independent of the chosen tag, note that if $\op{t}'$ is another tag then by definition the first coordinate agrees, i.e. $\op{t}(C)_1=\op{t}'(C)_1=i_{\rm max}$. Let $ \op{t}_{\rm max}$ and $\op{t}_{\rm min}$ be tags such that, $\op{t}_{\rm max}(C)_2$ is maximal and $\op{t}_{\rm min}(C)_2$ is minimal. If $\op{t}_{\rm max}(C)_2 = \op{t}_{\rm min}(C)_2$ then $\op{t}_{\rm max}(C) = \op{t}_{\rm min}(C)$ and there is only one possible image of $C$ under a tag and there is nothing to prove. Hence assume that the second coordinates are different.

We can decompose the circle $C$ into two connected components, $C_1$ and $C_2$, intersecting in $\op{t}_{\rm max}(C)$ and $\op{t}_{\rm min}(C)$ and chosen such that $C_1$ contains the cap containing $\op{t}_{\rm max}(C)$.

Since all vertices $(j,l')$ on the circle satisfy $j \leq i_{\rm max}$ by definition, the only possibility is that $C_1$ also contains the cup containing $\op{t}_{\rm min}(C)$ and all vertices $(j,l')$ on $C_1$ satisfy $j < i_{\rm max}$, except for $\op{t}_{\rm max}(C)$ and $\op{t}_{\rm min}(C)$. This implies that $\op{t}'(C) \in C_2$ for any tag $\op{t}'$. This implies that by admissibility $C_2$ cannot contain any dotted cups or caps. Since $C_2$ starts with an arc in the cup diagram containing $\op{t}_{\rm max}(C)$ and ends with an arc in the cap diagram containing $\op{t}_{\rm min}(C)$ it must contain an even total number of cups and caps. A similar argument also holds for any other tag $\op{t}'$ using a sequence of arcs connecting $\op{t}_{\rm max}(C)$ and $\op{t}'(C)$. This implies that ${\rm sign}_D(i,l)$ does not change when we change the tag.
\end{proof}

\begin{remark}\label{rmk:signonsequence}
In case two vertices  $(i,l)$ and $(i',l')$ lie on a common circle, the product  ${\rm sign}_D(i,l){\rm sign}_D(i',l')$ has an alternative expression, namely 
\begin{eqnarray}
\label{unverstformel}
{\rm sign}_D(i,l){\rm sign}_D(i',l') &=& (-1)^{\# \{ j \;\mid\; {\rm deco}(\gamma_j) = 0 \text{ and $\gamma_j$ is a cup/cap} \}}
\end{eqnarray}
where $\gamma_1,\ldots,\gamma_t$ is a sequence of arcs connecting the two vertices. Hence it can be computed by counting  undotted cups and caps on a sequence of arcs connecting the two vertices. 
\end{remark}

\begin{corollary} \label{cor:orientation_defined}
Let $D=\un{\lambda}(\mathbf{a},\bm{\nu})\ov{\mu}$ be an oriented stacked circle diagram of height $h$ and $C$ a circle in $D$. Let $\op{t},\op{t}'$ be tags of $D$ and let $\op{t}(C)=(i,l)$ and $\op{t}'(C)=(i,l')$. Then the symbols at the $i$-th vertex in $\nu_l$ and $\nu_{l'}$ agree.
\end{corollary}
\begin{proof}
By the proof of Lemma \ref{lem:sign_of_pos} the total number of undotted cups and caps between $(i,l)$ and $(i,l')$ is even. Since these are the only arcs that force a change from $\up$ to $\down$ or vice versa, the statement follows.
\end{proof}

\begin{definition} \label{def:componentorientation}
For an oriented stacked circle diagram $D$ and a tag $\op{t}$, a circle $C$ is called \emph{clockwise} if the symbol at $\op{t}(C)$ is $\down$ and \emph{anticlockwise} if it is $\up$. This is well-defined by Corollary \ref{cor:orientation_defined}. By definition we call the unique orientation of a line {\it anticlockwise}. Moreover we call a circle {\it small} if it only involves one cup and one cap.
\end{definition}

The following is a direct generalization of Definition~\ref{def:minimaldegree}.
\begin{definition} \label{def:minimaldegreestacked}
For an orientable  stacked circle diagram $D$ we denote by ${\rm mdeg}(D)$ the degree of the unique orientation such that all circles are oriented anticlockwise. (This exists by the same reason as in Lemma~\ref{lem:stupidlemmasub}.)
\end{definition}

Similar to the height $0$ case, see \eqref{DefB}, we define a vector space with a basis consisting of all possible oriented stacked circle diagram with a fixed internal part. 

\begin{definition}
For $\mathbf{a}=(a_1,\ldots,a_{h})$ a sequence of cup diagrams for  $\Theta$ set
$$ \mathbb{B}_\Lambda^\mathbf{a} = \left\{ \un{\lambda}(\mathbf{a},\bm{\nu})\ov{\mu} \mid  \text{oriented for } \lambda, \mu, \nu_i \in \Lambda \right\}.$$
By $\mathbb{D}_\Lambda^\mathbf{a}$ we denote the complex vector space with basis $\mathbb{B}_\Lambda^\mathbf{a}$ and for fixed $\la,\mu\in\Lambda$ by ${}_\lambda(\mathbb{D}_\Lambda^\mathbf{a})_\mu$ the subspace with basis consisting of oriented diagrams of the form  $\un{\lambda}(\mathbf{a},\bm{\nu})\ov{\mu}$ for arbitrary $\bm{\nu}$.
\end{definition}

\subsection{Diagrammatic surgery}

We first describe the surgery map more carefully diagrammatically, and afterwards give an algebraic interpretation of all the involved spaces and maps. The principle idea of the multiplication follows \cite{BS1}, but in contrast to the situation therein it can happen that the resulting diagrams can not be oriented at all in which case the result of the surgery procedure is declared to be zero.

Let still fix a block $\Lambda$. Let $D=\un{\lambda}(\mathbf{a})\ov{\mu}$ and $D'=\un{\lambda}(\mathbf{b})\ov{\mu}$ be two stacked circle diagrams of height $h$ for $\Lambda$ with $D'$ obtained from $D$ by a surgery at positions $i < j$ on level $l$. We now define the surgery map
\begin{eqnarray} \label{eqn:surgerymap}
{\rm surg}_{D,D'}: \quad{}_\lambda(\mathbb{D}_\Lambda^\mathbf{a})_\mu &\longrightarrow& {}_\lambda(\mathbb{D}_\Lambda^\mathbf{b})_\mu,
\end{eqnarray}
separately for each of the three possible scenarios Merge, Split, Reconnect in Definition~\ref{MSR}. For illustrating examples we refer to Section~\ref{sec:examples}.

\subsubsection{Merge:} 
\label{subsub}
Denote the component of $D$ containing $(i,l-1)$ by $C_{l-1}$, the component containing $(i,l)$ by $C_{l}$, and the component in $D'$ containing $(i,l-1)$ by $C$. The map ${\rm surg}_{D,D'}$ from \eqref{eqn:surgerymap} is given by
\begin{eqnarray*}
\un{\lambda} (\mathbf{a},\bm{\nu}) \ov{\mu} &\longmapsto& \left\lbrace \begin{array}{ll}
\un{\lambda} (\mathbf{b},\bm{\nu'}) \ov{\mu} &  \text{if $C_{l-1}$ and $C_l$ are both anticlockwise,}\\
\sigma_1 \un{\lambda} (\mathbf{b},\bm{\nu''}) \ov{\mu} &   \text{if $C_{l-1}$ is clockwise and}\\
& C_{l} \text{ is anticlockwise,} \\
\sigma_2 \un{\lambda} (\mathbf{b},\bm{\nu''}) \ov{\mu} &   \text{if $C_{l-1}$ is anticlockwise and}\\
&   \text{$C_{l}$ is clockwise,} \\
0 &   \text{if $C_{l-1}$ and $C_l$ are both clockwise,}
\end{array}\right.
\end{eqnarray*}
where $\bm{\nu'}$ respectively $\bm{\nu''}$ are obtained by changing $\bm{\nu}$ such that the component $C$ is oriented anticlockwise respectively clockwise. If $C$ cannot be oriented clockwise (i.e. it is a line) then the corresponding term is defined to be zero. Furthermore the signs $\sigma_1,\sigma_2\in\{\pm1\}$  are given by 
\begin{eqnarray*}
\sigma_1 &=& {\rm sign}_D(i,l-1){\rm sign}_{D'}(i,l-1) \quad \text{ and }\\
\sigma_2 &=& {\rm sign}_D(i,l){\rm sign}_{D'}(i,l).
\end{eqnarray*}
Instances of surgery maps in the merge case are the maps ${\rm surg}_{D,E_1}$ and ${\rm surg}_{D,E_2}$ in the two situations from Example \ref{ex:mergesplitexample}.

\begin{remark} \label{rem:interpretemerge}
The map ${\rm surg}_{D,D'}$ is based on the \emph{multiplication} map of the algebra $\mC[x]/(x^2)$ and the action on its trivial module $\mC = \mC y$ with basis $y$. 
\begin{equation*}
\begin{array}[t]{c|llll}
\mC[x]/(x^2)&1 \otimes 1 \mapsto 1,
& 1 \otimes x \mapsto x,
&
x \otimes 1 \mapsto x,
&
x \otimes x \mapsto 0, \\
\hline
\mC y &
 1 \otimes y \mapsto y, 
&& x \otimes y \mapsto 0, \nonumber
\end{array}
\end{equation*}
where $1$ is interpreted as an anticlockwise circle, $x$ as a clockwise circle and $y$ as a line, cf. \cite{BS1}. Merging a line with a circle looks as follows

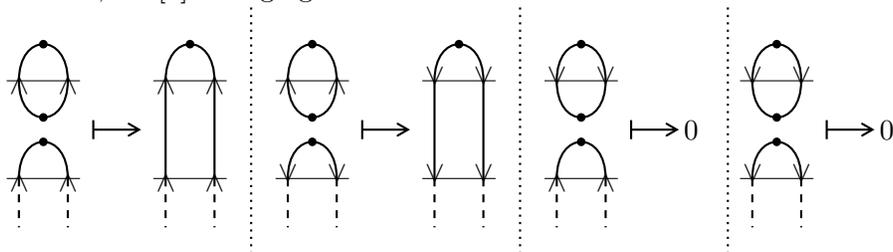
\begin{tikzpicture}[thick,>=angle 60, scale=0.65]

\draw (0,-1) .. controls +(0,1) and +(0,1) .. +(1,0);
\draw[thin] (-.25,-1) -- +(1.5,0);
\node at (0.005,-1.1) {$\up$};
\node at (1.005,-1.1) {$\up$};
\fill (0.5,-.25) circle(2.5pt);
\draw (0,-1) .. controls +(0,-1) and +(0,-1) .. +(1,0);
\fill (0.5,-1.75) circle(2.5pt);
\draw (0,-3) .. controls +(0,1) and +(0,1) .. +(1,0);
\draw[thin] (-.25,-3) -- +(1.5,0);
\node at (0.005,-3.1) {$\up$};
\node at (1.005,-3.1) {$\up$};
\draw[dashed](0,-3) -- +(0,-1);
\draw[dashed] (1,-3) -- +(0,-1);
\fill (0.5,-2.25) circle(2.5pt);

\draw[|->] (1.5,-2) -- +(1,0);

\draw (3,-1) .. controls +(0,1) and +(0,1) .. +(1,0);
\fill (3.5,-.25) circle(2.5pt);
\draw[thin] (2.75,-1) -- +(1.5,0);
\node at (3.005,-1.1) {$\up$};
\node at (4.005,-1.1) {$\up$};
\draw (3,-1) -- +(0,-2);
\draw (4,-1) -- +(0,-2);
\draw[thin] (2.75,-3) -- +(1.5,0);
\node at (3.005,-3.1) {$\up$};
\node at (4.005,-3.1) {$\up$};
\draw[dashed] (3,-3) -- +(0,-1);
\draw[dashed] (4,-3) -- +(0,-1);

\draw[dotted,thick] (4.75,.5) -- +(0,-5);

\begin{scope}[xshift=5.5cm]
\draw (0,-1) .. controls +(0,1) and +(0,1) .. +(1,0);
\draw[thin] (-.25,-1) -- +(1.5,0);
\node at (0.005,-1.1) {$\up$};
\node at (1.005,-1.1) {$\up$};
\fill (0.5,-.25) circle(2.5pt);
\draw (0,-1) .. controls +(0,-1) and +(0,-1) .. +(1,0);
\fill (0.5,-1.75) circle(2.5pt);
\draw (0,-3) .. controls +(0,1) and +(0,1) .. +(1,0);
\draw[thin] (-.25,-3) -- +(1.5,0);
\node at (0.005,-2.9) {$\down$};
\node at (1.005,-2.9) {$\down$};
\draw[dashed](0,-3) -- +(0,-1);
\draw[dashed] (1,-3) -- +(0,-1);
\fill (0.5,-2.25) circle(2.5pt);

\draw[|->] (1.5,-2) -- +(1,0);

\draw (3,-1) .. controls +(0,1) and +(0,1) .. +(1,0);
\fill (3.5,-.25) circle(2.5pt);
\draw[thin] (2.75,-1) -- +(1.5,0);
\node at (3.005,-0.9) {$\down$};
\node at (4.005,-0.9) {$\down$};
\draw (3,-1) -- +(0,-2);
\draw (4,-1) -- +(0,-2);
\draw[thin] (2.75,-3) -- +(1.5,0);
\node at (3.005,-2.9) {$\down$};
\node at (4.005,-2.9) {$\down$};
\draw[dashed] (3,-3) -- +(0,-1);
\draw[dashed] (4,-3) -- +(0,-1);

\draw[dotted,thick] (4.75,.5) -- +(0,-5);

\end{scope}

\begin{scope}[xshift=11cm]
\draw (0,-1) .. controls +(0,1) and +(0,1) .. +(1,0);
\draw[thin] (-.25,-1) -- +(1.5,0);
\node at (0.005,-0.9) {$\down$};
\node at (1.005,-0.9) {$\down$};
\fill (0.5,-.25) circle(2.5pt);
\draw (0,-1) .. controls +(0,-1) and +(0,-1) .. +(1,0);
\fill (0.5,-1.75) circle(2.5pt);
\draw (0,-3) .. controls +(0,1) and +(0,1) .. +(1,0);
\draw[thin] (-.25,-3) -- +(1.5,0);
\node at (0.005,-3.1) {$\up$};
\node at (1.005,-3.1) {$\up$};
\draw[dashed](0,-3) -- +(0,-1);
\draw[dashed] (1,-3) -- +(0,-1);
\fill (0.5,-2.25) circle(2.5pt);

\draw[|->] (1.5,-2) -- +(1,0);

\node at (2.75,-2) {$0$};

\draw[dotted,thick] (3.5,.5) -- +(0,-5);
\end{scope}

\begin{scope}[xshift=15cm]
\draw (0,-1) .. controls +(0,1) and +(0,1) .. +(1,0);
\draw[thin] (-.25,-1) -- +(1.5,0);
\node at (0.005,-0.9) {$\down$};
\node at (1.005,-0.9) {$\down$};
\fill (0.5,-.25) circle(2.5pt);
\draw (0,-1) .. controls +(0,-1) and +(0,-1) .. +(1,0);
\fill (0.5,-1.75) circle(2.5pt);
\draw (0,-3) .. controls +(0,1) and +(0,1) .. +(1,0);
\draw[thin] (-.25,-3) -- +(1.5,0);
\node at (0.005,-2.9) {$\down$};
\node at (1.005,-2.9) {$\down$};
\draw[dashed](0,-3) -- +(0,-1);
\draw[dashed] (1,-3) -- +(0,-1);
\fill (0.5,-2.25) circle(2.5pt);

\draw[|->] (1.5,-2) -- +(1,0);

\node at (2.75,-2) {$0$};
\end{scope}
\end{tikzpicture}

Here the dashed lines indicate that this piece of the diagram is contained in a component that forms a line. The formulas for lines arise from the formulas for the circles by remembering that lines can be only oriented anticlockwise. In particular the last two pictures give zero.
\end{remark}

\subsubsection{Split:} Denote by $C$ and $C_a$ respectively the components in $D$ containing $(i,l)$ and in $D'$ containing $(a,l)$ for $a=i,j$. The map ${\rm surg}_{D,D'}$ from \eqref{eqn:surgerymap} is then given by
\begin{eqnarray*}
\un{\lambda} (\mathbf{a},\bm{\nu}) \ov{\mu} \longmapsto \left\lbrace \begin{array}{ll}
(-1)^{{\rm p}(i)} \left(\sigma_1 \un{\lambda} (\mathbf{b},\bm{\nu'}) \ov{\mu} + \sigma_2 \un{\lambda} (\mathbf{b},\bm{\nu''}) \ov{\mu}\right) & \text{if $C$ is  anticlockwise} \\
& \text{and } i \DCupConnect j \text{ in } D,\\
(-1)^{{\rm p}(i)} \left(\sigma_1 \un{\lambda} (\mathbf{b},\bm{\nu'}) \ov{\mu} - \sigma_2 \un{\lambda} (\mathbf{b},\bm{\nu''}) \ov{\mu}\right) &  \text{if $C$  is anticlockwise} \\
&\text{and }  i \CupConnect j \text{ in } D,\\
(-1)^{{\rm p}(i)} \sigma_3 \un{\lambda} (\mathbf{b},\bm{\nu'''}) \ov{\mu} & \text{if $C$  is clockwise,}
\vspace{2mm}\\
0 & \text{if $D'$ is not orientable,}
\end{array}\right.
\end{eqnarray*}
where $\bm{\nu'}$ respectively $\bm{\nu''}$ are obtained by changing $\bm{\nu}$ such that the $C_j$ is oriented clockwise and $C_i$ is oriented anticlockwise respectively $C_j$ anticlockwise and $C_i$ clockwise, $\bm{\nu'''}$ is obtained by changing $\bm{\nu}$ such that both $C_i$ and $C_j$ are oriented clockwise. Furthermore the signs $\sigma_i$ are given by 
\begin{eqnarray*}
\sigma_1 &=& {\rm sign}_{D'}(j,l),\\
\sigma_2 &=& {\rm sign}_{D'}(i,l), \text{ and }\\
\sigma_3 &=& {\rm sign}_{D}(i,l) {\rm sign}_{D'}(i,l) {\rm sign}_{D'}(j,l).
\end{eqnarray*}

\begin{remark}
Using Remark~\ref{unverstformel}, the first two cases can in fact be combined into one case via the formula
$$\un{\lambda}(\mathbf{a},\bm{\nu})\ov{\mu} \longmapsto (-1)^{{\pos}(i)} \left(\sigma_1 \un{\lambda} (\mathbf{b},\bm{\nu'}) \ov{\mu} + {\rm sign}_D(i,l){\rm sign}_D(j,l) \sigma_2 \un{\lambda} (\mathbf{b},\bm{\nu''}) \ov{\mu}\right),$$
because ${\rm sign}_D(i,l){\rm sign}_D(j,l)$ equals $1$ if the cup-cap-pair is dotted and $-1$ if it is undotted.
\end{remark}

Instances for surgery maps in the split case are the maps ${\rm surg}_{D,E_1}$ and ${\rm surg}_{D,E_2}$ in the two situations in Example \ref{ex:splitmergeexample}. Furthermore Example \ref{ex:splitnotorientable} shows a surgery map in a split case, namely ${\rm surg}_{D,E_1}$,  where the resulting diagram is not orientable.

\begin{remark}
In this case the surgery map is based on the \emph{comultiplication} map of the algebra $\mC[x]/(x^2)$ and its trivial comodule $\mC y$
\begin{equation}
\label{comult}
\begin{array}[t]{c|cc}
\mC[x]/(x^2)&
1 \mapsto 1 \otimes x + x \otimes 1, & x \mapsto x \otimes x,\\
\hline
\mC y & y \mapsto y \otimes x, &\nonumber
\end{array}
\end{equation}
with the same interpretation of $1$ and $x$ as above in Remark \ref{rem:interpretemerge}. Again of special note is the rule involving lines, for example

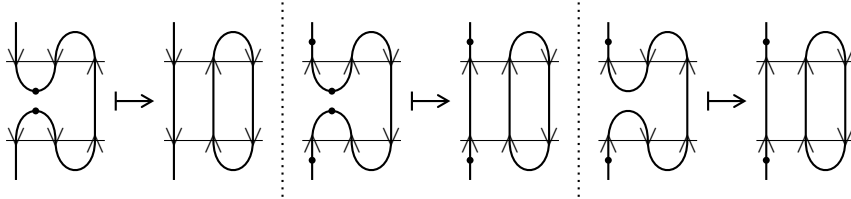
\begin{tikzpicture}[thick,>=angle 60, scale=0.525]

\draw (0,-1) -- +(0,1);
\draw (1,-1) .. controls +(0,1) and +(0,1) .. +(1,0);
\draw[thin] (-.25,-1) -- +(2.5,0);
\node at (0.005,-0.9) {$\down$};
\node at (1.005,-0.9) {$\down$};
\node at (2.005,-1.1) {$\up$};
\draw (0,-1) .. controls +(0,-1) and +(0,-1) .. +(1,0);
\fill (0.5,-1.75) circle(2.5pt);
\draw (2,-1) -- +(0,-2);
\draw (0,-3) .. controls +(0,1) and +(0,1) .. +(1,0);
\fill (0.5,-2.25) circle(2.5pt);
\draw (1,-3) .. controls +(0,-1) and +(0,-1) .. +(1,0);
\draw[thin] (-.25,-3) -- +(2.5,0);
\node at (0.005,-2.9) {$\down$};
\node at (1.005,-2.9) {$\down$};
\node at (2.005,-3.1) {$\up$};
\draw (0,-3) -- +(0,-1);

\draw[|->] (2.5,-2) -- +(1,0);

\draw (4,-1) -- +(0,1);
\draw (5,-1) .. controls +(0,1) and +(0,1) .. +(1,0);
\draw[thin] (3.75,-1) -- +(2.5,0);
\node at (4.005,-0.9) {$\down$};
\node at (5.005,-1.1) {$\up$};
\node at (6.005,-0.9) {$\down$};
\draw (4,-1) -- +(0,-2);
\draw (5,-1) -- +(0,-2);
\draw (6,-1) -- +(0,-2);
\draw (5,-3) .. controls +(0,-1) and +(0,-1) .. +(1,0);
\draw[thin] (3.75,-3) -- +(2.5,0);
\node at (4.005,-2.9) {$\down$};
\node at (5.005,-3.1) {$\up$};
\node at (6.005,-2.9) {$\down$};
\draw (4,-3) -- +(0,-1);

\draw[dotted,thick] (6.75,.5) -- +(0,-5);

\begin{scope}[xshift=7.5cm]
\draw (0,-1) -- +(0,1);
\fill (0,-.5) circle(2.5pt);
\draw (1,-1) .. controls +(0,1) and +(0,1) .. +(1,0);
\draw[thin] (-.25,-1) -- +(2.5,0);
\node at (0.005,-1.1) {$\up$};
\node at (1.005,-1.1) {$\up$};
\node at (2.005,-0.9) {$\down$};
\draw (0,-1) .. controls +(0,-1) and +(0,-1) .. +(1,0);
\fill (0.5,-1.75) circle(2.5pt);
\draw (2,-1) -- +(0,-2);
\draw (0,-3) .. controls +(0,1) and +(0,1) .. +(1,0);
\fill (0.5,-2.25) circle(2.5pt);
\draw (1,-3) .. controls +(0,-1) and +(0,-1) .. +(1,0);
\draw[thin] (-.25,-3) -- +(2.5,0);
\node at (0.005,-3.1) {$\up$};
\node at (1.005,-3.1) {$\up$};
\node at (2.005,-2.9) {$\down$};
\draw (0,-3) -- +(0,-1);
\fill (0,-3.5) circle(2.5pt);

\draw[|->] (2.5,-2) -- +(1,0);

\draw (4,-1) -- +(0,1);
\fill (4,-.5) circle(2.5pt);
\draw (5,-1) .. controls +(0,1) and +(0,1) .. +(1,0);
\draw[thin] (3.75,-1) -- +(2.5,0);
\node at (4.005,-1.1) {$\up$};
\node at (5.005,-1.1) {$\up$};
\node at (6.005,-0.9) {$\down$};
\draw (4,-1) -- +(0,-2);
\draw (5,-1) -- +(0,-2);
\draw (6,-1) -- +(0,-2);
\draw (5,-3) .. controls +(0,-1) and +(0,-1) .. +(1,0);
\draw[thin] (3.75,-3) -- +(2.5,0);
\node at (4.005,-3.1) {$\up$};
\node at (5.005,-3.1) {$\up$};
\node at (6.005,-2.9) {$\down$};
\draw (4,-3) -- +(0,-1);
\fill (4,-3.5) circle(2.5pt);

\draw[dotted,thick] (6.75,.5) -- +(0,-5);
\end{scope}

\begin{scope}[xshift=15cm]
\draw (0,-1) -- +(0,1);
\fill (0,-.5) circle(2.5pt);
\draw (1,-1) .. controls +(0,1) and +(0,1) .. +(1,0);
\draw[thin] (-.25,-1) -- +(2.5,0);
\node at (0.005,-1.1) {$\up$};
\node at (1.005,-0.9) {$\down$};
\node at (2.005,-1.1) {$\up$};
\draw (0,-1) .. controls +(0,-1) and +(0,-1) .. +(1,0);
\draw (2,-1) -- +(0,-2);
\draw (0,-3) .. controls +(0,1) and +(0,1) .. +(1,0);
\draw (1,-3) .. controls +(0,-1) and +(0,-1) .. +(1,0);
\draw[thin] (-.25,-3) -- +(2.5,0);
\node at (0.005,-3.1) {$\up$};
\node at (1.005,-2.9) {$\down$};
\node at (2.005,-3.1) {$\up$};
\draw (0,-3) -- +(0,-1);
\fill (0,-3.5) circle(2.5pt);

\draw[|->] (2.5,-2) -- +(1,0);

\draw (4,-1) -- +(0,1);
\fill (4,-.5) circle(2.5pt);
\draw (5,-1) .. controls +(0,1) and +(0,1) .. +(1,0);
\draw[thin] (3.75,-1) -- +(2.5,0);
\node at (4.005,-1.1) {$\up$};
\node at (5.005,-1.1) {$\up$};
\node at (6.005,-0.9) {$\down$};
\draw (4,-1) -- +(0,-2);
\draw (5,-1) -- +(0,-2);
\draw (6,-1) -- +(0,-2);
\draw (5,-3) .. controls +(0,-1) and +(0,-1) .. +(1,0);
\draw[thin] (3.75,-3) -- +(2.5,0);
\node at (4.005,-3.1) {$\up$};
\node at (5.005,-3.1) {$\up$};
\node at (6.005,-2.9) {$\down$};
\draw (4,-3) -- +(0,-1);
\fill (4,-3.5) circle(2.5pt);

\end{scope}
\end{tikzpicture}

Again we can think of this as a degenerate case of the first rule in \eqref{comult}, using that lines only admit anticlockwise orientations, which eliminates the second term in the formula for the comultiplication.
\end{remark}

\subsubsection{Reconnect:} This situation can only occur if the cup and cap lie on two distinct lines. In this case ${\rm surg}_{D,D'}$ is given by
$$\un{\lambda} (\mathbf{a},\bm{\nu}) \ov{\mu} \longmapsto \left\lbrace \begin{array}{ll}
\un{\lambda} (\mathbf{b},\bm{\nu}) \ov{\mu} & \text{if } \nu \text{ is an orientation for } D' \text{ and the }\\
& \text{two lines in $D$ were propagating,}\\
0 & \text{otherwise.}
\end{array}\right.$$

\begin{remark}
In the notation from Remark \ref{rem:interpretemerge}, this is the rule
\begin{eqnarray}
\label{yuk}
y \otimes y &\mapsto&
\left\{
\begin{array}{ll}
y\otimes y&\text{if both lines propagate and reconnecting}\\
&\text{as in \eqref{pic:surgery} gives an oriented diagram,}\\
0&\text{otherwise.}
\end{array}
\right.
\end{eqnarray}

Since lines have only one possible orientation, determined by the (forced) label on the rays at the ends of the line, the shape of the bottom cup diagram as well as the top cap diagram (in particular its rays) get not changed in the surgery  procedure; and therefore the labels don't change either. Hence the result is orientable if and only if we are in the first case of \eqref{yuk}. 
\end{remark}
\begin{example}
We now show some non-zero surgery moves for two lines.
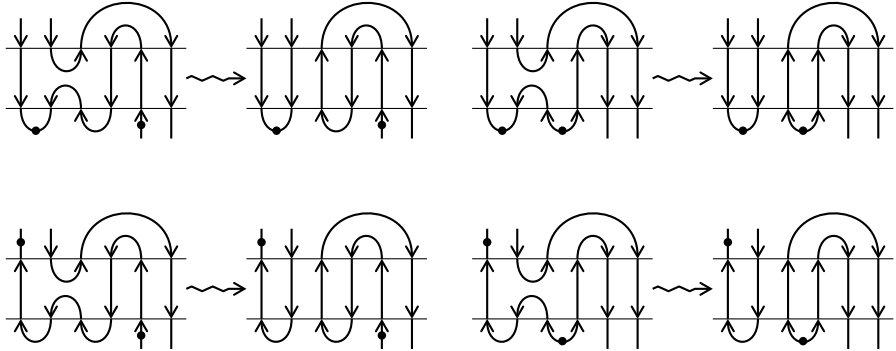
\begin{eqnarray*}
\usetikzlibrary{arrows}
\usetikzlibrary{decorations.pathmorphing}
\begin{tikzpicture}[thick,>=angle 60,scale=.4]
\draw[thin] (-.5,2) -- +(6,0);
\draw[thin] (-.5,0) -- +(6,0);
\draw [<-] (0,2) -- +(0,1);
\draw [<-] (1,2) -- +(0,1);
\draw [->] (2,2) .. controls +(0,2) and +(0,2) .. +(3,0);
\draw [<-] (3,2) .. controls +(0,1) and +(0,1) .. +(1,0);
\draw [<-] (0,0) -- +(0,2);
\draw [<-] (3,0) -- +(0,2);
\draw [->] (4,0) -- +(0,2);
\draw [->] (1,2) .. controls +(0,-1) and +(0,-1) .. +(1,0);
\draw [<-] (5,0) -- +(0,2);
\draw [<-] (1,0) .. controls +(0,1) and +(0,1) .. +(1,0);
\draw (0,0) .. controls +(0,-1) and +(0,-1) .. +(1,0);
\fill (0.5,-.74) circle(4pt);
\draw [<-] (2,0) .. controls +(0,-1) and +(0,-1) .. +(1,0);
\draw [->] (4,-1) -- +(0,1);
\fill (4,-.55) circle(4pt);
\draw (5,-1) -- +(0,1);

\draw [->,line join=round,decorate, decoration={zigzag,segment length=8,amplitude=.9,post=lineto,post length=4pt}]  (5.5,1) -- +(2,0);

\begin{scope}[xshift=8cm]
\draw[thin] (-.5,2) -- +(6,0);
\draw[thin] (-.5,0) -- +(6,0);
\draw [<-] (0,2) -- +(0,1);
\draw [<-] (1,2) -- +(0,1);

\draw [<-] (0,0) -- +(0,2);
\draw [<-] (1,0) -- +(0,2);
\draw [->] (2,0) -- +(0,2);
\draw [<-] (3,0) -- +(0,2);
\draw [->] (4,0) -- +(0,2);
\draw [<-] (5,0) -- +(0,2);

\draw [->] (2,2) .. controls +(0,2) and +(0,2) .. +(3,0);
\draw [<-] (3,2) .. controls +(0,1) and +(0,1) .. +(1,0);
\draw (0,0) .. controls +(0,-1) and +(0,-1) .. +(1,0);
\fill (0.5,-.74) circle(4pt);
\draw [<-] (2,0) .. controls +(0,-1) and +(0,-1) .. +(1,0);
\draw [->] (4,-1) -- +(0,1);
\fill (4,-.55) circle(4pt);
\draw (5,-1) -- +(0,1);
\end{scope}

\begin{scope}[xshift=15.5cm]
\draw[thin] (-.5,2) -- +(6,0);
\draw[thin] (-.5,0) -- +(6,0);
\draw [<-] (0,2) -- +(0,1);
\draw [<-] (1,2) -- +(0,1);
\draw [->] (2,2) .. controls +(0,2) and +(0,2) .. +(3,0);
\draw [->] (3,2) .. controls +(0,1) and +(0,1) .. +(1,0);

\draw [<-] (0,0) -- +(0,2);
\draw [->] (3,0) -- +(0,2);
\draw [<-] (4,0) -- +(0,2);
\draw [->] (1,2) .. controls +(0,-1) and +(0,-1) .. +(1,0);
\draw [<-] (5,0) -- +(0,2);
\draw [<-] (1,0) .. controls +(0,1) and +(0,1) .. +(1,0);

\draw (0,0) .. controls +(0,-1) and +(0,-1) .. +(1,0);
\fill (0.5,-.74) circle(4pt);
\draw [<->] (2,0) .. controls +(0,-1) and +(0,-1) .. +(1,0);
\fill (2.5,-.74) circle(4pt);
\draw (4,-1) -- +(0,1);
\draw (5,-1) -- +(0,1);

\draw [->,line join=round,decorate, decoration={zigzag,segment length=8,amplitude=.9,post=lineto,post length=4pt}]  (5.5,1) -- +(2,0);

\begin{scope}[xshift=8cm]
\draw[thin] (-.5,2) -- +(6,0);
\draw[thin] (-.5,0) -- +(6,0);
\draw [<-] (0,2) -- +(0,1);
\draw [<-] (1,2) -- +(0,1);

\draw [<-] (0,0) -- +(0,2);
\draw [<-] (1,0) -- +(0,2);
\draw [->] (2,0) -- +(0,2);
\draw [->] (3,0) -- +(0,2);
\draw [<-] (4,0) -- +(0,2);
\draw [<-] (5,0) -- +(0,2);

\draw [->] (2,2) .. controls +(0,2) and +(0,2) .. +(3,0);
\draw [->] (3,2) .. controls +(0,1) and +(0,1) .. +(1,0);
\draw (0,0) .. controls +(0,-1) and +(0,-1) .. +(1,0);
\fill (0.5,-.74) circle(4pt);
\draw [<->] (2,0) .. controls +(0,-1) and +(0,-1) .. +(1,0);
\fill (2.5,-.74) circle(4pt);
\draw (4,-1) -- +(0,1);
\draw (5,-1) -- +(0,1);
\end{scope}
\end{scope}

\begin{scope}[yshift=-7cm]
\draw[thin] (-.5,2) -- +(6,0);
\draw[thin] (-.5,0) -- +(6,0);
\draw (0,2) -- +(0,1);
\fill (0,2.55) circle(4pt);
\draw [<-] (1,2) -- +(0,1);
\draw [->] (2,2) .. controls +(0,2) and +(0,2) .. +(3,0);
\draw [<-] (3,2) .. controls +(0,1) and +(0,1) .. +(1,0);
\draw [->] (0,0) -- +(0,2);
\draw [<-] (3,0) -- +(0,2);
\draw [->] (4,0) -- +(0,2);
\draw [->] (1,2) .. controls +(0,-1) and +(0,-1) .. +(1,0);
\draw [<-] (5,0) -- +(0,2);
\draw [<-] (1,0) .. controls +(0,1) and +(0,1) .. +(1,0);
\draw [<-] (0,0) .. controls +(0,-1) and +(0,-1) .. +(1,0);
\draw [<-] (2,0) .. controls +(0,-1) and +(0,-1) .. +(1,0);
\draw [->] (4,-1) -- +(0,1);
\fill (4,-.55) circle(4pt);
\draw (5,-1) -- +(0,1);

\draw [->,line join=round,decorate, decoration={zigzag,segment length=8,amplitude=.9,post=lineto,post length=4pt}]  (5.5,1) -- +(2,0);

\begin{scope}[xshift=8cm]
\draw[thin] (-.5,2) -- +(6,0);
\draw[thin] (-.5,0) -- +(6,0);
\draw (0,2) -- +(0,1);
\fill (0,2.55) circle(4pt);
\draw [<-] (1,2) -- +(0,1);
\draw [->] (0,0) -- +(0,2);
\draw [<-] (1,0) -- +(0,2);
\draw [->] (2,0) -- +(0,2);
\draw [<-] (3,0) -- +(0,2);
\draw [->] (4,0) -- +(0,2);
\draw [<-] (5,0) -- +(0,2);

\draw [->] (2,2) .. controls +(0,2) and +(0,2) .. +(3,0);
\draw [<-] (3,2) .. controls +(0,1) and +(0,1) .. +(1,0);
\draw [<-] (0,0) .. controls +(0,-1) and +(0,-1) .. +(1,0);
\draw [<-] (2,0) .. controls +(0,-1) and +(0,-1) .. +(1,0);
\draw [->] (4,-1) -- +(0,1);
\fill (4,-.55) circle(4pt);
\draw (5,-1) -- +(0,1);
\end{scope}
\end{scope}

\begin{scope}[xshift=15.5cm,yshift=-7cm]
\draw[thin] (-.5,2) -- +(6,0);
\draw[thin] (-.5,0) -- +(6,0);
\draw (0,2) -- +(0,1);
\fill (0,2.55) circle(4pt);
\draw [<-] (1,2) -- +(0,1);
\draw [->] (2,2) .. controls +(0,2) and +(0,2) .. +(3,0);
\draw [->] (3,2) .. controls +(0,1) and +(0,1) .. +(1,0);

\draw [->] (0,0) -- +(0,2);
\draw [->] (3,0) -- +(0,2);
\draw [<-] (4,0) -- +(0,2);
\draw [->] (1,2) .. controls +(0,-1) and +(0,-1) .. +(1,0);
\draw [<-] (5,0) -- +(0,2);
\draw [<-] (1,0) .. controls +(0,1) and +(0,1) .. +(1,0);

\draw [<-] (0,0) .. controls +(0,-1) and +(0,-1) .. +(1,0);
\draw [<->] (2,0) .. controls +(0,-1) and +(0,-1) .. +(1,0);
\fill (2.5,-.74) circle(4pt);
\draw (4,-1) -- +(0,1);
\draw (5,-1) -- +(0,1);

\draw [->,line join=round,decorate, decoration={zigzag,segment length=8,amplitude=.9,post=lineto,post length=4pt}]  (5.5,1) -- +(2,0);

\begin{scope}[xshift=8cm]
\draw[thin] (-.5,2) -- +(6,0);
\draw[thin] (-.5,0) -- +(6,0);
\draw (0,2) -- +(0,1);
\fill (0,2.55) circle(4pt);
\draw [<-] (1,2) -- +(0,1);

\draw [->] (0,0) -- +(0,2);
\draw [<-] (1,0) -- +(0,2);
\draw [->] (2,0) -- +(0,2);
\draw [->] (3,0) -- +(0,2);
\draw [<-] (4,0) -- +(0,2);
\draw [<-] (5,0) -- +(0,2);

\draw [->] (2,2) .. controls +(0,2) and +(0,2) .. +(3,0);
\draw [->] (3,2) .. controls +(0,1) and +(0,1) .. +(1,0);
\draw [<-] (0,0) .. controls +(0,-1) and +(0,-1) .. +(1,0);
\draw [<->] (2,0) .. controls +(0,-1) and +(0,-1) .. +(1,0);
\fill (2.5,-.74) circle(4pt);
\draw (4,-1) -- +(0,1);
\draw (5,-1) -- +(0,1);
\end{scope}
\end{scope}
\end{tikzpicture}
\end{eqnarray*}
\end{example}
This completes the diagrammatic definition of the surgery procedure.

\subsection{Reduction to circle diagrams without lines}\
 \label{annoying}
By Definition \ref{def:componentorientation} we call the unique orientation of a line anticlockwise. This is made precise in the following observation, which also explains the special multiplication rule for the reconnect case in a more conceptual way, see \cite{Str09} for a similar observation in type $\rm A$. 

Fix a block $\Lambda = \Lambda_\Theta^{\overline{\epsilon}}$ with $s = | P_\bu(\Lambda)|$. Define a new block sequence $\widehat{\Theta}$ by replacing $s$ times the symbol $\circ$ by $\bu$ to the far right, i.e. choose $r$ minimal with $r > i$ for all $i \in P_\times(\Lambda) \cup P_\bu(\Lambda)$ and set $\widehat{\Theta}_{r+j} = \bu$ for $0 \leq j < s$ and $\widehat{\Theta}_i=\Theta_i$ otherwise. An example (with $s=2$ and $r=4$) would look like this
\begin{eqnarray} \label{eqn:blockadd}
\Theta = \bu \bu \times \circ \circ \circ \circ \, \cdots \qquad \rightsquigarrow \qquad  \widehat{\Theta}=\bu \bu \times \bu \bu \circ \circ \, \cdots.
\end{eqnarray}
This procedure defines a new block $\widehat{\Lambda} = \Lambda_{\widehat{\Theta}}^{\overline{\epsilon+s}}$ and any oriented cup diagram $\un\la\nu$ for $\Lambda$ can be turned into a diagram $\un{\widehat{\la}}\widehat{\nu}$ for $\widehat{\Lambda}$ in the following way
\begin{enumerate}[{\it Step 1:}]
\item Set $\widehat{\nu}_{r+j}=\up$ for $0 \leq j < s$ and otherwise $\widehat{\nu}_i=\nu_i$.
\item Starting from right to left replace any ray in $\un\la$ by a cup connecting the position for the ray with the first (from the left) free symbol $\up$ at positions $r,\ldots,r+s-1$. We choose the decoration of the cup to be the decoration of the deleted ray.
\item If there are symbols at positions $r,\ldots,r+s-1$ not connected to a cup, then connect neighbouring symbols from left to right by dotted cups.
\end{enumerate}
Due to admissibility of the original diagram this procedure is well-defined, and yields an (admissible) oriented cup diagram without rays. Here are two examples for the block sequence in \eqref{eqn:blockadd}:
\begin{eqnarray*}
\begin{tikzpicture}[thick,>=angle 60, scale=0.525]

\node at (0.005,-0.9) {$\down$};
\node at (1.005,-0.9) {$\down$};
\draw (0,-1) .. controls +(0,-1) and +(0,-1) .. +(1,0);
\fill (0.5,-1.75) circle(2.5pt);
\node at (2.005,-1) {$\times$};

\node at (3,-1.25) {$\rightsquigarrow$};

\node at (4.005,-0.9) {$\down$};
\node at (5.005,-0.9) {$\down$};
\draw (4,-1) .. controls +(0,-1) and +(0,-1) .. +(1,0);
\fill (4.5,-1.75) circle(2.5pt);
\node at (6.005,-1) {$\times$};
\node at (7.005,-1.1) {$\up$};
\node at (8.005,-1.1) {$\up$};
\draw (7,-1) .. controls +(0,-1) and +(0,-1) .. +(1,0);
\fill (7.5,-1.75) circle(2.5pt);

\begin{scope}[xshift=12cm]

\node at (0.005,-0.9) {$\down$};
\node at (1.005,-0.9) {$\down$};
\draw (0,-1) -- +(0,-1);
\draw (1,-1) -- +(0,-1);
\node at (2.005,-1) {$\times$};

\node at (3,-1.25) {$\rightsquigarrow$};

\node at (4.005,-0.9) {$\down$};
\node at (5.005,-0.9) {$\down$};
\draw (4,-1) .. controls +(0,-2) and +(0,-2) .. +(4,0);
\draw (5,-1) .. controls +(0,-1) and +(0,-1) .. +(2,0);
\node at (6.005,-1) {$\times$};
\node at (7.005,-1.1) {$\up$};
\node at (8.005,-1.1) {$\up$};

\end{scope}
\end{tikzpicture}
\end{eqnarray*}
By definition, all newly created cups will be oriented anticlockwise. This construction can also be applied to cap diagrams. In the case of stacked circle diagrams we perform this procedure only for the top cap diagram and the bottom cup diagram, for all internal cup diagrams we simple add $s$ undotted rays at positions $r,\ldots,r+s-1$. In this case any newly created circles, especially those obtained by replacing lines, are oriented anticlockwise. In particular we have embeddings
\begin{eqnarray}
\label{iota}
\iota \,:\, {}_\lambda (\D^{\mathbf{a}})_\mu &\hookrightarrow& {}_{\widehat{\lambda}} (\mD_{\widehat{\Lambda}}^{\widehat{\mathbf{a}}})_{\widehat{\mu}},
\end{eqnarray}
sending an oriented stacked circle diagram to its extension. The images have as basis the oriented diagrams inside ${}_{\widehat{\lambda}} (\mD_{\widehat{\Lambda}}^{\widehat{\mathbf{a}}})_{\widehat{\mu}}$ such that all labels at positions $(r,l),\ldots,(r+s-1,l)$ for $0 \leq l \leq h$ are $\up$. We denote by $\pi$ the projections onto these subspaces sending oriented stacked circle diagrams not contained in the images to zero.

The surgery rules Merge and Split applied to lines can now be replaced by the surgery rule on these extended diagrams followed by a projection onto the vector space spanned by diagrams where all the labels at positions $r,\ldots,r+s-1$ are $\up$.

The interpretation of the reconnect rule for two lines $L,L'$ needs some further explanation. Observe that, because of admissibility, there are no non-propagating lines to the right of a propagating line. This implies that when extending the diagram, the propagating lines are connected in order from right to left to the first new symbols in order from left to right and each such line produces exactly one circle. Moreover, if a non-propagating line connects positions $i$ and $j$, there can be no further line attached to any position between $i$ and $j$ by admissibility. This implies that the two ends of non-propagating lines are connected to two neighbouring new vertices.\\ \noindent
$\blacktriangleright$ If $L$ and $L'$ are propagating then the surgery on the extended diagram is a Merge and creates an anticlockwise circle (by the first observation). If the orientations do not match already we need to change one of the orientation of one of the original circles to obtain the new orientation, which changes one of the newly created labels to $\down$ which is sent to zero by the projection.\\ \noindent
$\blacktriangleright$ If $L$ is propagating and $L'$ non-propagating then the surgery on the extended diagram is a Merge. Since the surgery is assumed to produce an admissible diagram, there can't be any dots to the right of the cup-cap-pair on either of the two lines. This implies that the cup/cap belonging to the propagating line is automatically oriented anticlockwise, whereas the one of the non-propagating line is automatically oriented clockwise. In this case we need to reorient, producing a diagram killed by the projection.\\ \noindent
$\blacktriangleright$ If $L$ and $L'$ are non-propagating both Merge and Split can occur. If the surgery on the extended diagram is a Split, it always produces one clockwise circle. The result is killed by the projection since the tag of the clockwise circle is one of the newly created vertices. In the Merge situation we have the following two possible scenarios (and their mirror images), where in case i) the black dashed part can't have any dots and is therefore not orientable, while in case ii) some of the red dashed parts have to carry dots, but the Merge reorients one of the circles, hence the result is again zero after applying the projection $\pi$.
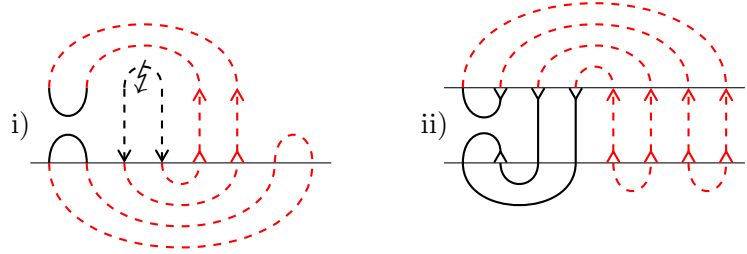
\vspace{-0.5cm}
\begin{eqnarray*}
\begin{tikzpicture}[thick,>=angle 60,scale=.5]
\node at (-.75,1) {i)};
\draw[thin] (-.5,0) -- +(8,0);
\draw[-] (0,2) .. controls +(0,-1) and +(0,-1) .. +(1,0);
\draw[dashed] (2,2) .. controls +(0,.75) and +(0,.75) .. +(1,0);
\draw[dashed, red] (0,2) .. controls +(0,2.25) and +(0,2.25) .. +(5,0);
\draw[dashed, red] (1,2) .. controls +(0,1.5) and +(0,1.5) .. +(3,0);
\draw[dashed,<-] (2,0) -- +(0,2);
\node at (2.5,2.3) {\Large $\lightning$};
\draw[dashed,<-] (3,0) -- +(0,2);
\draw[dashed,red,>->] (4,0) -- +(0,2);
\draw[dashed,red,>->] (5,0) -- +(0,2);
\draw[dashed, red] (2,0) .. controls +(0,-1.5) and +(0,-1.5) .. +(3,0);
\draw[dashed, red] (3,0) .. controls +(0,-.75) and +(0,-.75) .. +(1,0);
\draw[-] (0,0) .. controls +(0,1) and +(0,1) .. +(1,0);
\draw[dashed, red] (1,0) .. controls +(0,-2.25) and +(0,-2.25) .. +(5,0);
\draw[dashed, red] (0,0) .. controls +(0,-3) and +(0,-3) .. +(7,0);
\draw[dashed, red] (6,0) .. controls +(0,1) and +(0,1) .. +(1,0);

\begin{scope}[xshift=11cm]
\node at (-.75,1) {ii)};
\draw[thin] (-.5,2) -- +(8,0);
\draw[thin] (-.5,0) -- +(8,0);
\draw[dashed, red] (2,2) .. controls +(0,1.5) and +(0,1.5) .. +(3,0);
\draw[dashed, red] (3,2) .. controls +(0,.75) and +(0,.75) .. +(1,0);
\draw[dashed, red] (1,2) .. controls +(0,2.25) and +(0,2.25) .. +(5,0);
\draw[dashed, red] (0,2) .. controls +(0,3) and +(0,3) .. +(7,0);
\draw[-<] (0,2) .. controls +(0,-1) and +(0,-1) .. +(1,0);
\draw[-<] (0,0) .. controls +(0,1) and +(0,1) .. +(1,0);
\draw[-<] (2,0) -- +(0,2);
\draw[-<] (3,0) -- +(0,2);
\draw[dashed,red,>->] (4,0) -- +(0,2);
\draw[dashed,red,>->] (5,0) -- +(0,2);
\draw[dashed,red,>->] (6,0) -- +(0,2);
\draw[dashed,red,>->] (7,0) -- +(0,2);
\draw[-] (1,0) .. controls +(0,-.75) and +(0,-.75) .. +(1,0);
\draw[-] (0,0) .. controls +(0,-1.5) and +(0,-1.5) .. +(3,0);
\draw[dashed, red] (4,0) .. controls +(0,-1) and +(0,-1) .. +(1,0);
\draw[dashed, red] (6,0) .. controls +(0,-1) and +(0,-1) .. +(1,0);
\end{scope}
\end{tikzpicture}
\end{eqnarray*}

To summarise:
\begin{prop}
Our original surgery procedure agrees with the following procedure: first extend the diagrams using $\iota$ from \eqref{iota}, then perform the extended surgery, followed by the projection $\pi$ and finally take the preimage $\iota^{-1}$ of the resulting diagram.
\end{prop}

Let $N( \mathbf{a},\lambda,\mu) \subset {}_{\widehat{\lambda}} (\mD_{\widehat{\Lambda}}^{\widehat{\mathbf{a}}})_{\widehat{\mu}}$ denote the complement spanned by the diagrams where at least one of the labels at a vertex greater than $r$ is $\down$. Then by definition of Merge and Split
$$ {\rm surg}_{\widehat{D},\widehat{D'}}\left(N(\bm a,\lambda,\mu) \right) \subset N(\bm b,\lambda,\mu),$$
where $D'=\un\la({\bm a})\ov{\mu}$ and $D'=\un\la({\bm b})\ov{\mu}$.
For the algebra $\mD_{\widehat{\Lambda}}$ from Theorem \ref{algebra_structure} this implies inclusions
$$ \bigoplus_{\lambda,\mu \in \Lambda} N(\emptyset,\lambda,\mu) \subset \bigoplus_{\lambda,\mu \in \Lambda} {}_{\widehat{\lambda}}(\mD_{\widehat{\Lambda}})_{\widehat{\mu}} \subset \mD_{\widehat{\Lambda}},$$
the first as an ideal, and the second as a subalgebra. Hence 
\begin{eqnarray} \label{eqn:algebra_as_quotient}
\D &\cong& \bigoplus_{\lambda,\mu \in \Lambda} {}_{\widehat{\lambda}}(\mD_{\widehat{\Lambda}})_{\widehat{\mu}} \, / \bigoplus_{\lambda,\mu \in \Lambda} N(\emptyset,\lambda,\mu)
\end{eqnarray}
as algebras. This isomorphism of algebras allows us (namely via the extended surgery rule) to express the multiplication in $\D$ in terms of circle diagrams without rays.

\subsection{Surgery and grading}
We show the following.

\begin{prop}
\label{forgotgraded}
Let $D$ and $D'$ be stacked circle diagrams such that  $D'$ is obtained from $D$ by a surgery. Then the map ${\rm surg}_{D,D'}$ is of degree zero.
\end{prop}

\begin{proof}
Let us first assume that the surgery is performed at an undotted cup-cap-pair then we have one of the following situations:
\begin{equation} \label{eqn:cup-cap-config}
i)\quad
\usetikzlibrary{arrows}
\begin{tikzpicture}[thick,>=angle 60, scale=0.8]
\draw [>->] (0,0) .. controls +(0,-1) and +(0,-1) .. +(1,0);
\draw [<-<] (0,-2) .. controls +(0,1) and +(0,1) .. +(1,0);
\end{tikzpicture}
\quad\quad
ii)
\quad
\usetikzlibrary{arrows}
\begin{tikzpicture}[thick,>=angle 60, scale=0.8]
\draw [<-<] (0,0) .. controls +(0,-1) and +(0,-1) .. +(1,0);
\draw [>->] (0,-2) .. controls +(0,1) and +(0,1) .. +(1,0);
\end{tikzpicture}
\quad\quad
iii)
\quad
\usetikzlibrary{arrows}
\begin{tikzpicture}[thick,>=angle 60, scale=0.8]
\draw [>->] (0,0) .. controls +(0,-1) and +(0,-1) .. +(1,0);
\draw [>->] (0,-2) .. controls +(0,1) and +(0,1) .. +(1,0);
\end{tikzpicture}
\quad\quad
iv)
\quad
\usetikzlibrary{arrows}
\begin{tikzpicture}[thick,>=angle 60, scale=0.8]
\draw [<-<] (0,0) .. controls +(0,-1) and +(0,-1) .. +(1,0);
\draw [<-<] (0,-2) .. controls +(0,1) and +(0,1) .. +(1,0);
\end{tikzpicture}
\end{equation}
We distinguish between the cases Merge, Split and Reconnect.

Assume first that we have a merge from two circles, say $C_1$ and $C_2 $ to a circle $C$. Assume $C_1$ contains the cup. If $C_1$ and $C_2$  are both oriented clockwise, there is nothing to check since the map is zero. In case one of the two components is a line we use Section~\ref{annoying} to think of the line as an anticlockwise circle.\\ \noindent
$\blacktriangleright$ \textit{Case i):} If both circles are anticlockwise, the surgery sticks them together and does not change the orientation. Removing the cup-cap-pair clearly preserves the degree. If one, say $C_a$, is clockwise then by Lemma~ \ref{lem:circledomanating} below it must contain all possible tags of $C$ and the surgery does not change its orientation, again removing the cup-cap-pair preserves the degree.\\ \noindent
$\blacktriangleright$ \textit{Case ii):} If both are anticlockwise, then by Lemma \ref{lem:circledomanating} both must contain the possible tags of $C$, which is impossible. If one, say $C_a$, is clockwise, then the other one, say $C_b$, contains all the possible tags. Hence surgery will stick them together and for the orientation we need to change all labels of vertices in both circles, preserving the degree and afterwards the cup-cap-pair eliminated is of degree zero as well.\\ \noindent
$\blacktriangleright$ \textit{Case iii):} If both are anticlockwise, then by Lemma \ref{lem:circledomanating} $C_2$ contains all possible tags of $C$, thus to obtain the orientation of the image we need to reorient $C_1$ increasing the degree by $2$, but afterwards the cup-cap-pair eliminated is of degree $2$ as well, giving us degree zero in total. If $C_1$ is clockwise and $C_2$ is anticlockwise, then by Lemma~\ref{lem:circledomanating} both contain all tags, which is impossible. If on the contrary $C_1$ is anticlockwise and $C_2$ clockwise then no matter which of the two circles contains a tag of $C$ we need to reorient $C_1$. We first increase the degree of the diagram by $2$ and afterwards delete a cup-cap-pair of degree $2$. If one component is a line then by definition this must be $C_2$ thus the argument is valid as well.\\ \noindent
$\blacktriangleright$ \textit{Case iv):} This is done exactly analogous to Case iii) by just switching the roles of $C_1$ and $C_2$.

Since we only used Lemma \ref{lem:circledomanating} to see which circle must be reoriented, putting dots on the cup-cap-pair in \eqref{eqn:cup-cap-config} and changing the left labels does not change any of the arguments. Hence the claim follows for merges.

Assume now that the surgery is a split. Let  $C$ be the original circle and assume the surgery is performed at positions $i < j$ at level $l$ and denote by $C_i$ the circle containing $(i,l)$ and by $C_j$ the circle containing $(j,l)$ inside $D'$. We can assume that $D'$ is orientable, otherwise the surgery is zero anyway. Observe that due to orientability it follows that Cases iii) and iv) are not possible. The general shapes of the possible configurations are 
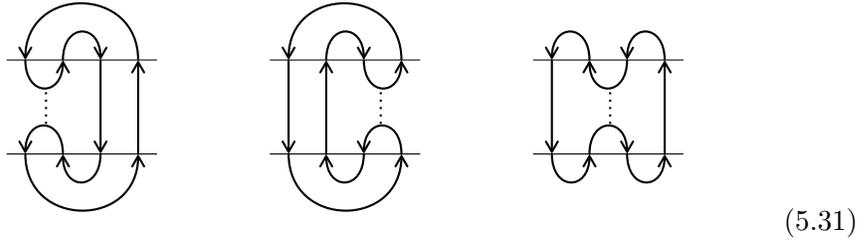
\begin{equation}
\label{kidneys}
\begin{array}[t]{c}
\usetikzlibrary{arrows}
\begin{tikzpicture}[thick,>=angle 60,scale=.5]
\draw[thin] (1.5,2) -- +(4,0);
\draw[thin] (1.5,-.5) -- +(4,0);
\draw [<-] (2,2) .. controls +(0,2) and +(0,2) .. +(3,0);
\draw [->] (3,2) .. controls +(0,1) and +(0,1) .. +(1,0);

\draw [<-] (4,-.5) -- +(0,2.5);
\draw [->] (2,2) .. controls +(0,-
1) and +(0,-1) .. +(1,0);
\draw[dotted] (2.55,.3) -- +(0,.95);
\draw [<-] (2,-.5) .. controls +(0,1) and +(0,1) .. +(1,0);
\draw [->] (5,-.5) -- +(0,2.5);

\draw [->] (2,-.5) .. controls +(0,-2) and +(0,-2) .. +(3,0);
\draw [<-] (3,-.5) .. controls +(0,-1) and +(0,-1) .. +(1,0);

\begin{scope}[xshift=14cm,xscale=-1]
\draw[thin] (1.5,2) -- +(4,0);
\draw[thin] (1.5,-.5) -- +(4,0);
\draw [->] (2,2) .. controls +(0,2) and +(0,2) .. +(3,0);
\draw [<-] (3,2) .. controls +(0,1) and +(0,1) .. +(1,0);

\draw [->] (4,-.5) -- +(0,2.5);
\draw [<-] (2,2) .. controls +(0,-
1) and +(0,-1) .. +(1,0);
\draw[dotted] (2.55,.3) -- +(0,.95);
\draw [->] (2,-.5) .. controls +(0,1) and +(0,1) .. +(1,0);
\draw [<-] (5,-.5) -- +(0,2.5);

\draw [<-] (2,-.5) .. controls +(0,-2) and +(0,-2) .. +(3,0);
\draw [->] (3,-.5) .. controls +(0,-1) and +(0,-1) .. +(1,0);
\end{scope}

\begin{scope}[xshift=14cm]
\draw[thin] (1.5,2) -- +(4,0);
\draw[thin] (1.5,-.5) -- +(4,0);

\draw [<-] (2,2) .. controls +(0,1) and +(0,1) .. +(1,0);
\draw [<-] (4,2) .. controls +(0,1) and +(0,1) .. +(1,0);

\draw [<-] (2,-.5) -- +(0,2.5);
\draw [<-] (3,2) .. controls +(0,-1) and +(0,-1) .. +(1,0);
\draw[dotted] (3.55,.3) -- +(0,.95);
\draw [->] (3,-.5) .. controls +(0,1) and +(0,1) .. +(1,0);
\draw [->] (5,-.5) -- +(0,2.5);

\draw [->] (2,-.5) .. controls +(0,-1) and +(0,-1) .. +(1,0);
\draw [->] (4,-.5) .. controls +(0,-1) and +(0,-1) .. +(1,0);
\end{scope}
\end{tikzpicture}
\end{array}
\end{equation}
with possibly additionally some dots. Again if one of the two components is a line we use Section~\ref{annoying} to reduce it to the circle case.\\ \noindent
Consider  the first shape in \eqref{kidneys}. Assume $C$ is anticlockwise and the sequence of arcs connecting $(i,l)$ and a tag of $C$ contains an even number of dots then the cup-cap-pair is of Case i). The split will create $C_i$ anticlockwise and $C_j$ clockwise if we leave the orientation preserving the degree. Again, to obtain the second term one reorients both circles without changing the degree. If the original circle was the closure of a line then $C_i$ would be after the surgery, so the arguments are also valid in this case. If $C$ is clockwise and the condition on dots stays the same we see that the cup-cap-pair is of Case ii). The split creates $C_i$ clockwise and $C_j$ anticlockwise if we leave the orientation. Thus we need to reorient $C_j$ increasing the degree by $2$ which cancels with the fact that we deleted a degree $2$ cup-cap-pair. If we choose an odd number of dots and $C$ anticlockwise, then the cup-cap pair is of Case ii) and both $C_i$ and $C_j$ are anticlockwise with the original orientation. In this case we need to reorient one of them, but also deleted a degree $2$ cup-cap pair, thus preserving the degree. If we choose an odd number of dots and $C$ clockwise, both circles would be clockwise automatically.

Now consider the second shape in \eqref{kidneys}. If the circle $C$ is oriented anticlockwise the cup-cap-pair is of Case i) and then $C_i$ is clockwise and $C_j$ anticlockwise if stitched together with the same orientation and removing the degree $0$ cup-cap-pair. To obtain the second term in the surgery we need to reorient both circles afterwards, but the two degree changed cancel. If the circle is oriented clockwise the cup-cap-pair is of Case ii) we split into $C_i$ anticlockwise and $C_j$ clockwise if we preserve the orientation. Thus we need to change the label of all the vertices in $C_i$ increasing the degree by $2$, but this cancels with the fact that we deleted a degree $2$ cup-cap-pair. The arguments for the third shape in \eqref{kidneys} are analogous.

In case our cup-cap-pair is dotted, then we just consider the shapes displayed in \eqref{4cases} and apply the same arguments.

Assume finally that we have a reconnect. In this case the map is only non-zero for propagating lines  where the orientations match. We leave it to the reader to check that, since there are no dots to the right of the two propagating lines, we are back in Case i) and the claim is clear.
\end{proof}

\begin{lemma}\label{lem:circledomanating}
Let $\un\la(\bm a,\bm \nu)\ov{\mu}$ be an oriented stacked circle diagram which allows a Merge at positions $i<j$ at level $l$ resulting in a circle $C'$. Denote by $\gamma$ one of the two arcs connecting $(i,l)$ with $(j,l)$ or $(i,l-1)$ with $(j,l-1)$ and let $C$ be the circle containing $\gamma$. Then if $\gamma$ and $C$ have different orientations then $C$ contains all possible tags of $C'$.
\end{lemma}
\begin{proof}
Let us first assume $\gamma$ is an undotted cup. Choose a  tag $(c,l')$ for $C$ and look at a path $\bm \alpha$ of arcs connecting $(j,l)$ and $(c,l')$  not containing $(i,l)$.  With  $\bm \alpha$ indicated by a dotted curve, the two possibilities  for the non-ray arc  closest to $(c,l')$ in $\bm \alpha$ are shown below, namely  a cap in i) and a cup in ii).
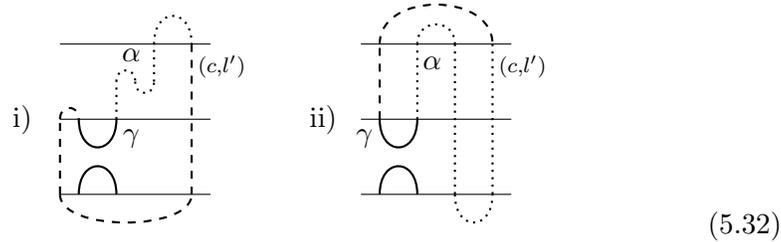
\begin{eqnarray} \label{eqn:dominatingcirles}
\begin{tikzpicture}[thick,decoration={zigzag,pre length=1mm,post length=1mm}]
\begin{scope}[xshift=4cm]
\node at (-.75,1) {ii)};

\draw[thick] (0,1) .. controls +(0,-.5) and +(0,-.5) .. +(.5,0);
\node at (-.2,.75) {$\gamma$};
\draw (0,0) .. controls +(0,.5) and +(0,.5) .. +(.5,0);

\draw[dotted] (.5,1) -- +(0,1);
\node at (.7,1.75) {$\alpha$};
\draw[dotted] (.5,2) .. controls +(0,.35) and +(0,.35) .. +(.5,0);
\draw[dotted] (1,2) -- +(0,-2);
\draw[dotted] (1.5,2) -- +(0,-2);
\draw[dotted] (1,0) .. controls +(0,-.5) and +(0,-.5) .. +(.5,0);

\draw[dashed] (0,1) -- +(0,1);
\draw[dashed] (0,2) .. controls +(0,.7) and +(0,.7) .. +(1.5,0);

\node at (1.9,1.7) {$\scriptstyle (c,l')$};

\draw[thin] (-.25,1) -- +(2,0);
\draw[thin] (-.25,0) -- +(2,0);
\draw[thin] (-.25,2) -- +(2,0);
\end{scope}

\begin{scope}
\node at (-.75,1) {i)};
\draw (0,1) .. controls +(0,-.5) and +(0,-.5) .. +(.5,0);
\draw (0,0) .. controls +(0,.5) and +(0,.5) .. +(.5,0);
\node at (.7,.75) {$\gamma$};

\draw[dotted] (1,2) .. controls +(0,.5) and +(0,.5) .. +(.5,0);
\draw[dotted] (.5,1.5) .. controls +(0,.2) and +(0,.2) .. +(.25,0);
\draw[dotted] (.75,1.5) .. controls +(0,-.2) and +(0,-.2) .. +(.25,0);
\draw[dotted] (.5,1) -- +(0,.5);
\draw[dotted] (1,1.5) -- +(0,.5);
\node at (.7,1.85) {$\alpha$};

\draw[dashed] (1.5,2) -- +(0,-2);
\draw[dashed] (-.25,0) .. controls +(0,-.5) and +(0,-.5) .. +(1.75,0);
\draw[dashed] (-.25,0) -- +(0,1);
\draw[dashed] (-.25,1) .. controls +(0,.2) and +(0,.2) .. +(.25,0);

\node at (1.9,1.7) {$\scriptstyle (c,l')$};

\draw[thin] (-.25,1) -- +(2,0);
\draw[thin] (-.25,0) -- +(2,0);
\draw[thin] (-.25,2) -- +(2,0);
\end{scope}
\end{tikzpicture}
\end{eqnarray}
In both cases the dashed connection indicates how the tag $(c,l')$ is connected to $(i,l)$. Since the point $(c,l')$ was assumed to be a tag of $C$, the dashed part cannot contain any points to the right of $(c,l')$ resulting in the given shape. In case i) it is obvious that all points of the circle containing the cap must be strictly to the left of $(c,l')$ proving the claim in this case. In case ii) it follows that the dotted segment must contain an odd number of dots since the two vertices $(j,l)$ and $(c,l')$ are assumed to be labelled differently. By admissibility this implies that the dotted segment must contain at least one dotted cup in $a_s$ for $s \leq l-1$, which implies again that the circle containing the cap can only contain vertices strictly to the left of $(c,l')$.

The case of the cap is done exactly in the same way and since we only looked at the label of the right vertex of the cup respectively cap and how it differs from the label at the tag it also makes no difference whether the cup/cap is dotted or not.
\end{proof}

\subsection{Algebraic surgery}

In this section we present an alternative description of the surgery maps in terms of commutative algebra. This approach is often better for general proofs and also connects to the geometry of Springer fibres, see \cite{ES1}, \cite{Wilbert}. We still fix a block $\Lambda$.

\begin{definition}
Let $D=\un{\lambda}(\mathbf{a})\ov{\mu}$ be a stacked circle diagram of height $h$ for $\Lambda$. We define the vector space $\cM(D)$, where $\cM(D)=\{0\}$ if $D$ cannot be equipped with an orientation, and otherwise set
\begin{eqnarray*}
\cM(D)=\mC[X_{(i,l)} \mid 0 \leq l \leq h, \, i \in P_\bu(\Lambda)] / I(D) \left\langle {\rm mdeg}(D) \right\rangle
\end{eqnarray*}
where ${\rm mdeg}$ is as in Definition \ref{def:minimaldegree} and the ideal $I(D)$ is generated by the following relations for $0 \leq l \leq h$ and $i,j \in P_\bu(\Lambda)$
\begin{eqnarray}\label{idealID}
\left\lbrace
\begin{array}{rl}
 X_{(i,l)}^2 & \\
 X_{(i,l)} + X_{(j,l)} & \text{ if $i \CupConnect j$ in $a_{l}a_{l+1}^*$},\\
 X_{(i,l)} - X_{(j,l)} & \text{ if $i \DCupConnect j$ in $a_{l}a_{l+1}^*$},\\
 X_{(i,h)} & \text{ if $i \RayConnect$ or $i \DRayConnect$ in $a_{h+1}^*$},\\
 X_{(i,0)} & \text{ if $i \RayConnect$ or $i \DRayConnect$ in $a_0$},\\
 X_{(i,l-1)} - X_{(i,l)} & \text{ if $i \RayConnect$ or $i \DRayConnect$ in $a_{l}$ and $l \notin \{0,h+1\}$}.
\end{array} \right\rbrace
\end{eqnarray}
Again we put for simplicity $a_0 = \un{\lambda}$ and $a_{h+1}^* = \ov{\mu}$.
\end{definition}

\begin{remark}
For stacked circle diagrams of height $0$ the definition agrees with the one from Proposition \ref{coho} when identifying $X_i$ with $X_{(i,0)}$.
\end{remark}

\begin{lemma} \label{lem:orientations_dimensions_agree}
The dimension ${\rm dim}(\cM(D))$ of $\cM(D)$ equals the number of possible orientations of $D$. In particular $\cM(D)\not=\{0\}$ if $D$ can be equipped with an orientation. 
\end{lemma}

\begin{proof}
By definition $\cM(D)=\{0\}$ if $D$ is not orientable. If $D$ is orientable, the defining relations \eqref{idealID} of $\cM(D)$ only involve variables $X_{(i,l)}$, where the indices $(i,l)$ label vertices on the same component, the  $X_{(i,l)}$'s from different components are linearly independent or zero. Similar to the proof of Proposition~\ref{coho} it follows that the indeterminants belonging to the same component only differ by at most a sign and are non-zero in the case of circles. Which implies that $\cM(D)$ is isomorphic to a tensor product of vector spaces, one for each connected component of $D$, which are 1-dimensional for lines and 2-dimensional for circles, proving the statement.
\end{proof}

The following generalises Proposition~\ref{coho}.

\begin{prop}
Let $D=\un{\lambda}(\mathbf{a})\ov{\mu}$ be an admissible stacked circle diagram for $\Lambda$. Then there is a natural isomorphism of graded vector spaces
\begin{eqnarray} \label{eqn:diagram_algebraic}
\Psi_D:\quad{}_\lambda(\D^\mathbf{a})_\mu{}  &\longrightarrow & \cM(D) \nonumber\\ 
\underline{\la}(\mathbf{a},\bm{\nu})\ov{\mu}&\longmapsto& \prod_{C \in \mathscr{C}_{\rm clock}(\underline{\la}(\mathbf{a},\bm{\nu})\ov{\mu})}X_{\op{t}(C)},
\end{eqnarray}
where $\mathscr{C}_{\rm clock}(\underline{\la}(\mathbf{a},\bm{\nu})\ov{\mu})$ is the set of clockwise oriented circles in $\underline{\la}(\mathbf{a},\bm{\nu})\ov{\mu}$ and $\op{t}$ is a tag for $D$.
\end{prop}
\begin{proof}
That the map is independent from the chosen tag follows directly from Corollary \ref{cor:orientation_defined}. That it is an isomorphism follows from Lemma \ref{lem:orientations_dimensions_agree} and the same argument as in the proof of Proposition \ref{coho} which shows that the image of the standard basis is a basis.
\end{proof}

For two stacked circle diagrams $D=\un{\lambda}(\mathbf{a})\ov{\mu}$ and $D'=\un{\mu}(\mathbf{a'})\ov{\eta}$ of height $h$ respectively $h'$ denote by $D \circ D'=\un{\lambda}(\mathbf{b})\ov{\eta}$ the stacked circle diagram of height $h+h'+1$, where $\mathbf{b}=(a_1,\ldots,a_{h},b,a_1',\ldots,a_{h'}')$, with $b$ being equal to $\un{\mu}$ except that all rays are undotted.

\begin{lemma}
\label{lem:glue_stacked}
Let $D=\un{\lambda}(\mathbf{a})\ov{\mu}$ and $D'=\un{\mu}(\mathbf{a'})\ov{\eta}$ be stacked circle diagram of height $h$ respectively $h'$ and $D \circ D'=\un{\lambda}(\mathbf{b})\ov{\eta}$. Then there exists an injective map of graded vector spaces
\begin{eqnarray*}
{\rm glue}_{D,D'}:\quad {}_\lambda(\D^\mathbf{a})_\mu \otimes {}_\mu(\D^\mathbf{a'})_\eta \quad &\longrightarrow &{}_\lambda(\D^\mathbf{b})_\eta,\nonumber\\
\un{\lambda}(\mathbf{a},\bm{\nu})\ov{\mu} \otimes \un{\mu}(\mathbf{a'},\bm{\nu'})\ov{\eta} &\longmapsto& \un{\lambda}(\mathbf{b},\bm{\nu''})\ov{\eta},\\
& &\text{ with } \bm{\nu''} = (\nu_0,\ldots,\nu_h,\nu_0',\ldots,\nu_{h'}').
\end{eqnarray*}
\end{lemma}
\begin{proof}
By definition the map is well-defined, and obviously injective and of degree zero.
\end{proof}

\begin{remark}
We avoid giving the algebraic definition of the glueing map, since its definition is more involved than the one using the diagrammatics. This is due to the fact that when stitching two stacked circle diagrams together two non-propagating lines might combine to form a circle (automatically oriented clockwise). This phenomenon is easy to describe in terms of diagrams, but cumbersome in algebraic terms.
\end{remark}

Internal parts consisting of rays only can be collapsed:

\begin{lemma}
\label{lem:collapsing_stacked}
Let $D=\un{\lambda}(\mathbf{a})\ov{\mu}$ be a stacked circle diagram of height $h$ such that $a_1, \ldots, a_{h}$ only consist of rays. Then there exists an isomorphism of graded vector spaces
\begin{eqnarray*}
{\rm collapse}_D:\quad {}_\lambda(\D^\mathbf{a})_\mu &\longrightarrow &{}_\lambda(\D)_\mu,\nonumber\\
\un{\lambda}(\mathbf{a},\bm{\nu})\ov{\mu} &\longmapsto & \un{\lambda} \nu_0 \ov{\mu}.
\end{eqnarray*}
\end{lemma}
\begin{proof}
That the map is well-defined follows directly from the definitions, since $\nu_0 = \ldots = \nu_h$ by assumption on $D$. It is obviously of degree zero.
\end{proof}

\begin{lemma}
\label{lem:algebraic_merge}
Let $D=\un{\lambda}(\mathbf{a})\ov{\mu}$ and $D'=\un{\lambda}(\mathbf{b})\ov{\mu}$ be stacked circle diagrams such that $D'$ is obtained from $D$ by a surgery at level $l$ and positions $i < j$. Assume that the surgery is a merge. Then under the identification \eqref{eqn:diagram_algebraic} the surgery map from \eqref{eqn:surgerymap} is given by
\begin{eqnarray*}
\op{surg}_{D,D'}:\quad\cM(D)&\longrightarrow&\cM(D')\\
X_{(s,a)}&\longmapsto&X_{(s,a)}
\end{eqnarray*}
\end{lemma}
\begin{proof}
Let $C_l$ be the component of $D$ containing $(i,l)$, $C_{l-1}$ the one containing $(i,l-1)$, and $C$ the component of $D'$ containing both. Let $\op{t}_D$ be a tag of $D$ and $\op{t}_{D'}$ a tag of $D'$. Let first $(s,a)$ be a vertex in $C_l$ then using the identification \eqref{eqn:diagram_algebraic} and the diagrammatic surgery map
\begin{eqnarray*}
X_{(s,a)} &=& {\rm sign}_D(s,a) X_{\op{t}_D(C_l)} \\
&\longmapsto& {\rm sign}_D(s,a) {\rm sign}_D(i,l){\rm sign}_{D'}(i,l) X_{\op{t}_{D'}(C)}\\
& =& {\rm sign}_{D'}(s,a) {\rm sign}_{D'}(i,l) {\rm sign}_D(s,a) {\rm sign}_D(i,l) X_{(s,a)}.
\end{eqnarray*}
This final coefficient is equal to $1$, since there is a sequence of arcs connecting $(s,a)$ and $(i,l)$ that is not changed by the surgery. The case $(s,a)$ in $C_{l-1}$ is done in the same way.
\end{proof}

\begin{lemma}
\label{lem:algebraic_split}
Let $D=\un{\lambda}(\mathbf{a})\ov{\mu}$ and $D'=\un{\lambda}(\mathbf{b})\ov{\mu}$ be stacked circle diagrams such that $D'$ is obtained from $D$ by a surgery at level $l$ and positions $i < j$. Assume that the surgery is a split. Then under the identification \eqref{eqn:diagram_algebraic} the surgery map is given by
\begin{eqnarray*}
\op{surg}_{D,D'}:\quad\cM(D)&\longrightarrow&\cM(D')
\end{eqnarray*}
\begin{eqnarray*}
f&\longmapsto&
\begin{cases}
(-1)^{\pos(i)}{(X_{(j,l)}-X_{(i,l)})f}&\text{if $i \CupConnect j$ in $a_{l}$,}\\
(-1)^{\pos(i)}{(X_{(j,l)}+X_{(i,l)})f}&\text{if $i \DCupConnect j$ in $a_{l}$},\\
0&\text{if $D'$ is not orientable,}
\end{cases}
\end{eqnarray*}
where $\pos$ is calculated with respect to the fixed block $\Lambda$.
\end{lemma}
\begin{proof}
The map is well-defined as a map from $\mC[X_{(i,l)} \mid 0 \leq l \leq h, \, i \in P_\bu(\Lambda)]$ to $\cM(D')$. We only need to check that those generators of $I(D)$ that belong to the cup-cap-pair between $(i,l)$ and $(j,l)$ respectively $(i,l-1)$ and $(j,l-1)$ are sent to zero, since all other generators are also generators of $I(D')$ and hence are sent to zero anyway. If the cup-cap-pair is undotted the generator in $I(D)$ corresponding to the cup is $X_{(j,l)}+X_{(i,l)}$. Hence multiplied with $X_{(j,l)}-X_{(i,l)}$ it is zero in $\cM(D')$. Now the generator from the cap, i.e. $X_{(j,l-1)}+X_{(i,l-1)}$ gets multiplied with $X_{(j,l)}-X_{(i,l)}$ which we verify to be zero in $\cM(D')$ by using the relations $X_{(i,l)}-X_{(i,l-1)}$ and $X_{(j,l)}-X_{(j,l-1)}$ in $\cM(D')$.  The same holds for the dotted case, and so the map is well-defined.

In case that the resulting diagram is not orientable the maps agree by definition. Thus assume that $D'$ is orientable. In the case of applying the map to an anticlockwise component, one observes that the signs in the diagrammatic definition of the surgery are exactly the ones by which $X_{(j,l)}$ and $X_{(i,l)}$ differ from a tag for their respective components;
whereas in case of a clockwise circle, the circle corresponds to ${\rm sign}_D(i,l)X_{(i,l)}$ under identification \eqref{eqn:diagram_algebraic}, which in both cases is mapped to $(-1)^{{\pos}(i)}{\rm sign}_D(i,l)X_{(i,l)}X_{(j,l)}$. Rewriting this in terms of tags for both circles gives the desired signs.
\end{proof}

\begin{remark}
As in the diagrammatic definition the signs in the first two cases in Lemma \ref{lem:algebraic_split} can be combined by using ${\rm sign}_D(i,l){\rm sign}_D(j,l)$ again.
\end{remark}

\subsection{Commutativity of surgeries}

In this section we show that if two successive surgeries can be performed in the opposite order then composing the surgery maps in any orders will yield the same result. This will allow to choose an arbitrary sequence of surgeries to define the multiplication.

\begin{definition}
Let $D=\un{\lambda}(\mathbf{a})\ov{\mu}$ be a stacked circle diagram of height $h$ for $\Lambda$ and $J \subset P_\bu(\Lambda) \times \{0,\ldots,h\}$. Two elements $x_1$ and $x_2$ of $J$ are \emph{directly connected in $D$ with respect to $J$} if there exists a sequence of arcs in $D$ connecting $x_1$ with $x_2$ without containing any $x \in J \setminus \{x_1,x_2\}$.
\end{definition}

For a cup-cap-pair in a stacked circle diagram $D$ connecting the vertices of $J= \{(i,l-1),(j,l-1),(i,l),(j,l)\}$, we say that two distinct vertices in $J$ are \emph{neighboured} if at least one of the two coordinates coincide, otherwise we say they are \emph{opposite}.

The next lemma is a technical tool to ensure that two vertices diagonally opposite of each other in a given outer cup-cap-pair cannot be directly connected with respect to the vertices of the pair.

\begin{lemma} \label{lem:opposite_cannot_connect}
Let $D=\un{\lambda}(\mathbf{a})\ov{\mu}$ be a stacked circle diagram of height $h$ for $\Lambda$. Assume that the vertices $i < j$ are connected by an outer cup in $a_l$ for some $1 \leq l \leq h$. Let $J=\{(i,l-1),(j,l-1),(i,l),(j,l)\}$. Then opposite vertices in $J$ are not directly connected with respect to $J$.
\end{lemma}
\begin{proof}
Without loss of generality, assume that $(i,l)$ and $(j,l-1)$ are directly connected in $D$ with respect to $J$. By assumption the cup-cap-pair is an outer cup-cap-pair in $a_l$. Hence we can add in an (auxiliary) curve connecting $(i,l)$ and $(j,l-1)$ without intersecting any arc of $D$ (see the dashed curve in the pictures below).

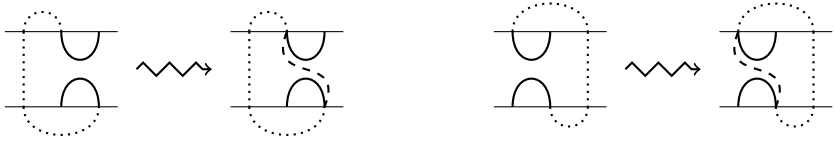
\begin{eqnarray*}
\begin{tikzpicture}[thick,decoration={zigzag,pre length=5mm,post length=5mm}]
\begin{scope}
\draw[dotted] (-.5,1) .. controls +(0,.35) and +(0,.35) .. +(.5,0);
\draw (0,1) .. controls +(0,-.5) and +(0,-.5) .. +(.5,0);
\draw (0,0) .. controls +(0,.5) and +(0,.5) .. +(.5,0);
\draw[dotted] (-.5,0) .. controls +(0,-.5) and +(0,-.5) .. +(1,0);
\draw[dotted] (-.5,1) -- +(0,-1);
\draw[thin] (-.75,1) -- +(1.5,0);
\draw[thin] (-.75,0) -- +(1.5,0);
\draw[snake,->] (1,.5) -- +(1,0);
\end{scope}

\begin{scope}[xshift=3cm]
\draw[dotted] (-.5,1) .. controls +(0,.35) and +(0,.35) .. +(.5,0);
\draw (0,1) .. controls +(0,-.5) and +(0,-.5) .. +(.5,0);
\draw (0,0) .. controls +(0,.5) and +(0,.5) .. +(.5,0);
\draw[thick, dashed] (0,1) .. controls +(-.3,-.75) and +(.3,.75) .. +(.5,-1);
\draw[dotted] (-.5,0) .. controls +(0,-.5) and +(0,-.5) .. +(1,0);
\draw[dotted] (-.5,1) -- +(0,-1);
\draw[thin] (-.75,1) -- +(1.5,0);
\draw[thin] (-.75,0) -- +(1.5,0);
\end{scope}

\begin{scope}[xshift=6cm]
\draw[dotted] (0,1) .. controls +(0,.5) and +(0,.5) .. +(1,0);
\draw (0,1) .. controls +(0,-.5) and +(0,-.5) .. +(.5,0);
\draw (0,0) .. controls +(0,.5) and +(0,.5) .. +(.5,0);
\draw[dotted] (.5,0) .. controls +(0,-.35) and +(0,-.35) .. +(.5,0);
\draw[dotted] (1,1) -- +(0,-1);
\draw[thin] (-.25,1) -- +(1.5,0);
\draw[thin] (-.25,0) -- +(1.5,0);
\draw[snake,->] (1.5,.5) -- +(1,0);
\end{scope}

\begin{scope}[xshift=9cm]
\draw[dotted] (0,1) .. controls +(0,.5) and +(0,.5) .. +(1,0);
\draw (0,1) .. controls +(0,-.5) and +(0,-.5) .. +(.5,0);
\draw (0,0) .. controls +(0,.5) and +(0,.5) .. +(.5,0);
\draw[dotted] (.5,0) .. controls +(0,-.35) and +(0,-.35) .. +(.5,0);
\draw[dotted] (1,1) -- +(0,-1);
\draw[thin] (-.25,1) -- +(1.5,0);
\draw[thin] (-.25,0) -- +(1.5,0);
\draw[thick, dashed] (0,1) .. controls +(-.3,-.75) and +(.3,.75) .. +(.5,-1);
\end{scope}

\end{tikzpicture}
\end{eqnarray*}
But now $(j,l)$ and $(i,l-1)$ cannot be connected without intersecting the auxiliary curve, which contradicts that the cup-cap-pair was outer. \end{proof}

Lemma \ref{lem:opposite_cannot_connect} restricts the possible configurations of connections between two distinct outer cup-cap-pairs.

\begin{lemma} \label{lem:two_pairs_connect}
Let $D=\un{\lambda}(\mathbf{a})\ov{\mu}$ be a stacked circle diagram of height $h$. Assume that the vertices $a < b$ are connected by an outer cup in $a_l$ and vertices $c < d$ by an outer cup in $a_{l'}$ for some $1 \leq l,l' \leq h$. Define 
\begin{align}
\label{Js}
J&=\,\{(a,l-1),(b,l-1),(a,l),(b,l)\}, \nonumber\\
J'&=\,\{(c,l'-1),(d,l'-1),(c,l'),(d,l')\}. 
\end{align}
Then two neighbouring vertices in $J$ are not directly connected  to two opposite vertices in $J'$ with respect to $J \cup J'$.
\end{lemma}
\begin{proof}
Without loss of generality, assume that say $(a,l)$ is directly connected to $(d,l')$ and $(b,l)$ is directly connected to $(c,l'-1)$. This implies that $(d,l')$ and $(c,l'-1)$ are directly connected with respect to $J'$, contradicting Lemma~\ref{lem:opposite_cannot_connect}. If we assume instead that $(a,l-1)$ is directly connected to $(c,l'-1)$, we first change the cup-cap-pair in $D$ connecting $(a,l)$ and $(b,l)$ in $a_l$ into two lines (like in a surgery) and then argue as above.
\end{proof}

We now show that surgeries commute if they fit into a square configuration of four stacked circle diagrams obtained from each other by surgeries as follows:
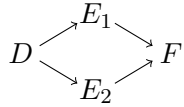
\begin{eqnarray*}
\begin{tikzpicture}[thick,decoration={zigzag,pre length=5mm,post length=5mm}]
\begin{scope}
\node at (0,0) {$D$};
\node at (1,-.5) {$E_2$};
\node at (1,.5) {$E_1$};
\node at (2,0) {$F$};
\draw[thin,->] (.25,.1) -- +(.5,.3);
\draw[thin,->] (.25,-.1) -- +(.5,-.3);
\draw[thin,->] (1.25,.4) -- +(.5,-.3);
\draw[thin,->] (1.25,-.4) -- +(.5,.3);
\end{scope}
\end{tikzpicture}
\end{eqnarray*}

\begin{theorem} \label{thm:surgeries_commute}
Let $D=\un{\lambda}(\mathbf{a})\ov{\mu}$ be a stacked circle diagram. Assume $E_1=\un{\lambda}(\mathbf{b^{\scriptscriptstyle(1)}})\ov{\mu}$ and $E_2=\un{\lambda}(\mathbf{b^{\scriptscriptstyle(2)}})\ov{\mu}$ are obtained from $D$ by a surgery. Moreover, assume $F=\un{\lambda}(\mathbf{c})\ov{\mu}$ is obtained from both, $E_1$ and $E_2$, by a surgery. Then
\begin{eqnarray*}
{\rm surg}_{E_1,F} \circ {\rm surg}_{D,E_1}& =& {\rm surg}_{E_2,F} \circ {\rm surg}_{D,E_2}.
\end{eqnarray*}
\end{theorem}
\begin{proof}
The proof requires to distinguish between a number of cases depending on how the two surgeries interact with each other. Admissibility only appears in the very last step of the proof to ensure that certain signs vanish.

Let $l,a,b$ be such that $E_1$ is obtained from $D$ by a surgery at level $l$ at positions $a < b$ (or equivalently $F$ is obtained from $E_2$ by a surgery at the same level and positions) and let $l',c,d$ be such that $E_2$ is obtained from $D$ by a surgery at level $l'$ at positions $c < d$ (or the same statement for $F$ and $E_1$). Define $J$ and $J'$ as in \eqref{Js} of Lemma \ref{lem:two_pairs_connect}. In the following ``directly connected'' always means directly connected with respect to $J \cup J'$.  We will distinguish the cases by how $(a,l)$ is connected to the other vertices in $J \cup J'$. By Lemma~\ref{lem:two_pairs_connect} there are $6$ cases (with some subcases): 
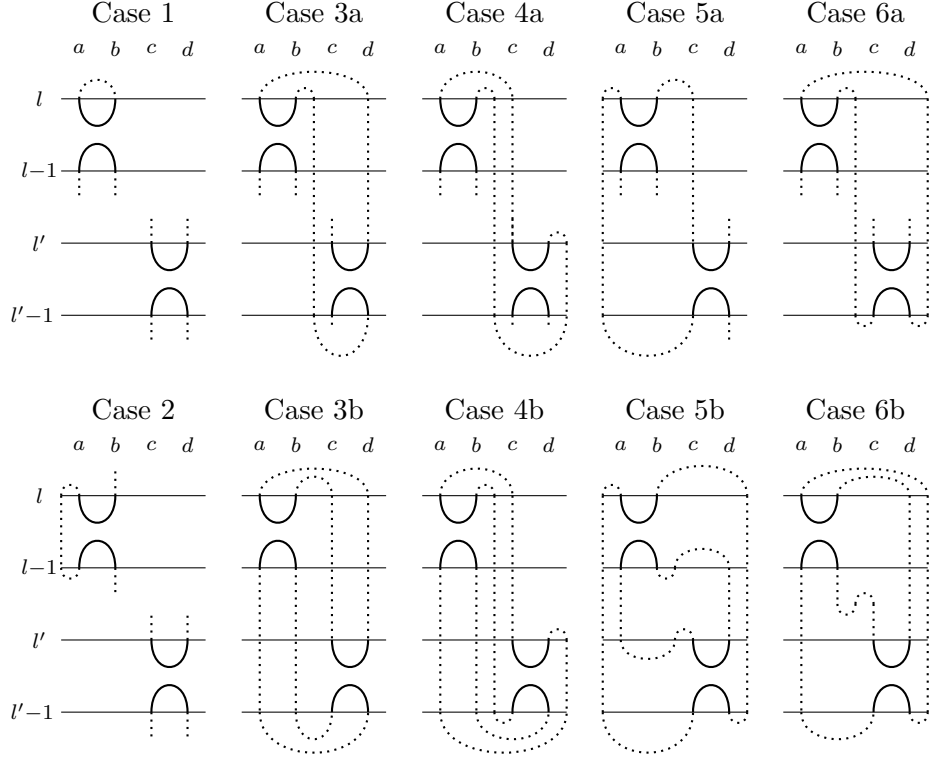
\begin{center}
\begin{tikzpicture}[thick,decoration={zigzag,pre length=1mm,post length=1mm},scale=0.96]
\begin{scope}
\draw[dotted] (0,1) .. controls +(0,.35) and +(0,.35) .. +(.5,0);
\draw (0,1) .. controls +(0,-.5) and +(0,-.5) .. +(.5,0);
\draw (0,0) .. controls +(0,.5) and +(0,.5) .. +(.5,0);
\draw[dotted] (0,0) -- +(0,-.4);
\draw[dotted] (.5,0) -- +(0,-.4);

\draw (1,-1) .. controls +(0,-.5) and +(0,-.5) .. +(.5,0);
\draw (1,-2) .. controls +(0,.5) and +(0,.5) .. +(.5,0);
\draw[dotted] (1,-1) -- +(0,.4);
\draw[dotted] (1.5,-1) -- +(0,.4);
\draw[dotted] (1,-2) -- +(0,-.4);
\draw[dotted] (1.5,-2) -- +(0,-.4);

\node at (0.75,2.2) {Case 1};

\node at (0,1.7) {$\scriptstyle a$};
\node at (.5,1.7) {$\scriptstyle b$};
\node at (1,1.7) {$\scriptstyle c$};
\node at (1.5,1.7) {$\scriptstyle d$};

\node at (-.55,1) {$\scriptstyle l$};
\node at (-.55,0) {$\scriptstyle l-1$};
\node at (-.55,-1) {$\scriptstyle l'$};
\node at (-.65,-2) {$\scriptstyle l'-1$};
\draw[thin] (-.25,1) -- +(2,0);
\draw[thin] (-.25,0) -- +(2,0);
\draw[thin] (-.25,-1) -- +(2,0);
\draw[thin] (-.25,-2) -- +(2,0);
\end{scope}

\begin{scope}[xshift=0cm,yshift=-5.5cm]
\draw[dotted] (-.25,1) .. controls +(0,.2) and +(0,.2) .. +(.25,0);
\draw[dotted] (-.25,0) .. controls +(0,-.2) and +(0,-.2) .. +(.25,0);
\draw[dotted] (-.25,1) -- +(0,-1);
\draw (0,1) .. controls +(0,-.5) and +(0,-.5) .. +(.5,0);
\draw (0,0) .. controls +(0,.5) and +(0,.5) .. +(.5,0);
\draw[dotted] (.5,1) -- +(0,.4);
\draw[dotted] (.5,0) -- +(0,-.4);

\draw (1,-1) .. controls +(0,-.5) and +(0,-.5) .. +(.5,0);
\draw (1,-2) .. controls +(0,.5) and +(0,.5) .. +(.5,0);
\draw[dotted] (1,-1) -- +(0,.4);
\draw[dotted] (1.5,-1) -- +(0,.4);
\draw[dotted] (1,-2) -- +(0,-.4);
\draw[dotted] (1.5,-2) -- +(0,-.4);

\node at (0.75,2.2) {Case 2};

\node at (0,1.7) {$\scriptstyle a$};
\node at (.5,1.7) {$\scriptstyle b$};
\node at (1,1.7) {$\scriptstyle c$};
\node at (1.5,1.7) {$\scriptstyle d$};

\node at (-.55,1) {$\scriptstyle l$};
\node at (-.55,0) {$\scriptstyle l-1$};
\node at (-.55,-1) {$\scriptstyle l'$};
\node at (-.65,-2) {$\scriptstyle l'-1$};

\draw[thin] (-.25,1) -- +(2,0);
\draw[thin] (-.25,0) -- +(2,0);
\draw[thin] (-.25,-1) -- +(2,0);
\draw[thin] (-.25,-2) -- +(2,0);
\end{scope}

\begin{scope}[xshift=2.5cm]
\draw[dotted] (0,1) .. controls +(0,.5) and +(0,.5) .. +(1.5,0);
\draw[dotted] (1.5,1) -- +(0,-2);
\draw[dotted] (.5,1) .. controls +(0,.2) and +(0,.2) .. +(.25,0);
\draw[dotted] (.75,1) -- +(0,-3);
\draw[dotted] (.75,-2) .. controls +(0,-.75) and +(0,-.75) .. +(.75,0);
\draw (0,1) .. controls +(0,-.5) and +(0,-.5) .. +(.5,0);
\draw (0,0) .. controls +(0,.5) and +(0,.5) .. +(.5,0);
\draw[dotted] (0,0) -- +(0,-.4);
\draw[dotted] (.5,0) -- +(0,-.4);

\draw (1,-1) .. controls +(0,-.5) and +(0,-.5) .. +(.5,0);
\draw (1,-2) .. controls +(0,.5) and +(0,.5) .. +(.5,0);
\draw[dotted] (1,-1) -- +(0,.4);
\draw[dotted] (1,-2) -- +(0,-.2);

\node at (0.75,2.2) {Case 3a};

\node at (0,1.7) {$\scriptstyle a$};
\node at (.5,1.7) {$\scriptstyle b$};
\node at (1,1.7) {$\scriptstyle c$};
\node at (1.5,1.7) {$\scriptstyle d$};

\draw[thin] (-.25,1) -- +(2,0);
\draw[thin] (-.25,0) -- +(2,0);
\draw[thin] (-.25,-1) -- +(2,0);
\draw[thin] (-.25,-2) -- +(2,0);
\end{scope}

\begin{scope}[xshift=2.5cm,yshift=-5.5cm]
\draw[dotted] (0,1) .. controls +(0,.5) and +(0,.5) .. +(1.5,0);
\draw[dotted] (1.5,1) -- +(0,-2);
\draw[dotted] (1,1) -- +(0,-2);
\draw[dotted] (.5,1) .. controls +(0,.35) and +(0,.35) .. +(.5,0);
\draw (0,1) .. controls +(0,-.5) and +(0,-.5) .. +(.5,0);
\draw (0,0) .. controls +(0,.5) and +(0,.5) .. +(.5,0);
\draw[dotted] (0,0) -- +(0,-2);
\draw[dotted] (.5,0) -- +(0,-2);

\draw (1,-1) .. controls +(0,-.5) and +(0,-.5) .. +(.5,0);
\draw (1,-2) .. controls +(0,.5) and +(0,.5) .. +(.5,0);
\draw[dotted] (0,-2) .. controls +(0,-.75) and +(0,-.75) .. +(1.5,0);
\draw[dotted] (.5,-2) .. controls +(0,-.5) and +(0,-.5) .. +(.5,0);

\node at (0.75,2.2) {Case 3b};

\node at (0,1.7) {$\scriptstyle a$};
\node at (.5,1.7) {$\scriptstyle b$};
\node at (1,1.7) {$\scriptstyle c$};
\node at (1.5,1.7) {$\scriptstyle d$};

\draw[thin] (-.25,1) -- +(2,0);
\draw[thin] (-.25,0) -- +(2,0);
\draw[thin] (-.25,-1) -- +(2,0);
\draw[thin] (-.25,-2) -- +(2,0);
\end{scope}

\begin{scope}[xshift=10cm]
\draw[dotted] (0,1) .. controls +(0,.5) and +(0,.5) .. +(1.75,0);
\draw[dotted] (1.75,1) -- +(0,-3);
\draw[dotted] (1.5,-2) .. controls +(0,-.2) and +(0,-.2) .. +(.25,0);
\draw[dotted] (.5,1) .. controls +(0,.2) and +(0,.2) .. +(.25,0);
\draw[dotted] (.75,1) -- +(0,-3);
\draw[dotted] (.75,-2) .. controls +(0,-.2) and +(0,-.2) .. +(.25,0);
\draw (0,1) .. controls +(0,-.5) and +(0,-.5) .. +(.5,0);
\draw (0,0) .. controls +(0,.5) and +(0,.5) .. +(.5,0);
\draw[dotted] (0,0) -- +(0,-.4);
\draw[dotted] (.5,0) -- +(0,-.4);

\draw (1,-1) .. controls +(0,-.5) and +(0,-.5) .. +(.5,0);
\draw (1,-2) .. controls +(0,.5) and +(0,.5) .. +(.5,0);
\draw[dotted] (1,-1) -- +(0,.4);
\draw[dotted] (1.5,-1) -- +(0,.4);

\node at (0.75,2.2) {Case 6a};

\node at (0,1.7) {$\scriptstyle a$};
\node at (.5,1.7) {$\scriptstyle b$};
\node at (1,1.7) {$\scriptstyle c$};
\node at (1.5,1.7) {$\scriptstyle d$};

\draw[thin] (-.25,1) -- +(2,0);
\draw[thin] (-.25,0) -- +(2,0);
\draw[thin] (-.25,-1) -- +(2,0);
\draw[thin] (-.25,-2) -- +(2,0);
\end{scope}

\begin{scope}[xshift=5cm]
\draw[dotted] (0,1) .. controls +(0,.5) and +(0,.5) .. +(1,0);
\draw[dotted] (1,1) -- +(0,-2);
\draw[dotted] (.5,1) .. controls +(0,.2) and +(0,.2) .. +(.25,0);
\draw[dotted] (.75,1) -- +(0,-3);
\draw[dotted] (.75,-2) .. controls +(0,-.75) and +(0,-.75) .. +(1,0);
\draw (0,1) .. controls +(0,-.5) and +(0,-.5) .. +(.5,0);
\draw (0,0) .. controls +(0,.5) and +(0,.5) .. +(.5,0);
\draw[dotted] (0,0) -- +(0,-.4);
\draw[dotted] (.5,0) -- +(0,-.4);

\draw[dotted] (1.5,-1) .. controls +(0,.2) and +(0,.2) .. +(.25,0);
\draw[dotted] (1.75,-1) -- +(0,-1);
\draw (1,-1) .. controls +(0,-.5) and +(0,-.5) .. +(.5,0);
\draw (1,-2) .. controls +(0,.5) and +(0,.5) .. +(.5,0);
\draw[dotted] (1,-1) -- +(0,.4);
\draw[dotted] (1,-2) -- +(0,-.2);
\draw[dotted] (1.5,-2) -- +(0,-.2);

\node at (0.75,2.2) {Case 4a};

\node at (0,1.7) {$\scriptstyle a$};
\node at (.5,1.7) {$\scriptstyle b$};
\node at (1,1.7) {$\scriptstyle c$};
\node at (1.5,1.7) {$\scriptstyle d$};

\draw[thin] (-.25,1) -- +(2,0);
\draw[thin] (-.25,0) -- +(2,0);
\draw[thin] (-.25,-1) -- +(2,0);
\draw[thin] (-.25,-2) -- +(2,0);
\end{scope}

\begin{scope}[xshift=5cm,yshift=-5.5cm]
\draw[dotted] (0,1) .. controls +(0,.5) and +(0,.5) .. +(1,0);
\draw[dotted] (1,1) -- +(0,-2);
\draw[dotted] (.5,1) .. controls +(0,.2) and +(0,.2) .. +(.25,0);
\draw[dotted] (.75,1) -- +(0,-3);
\draw[dotted] (.75,-2) .. controls +(0,-.2) and +(0,-.2) .. +(.25,0);
\draw (0,1) .. controls +(0,-.5) and +(0,-.5) .. +(.5,0);
\draw (0,0) .. controls +(0,.5) and +(0,.5) .. +(.5,0);
\draw[dotted] (0,0) -- +(0,-2);
\draw[dotted] (.5,0) -- +(0,-2);
\draw[dotted] (.5,-2) .. controls +(0,-.5) and +(0,-.5) .. +(1,0);
\draw[dotted] (0,-2) .. controls +(0,-.75) and +(0,-.75) .. +(1.75,0);

\draw[dotted] (1.5,-1) .. controls +(0,.2) and +(0,.2) .. +(.25,0);
\draw[dotted] (1.75,-1) -- +(0,-1);
\draw (1,-1) .. controls +(0,-.5) and +(0,-.5) .. +(.5,0);
\draw (1,-2) .. controls +(0,.5) and +(0,.5) .. +(.5,0);

\node at (0.75,2.2) {Case 4b};

\node at (0,1.7) {$\scriptstyle a$};
\node at (.5,1.7) {$\scriptstyle b$};
\node at (1,1.7) {$\scriptstyle c$};
\node at (1.5,1.7) {$\scriptstyle d$};

\draw[thin] (-.25,1) -- +(2,0);
\draw[thin] (-.25,0) -- +(2,0);
\draw[thin] (-.25,-1) -- +(2,0);
\draw[thin] (-.25,-2) -- +(2,0);
\end{scope}

\begin{scope}[xshift=7.5cm]
\draw[dotted] (-.25,1) .. controls +(0,.2) and +(0,.2) .. +(.25,0);
\draw[dotted] (-.25,1) -- +(0,-3);
\draw[dotted] (.5,1) .. controls +(0,.35) and +(0,.35) .. +(.5,0);
\draw[dotted] (1,1) -- +(0,-2);
\draw[dotted] (-.25,-2) .. controls +(0,-.75) and +(0,-.75) .. +(1.25,0);
\draw (0,1) .. controls +(0,-.5) and +(0,-.5) .. +(.5,0);
\draw (0,0) .. controls +(0,.5) and +(0,.5) .. +(.5,0);
\draw[dotted] (0,0) -- +(0,-.4);
\draw[dotted] (.5,0) -- +(0,-.4);

\draw (1,-1) .. controls +(0,-.5) and +(0,-.5) .. +(.5,0);
\draw (1,-2) .. controls +(0,.5) and +(0,.5) .. +(.5,0);
\draw[dotted] (1.5,-1) -- +(0,.4);
\draw[dotted] (1.5,-2) -- +(0,-.4);

\node at (0.75,2.2) {Case 5a};

\node at (0,1.7) {$\scriptstyle a$};
\node at (.5,1.7) {$\scriptstyle b$};
\node at (1,1.7) {$\scriptstyle c$};
\node at (1.5,1.7) {$\scriptstyle d$};

\draw[thin] (-.25,1) -- +(2,0);
\draw[thin] (-.25,0) -- +(2,0);
\draw[thin] (-.25,-1) -- +(2,0);
\draw[thin] (-.25,-2) -- +(2,0);
\end{scope}

\begin{scope}[xshift=7.5cm,yshift=-5.5cm]
\draw[dotted] (-.25,1) .. controls +(0,.2) and +(0,.2) .. +(.25,0);
\draw[dotted] (-.25,1) -- +(0,-3);
\draw[dotted] (.5,1) .. controls +(0,.55) and +(0,.55) .. +(1.25,0);
\draw[dotted] (1.75,1) -- +(0,-3);
\draw[dotted] (-.25,-2) .. controls +(0,-.75) and +(0,-.75) .. +(1.25,0);
\draw[dotted] (1.5,-2) .. controls +(0,-.2) and +(0,-.2) .. +(.25,0);
\draw (0,1) .. controls +(0,-.5) and +(0,-.5) .. +(.5,0);
\draw (0,0) .. controls +(0,.5) and +(0,.5) .. +(.5,0);

\draw[dotted] (.5,0) .. controls +(0,-.2) and +(0,-.2) .. +(.25,0);
\draw[dotted] (.75,0) .. controls +(0,.35) and +(0,.35) .. +(.75,0);
\draw[dotted] (.75,-1) .. controls +(0,.2) and +(0,.2) .. +(.25,0);
\draw[dotted] (0,-1) .. controls +(0,-.35) and +(0,-.35) .. +(.75,0);
\draw[dotted] (0,0) -- +(0,-1);
\draw[dotted] (1.5,0) -- +(0,-1);

\draw (1,-1) .. controls +(0,-.5) and +(0,-.5) .. +(.5,0);
\draw (1,-2) .. controls +(0,.5) and +(0,.5) .. +(.5,0);

\node at (0.75,2.2) {Case 5b};

\node at (0,1.7) {$\scriptstyle a$};
\node at (.5,1.7) {$\scriptstyle b$};
\node at (1,1.7) {$\scriptstyle c$};
\node at (1.5,1.7) {$\scriptstyle d$};

\draw[thin] (-.25,1) -- +(2,0);
\draw[thin] (-.25,0) -- +(2,0);
\draw[thin] (-.25,-1) -- +(2,0);
\draw[thin] (-.25,-2) -- +(2,0);
\end{scope}

\begin{scope}[xshift=10cm,yshift=-5.5cm]
\draw[dotted] (0,1) .. controls +(0,.5) and +(0,.5) .. +(1.75,0);
\draw[dotted] (1.75,1) -- +(0,-3);
\draw[dotted] (1.5,-2) .. controls +(0,-.2) and +(0,-.2) .. +(.25,0);
\draw[dotted] (.5,1) .. controls +(0,.35) and +(0,.35) .. +(1,0);
\draw[dotted] (1.5,1) -- +(0,-2);
\draw[dotted] (0,-2) .. controls +(0,-.5) and +(0,-.5) .. +(1,0);
\draw[dotted] (0,0) -- +(0,-2);

\draw[dotted] (.5,-.5) .. controls +(0,-.2) and +(0,-.2) .. +(.25,0);
\draw[dotted] (.75,-.5) .. controls +(0,.2) and +(0,.2) .. +(.25,0);
\draw[dotted] (.5,0) -- +(0,-.5);
\draw[dotted] (1,-.5) -- +(0,-.5);

\draw (0,1) .. controls +(0,-.5) and +(0,-.5) .. +(.5,0);
\draw (0,0) .. controls +(0,.5) and +(0,.5) .. +(.5,0);

\draw (1,-1) .. controls +(0,-.5) and +(0,-.5) .. +(.5,0);
\draw (1,-2) .. controls +(0,.5) and +(0,.5) .. +(.5,0);

\node at (0.75,2.2) {Case 6b};

\node at (0,1.7) {$\scriptstyle a$};
\node at (.5,1.7) {$\scriptstyle b$};
\node at (1,1.7) {$\scriptstyle c$};
\node at (1.5,1.7) {$\scriptstyle d$};

\draw[thin] (-.25,1) -- +(2,0);
\draw[thin] (-.25,0) -- +(2,0);
\draw[thin] (-.25,-1) -- +(2,0);
\draw[thin] (-.25,-2) -- +(2,0);
\end{scope}

\end{tikzpicture}
\end{center}
\noindent
\textit{$\blacktriangleright$ Case 1: $(a,l)$ is directly connected to $(b,l)$:}
In this case ${\rm surg}_{D,E_1}$ and ${\rm surg}_{E_2,F}$, the surgeries involving the cup-cap-pair around the indices in $J$, are both merges. Furthermore it is easy to check that the types of both ${\rm surg}_{E_1,F}$ and ${\rm surg}_{D,E_2}$ also agree. Thus the composites agree in both cases.\\ \noindent
\textit{$\blacktriangleright$ Case 2: $(a,l)$ is directly connected to $(a,l-1)$:}
In this case ${\rm surg}_{D,E_1}$ and ${\rm surg}_{E_2,F}$ are both splits that multiply with the same expression, given in Lemma~\ref{lem:algebraic_split}, independent of the other surgery. Again it follows that ${\rm surg}_{E_1,F}$ and ${\rm surg}_{D,E_2}$ are both also of the same type and the two composites agree.

These are the only two possibilities how $(a,l)$ can be directly connected to vertices in $J$ by Lemma~\ref{lem:opposite_cannot_connect}.  These arguments are of course also valid for any of the other vertices instead of $(a,l)$.  Hence we may assume 
\begin{equation}
\label{ass}
\text{\it all vertices in $J$ are directly connected to vertices in $J'$.}
\end{equation}
\noindent 
\textit{$\blacktriangleright$ Case 3: $(a,l)$ is directly connected to $(d,l')$:}
By Lemma~\ref{lem:two_pairs_connect} there are two possibilities how $(b,l)$ can be connected to vertices in $J'$, namely $(b,l)$ is directly connected to $(d,l'-1)$, see Case 3a, or  $(b,l)$ is directly connected to $(c,l')$, see Case 3b. In the first case we are done by \eqref{ass}, since we observe that the remaining two vertices in $J$ or the remaining two in $J'$ must be connected to each other.  In the latter case Lemma \ref{lem:two_pairs_connect} implies that $(a,l-1)$ is connected to $(d,l'-1)$, and $(b,l-1)$ to $(c,l'-1)$. Then ${\rm surg}_{D,E_1}$ and ${\rm surg}_{D,E_2}$ are merges, while ${\rm surg}_{E_1,F}$ and ${\rm surg}_{E_2,F}$ are splits. Using Lemmas \ref{lem:algebraic_merge} and \ref{lem:algebraic_split} we know that if $F$ is orientable 
\begin{eqnarray*}
&{\rm surg}_{E_1,F} \circ {\rm surg}_{D,E_1} (f) = {\small (-1)^{{\pos}(c)}}(X_{(d,l')} + {\rm sign}_{D}(c,l'){\rm sign}_{D}(d,l') X_{(c,l')})f,&
\end{eqnarray*}
whereas
\begin{equation*}
{\rm surg}_{E_2,F} \circ {\rm surg}_{D,E_2}(f) = (-1)^{{\pos}(a)}(X_{(b,l)} +{\rm sign}_{D}(a,l){\rm sign}_{D}(b,l) X_{(a,l)})f,
\end{equation*}
where we use the observation that the signs do not depend on whether we determine them in $D$, $E_1$ or $E_2$. Using the  assumption on how the vertices are connected we know that in $\cM(F)$ the following equalities hold
\begin{eqnarray}
\label{formulas}
X_{(d,l')} = {\scriptstyle {\rm sign}_{D}(a,l'){\rm sign}_{D}(a,l)}X_{(a,l)}, \quad X_{(c,l')} = {\scriptstyle {\rm sign}_{D}(c,l'){\rm sign}_{D}(b,l)}X_{(b,l)}.
\end{eqnarray}
We are allowed to determine the signs in $D$ due to the fact that there are connections between $(d,l')$ and $(a,l)$ respectively $(c,l')$ and $(b,l)$ which are not changed by the surgeries.
Substituting \eqref{formulas} into ${\rm surg}_{E_1,F} \circ {\rm surg}_{D,E_1}(f)$ we see that it differs from ${\rm surg}_{E_2,F} \circ {\rm surg}_{D,E_2}(f)$ by the scalar
$$(-1)^{{\pos}(c)-{\pos}(a)} {\rm sign}_{D}(b,l) {\rm sign}_{D}(d,l').$$
We will argue at the end of the proof why this term is $1$ if we use that all diagrams are admissible.  \\ \noindent
\textit{$\blacktriangleright$ Case 4: $(a,l)$ is directly connected to $(c,l')$:}
By Lemma \ref{lem:two_pairs_connect} there are again two possibilities how $(b,l)$ can be connected to vertices in $J'$, namely Case 4a, where $(b,l)$ is directly connected to $(d,l')$ in which case we are done by \eqref{ass}, or Case 4b, where $(b,l)$ is directly connected to $(c,l'-1)$.
In this case Lemma~\ref{lem:two_pairs_connect} forces that $(a,l-1)$ is connected to $(d,l')$ and $(b,l-1)$ to $(d,l'-1)$. Then ${\rm surg}_{D,E_1}$ and ${\rm surg}_{D,E_2}$ are splits, while ${\rm surg}_{E_1,F}$ and ${\rm surg}_{E_2,F}$ are merges. In contrast to Case 3b, here we have to make sure that if one of the two surgeries produces a non-orientable diagram then the result of the other split map is zero as well. Let us first assume that both splits produce orientable diagrams. Then we have
\begin{eqnarray}
{\rm surg}_{D,E_1} (f) = (-1)^{{\pos}(a)}(X_{(b,l)} +{\rm sign}_{D}(a,l){\rm sign}_{D}(b,l) X_{(a,l)})f\label{4one},\\
{\rm surg}_{D,E_2}(f)=(-1)^{{\pos}(c)}(X_{(d,l')} + {\rm sign}_{D}(c,l'){\rm sign}_{D}(d,l') X_{(c,l')})f.\label{4two}
\end{eqnarray}
Again we use the assumptions on how the vertices are connected to obtain for $\cM(E_2)$ the following equalities  
$$X_{(d,l')} = {\scriptstyle {\rm sign}_{D}(d,l'){\rm sign}_{D}(b,l-1)}X_{(b,l-1)}, \quad X_{(c,l')} = {\scriptstyle {\rm sign}_{D}(c,l'){\rm sign}_{D}(b,l)} X_{(b,l)}.$$
Substituting these into equality \eqref{4two} and using additionally  the relation $X_{(b,l-1)} = {\scriptstyle {\rm sign}_{D}(a,l-1){\rm sign}_{D}(b,l-1)} X_{(a,l-1)}$ in $\cM(E_2)$ we obtain
\begin{eqnarray} \label{eqn:surgery4b}
& &{\rm surg}_{D,E_2}(f)\\
&=&(-1)^{{\pos}(c)}( {\scriptstyle {\rm sign}_{D}(d,l'){\rm sign}_{D}(b,l-1)}X_{(b,l-1)} + {\scriptstyle {\rm sign}_{D}(d,l') {\rm sign}_{D}(b,l)} X_{(b,l)})f \nonumber \\
&=&(-1)^{{\pos}(c)}( {\scriptstyle {\rm sign}_{D}(d,l'){\rm sign}_{D}(a,l-1)}X_{(a,l-1)} + {\scriptstyle {\rm sign}_{D}(d,l') {\rm sign}_{D}(b,l)} X_{(b,l)})f. \nonumber
\end{eqnarray}
Since we assumed that $E_1$ is orientable, it holds 
$${\rm sign}_{D}(d,l'){\rm sign}_{D}(a,l) = {\rm sign}_{D}(d,l'){\rm sign}_{D}(a,l-1).$$
Applying the second surgery and using $X_{(a,l-1)}=X_{(a,l)}$ in $\cM(F)$ we obtain
\begin{eqnarray*}
{\rm surg}_{E_2,F} \circ {\rm surg}_{D,E_2} = \left({\scriptstyle (-1)^{{\pos}(c)-{\pos}(a)} {\rm sign}_{D}(b,l) {\rm sign}_{D}(d,l')}\right) {\rm surg}_{E_1,F} \circ {\rm surg}_{D,E_1}.
\end{eqnarray*}
We will deal with the expression involving signs at the end of the proof.\noindent\\
\textit{$\blacktriangleright$ Case 5: $(a,l)$ is directly connected to $(c,l'-1)$:}
Again, by Lemma~\ref{lem:two_pairs_connect} either $(b,l)$ is directly connected to $(c,l')$, see Case 5a, in which case we are done by \eqref{ass}, or $(b,l)$ is directly connected to $(d,l'-1)$, see {Case 5b}. This latter case is done analogously to Case 3b. As a result one obtains the same expression involving signs.  \\ \noindent
\textit{$\blacktriangleright$ Case 6: $(a,l)$ is directly connected to $(d,l'-1)$:}
This behaves parallel to Case 4. The Case 6a, where $(b,l)$ is directly connected to $(c,l'-1)$, is again done by assumption \eqref{ass}.
The Case 6b, where $(b,l)$ is directly connected to $(d,l')$ is completely parallel to Case 4b; the two maps in question differ by the same sign as in Case 4b. Moreover, we observe that only one of the two sequences of surgeries might yield an orientable sequence of diagrams. We now deal first with this orientability question.\noindent\\
\textit{$\blacktriangleright$ Orientability in Cases 4b and 6b:}
Let us consider Case 4b and assume that $E_1$ is not orientable (the proof for $E_2$ being not orientable is exactly analogous). We rewrite the surgery map from Case 4b by rewriting the last line of \eqref{eqn:surgery4b} as
\begin{eqnarray*}
& & {\rm surg}_{D,E_2}(f) \\
&=& (-1)^{{\pos}(c)} {\scriptstyle {\rm sign}_{D}(d,l') {\rm sign}_{D}(b,l-1)} (X_{(b,l-1)} + {\scriptstyle {\rm sign}_{D}(b,l-1){\rm sign}_{D}(b,l)} X_{(b,l)})f.
\end{eqnarray*}
Since both, $(b,l-1)$ and $(b,l)$, are on the same connected component in $D$, ${\rm sign}_{D}(b,l-1){\rm sign}_{D}(b,l)$ agrees with $(-1)^z$ where $z$ is the number of undotted arcs in the following sequence ${\bm{\alpha}}$ of arcs: first take the direct connection from $(b,l-1)$ to $(d,l'-1)$, then from $(d,l'-1)$ to $(c,l'-1)$ and finally from $(c,l'-1)$ to $(b,l)$. By assumption this forms a sequence of arcs connecting $(b,l-1)$ and $(b,l)$ that will neither go through $(a,l-1)$ nor through $(a,l)$. Since we assumed that $E_1$ is not orientable, $z$ must be odd since  ${\bm\alpha}$ together with the new arc in $E_1$ connecting $(b,l-1)$ and $(b,l)$ forms one of the two new connected components in $E_1$. But this implies
$ {\rm surg}_{E_2,F}{\rm surg}_{D,E_2}(f) = 0$ 
and we are done. For Case 6b one argues in exactly the same way.

To finish the proof we exploit the admissibility of all involved diagrams and  show that the total signs appearing in 3b, 4b, 5b, and 6b equal $1$.  \\ \noindent
\textit{$\blacktriangleright$ Admissible diagrams and signs in Cases 3b, 4b, 5b, and 6b:} In these cases the two composites of surgery maps are either zero or differ by
\begin{equation}
\label{signscomp}
(-1)^{{\pos}(c)-{\pos}(a)} {\rm sign}_{D}(b,l) {\rm sign}_{D}(d,l').
\end{equation}
We claim this equals $1$. We only give the details for the Cases 3b and 4b, since the other two (5b and 6b) are done completely analogously.

In Case 3b, the vertices in $J\cup J'$ are connected such that ${\pos}(a)-{\pos}(c)$ is even. Furthermore all direct connections (i.e. the sequence of arcs connecting two vertices in $J \cup J'$ without going through any of the other vertices in $J \cup J'$) must contain an odd total number of cups and caps. This implies that 
$${\rm sign}_{D}(b,l) {\rm sign}_{D}(d,l') = (-1)^{ \# \{ \text{dotted cups on connection from} (b,l) \text{ to } (d,l') \} }.$$
There are two possible connections from $(b,l)$ to $(d,l')$, namely 
\begin{eqnarray}
&&(b,l) \rightarrow (c,l') \rightarrow (d,l')  \label{eqn:pathone}\\
\text{ and } &&(b,l) \rightarrow (a,l) \rightarrow (d,l').\label{eqn:pathtwo}
\end{eqnarray}
Since \eqref{eqn:pathone} and \eqref{eqn:pathtwo} form an orientable circle in $D$,  the parity of the total number of dotted cups and caps on each must be the same by Lemma~\ref{lem:stupidlemmasub}. On the other hand there are the connections
\begin{eqnarray}
&&(b,l-1) \rightarrow (c,l'-1) \rightarrow (d,l'-1)  \label{eqn:paththree} \\
\text{ and } &&(b,l-1) \rightarrow (a,l-1) \rightarrow (d,l'-1).\label{eqn:pathfour}
\end{eqnarray}
Assuming that $F$ is orientable, as otherwise the surgery is zero, we obtain that the parity of dotted cups and caps on \eqref{eqn:paththree} agrees with the one on \eqref{eqn:pathone}. Similarly the one for \eqref{eqn:pathfour} and \eqref{eqn:pathtwo} agrees, hence all of them agree.
In this way we obtain two non-intersecting arc sequences from $(b,l)$ to $(d,l')$ and two from $(b,l-1)$ to $(d,l'-1)$. An easy observation shows that two of these sequences may not contain any dotted cups/caps since we assumed both surgeries to be admissible. Which in turn implies that the parity of number of dotted cups/caps is even and hence \eqref{signscomp} equals $1$.

In Case 4b we argue in a similar manner, the only difference is that now ${\pos}(c)-{\pos}(a)$ is odd and the total number of cups and caps in the following four connections is odd
\begin{eqnarray*}
(b,l-1) \rightarrow (d,l'-1), & (b,l-1) \rightarrow (a,l-1) \rightarrow (d,l'), \\
(b,l) \rightarrow (c,l'-1) \rightarrow (d,l'-1), & (b,l) \rightarrow (a,l) \rightarrow (c,l') \rightarrow (d,l').
\end{eqnarray*}
Moreover, by orientability,  they all have the same parity for the total number of dotted cups and caps. But there are two connections that are not allowed to have any dots due to admissibility. Therefore, the parity of the total number of undotted cups and caps must be odd in each case. Hence the two signs in \eqref{signscomp} cancel and we obtain $1$ again. Hence the claim is proved and the theorem established.
\end{proof}

\section{The generalized Khovanov algebra $\D$}
In this section we finally finish the construction of the graded generalized Khovanov algebra $\D$ of type $\rm D$ and illustrate the multiplication rule in a few examples. Again we fix a block $\Lambda$.
\subsection{Multiplication in $\D$} \label{sec:multiplication} The multiplication on $\D$ can be illustrated by the following commutative diagram:
\begin{eqnarray*}
\begin{xy}
  \xymatrix{
{}_{\lambda_0}(\D^\mathbf{a^{\scriptscriptstyle(1)}})_{\lambda_1} \otimes \ldots \otimes {}_{\lambda_{t-1}}(\D^\mathbf{a^{(t)}})_{\lambda_t} \ar@{-->}[rr]^{\qquad {\rm mult}} \ar@{^{(}->}[d]^{\rm glue} & & {}_{\lambda_0}(\D)_{\lambda_t} \\
{}_{\lambda_0}(\D^\mathbf{b})_{\lambda_t} \ar[rr]^{\qquad {\rm surg}} & & {}_{\lambda_0}(\D^\mathbf{e})_{\lambda_t} \ar[u]^{\rm collapse}
}
\end{xy}
\end{eqnarray*}
Here $\mathbf{b}=(b_1,\ldots,b_p)$ is the sequence $(\mathbf{a^{\scriptscriptstyle(1)}},\un{\lambda_1}',\mathbf{a^{\scriptscriptstyle(2)}},\un{\lambda_2}',\ldots, \un{\lambda_{t-1}}',\mathbf{a^{(t)}})$ with $\un{\lambda_i}'$ being $\un{\lambda_i}$ but with dots on rays removed, and $\mathbf{e}$ is the sequence of the same length as $\mathbf{b}$ but with  internal rays only.

\begin{definition}\label{def:multiplication}
For stacked circle diagrams $D=\un{\lambda}(\mathbf{a})\ov{\mu}$, $D'=\un{\eta}(\mathbf{a'})\ov{\nu}$ let
$$ {\rm mult}_{D,D'}: {}_\lambda(\D^\mathbf{a})_\mu \otimes {}_\eta(\D^\mathbf{a'})_\nu \longrightarrow {}_\lambda(\D)_\nu$$
be the linear map defined as ${\rm mult}_{D,D'}=0$ if $\mu \neq \eta$ and otherwise as follows: let $E_0=D \circ D' = \un{\lambda}(\mathbf{b^0})\ov{\eta}$ and choose $E_i=\un{\lambda}(\mathbf{b^i})\ov{\eta}$ for $1 \leq i \leq t$, a sequence of stacked circle diagrams such that $E_i$ is obtained from $E_{i-1}$ by a surgery\footnote{
The artificially looking choice of only allowing {\it admissible} surgeries between {\it admissible} stacked circle diagrams becomes transparent and consistent with \cite{BS1} if we use the approach from \cite{LS}. Rewriting the dotted cup diagrams in terms of symmetric cup diagrams, as in \cite{LS}, turns neighboured cup-cap-pairs into nested ones and our admissibility assures that we only use those that are applicable for surgery in the sense of \cite{BS1}.
} such that $\mathbf{b^t}$ consists of cup diagrams containing only rays. Then
\begin{eqnarray} \label{eqn:multiplication}
{\rm mult}_{D,D'} = {\rm collapse}_{E_t} \circ {\rm surg}_{E_t,E_{t-1}} \circ \ldots \circ {\rm surg}_{E_0,E_1} \circ {\rm glue}_{D,D'}.
\end{eqnarray} 
By Theorem \ref{thm:surgeries_commute} this is independent of the choice of sequence of surgeries and stacked circle diagrams $E_i$. By Proposition \ref{forgotgraded} it is of degree zero. 
\end{definition}

Using $\D^\mathbf{a} = \bigoplus_{\lambda,\mu \in \Lambda} {}_\lambda(\D^\mathbf{a})_\mu$, the map \eqref{eqn:multiplication} gives rise to  
\begin{eqnarray*}
{\rm mult}:\,\,\D^\mathbf{a} \otimes \D^\mathbf{a'}&\longrightarrow&\D.
\end{eqnarray*}
Specialising to the case of circle diagrams we obtain maps
$$ {\rm mult}_{\un{\lambda}\ov{\mu},\un{\eta}\ov{\nu}}: {}_\lambda(\D)_\mu \otimes {}_\eta(\D)_\nu \longrightarrow {}_\lambda(\D)_\nu,$$
for $\lambda,\mu,\eta,\nu \in \Lambda$ which extend to a common map ${\rm mult} : \D \otimes \D \longrightarrow \D$.

\begin{theorem}
\label{algebra_structure}
The map ${\rm mult}:\D \otimes \D \rightarrow \D$ defines a graded associative unitary algebra structure on $\D$ with pairwise orthogonal primitive
idempotents ${}_{\lambda} \mathbbm{1}_\la = \underline{\la} \la \overline{\la}$ for $\la\in\La$ spanning the semisimple part of degree zero and unit $\mathbbm{1}=\sum_{\la\in \La} {}_{\lambda} \mathbbm{1}_\la$.
\end{theorem}
\begin{proof}
It is clear from the definitions that the only circle diagrams of degree zero in $\B$ are the diagrams of the form ${}_{\lambda} \mathbbm{1}_\la$ and
\begin{eqnarray*}
{}_{\lambda} \mathbbm{1}_\la ( a \mu b )=
\begin{cases}
\,( a \mu b )&\text{if $\underline{\la}=a$,}\\
\,0&\text{otherwise,}
\end{cases}
&\quad&
( a \mu b ) {}_{\lambda} \mathbbm{1}_\la
=
\begin{cases}
\, ( a \mu b ) &\text{if $b = \overline{\la}$,}\\
\, 0&\text{otherwise,}
\end{cases}
\end{eqnarray*}
for $\la \in \La$ and any basis vector $(a \mu b)\in \D$.
This implies that $\{{}_{\lambda} \mathbbm{1}_\la \:|\:\la \in \La\}$ is a set of mutually orthogonal idempotents whose sum is the identity in  $\D$. The multiplication is graded by Proposition~\ref{forgotgraded} above. The associativity follows from Theorem \ref{thm:surgeries_commute}.
\end{proof}

\subsection{Explicit multiplication rules}
\label{sec:explicitmult}
Let us summarise the algebraic multiplication in the special case of circle diagrams. \\[0.25cm]

\begin{mdframed}
To multiply two oriented circle diagrams rewrite both as polynomials in the $X_{(s,0)}$'s respectively $X_{(s,1)}$'s using Proposition \ref{coho}. Use the map $X_{(s,a)} \longmapsto X_s$ for all $s$ and  $a \in \{0,1\}$ and multiply the results. Then perform the surgeries on the stacked circle diagram as in \eqref{pic:surgery}:

$$\begin{array}{ll}
\text{\bf Merge:} & \text{leave as it is,}\\
\text{\bf Split:} & \text{multiply with} \\
&\begin{tikzpicture}[thick,scale=0.4]
\draw (0,0) node[above]{$i$} .. controls +(0,-1) and +(0,-1) ..
+(1.5,0) node[above]{$j$};
\draw (0,-2) .. controls +(0,1) and +(0,1) .. +(1.5,0);
\node at (-6,-1) {$(-1)^{\pos(i)}(X_{j}-X_{i})$ if};
\end{tikzpicture}
\\
&\begin{tikzpicture}[thick,scale=0.4]
\draw (0,0) node[above]{$i$} .. controls +(0,-1) and +(0,-1) ..
+(1.5,0) node[above]{$j$};
\fill (0.75,-.75) circle(4pt);
\draw (0,-2) .. controls +(0,1) and +(0,1) .. +(1.5,0);
\fill (0.75,-1.25) circle(4pt);
\node at (-6,-1) {$(-1)^{\pos(i)}(X_{j}+X_{i})$ if};
\end{tikzpicture}\\
\text{\bf Reconnect:} & \text{if the lines are not propagating with}\\
& \text{matching orientations, multiply with } 0.
\end{array}$$
\end{mdframed}
\vspace{0.25cm}
\begin{corollary}
\label{atypicality}
The algebra $\D$ depends up to canonical isomorphism only on the atypicality or defect of the block, not on the block itself.
\end{corollary}

\begin{proof}
Given two blocks $\La$ and $\La'$ with the same atypicality then $|P_\bu(\Lambda)|=|P_\bu(\Lambda')|$ and hence the cup diagrams and basis vectors in $\D$ and $\mathbb{D}_{\La'}$ are the same up to some vertices labelled $\circ$ or $\times$. If $P_\bu(\Lambda)=\{i_1<\ldots<i_r\}$ and  $P_\bu(\Lambda')=\{i'_1<\ldots<i'_k\}$ then the identification $X_{i_s}\mapsto X_{i'_s}$ for $1\leq s\leq r$ defines the canonical isomorphism.
\end{proof}

\begin{corollary}
\label{antiaut}
The assignment $a\la b\mapsto (a \la b)^*= b^* \la a^*$, on elements from $\B$ in \eqref{DefB}, defines a graded algebra anti-automorphism of $\D$.
\end{corollary}

\begin{proof}
Note that $b^* \la a^*$ is the diagram $a\la b$ reflected in the real line, but with the weight $\la$ fixed. According to the above multiplication rules the product $(d^* \mu c^*)(b^* \la a^*)$ gives precisely the diagrams mirror the ones from $(a^* \la b^*)(c^* \mu d^*)$, except that weights are kept, i.e. not reflected. Then the claim is obvious.
\end{proof}

\subsection{Examples for the multiplication} \label{sec:examples} Since the definition of the multiplication is rather involved, we present here some explicit  examples. For simplicity we assume that the position of the leftmost vertex is $1$. These examples also illustrate the importance of Theorem~\ref{thm:surgeries_commute}.

\begin{example} \label{ex:mergesplitexample}
The following gives an example of a multiplication and illustrates that it does not depend on the order we choose for the cup-cap-pairs.

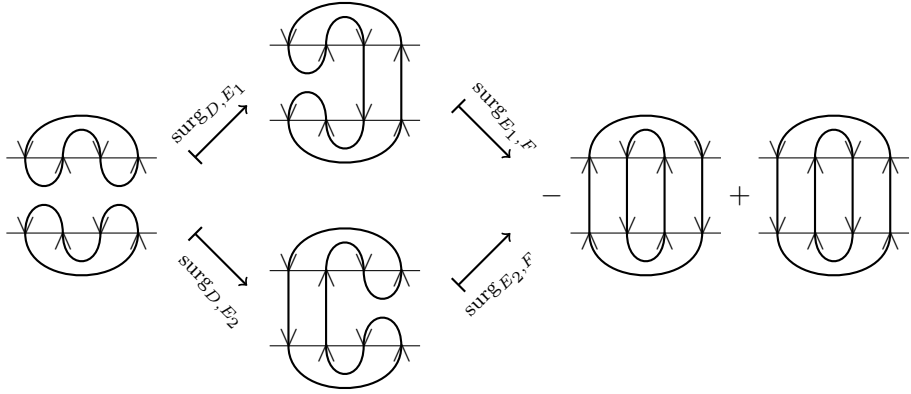
\begin{tikzpicture}[thick,decoration={zigzag,pre length=1mm,post length=1mm}]
\begin{scope}
\draw (0,1) .. controls +(0,-.5) and +(0,-.5) .. +(.5,0);
\draw (.5,1) .. controls +(0,.5) and +(0,.5) .. +(.5,0);
\draw (1,1) .. controls +(0,-.5) and +(0,-.5) .. +(.5,0);
\draw (0,1) .. controls +(0,.75) and +(0,.75) .. +(1.5,0);
\draw[thin] (-.25,1) -- +(2,0);
\node at (.005,1.1) {$\down$};
\node at (.505,0.9) {$\up$};
\node at (1.005,1.1) {$\down$};
\node at (1.505,0.9) {$\up$};

\draw (0,0) .. controls +(0,.5) and +(0,.5) .. +(.5,0);
\draw (.5,0) .. controls +(0,-.5) and +(0,-.5) .. +(.5,0);
\draw (1,0) .. controls +(0,.5) and +(0,.5) .. +(.5,0);
\draw (0,0) .. controls +(0,-.75) and +(0,-.75) .. +(1.5,0);
\draw[thin] (-.25,0) -- +(2,0);
\node at (.005,.1) {$\down$};
\node at (.505,-.1) {$\up$};
\node at (1.005,.1) {$\down$};
\node at (1.505,-.1) {$\up$};
\end{scope}

\begin{scope}[xshift=2.25cm,yshift=1cm]
\node[rotate=45] at (.2,.6) {$\scriptstyle {\rm surg}_{D,E_1}$};
\draw[|->] (0,0) -- +(.7,.7);
\end{scope}
\begin{scope}[xshift=2.25cm]
\node[rotate=-45] at (.2,-.8) {$\scriptstyle {\rm surg}_{D,E_2}$};
\draw[|->] (0,0) -- +(.7,-.7);
\end{scope}

\begin{scope}[xshift=3.5cm, yshift=1.5cm]
\draw (0,1) .. controls +(0,-.5) and +(0,-.5) .. +(.5,0);
\draw (.5,1) .. controls +(0,.5) and +(0,.5) .. +(.5,0);
\draw (0,1) .. controls +(0,.75) and +(0,.75) .. +(1.5,0);
\draw[thin] (-.25,1) -- +(2,0);
\node at (.005,1.1) {$\down$};
\node at (.505,0.9) {$\up$};
\node at (1.005,1.1) {$\down$};
\node at (1.505,0.9) {$\up$};

\draw (1,0) -- +(0,1);
\draw (1.5,0) -- +(0,1);

\draw (0,0) .. controls +(0,.5) and +(0,.5) .. +(.5,0);
\draw (.5,0) .. controls +(0,-.5) and +(0,-.5) .. +(.5,0);
\draw (0,0) .. controls +(0,-.75) and +(0,-.75) .. +(1.5,0);
\draw[thin] (-.25,0) -- +(2,0);
\node at (.005,.1) {$\down$};
\node at (.505,-.1) {$\up$};
\node at (1.005,.1) {$\down$};
\node at (1.505,-.1) {$\up$};
\end{scope}

\begin{scope}[xshift=3.5cm, yshift=-1.5cm]
\draw (1,1) .. controls +(0,-.5) and +(0,-.5) .. +(.5,0);
\draw (.5,1) .. controls +(0,.5) and +(0,.5) .. +(.5,0);
\draw (0,1) .. controls +(0,.75) and +(0,.75) .. +(1.5,0);
\draw[thin] (-.25,1) -- +(2,0);
\node at (.005,1.1) {$\down$};
\node at (.505,0.9) {$\up$};
\node at (1.005,1.1) {$\down$};
\node at (1.505,0.9) {$\up$};

\draw (0,0) -- +(0,1);
\draw (0.5,0) -- +(0,1);

\draw (1,0) .. controls +(0,.5) and +(0,.5) .. +(.5,0);
\draw (.5,0) .. controls +(0,-.5) and +(0,-.5) .. +(.5,0);
\draw (0,0) .. controls +(0,-.75) and +(0,-.75) .. +(1.5,0);
\draw[thin] (-.25,0) -- +(2,0);
\node at (.005,.1) {$\down$};
\node at (.505,-.1) {$\up$};
\node at (1.005,.1) {$\down$};
\node at (1.505,-.1) {$\up$};
\end{scope}

\begin{scope}[xshift=5.75cm,yshift=1.7cm]
\node[rotate=-45] at (.6,-.2) {$\scriptstyle {\rm surg}_{E_1,F}$};
\draw[|->] (0,0) -- +(.7,-.7);
\end{scope}
\begin{scope}[xshift=5.75cm,yshift=-0.7cm]
\node[rotate=45] at (.6,0) {$\scriptstyle {\rm surg}_{E_2,F}$};
\draw[|->] (0,0) -- +(.7,.7);
\end{scope}

\begin{scope}[xshift=7cm]
\node at (0,.5) {$-$};
\end{scope}
\begin{scope}[xshift=7.5cm]
\draw (.5,1) .. controls +(0,.5) and +(0,.5) .. +(.5,0);
\draw (0,1) .. controls +(0,.75) and +(0,.75) .. +(1.5,0);
\draw[thin] (-.25,1) -- +(2,0);
\node at (.005,0.9) {$\up$};
\node at (.505,1.1) {$\down$};
\node at (1.005,0.9) {$\up$};
\node at (1.505,1.1) {$\down$};

\draw (0,0) -- +(0,1);
\draw (0.5,0) -- +(0,1);
\draw (1,0) -- +(0,1);
\draw (1.5,0) -- +(0,1);

\draw (.5,0) .. controls +(0,-.5) and +(0,-.5) .. +(.5,0);
\draw (0,0) .. controls +(0,-.75) and +(0,-.75) .. +(1.5,0);
\draw[thin] (-.25,0) -- +(2,0);
\node at (.005,-.1) {$\up$};
\node at (.505,.1) {$\down$};
\node at (1.005,-.1) {$\up$};
\node at (1.505,.1) {$\down$};
\end{scope}
\begin{scope}[xshift=9.5cm]
\node at (0,.5) {$+$};
\end{scope}
\begin{scope}[xshift=10cm]
\draw (.5,1) .. controls +(0,.5) and +(0,.5) .. +(.5,0);
\draw (0,1) .. controls +(0,.75) and +(0,.75) .. +(1.5,0);
\draw[thin] (-.25,1) -- +(2,0);
\node at (.005,1.1) {$\down$};
\node at (.505,0.9) {$\up$};
\node at (1.005,1.1) {$\down$};
\node at (1.505,0.9) {$\up$};

\draw (0,0) -- +(0,1);
\draw (0.5,0) -- +(0,1);
\draw (1,0) -- +(0,1);
\draw (1.5,0) -- +(0,1);

\draw (.5,0) .. controls +(0,-.5) and +(0,-.5) .. +(.5,0);
\draw (0,0) .. controls +(0,-.75) and +(0,-.75) .. +(1.5,0);
\draw[thin] (-.25,0) -- +(2,0);
\node at (.005,.1) {$\down$};
\node at (.505,-.1) {$\up$};
\node at (1.005,.1) {$\down$};
\node at (1.505,-.1) {$\up$};
\end{scope}
\end{tikzpicture}

The next example is similar, but with a different number of dotted arcs.

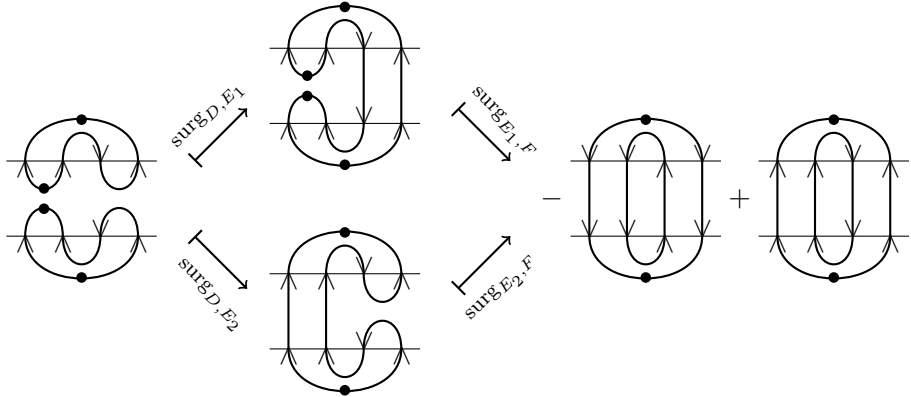
\begin{tikzpicture}[thick,decoration={zigzag,pre length=1mm,post length=1mm}]
\begin{scope}
\draw (0,1) .. controls +(0,-.5) and +(0,-.5) .. +(.5,0);
\fill (.25,0.635) circle(2pt);
\draw (.5,1) .. controls +(0,.5) and +(0,.5) .. +(.5,0);
\draw (1,1) .. controls +(0,-.5) and +(0,-.5) .. +(.5,0);
\draw (0,1) .. controls +(0,.75) and +(0,.75) .. +(1.5,0);
\fill (.75,1.55) circle(2pt);
\draw[thin] (-.25,1) -- +(2,0);
\node at (0.005,0.9) {$\up$};
\node at (0.505,0.9) {$\up$};
\node at (1.005,1.1) {$\down$};
\node at (1.505,0.9) {$\up$};

\fill (.25,.365) circle(2pt);
\draw (0,0) .. controls +(0,.5) and +(0,.5) .. +(.5,0);
\draw (.5,0) .. controls +(0,-.5) and +(0,-.5) .. +(.5,0);
\draw (1,0) .. controls +(0,.5) and +(0,.5) .. +(.5,0);
\draw (0,0) .. controls +(0,-.75) and +(0,-.75) .. +(1.5,0);
\fill (.75,-.55) circle(2pt);
\draw[thin] (-.25,0) -- +(2,0);
\node at (.005,-.1) {$\up$};
\node at (.505,-.1) {$\up$};
\node at (1.005,.1) {$\down$};
\node at (1.505,-.1) {$\up$};
\end{scope}

\begin{scope}[xshift=2.25cm,yshift=1cm]
\node[rotate=45] at (.2,.6) {$\scriptstyle {\rm surg}_{D,E_1}$};
\draw[|->] (0,0) -- +(.7,.7);
\end{scope}
\begin{scope}[xshift=2.25cm]
\node[rotate=-45] at (.2,-.8) {$\scriptstyle {\rm surg}_{D,E_2}$};
\draw[|->] (0,0) -- +(.7,-.7);
\end{scope}

\begin{scope}[xshift=3.5cm, yshift=1.5cm]
\draw (0,1) .. controls +(0,-.5) and +(0,-.5) .. +(.5,0);
\fill (.25,0.635) circle(2pt);
\draw (.5,1) .. controls +(0,.5) and +(0,.5) .. +(.5,0);
\draw (0,1) .. controls +(0,.75) and +(0,.75) .. +(1.5,0);
\fill (.75,1.55) circle(2pt);
\draw[thin] (-.25,1) -- +(2,0);
\node at (0.005,0.9) {$\up$};
\node at (0.505,0.9) {$\up$};
\node at (1.005,1.1) {$\down$};
\node at (1.505,0.9) {$\up$};

\draw (1,0) -- +(0,1);
\draw (1.5,0) -- +(0,1);

\fill (.25,.365) circle(2pt);
\draw (0,0) .. controls +(0,.5) and +(0,.5) .. +(.5,0);
\draw (.5,0) .. controls +(0,-.5) and +(0,-.5) .. +(.5,0);
\draw (0,0) .. controls +(0,-.75) and +(0,-.75) .. +(1.5,0);
\fill (.75,-.55) circle(2pt);
\draw[thin] (-.25,0) -- +(2,0);
\node at (.005,-.1) {$\up$};
\node at (.505,-.1) {$\up$};
\node at (1.005,.1) {$\down$};
\node at (1.505,-.1) {$\up$};
\end{scope}

\begin{scope}[xshift=3.5cm, yshift=-1.5cm]
\draw (1,1) .. controls +(0,-.5) and +(0,-.5) .. +(.5,0);
\draw (.5,1) .. controls +(0,.5) and +(0,.5) .. +(.5,0);
\draw (0,1) .. controls +(0,.75) and +(0,.75) .. +(1.5,0);
\fill (.75,1.55) circle(2pt);
\draw[thin] (-.25,1) -- +(2,0);
\node at (0.005,0.9) {$\up$};
\node at (0.505,0.9) {$\up$};
\node at (1.005,1.1) {$\down$};
\node at (1.505,0.9) {$\up$};

\draw (0,0) -- +(0,1);
\draw (0.5,0) -- +(0,1);

\draw (1,0) .. controls +(0,.5) and +(0,.5) .. +(.5,0);
\draw (.5,0) .. controls +(0,-.5) and +(0,-.5) .. +(.5,0);
\draw (0,0) .. controls +(0,-.75) and +(0,-.75) .. +(1.5,0);
\fill (.75,-.55) circle(2pt);
\draw[thin] (-.25,0) -- +(2,0);
\node at (.005,-.1) {$\up$};
\node at (.505,-.1) {$\up$};
\node at (1.005,.1) {$\down$};
\node at (1.505,-.1) {$\up$};
\end{scope}

\begin{scope}[xshift=5.75cm,yshift=1.7cm]
\node[rotate=-45] at (.6,-.2) {$\scriptstyle {\rm surg}_{E_1,F}$};
\draw[|->] (0,0) -- +(.7,-.7);
\end{scope}
\begin{scope}[xshift=5.75cm,yshift=-0.7cm]
\node[rotate=45] at (.6,0) {$\scriptstyle {\rm surg}_{E_2,F}$};
\draw[|->] (0,0) -- +(.7,.7);
\end{scope}

\begin{scope}[xshift=7cm]
\node at (0,.5) {$-$};
\end{scope}
\begin{scope}[xshift=7.5cm]
\draw (.5,1) .. controls +(0,.5) and +(0,.5) .. +(.5,0);
\draw (0,1) .. controls +(0,.75) and +(0,.75) .. +(1.5,0);
\fill (.75,1.55) circle(2pt);
\draw[thin] (-.25,1) -- +(2,0);
\node at (.005,1.1) {$\down$};
\node at (.505,1.1) {$\down$};
\node at (1.005,0.9) {$\up$};
\node at (1.505,1.1) {$\down$};

\draw (0,0) -- +(0,1);
\draw (0.5,0) -- +(0,1);
\draw (1,0) -- +(0,1);
\draw (1.5,0) -- +(0,1);

\draw (.5,0) .. controls +(0,-.5) and +(0,-.5) .. +(.5,0);
\draw (0,0) .. controls +(0,-.75) and +(0,-.75) .. +(1.5,0);
\fill (.75,-.55) circle(2pt);
\draw[thin] (-.25,0) -- +(2,0);
\node at (0.005,.1) {$\down$};
\node at (.505,.1) {$\down$};
\node at (1.005,-.1) {$\up$};
\node at (1.505,.1) {$\down$};
\end{scope}
\begin{scope}[xshift=9.5cm]
\node at (0,.5) {$+$};
\end{scope}
\begin{scope}[xshift=10cm]
\draw (.5,1) .. controls +(0,.5) and +(0,.5) .. +(.5,0);
\draw (0,1) .. controls +(0,.75) and +(0,.75) .. +(1.5,0);
\fill (.75,1.55) circle(2pt);
\draw[thin] (-.25,1) -- +(2,0);
\node at (0.005,0.9) {$\up$};
\node at (.505,0.9) {$\up$};
\node at (1.005,1.1) {$\down$};
\node at (1.505,0.9) {$\up$};

\draw (0,0) -- +(0,1);
\draw (0.5,0) -- +(0,1);
\draw (1,0) -- +(0,1);
\draw (1.5,0) -- +(0,1);

\draw (.5,0) .. controls +(0,-.5) and +(0,-.5) .. +(.5,0);
\draw (0,0) .. controls +(0,-.75) and +(0,-.75) .. +(1.5,0);
\fill (.75,-.55) circle(2pt);
\draw[thin] (-.25,0) -- +(2,0);
\node at (.005,-.1) {$\up$};
\node at (.505,-.1) {$\up$};
\node at (1.005,.1) {$\down$};
\node at (1.505,-.1) {$\up$};
\end{scope}
\end{tikzpicture}\hfill\\
Note that we always multiply two degree $1$ elements resulting in a linear combination of degree $2$ elements.
\end{example}

\begin{example}\label{ex:splitmergeexample}
In the previous Example~\ref{ex:mergesplitexample} the first of the two involved surgeries was always a merge while the second was a split. The following is an example for the reversed situation.

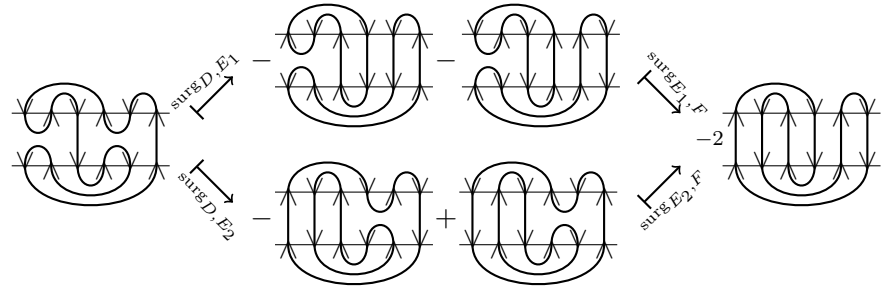
\begin{tikzpicture}[thick,decoration={zigzag,pre length=1mm,post length=1mm},scale=0.7]
\begin{scope}
\draw (0,1) .. controls +(0,-.5) and +(0,-.5) .. +(.5,0);
\draw (.5,1) .. controls +(0,.5) and +(0,.5) .. +(.5,0);
\draw (1.5,1) .. controls +(0,-.5) and +(0,-.5) .. +(.5,0);
\draw (2,1) .. controls +(0,.5) and +(0,.5) .. +(.5,0);
\draw (0,1) .. controls +(0,.75) and +(0,.75) .. +(1.5,0);
\draw[thin] (-.25,1) -- +(3,0);
\node at (.005,1.1) {$\down$};
\node at (.505,0.9) {$\up$};
\node at (1.005,1.1) {$\down$};
\node at (1.505,0.9) {$\up$};
\node at (2.005,1.1) {$\down$};
\node at (2.505,0.9) {$\up$};

\draw (1,0) -- +(0,1);
\draw (2.5,0) -- +(0,1);

\draw (0,0) .. controls +(0,.5) and +(0,.5) .. +(.5,0);
\draw (1.5,0) .. controls +(0,.5) and +(0,.5) .. +(.5,0);
\draw (1,0) .. controls +(0,-.5) and +(0,-.5) .. +(.5,0);
\draw (.5,0) .. controls +(0,-.75) and +(0,-.75) .. +(1.5,0);
\draw (0,0) .. controls +(0,-1) and +(0,-1) .. +(2.5,0);
\draw[thin] (-.25,0) -- +(3,0);
\node at (.005,.1) {$\down$};
\node at (.505,-.1) {$\up$};
\node at (1.005,.1) {$\down$};
\node at (1.505,-.1) {$\up$};
\node at (2.005,.1) {$\down$};
\node at (2.505,-.1) {$\up$};
\end{scope}

\begin{scope}[xshift=3.25cm,yshift=1cm]
\node[rotate=45] at (.2,.6) {$\scriptscriptstyle {\rm surg}_{D,E_1}$};
\draw[|->] (0,0) -- +(.7,.7);
\end{scope}
\begin{scope}[xshift=3.25cm]
\node[rotate=-45] at (.2,-.8) {$\scriptscriptstyle {\rm surg}_{D,E_2}$};
\draw[|->] (0,0) -- +(.7,-.7);
\end{scope}

\begin{scope}[xshift=4.5cm,yshift=1.5cm]
\node at (0,.5) {$-$};
\end{scope}
\begin{scope}[xshift=8.55cm,yshift=1.5cm]
\draw (0,1) .. controls +(0,-.5) and +(0,-.5) .. +(.5,0);
\draw (.5,1) .. controls +(0,.5) and +(0,.5) .. +(.5,0);
\draw (2,1) .. controls +(0,.5) and +(0,.5) .. +(.5,0);
\draw (0,1) .. controls +(0,.75) and +(0,.75) .. +(1.5,0);
\draw[thin] (-.25,1) -- +(3,0);
\node at (.005,1.1) {$\down$};
\node at (.505,0.9) {$\up$};
\node at (1.005,1.1) {$\down$};
\node at (1.505,0.9) {$\up$};
\node at (2.005,0.9) {$\up$};
\node at (2.505,1.1) {$\down$};

\draw (1,0) -- +(0,1);
\draw (1.5,0) -- +(0,1);
\draw (2,0) -- +(0,1);
\draw (2.5,0) -- +(0,1);

\draw (0,0) .. controls +(0,.5) and +(0,.5) .. +(.5,0);
\draw (1,0) .. controls +(0,-.5) and +(0,-.5) .. +(.5,0);
\draw (.5,0) .. controls +(0,-.75) and +(0,-.75) .. +(1.5,0);
\draw (0,0) .. controls +(0,-1) and +(0,-1) .. +(2.5,0);
\draw[thin] (-.25,0) -- +(3,0);
\node at (.505,.1) {$\down$};
\node at (.005,-.1) {$\up$};
\node at (1.005,.1) {$\down$};
\node at (1.505,-.1) {$\up$};
\node at (2.505,.1) {$\down$};
\node at (2.005,-.1) {$\up$};
\end{scope}

\begin{scope}[xshift=8cm,yshift=1.5cm]
\node at (0,.5) {$-$};
\end{scope}
\begin{scope}[xshift=5cm,yshift=1.5cm]
\draw (0,1) .. controls +(0,-.5) and +(0,-.5) .. +(.5,0);
\draw (.5,1) .. controls +(0,.5) and +(0,.5) .. +(.5,0);
\draw (2,1) .. controls +(0,.5) and +(0,.5) .. +(.5,0);
\draw (0,1) .. controls +(0,.75) and +(0,.75) .. +(1.5,0);
\draw[thin] (-.25,1) -- +(3,0);
\node at (.505,1.1) {$\down$};
\node at (.005,0.9) {$\up$};
\node at (1.505,1.1) {$\down$};
\node at (1.005,0.9) {$\up$};
\node at (2.005,1.1) {$\down$};
\node at (2.505,0.9) {$\up$};

\draw (1,0) -- +(0,1);
\draw (1.5,0) -- +(0,1);
\draw (2,0) -- +(0,1);
\draw (2.5,0) -- +(0,1);

\draw (0,0) .. controls +(0,.5) and +(0,.5) .. +(.5,0);
\draw (1,0) .. controls +(0,-.5) and +(0,-.5) .. +(.5,0);
\draw (.5,0) .. controls +(0,-.75) and +(0,-.75) .. +(1.5,0);
\draw (0,0) .. controls +(0,-1) and +(0,-1) .. +(2.5,0);
\draw[thin] (-.25,0) -- +(3,0);
\node at (.005,.1) {$\down$};
\node at (.505,-.1) {$\up$};
\node at (1.505,.1) {$\down$};
\node at (1.005,-.1) {$\up$};
\node at (2.005,.1) {$\down$};
\node at (2.505,-.1) {$\up$};
\end{scope}

\begin{scope}[xshift=4.5cm,yshift=-1.5cm]
\node at (0,.5) {$-$};
\end{scope}
\begin{scope}[xshift=5cm,yshift=-1.5cm]
\draw (.5,1) .. controls +(0,.5) and +(0,.5) .. +(.5,0);
\draw (1.5,1) .. controls +(0,-.5) and +(0,-.5) .. +(.5,0);
\draw (2,1) .. controls +(0,.5) and +(0,.5) .. +(.5,0);
\draw (0,1) .. controls +(0,.75) and +(0,.75) .. +(1.5,0);
\draw[thin] (-.25,1) -- +(3,0);
\node at (1.505,1.1) {$\down$};
\node at (1.005,0.9) {$\up$};
\node at (.505,1.1) {$\down$};
\node at (0.005,0.9) {$\up$};
\node at (2.505,1.1) {$\down$};
\node at (2.005,0.9) {$\up$};

\draw (0,0) -- +(0,1);
\draw (.5,0) -- +(0,1);
\draw (1,0) -- +(0,1);
\draw (2.5,0) -- +(0,1);

\draw (1,0) .. controls +(0,-.5) and +(0,-.5) .. +(.5,0);
\draw (1.5,0) .. controls +(0,.5) and +(0,.5) .. +(.5,0);
\draw (.5,0) .. controls +(0,-.75) and +(0,-.75) .. +(1.5,0);
\draw (0,0) .. controls +(0,-1) and +(0,-1) .. +(2.5,0);
\draw[thin] (-.25,0) -- +(3,0);
\node at (2.505,.1) {$\down$};
\node at (1.005,-.1) {$\up$};
\node at (.505,.1) {$\down$};
\node at (2.005,-.1) {$\up$};
\node at (1.505,.1) {$\down$};
\node at (.005,-.1) {$\up$};
\end{scope}

\begin{scope}[xshift=8cm,yshift=-1.5cm]
\node at (0,.5) {$+$};
\end{scope}
\begin{scope}[xshift=8.5cm,yshift=-1.5cm]
\draw (.5,1) .. controls +(0,.5) and +(0,.5) .. +(.5,0);
\draw (1.5,1) .. controls +(0,-.5) and +(0,-.5) .. +(.5,0);
\draw (2,1) .. controls +(0,.5) and +(0,.5) .. +(.5,0);
\draw (0,1) .. controls +(0,.75) and +(0,.75) .. +(1.5,0);
\draw[thin] (-.25,1) -- +(3,0);
\node at (.005,1.1) {$\down$};
\node at (.505,0.9) {$\up$};
\node at (1.005,1.1) {$\down$};
\node at (1.505,0.9) {$\up$};
\node at (2.005,1.1) {$\down$};
\node at (2.505,0.9) {$\up$};

\draw (0,0) -- +(0,1);
\draw (.5,0) -- +(0,1);
\draw (1,0) -- +(0,1);
\draw (2.5,0) -- +(0,1);

\draw (1,0) .. controls +(0,-.5) and +(0,-.5) .. +(.5,0);
\draw (1.5,0) .. controls +(0,.5) and +(0,.5) .. +(.5,0);
\draw (.5,0) .. controls +(0,-.75) and +(0,-.75) .. +(1.5,0);
\draw (0,0) .. controls +(0,-1) and +(0,-1) .. +(2.5,0);
\draw[thin] (-.25,0) -- +(3,0);
\node at (.005,.1) {$\down$};
\node at (.505,-.1) {$\up$};
\node at (1.005,.1) {$\down$};
\node at (1.505,-.1) {$\up$};
\node at (2.005,.1) {$\down$};
\node at (2.505,-.1) {$\up$};
\end{scope}

\begin{scope}[xshift=11.75cm,yshift=1.7cm]
\node[rotate=-45] at (.6,-.2) {$\scriptscriptstyle {\rm surg}_{E_1,F}$};
\draw[|->] (0,0) -- +(.7,-.7);
\end{scope}
\begin{scope}[xshift=11.75cm,yshift=-0.7cm]
\node[rotate=45] at (.6,0) {$\scriptscriptstyle {\rm surg}_{E_2,F}$};
\draw[|->] (0,0) -- +(.7,.7);
\end{scope}

\begin{scope}[xshift=13cm]
\node at (0,.5) {$\scriptstyle -2$};
\end{scope}
\begin{scope}[xshift=13.5cm]
\draw (.5,1) .. controls +(0,.5) and +(0,.5) .. +(.5,0);
\draw (2,1) .. controls +(0,.5) and +(0,.5) .. +(.5,0);
\draw (0,1) .. controls +(0,.75) and +(0,.75) .. +(1.5,0);
\draw[thin] (-.25,1) -- +(3,0);
\node at (.505,1.1) {$\down$};
\node at (.005,0.9) {$\up$};
\node at (1.505,1.1) {$\down$};
\node at (1.005,0.9) {$\up$};
\node at (2.505,1.1) {$\down$};
\node at (2.005,0.9) {$\up$};

\draw (0,0) -- +(0,1);
\draw (.5,0) -- +(0,1);
\draw (1,0) -- +(0,1);
\draw (2.5,0) -- +(0,1);
\draw (1.5,0) -- +(0,1);
\draw (2,0) -- +(0,1);

\draw (1,0) .. controls +(0,-.5) and +(0,-.5) .. +(.5,0);
\draw (.5,0) .. controls +(0,-.75) and +(0,-.75) .. +(1.5,0);
\draw (0,0) .. controls +(0,-1) and +(0,-1) .. +(2.5,0);
\draw[thin] (-.25,0) -- +(3,0);
\node at (.505,.1) {$\down$};
\node at (.005,-.1) {$\up$};
\node at (1.505,.1) {$\down$};
\node at (1.005,-.1) {$\up$};
\node at (2.505,.1) {$\down$};
\node at (2.005,-.1) {$\up$};
\end{scope}
\end{tikzpicture}

We could also include some decorations (as long as we make sure that all the diagrams stay admissible): \hfill\\

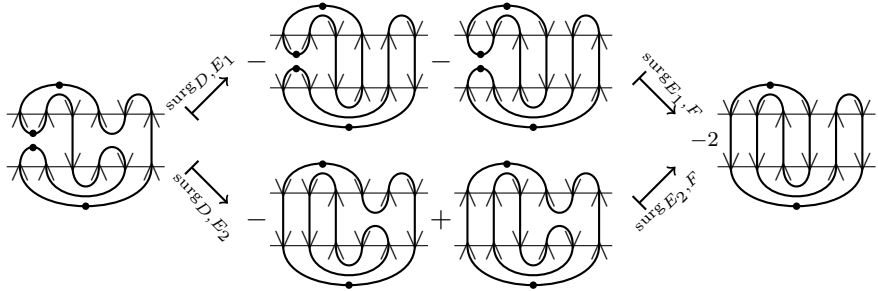
\begin{tikzpicture}[thick,decoration={zigzag,pre length=1mm,post length=1mm},scale=0.7]
\begin{scope}
\draw (0,1) .. controls +(0,-.5) and +(0,-.5) .. +(.5,0);
\draw (.5,1) .. controls +(0,.5) and +(0,.5) .. +(.5,0);
\draw (1.5,1) .. controls +(0,-.5) and +(0,-.5) .. +(.5,0);
\draw (2,1) .. controls +(0,.5) and +(0,.5) .. +(.5,0);
\draw (0,1) .. controls +(0,.75) and +(0,.75) .. +(1.5,0);
\draw[thin] (-.25,1) -- +(3,0);
\node at (.005,0.9) {$\up$};
\node at (.505,0.9) {$\up$};
\node at (1.005,1.1) {$\down$};
\node at (1.505,0.9) {$\up$};
\node at (2.005,1.1) {$\down$};
\node at (2.505,0.9) {$\up$};

\fill (0.25,.365) circle(2pt);
\fill (0.25,.635) circle(2pt);
\fill (.75,1.55) circle(2pt);
\fill (1.25,-0.75) circle(2pt);

\draw (1,0) -- +(0,1);
\draw (2.5,0) -- +(0,1);

\draw (0,0) .. controls +(0,.5) and +(0,.5) .. +(.5,0);
\draw (1.5,0) .. controls +(0,.5) and +(0,.5) .. +(.5,0);
\draw (1,0) .. controls +(0,-.5) and +(0,-.5) .. +(.5,0);
\draw (.5,0) .. controls +(0,-.75) and +(0,-.75) .. +(1.5,0);
\draw (0,0) .. controls +(0,-1) and +(0,-1) .. +(2.5,0);
\draw[thin] (-.25,0) -- +(3,0);
\node at (.005,-.1) {$\up$};
\node at (.505,-.1) {$\up$};
\node at (1.005,.1) {$\down$};
\node at (1.505,-.1) {$\up$};
\node at (2.005,.1) {$\down$};
\node at (2.505,-.1) {$\up$};
\end{scope}

\begin{scope}[xshift=3.25cm,yshift=1cm]
\node[rotate=45] at (.2,.6) {$\scriptscriptstyle {\rm surg}_{D,E_1}$};
\draw[|->] (0,0) -- +(.7,.7);
\end{scope}
\begin{scope}[xshift=3.25cm]
\node[rotate=-45] at (.2,-.8) {$\scriptscriptstyle {\rm surg}_{D,E_2}$};
\draw[|->] (0,0) -- +(.7,-.7);
\end{scope}

\begin{scope}[xshift=4.5cm,yshift=1.5cm]
\node at (0,.5) {$-$};
\end{scope}
\begin{scope}[xshift=5cm,yshift=1.5cm]
\draw (0,1) .. controls +(0,-.5) and +(0,-.5) .. +(.5,0);
\draw (.5,1) .. controls +(0,.5) and +(0,.5) .. +(.5,0);
\draw (2,1) .. controls +(0,.5) and +(0,.5) .. +(.5,0);
\draw (0,1) .. controls +(0,.75) and +(0,.75) .. +(1.5,0);
\draw[thin] (-.25,1) -- +(3,0);
\node at (.005,0.9) {$\up$};
\node at (.505,0.9) {$\up$};
\node at (1.005,1.1) {$\down$};
\node at (1.505,0.9) {$\up$};
\node at (2.005,0.9) {$\up$};
\node at (2.505,1.1) {$\down$};

\fill (0.25,.365) circle(2pt);
\fill (0.25,.635) circle(2pt);
\fill (.75,1.55) circle(2pt);
\fill (1.25,-0.75) circle(2pt);

\draw (1,0) -- +(0,1);
\draw (1.5,0) -- +(0,1);
\draw (2,0) -- +(0,1);
\draw (2.5,0) -- +(0,1);

\draw (0,0) .. controls +(0,.5) and +(0,.5) .. +(.5,0);
\draw (1,0) .. controls +(0,-.5) and +(0,-.5) .. +(.5,0);
\draw (.5,0) .. controls +(0,-.75) and +(0,-.75) .. +(1.5,0);
\draw (0,0) .. controls +(0,-1) and +(0,-1) .. +(2.5,0);
\draw[thin] (-.25,0) -- +(3,0);
\node at (.505,.1) {$\down$};
\node at (.005,.1) {$\down$};
\node at (1.005,.1) {$\down$};
\node at (1.505,-.1) {$\up$};
\node at (2.505,.1) {$\down$};
\node at (2.005,-.1) {$\up$};
\end{scope}

\begin{scope}[xshift=8cm,yshift=1.5cm]
\node at (0,.5) {$-$};
\end{scope}
\begin{scope}[xshift=8.5cm,yshift=1.5cm]
\draw (0,1) .. controls +(0,-.5) and +(0,-.5) .. +(.5,0);
\draw (.5,1) .. controls +(0,.5) and +(0,.5) .. +(.5,0);
\draw (2,1) .. controls +(0,.5) and +(0,.5) .. +(.5,0);
\draw (0,1) .. controls +(0,.75) and +(0,.75) .. +(1.5,0);
\draw[thin] (-.25,1) -- +(3,0);
\node at (.005,1.1) {$\down$};
\node at (.505,1.1) {$\down$};
\node at (1.505,1.1) {$\down$};
\node at (1.005,0.9) {$\up$};
\node at (2.005,1.1) {$\down$};
\node at (2.505,0.9) {$\up$};

\fill (0.25,.365) circle(2pt);
\fill (0.25,.635) circle(2pt);
\fill (.75,1.55) circle(2pt);
\fill (1.25,-0.75) circle(2pt);

\draw (1,0) -- +(0,1);
\draw (1.5,0) -- +(0,1);
\draw (2,0) -- +(0,1);
\draw (2.5,0) -- +(0,1);

\draw (0,0) .. controls +(0,.5) and +(0,.5) .. +(.5,0);
\draw (1,0) .. controls +(0,-.5) and +(0,-.5) .. +(.5,0);
\draw (.5,0) .. controls +(0,-.75) and +(0,-.75) .. +(1.5,0);
\draw (0,0) .. controls +(0,-1) and +(0,-1) .. +(2.5,0);
\draw[thin] (-.25,0) -- +(3,0);
\node at (.005,-.1) {$\up$};
\node at (.505,-.1) {$\up$};
\node at (1.505,.1) {$\down$};
\node at (1.005,-.1) {$\up$};
\node at (2.005,.1) {$\down$};
\node at (2.505,-.1) {$\up$};
\end{scope}

\begin{scope}[xshift=4.5cm,yshift=-1.5cm]
\node at (0,.5) {$-$};
\end{scope}
\begin{scope}[xshift=5cm,yshift=-1.5cm]
\draw (.5,1) .. controls +(0,.5) and +(0,.5) .. +(.5,0);
\draw (1.5,1) .. controls +(0,-.5) and +(0,-.5) .. +(.5,0);
\draw (2,1) .. controls +(0,.5) and +(0,.5) .. +(.5,0);
\draw (0,1) .. controls +(0,.75) and +(0,.75) .. +(1.5,0);
\draw[thin] (-.25,1) -- +(3,0);
\node at (1.505,1.1) {$\down$};
\node at (1.005,0.9) {$\up$};
\node at (.505,1.1) {$\down$};
\node at (.005,1.1) {$\down$};
\node at (2.505,1.1) {$\down$};
\node at (2.005,0.9) {$\up$};

\fill (.75,1.55) circle(2pt);
\fill (1.25,-0.75) circle(2pt);

\draw (0,0) -- +(0,1);
\draw (.5,0) -- +(0,1);
\draw (1,0) -- +(0,1);
\draw (2.5,0) -- +(0,1);

\draw (1,0) .. controls +(0,-.5) and +(0,-.5) .. +(.5,0);
\draw (1.5,0) .. controls +(0,.5) and +(0,.5) .. +(.5,0);
\draw (.5,0) .. controls +(0,-.75) and +(0,-.75) .. +(1.5,0);
\draw (0,0) .. controls +(0,-1) and +(0,-1) .. +(2.5,0);
\draw[thin] (-.25,0) -- +(3,0);
\node at (2.505,.1) {$\down$};
\node at (1.005,-.1) {$\up$};
\node at (.505,.1) {$\down$};
\node at (2.005,-.1) {$\up$};
\node at (1.505,.1) {$\down$};
\node at (.005,.1) {$\down$};
\end{scope}

\begin{scope}[xshift=8cm,yshift=-1.5cm]
\node at (0,.5) {$+$};
\end{scope}
\begin{scope}[xshift=8.5cm,yshift=-1.5cm]
\draw (.5,1) .. controls +(0,.5) and +(0,.5) .. +(.5,0);
\draw (1.5,1) .. controls +(0,-.5) and +(0,-.5) .. +(.5,0);
\draw (2,1) .. controls +(0,.5) and +(0,.5) .. +(.5,0);
\draw (0,1) .. controls +(0,.75) and +(0,.75) .. +(1.5,0);
\draw[thin] (-.25,1) -- +(3,0);
\node at (.005,0.9) {$\up$};
\node at (.505,0.9) {$\up$};
\node at (1.005,1.1) {$\down$};
\node at (1.505,0.9) {$\up$};
\node at (2.005,1.1) {$\down$};
\node at (2.505,0.9) {$\up$};

\fill (.75,1.55) circle(2pt);
\fill (1.25,-0.75) circle(2pt);

\draw (0,0) -- +(0,1);
\draw (.5,0) -- +(0,1);
\draw (1,0) -- +(0,1);
\draw (2.5,0) -- +(0,1);

\draw (1,0) .. controls +(0,-.5) and +(0,-.5) .. +(.5,0);
\draw (1.5,0) .. controls +(0,.5) and +(0,.5) .. +(.5,0);
\draw (.5,0) .. controls +(0,-.75) and +(0,-.75) .. +(1.5,0);
\draw (0,0) .. controls +(0,-1) and +(0,-1) .. +(2.5,0);
\draw[thin] (-.25,0) -- +(3,0);
\node at (.005,-.1) {$\up$};
\node at (.505,-.1) {$\up$};
\node at (1.005,.1) {$\down$};
\node at (1.505,-.1) {$\up$};
\node at (2.005,.1) {$\down$};
\node at (2.505,-.1) {$\up$};
\end{scope}

\begin{scope}[xshift=11.75cm,yshift=1.7cm]
\node[rotate=-45] at (.6,-.2) {$\scriptscriptstyle {\rm surg}_{E_1,F}$};
\draw[|->] (0,0) -- +(.7,-.7);
\end{scope}
\begin{scope}[xshift=11.75cm,yshift=-0.7cm]
\node[rotate=45] at (.6,0) {$\scriptscriptstyle {\rm surg}_{E_2,F}$};
\draw[|->] (0,0) -- +(.7,.7);
\end{scope}

\begin{scope}[xshift=13cm]
\node at (0,.5) {$\scriptstyle -2$};
\end{scope}
\begin{scope}[xshift=13.5cm]
\draw (.5,1) .. controls +(0,.5) and +(0,.5) .. +(.5,0);
\draw (2,1) .. controls +(0,.5) and +(0,.5) .. +(.5,0);
\draw (0,1) .. controls +(0,.75) and +(0,.75) .. +(1.5,0);
\draw[thin] (-.25,1) -- +(3,0);
\node at (.005,1.1) {$\down$};
\node at (.505,1.1) {$\down$};
\node at (1.505,1.1) {$\down$};
\node at (1.005,0.9) {$\up$};
\node at (2.505,1.1) {$\down$};
\node at (2.005,0.9) {$\up$};

\fill (.75,1.55) circle(2pt);
\fill (1.25,-0.75) circle(2pt);

\draw (0,0) -- +(0,1);
\draw (.5,0) -- +(0,1);
\draw (1,0) -- +(0,1);
\draw (2.5,0) -- +(0,1);
\draw (1.5,0) -- +(0,1);
\draw (2,0) -- +(0,1);

\draw (1,0) .. controls +(0,-.5) and +(0,-.5) .. +(.5,0);
\draw (.5,0) .. controls +(0,-.75) and +(0,-.75) .. +(1.5,0);
\draw (0,0) .. controls +(0,-1) and +(0,-1) .. +(2.5,0);
\draw[thin] (-.25,0) -- +(3,0);
\node at (.005,.1) {$\down$};
\node at (.505,.1) {$\down$};
\node at (1.505,.1) {$\down$};
\node at (1.005,-.1) {$\up$};
\node at (2.505,.1) {$\down$};
\node at (2.005,-.1) {$\up$};
\end{scope}
\end{tikzpicture}
\end{example}

Since we only defined the surgery maps for admissible diagrams, not all surgery maps in the next example are formally defined.  It should however illustrate what goes wrong if one would ignore the admissibility assumptions and define the surgery in the obvious way for all stacked circle diagrams.

\begin{example}
To illustrate that the admissibility assumptions are really necessary to get a well-defined muliplication, we apply the surgery maps formally ignoring whether we actually stay in admissible diagrams. One observes that the two maps differ exactly by a sign (namely the sign $w$ which in the proof of  Theorem~\ref{thm:surgeries_commute} turned out to be $1$ for admissible diagrams). \hfill\\

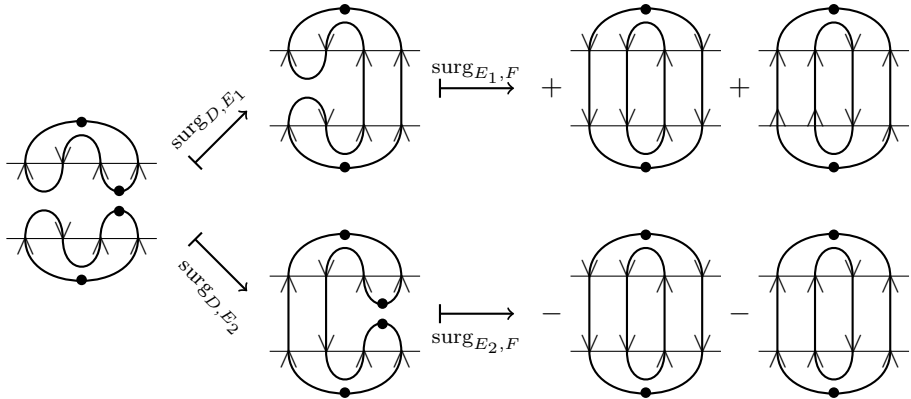
\begin{tikzpicture}[thick,decoration={zigzag,pre length=1mm,post length=1mm}]
\begin{scope}
\draw (0,1) .. controls +(0,-.5) and +(0,-.5) .. +(.5,0);
\draw (.5,1) .. controls +(0,.5) and +(0,.5) .. +(.5,0);
\draw (1,1) .. controls +(0,-.5) and +(0,-.5) .. +(.5,0);
\fill (1.25,0.635) circle(2pt);
\draw (0,1) .. controls +(0,.75) and +(0,.75) .. +(1.5,0);
\fill (.75,1.55) circle(2pt);
\draw[thin] (-.25,1) -- +(2,0);
\node at (0.005,0.9) {$\up$};
\node at (.505,1.1) {$\down$};
\node at (1.005,0.9) {$\up$};
\node at (1.505,0.9) {$\up$};

\draw (0,0) .. controls +(0,.5) and +(0,.5) .. +(.5,0);
\draw (.5,0) .. controls +(0,-.5) and +(0,-.5) .. +(.5,0);
\draw (1,0) .. controls +(0,.5) and +(0,.5) .. +(.5,0);
\fill (1.25,.365) circle(2pt);
\draw (0,0) .. controls +(0,-.75) and +(0,-.75) .. +(1.5,0);
\fill (.75,-.55) circle(2pt);
\draw[thin] (-.25,0) -- +(2,0);
\node at (.005,-.1) {$\up$};
\node at (.505,.1) {$\down$};
\node at (1.005,-.1) {$\up$};
\node at (1.505,-.1) {$\up$};
\end{scope}

\begin{scope}[xshift=2.25cm,yshift=1cm]
\node[rotate=45] at (.2,.6) {$\scriptstyle {\rm surg}_{D,E_1}$};
\draw[|->] (0,0) -- +(.7,.7);
\end{scope}
\begin{scope}[xshift=2.25cm]
\node[rotate=-45] at (.2,-.8) {$\scriptstyle {\rm surg}_{D,E_2}$};
\draw[|->] (0,0) -- +(.7,-.7);
\end{scope}

\begin{scope}[xshift=3.5cm, yshift=1.5cm]
\draw (0,1) .. controls +(0,-.5) and +(0,-.5) .. +(.5,0);
\draw (.5,1) .. controls +(0,.5) and +(0,.5) .. +(.5,0);
\draw (0,1) .. controls +(0,.75) and +(0,.75) .. +(1.5,0);
\fill (.75,1.55) circle(2pt);
\draw[thin] (-.25,1) -- +(2,0);
\node at (0.005,0.9) {$\up$};
\node at (.505,1.1) {$\down$};
\node at (1.005,0.9) {$\up$};
\node at (1.505,0.9) {$\up$};

\draw (1,0) -- +(0,1);
\draw (1.5,0) -- +(0,1);

\draw (0,0) .. controls +(0,.5) and +(0,.5) .. +(.5,0);
\draw (.5,0) .. controls +(0,-.5) and +(0,-.5) .. +(.5,0);
\draw (0,0) .. controls +(0,-.75) and +(0,-.75) .. +(1.5,0);
\fill (.75,-.55) circle(2pt);
\draw[thin] (-.25,0) -- +(2,0);
\node at (.005,-.1) {$\up$};
\node at (.505,.1) {$\down$};
\node at (1.005,-.1) {$\up$};
\node at (1.505,-.1) {$\up$};
\end{scope}

\begin{scope}[xshift=3.5cm, yshift=-1.5cm]
\draw (1,1) .. controls +(0,-.5) and +(0,-.5) .. +(.5,0);
\fill (1.25,0.635) circle(2pt);
\draw (.5,1) .. controls +(0,.5) and +(0,.5) .. +(.5,0);
\draw (0,1) .. controls +(0,.75) and +(0,.75) .. +(1.5,0);
\fill (.75,1.55) circle(2pt);
\draw[thin] (-.25,1) -- +(2,0);
\node at (0.005,0.9) {$\up$};
\node at (.505,1.1) {$\down$};
\node at (1.005,0.9) {$\up$};
\node at (1.505,0.9) {$\up$};

\draw (0,0) -- +(0,1);
\draw (0.5,0) -- +(0,1);

\draw (1,0) .. controls +(0,.5) and +(0,.5) .. +(.5,0);
\fill (1.25,0.365) circle(2pt);
\draw (.5,0) .. controls +(0,-.5) and +(0,-.5) .. +(.5,0);
\draw (0,0) .. controls +(0,-.75) and +(0,-.75) .. +(1.5,0);
\fill (.75,-.55) circle(2pt);
\draw[thin] (-.25,0) -- +(2,0);
\node at (.005,-.1) {$\up$};
\node at (.505,.1) {$\down$};
\node at (1.005,-.1) {$\up$};
\node at (1.505,-.1) {$\up$};
\end{scope}

\begin{scope}[xshift=5.5cm,yshift=2cm]
\node at (.5,.2) {$\scriptstyle {\rm surg}_{E_1,F}$};
\draw[|->] (0,0) -- +(1,0);
\end{scope}
\begin{scope}[xshift=5.5cm,yshift=-1cm]
\node at (.5,-.4) {$\scriptstyle {\rm surg}_{E_2,F}$};
\draw[|->] (0,0) -- +(1,0);
\end{scope}

\begin{scope}[xshift=7cm,yshift=1.5cm]
\node at (0,.5) {$+$};
\end{scope}
\begin{scope}[xshift=7.5cm,yshift=1.5cm]
\draw (.5,1) .. controls +(0,.5) and +(0,.5) .. +(.5,0);
\draw (0,1) .. controls +(0,.75) and +(0,.75) .. +(1.5,0);
\fill (.75,1.55) circle(2pt);
\draw[thin] (-.25,1) -- +(2,0);
\node at (0.005,1.1) {$\down$};
\node at (.505,1.1) {$\down$};
\node at (1.005,0.9) {$\up$};
\node at (1.505,1.1) {$\down$};

\draw (0,0) -- +(0,1);
\draw (0.5,0) -- +(0,1);
\draw (1,0) -- +(0,1);
\draw (1.5,0) -- +(0,1);

\draw (.5,0) .. controls +(0,-.5) and +(0,-.5) .. +(.5,0);
\draw (0,0) .. controls +(0,-.75) and +(0,-.75) .. +(1.5,0);
\fill (.75,-.55) circle(2pt);
\draw[thin] (-.25,0) -- +(2,0);
\node at (.005,.1) {$\down$};
\node at (.505,.1) {$\down$};
\node at (1.005,-.1) {$\up$};
\node at (1.505,.1) {$\down$};
\end{scope}
\begin{scope}[xshift=9.5cm,yshift=1.5cm]
\node at (0,.5) {$+$};
\end{scope}
\begin{scope}[xshift=10cm,yshift=1.5cm]
\draw (.5,1) .. controls +(0,.5) and +(0,.5) .. +(.5,0);
\draw (0,1) .. controls +(0,.75) and +(0,.75) .. +(1.5,0);
\fill (.75,1.55) circle(2pt);
\draw[thin] (-.25,1) -- +(2,0);
\node at (.005,0.9) {$\up$};
\node at (.505,0.9) {$\up$};
\node at (1.005,1.1) {$\down$};
\node at (1.505,0.9) {$\up$};

\draw (0,0) -- +(0,1);
\draw (0.5,0) -- +(0,1);
\draw (1,0) -- +(0,1);
\draw (1.5,0) -- +(0,1);

\draw (.5,0) .. controls +(0,-.5) and +(0,-.5) .. +(.5,0);
\draw (0,0) .. controls +(0,-.75) and +(0,-.75) .. +(1.5,0);
\fill (.75,-.55) circle(2pt);
\draw[thin] (-.25,0) -- +(2,0);
\node at (.005,.1) {$\up$};
\node at (.505,.1) {$\up$};
\node at (1.005,.1) {$\down$};
\node at (1.505,-.1) {$\up$};
\end{scope}

\begin{scope}[xshift=7cm,yshift=-1.5cm]
\node at (0,.5) {$-$};
\end{scope}
\begin{scope}[xshift=7.5cm,yshift=-1.5cm]
\draw (.5,1) .. controls +(0,.5) and +(0,.5) .. +(.5,0);
\draw (0,1) .. controls +(0,.75) and +(0,.75) .. +(1.5,0);
\fill (.75,1.55) circle(2pt);
\draw[thin] (-.25,1) -- +(2,0);
\node at (.005,1.1) {$\down$};
\node at (.505,1.1) {$\down$};
\node at (1.005,0.9) {$\up$};
\node at (1.505,1.1) {$\down$};

\draw (0,0) -- +(0,1);
\draw (0.5,0) -- +(0,1);
\draw (1,0) -- +(0,1);
\draw (1.5,0) -- +(0,1);

\draw (.5,0) .. controls +(0,-.5) and +(0,-.5) .. +(.5,0);
\draw (0,0) .. controls +(0,-.75) and +(0,-.75) .. +(1.5,0);
\fill (.75,-.55) circle(2pt);
\draw[thin] (-.25,0) -- +(2,0);
\node at (.005,.1) {$\down$};
\node at (.505,.1) {$\down$};
\node at (1.005,-.1) {$\up$};
\node at (1.505,.1) {$\down$};
\end{scope}
\begin{scope}[xshift=9.5cm,yshift=-1.5cm]
\node at (0,.5) {$-$};
\end{scope}
\begin{scope}[xshift=10cm,yshift=-1.5cm]
\draw (.5,1) .. controls +(0,.5) and +(0,.5) .. +(.5,0);
\draw (0,1) .. controls +(0,.75) and +(0,.75) .. +(1.5,0);
\fill (.75,1.55) circle(2pt);
\draw[thin] (-.25,1) -- +(2,0);
\node at (.005,0.9) {$\up$};
\node at (.505,0.9) {$\up$};
\node at (1.005,1.1) {$\down$};
\node at (1.505,0.9) {$\up$};

\draw (0,0) -- +(0,1);
\draw (0.5,0) -- +(0,1);
\draw (1,0) -- +(0,1);
\draw (1.5,0) -- +(0,1);

\draw (.5,0) .. controls +(0,-.5) and +(0,-.5) .. +(.5,0);
\draw (0,0) .. controls +(0,-.75) and +(0,-.75) .. +(1.5,0);
\fill (.75,-.55) circle(2pt);
\draw[thin] (-.25,0) -- +(2,0);
\node at (.005,-.1) {$\up$};
\node at (.505,-.1) {$\up$};
\node at (1.005,.1) {$\down$};
\node at (1.505,-.1) {$\up$};
\end{scope}
\end{tikzpicture}

Here the surgery starting from a non-admissible diagram was a split, but we could also have a merge:

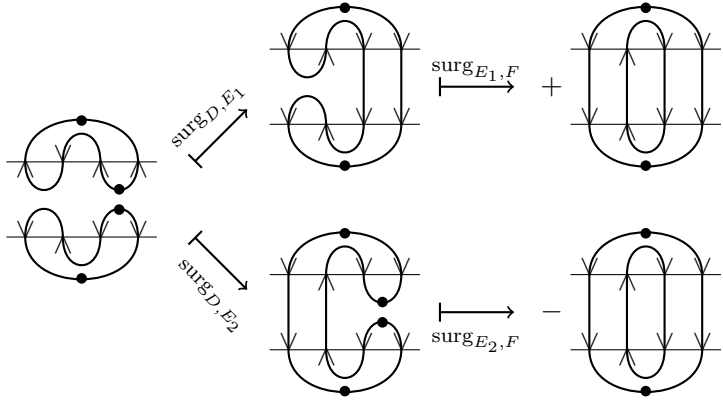
\begin{tikzpicture}[thick,decoration={zigzag,pre length=1mm,post length=1mm}]
\begin{scope}
\draw (0,1) .. controls +(0,-.5) and +(0,-.5) .. +(.5,0);
\draw (.5,1) .. controls +(0,.5) and +(0,.5) .. +(.5,0);
\draw (1,1) .. controls +(0,-.5) and +(0,-.5) .. +(.5,0);
\fill (1.25,0.635) circle(2pt);
\draw (0,1) .. controls +(0,.75) and +(0,.75) .. +(1.5,0);
\fill (.75,1.55) circle(2pt);
\draw[thin] (-.25,1) -- +(2,0);
\node at (0.005,0.9) {$\up$};
\node at (.505,1.1) {$\down$};
\node at (1.005,0.9) {$\up$};
\node at (1.505,0.9) {$\up$};

\draw (0,0) .. controls +(0,.5) and +(0,.5) .. +(.5,0);
\draw (.5,0) .. controls +(0,-.5) and +(0,-.5) .. +(.5,0);
\draw (1,0) .. controls +(0,.5) and +(0,.5) .. +(.5,0);
\fill (1.25,.365) circle(2pt);
\draw (0,0) .. controls +(0,-.75) and +(0,-.75) .. +(1.5,0);
\fill (.75,-.55) circle(2pt);
\draw[thin] (-.25,0) -- +(2,0);
\node at (.005,.1) {$\down$};
\node at (.505,-.1) {$\up$};
\node at (1.005,.1) {$\down$};
\node at (1.505,.1) {$\down$};
\end{scope}

\begin{scope}[xshift=2.25cm,yshift=1cm]
\node[rotate=45] at (.2,.6) {$\scriptstyle {\rm surg}_{D,E_1}$};
\draw[|->] (0,0) -- +(.7,.7);
\end{scope}
\begin{scope}[xshift=2.25cm]
\node[rotate=-45] at (.2,-.8) {$\scriptstyle {\rm surg}_{D,E_2}$};
\draw[|->] (0,0) -- +(.7,-.7);
\end{scope}

\begin{scope}[xshift=3.5cm, yshift=1.5cm]
\draw (0,1) .. controls +(0,-.5) and +(0,-.5) .. +(.5,0);
\draw (.5,1) .. controls +(0,.5) and +(0,.5) .. +(.5,0);
\draw (0,1) .. controls +(0,.75) and +(0,.75) .. +(1.5,0);
\fill (.75,1.55) circle(2pt);
\draw[thin] (-.25,1) -- +(2,0);
\node at (.005,1.1) {$\down$};
\node at (.505,0.9) {$\up$};
\node at (1.005,1.1) {$\down$};
\node at (1.505,1.1) {$\down$};

\draw (1,0) -- +(0,1);
\draw (1.5,0) -- +(0,1);

\draw (0,0) .. controls +(0,.5) and +(0,.5) .. +(.5,0);
\draw (.5,0) .. controls +(0,-.5) and +(0,-.5) .. +(.5,0);
\draw (0,0) .. controls +(0,-.75) and +(0,-.75) .. +(1.5,0);
\fill (.75,-.55) circle(2pt);
\draw[thin] (-.25,0) -- +(2,0);
\node at (.005,.1) {$\down$};
\node at (.505,-.1) {$\up$};
\node at (1.005,.1) {$\down$};
\node at (1.505,.1) {$\down$};
\end{scope}

\begin{scope}[xshift=3.5cm, yshift=-1.5cm]
\draw (1,1) .. controls +(0,-.5) and +(0,-.5) .. +(.5,0);
\fill (1.25,0.635) circle(2pt);
\draw (.5,1) .. controls +(0,.5) and +(0,.5) .. +(.5,0);
\draw (0,1) .. controls +(0,.75) and +(0,.75) .. +(1.5,0);
\fill (.75,1.55) circle(2pt);
\draw[thin] (-.25,1) -- +(2,0);
\node at (.005,1.1) {$\down$};
\node at (.505,0.9) {$\up$};
\node at (1.005,1.1) {$\down$};
\node at (1.505,1.1) {$\down$};

\draw (0,0) -- +(0,1);
\draw (0.5,0) -- +(0,1);

\draw (1,0) .. controls +(0,.5) and +(0,.5) .. +(.5,0);
\fill (1.25,0.365) circle(2pt);
\draw (.5,0) .. controls +(0,-.5) and +(0,-.5) .. +(.5,0);
\draw (0,0) .. controls +(0,-.75) and +(0,-.75) .. +(1.5,0);
\fill (.75,-.55) circle(2pt);
\draw[thin] (-.25,0) -- +(2,0);
\node at (.005,.1) {$\down$};
\node at (.505,-.1) {$\up$};
\node at (1.005,.1) {$\down$};
\node at (1.505,.1) {$\down$};
\end{scope}

\begin{scope}[xshift=5.5cm,yshift=2cm]
\node at (.5,.2) {$\scriptstyle {\rm surg}_{E_1,F}$};
\draw[|->] (0,0) -- +(1,0);
\end{scope}
\begin{scope}[xshift=5.5cm,yshift=-1cm]
\node at (.5,-.4) {$\scriptstyle {\rm surg}_{E_2,F}$};
\draw[|->] (0,0) -- +(1,0);
\end{scope}

\begin{scope}[xshift=7cm,yshift=1.5cm]
\node at (0,.5) {$+$};
\end{scope}
\begin{scope}[xshift=7.5cm,yshift=1.5cm]
\draw (.5,1) .. controls +(0,.5) and +(0,.5) .. +(.5,0);
\draw (0,1) .. controls +(0,.75) and +(0,.75) .. +(1.5,0);
\fill (.75,1.55) circle(2pt);
\draw[thin] (-.25,1) -- +(2,0);
\node at (0.005,1.1) {$\down$};
\node at (1.005,1.1) {$\down$};
\node at (0.505,0.9) {$\up$};
\node at (1.505,1.1) {$\down$};

\draw (0,0) -- +(0,1);
\draw (0.5,0) -- +(0,1);
\draw (1,0) -- +(0,1);
\draw (1.5,0) -- +(0,1);

\draw (.5,0) .. controls +(0,-.5) and +(0,-.5) .. +(.5,0);
\draw (0,0) .. controls +(0,-.75) and +(0,-.75) .. +(1.5,0);
\fill (.75,-.55) circle(2pt);
\draw[thin] (-.25,0) -- +(2,0);
\node at (.005,.1) {$\down$};
\node at (1.005,.1) {$\down$};
\node at (.505,-.1) {$\up$};
\node at (1.505,.1) {$\down$};
\end{scope}

\begin{scope}[xshift=7cm,yshift=-1.5cm]
\node at (0,.5) {$-$};
\end{scope}
\begin{scope}[xshift=7.5cm,yshift=-1.5cm]
\draw (.5,1) .. controls +(0,.5) and +(0,.5) .. +(.5,0);
\draw (0,1) .. controls +(0,.75) and +(0,.75) .. +(1.5,0);
\fill (.75,1.55) circle(2pt);
\draw[thin] (-.25,1) -- +(2,0);
\node at (.005,1.1) {$\down$};
\node at (1.005,1.1) {$\down$};
\node at (.505,0.9) {$\up$};
\node at (1.505,1.1) {$\down$};

\draw (0,0) -- +(0,1);
\draw (0.5,0) -- +(0,1);
\draw (1,0) -- +(0,1);
\draw (1.5,0) -- +(0,1);

\draw (.5,0) .. controls +(0,-.5) and +(0,-.5) .. +(.5,0);
\draw (0,0) .. controls +(0,-.75) and +(0,-.75) .. +(1.5,0);
\fill (.75,-.55) circle(2pt);
\draw[thin] (-.25,0) -- +(2,0);
\node at (.005,.1) {$\down$};
\node at (1.005,.1) {$\down$};
\node at (.505,-.1) {$\up$};
\node at (1.505,.1) {$\down$};
\end{scope}
\end{tikzpicture}
\end{example}

The final example illustrates the problem with orientability when multiple surgeries are possible, as discussed in the proof of Theorem~\ref{thm:surgeries_commute}.

\begin{example} \label{ex:splitnotorientable}
This example shows a special phenomenon which can occur if we compare the two compositions of two surgeries.
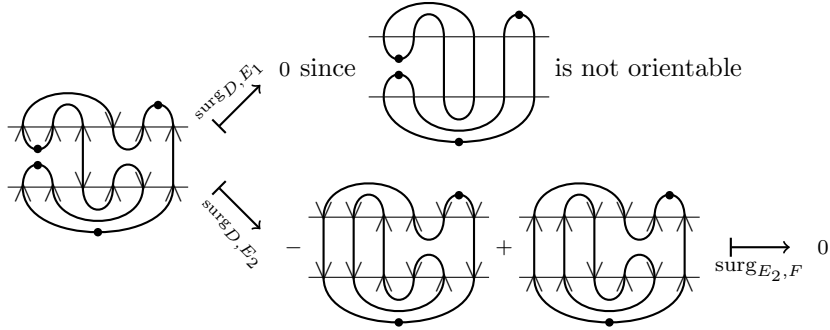
\begin{center}
\begin{tikzpicture}[thick,decoration={zigzag,pre length=1mm,post length=1mm}, scale=0.8]
\begin{scope}
\draw (0,1) .. controls +(0,-.5) and +(0,-.5) .. +(.5,0);
\draw (.5,1) .. controls +(0,.5) and +(0,.5) .. +(.5,0);
\draw (1.5,1) .. controls +(0,-.5) and +(0,-.5) .. +(.5,0);
\draw (2,1) .. controls +(0,.5) and +(0,.5) .. +(.5,0);
\draw (0,1) .. controls +(0,.75) and +(0,.75) .. +(1.5,0);
\draw[thin] (-.25,1) -- +(3,0);
\node at (.005,0.9) {$\up$};
\node at (.505,0.9) {$\up$};
\node at (1.005,0.9) {$\up$};
\node at (1.505,1.1) {$\down$};
\node at (2.005,0.9) {$\up$};
\node at (2.505,0.9) {$\up$};

\fill (0.25,.365) circle(2pt);
\fill (0.25,.635) circle(2pt);
\fill (2.25,1.365) circle(2pt);
\fill (1.25,-0.75) circle(2pt);

\draw (1,0) -- +(0,1);
\draw (2.5,0) -- +(0,1);

\draw (0,0) .. controls +(0,.5) and +(0,.5) .. +(.5,0);
\draw (1.5,0) .. controls +(0,.5) and +(0,.5) .. +(.5,0);
\draw (1,0) .. controls +(0,-.5) and +(0,-.5) .. +(.5,0);
\draw (.5,0) .. controls +(0,-.75) and +(0,-.75) .. +(1.5,0);
\draw (0,0) .. controls +(0,-1) and +(0,-1) .. +(2.5,0);
\draw[thin] (-.25,0) -- +(3,0);
\node at (.005,-.1) {$\up$};
\node at (.505,-.1) {$\up$};
\node at (1.005,.1) {$\down$};
\node at (1.505,-.1) {$\up$};
\node at (2.005,.1) {$\down$};
\node at (2.505,-.1) {$\up$};
\end{scope}

\begin{scope}[xshift=3.25cm,yshift=1cm]
\node[rotate=45] at (.2,.6) {$\scriptscriptstyle {\rm surg}_{D,E_1}$};
\draw[|->] (0,0) -- +(.7,.7);
\end{scope}
\begin{scope}[xshift=3.25cm]
\node[rotate=-45] at (.2,-.8) {$\scriptscriptstyle {\rm surg}_{D,E_2}$};
\draw[|->] (0,0) -- +(.7,-.7);
\end{scope}

\begin{scope}[xshift=4.9cm,yshift=1.5cm]
\node at (0,.5) {$\scriptstyle 0$ \small since};
\end{scope}
\begin{scope}[xshift=6cm,yshift=1.5cm]
\draw (0,1) .. controls +(0,-.5) and +(0,-.5) .. +(.5,0);
\draw (.5,1) .. controls +(0,.5) and +(0,.5) .. +(.5,0);
\draw (2,1) .. controls +(0,.5) and +(0,.5) .. +(.5,0);
\draw (0,1) .. controls +(0,.75) and +(0,.75) .. +(1.5,0);
\draw[thin] (-.25,1) -- +(3,0);

\fill (0.25,.365) circle(2pt);
\fill (0.25,.635) circle(2pt);
\fill (2.25,1.365) circle(2pt);
\fill (1.25,-0.75) circle(2pt);

\draw (1,0) -- +(0,1);
\draw (1.5,0) -- +(0,1);
\draw (2,0) -- +(0,1);
\draw (2.5,0) -- +(0,1);

\draw (0,0) .. controls +(0,.5) and +(0,.5) .. +(.5,0);
\draw (1,0) .. controls +(0,-.5) and +(0,-.5) .. +(.5,0);
\draw (.5,0) .. controls +(0,-.75) and +(0,-.75) .. +(1.5,0);
\draw (0,0) .. controls +(0,-1) and +(0,-1) .. +(2.5,0);
\draw[thin] (-.25,0) -- +(3,0);
\node at (4.4,0.5) {\small is not orientable};
\end{scope}

\begin{scope}[xshift=4.5cm,yshift=-1.5cm]
\node at (0,.5) {$\scriptstyle -$};
\end{scope}
\begin{scope}[xshift=5cm,yshift=-1.5cm]
\draw (.5,1) .. controls +(0,.5) and +(0,.5) .. +(.5,0);
\draw (1.5,1) .. controls +(0,-.5) and +(0,-.5) .. +(.5,0);
\draw (2,1) .. controls +(0,.5) and +(0,.5) .. +(.5,0);
\draw (0,1) .. controls +(0,.75) and +(0,.75) .. +(1.5,0);
\draw[thin] (-.25,1) -- +(3,0);
\node at (1.005,0.9) {$\up$};
\node at (1.505,0.9) {$\up$};
\node at (.505,1.1) {$\down$};
\node at (.005,1.1) {$\down$};
\node at (2.005,1.1) {$\down$};
\node at (2.505,1.1) {$\down$};

\fill (1.25,-0.75) circle(2pt);
\fill (2.25,1.365) circle(2pt);

\draw (0,0) -- +(0,1);
\draw (.5,0) -- +(0,1);
\draw (1,0) -- +(0,1);
\draw (2.5,0) -- +(0,1);

\draw (1,0) .. controls +(0,-.5) and +(0,-.5) .. +(.5,0);
\draw (1.5,0) .. controls +(0,.5) and +(0,.5) .. +(.5,0);
\draw (.5,0) .. controls +(0,-.75) and +(0,-.75) .. +(1.5,0);
\draw (0,0) .. controls +(0,-1) and +(0,-1) .. +(2.5,0);
\draw[thin] (-.25,0) -- +(3,0);
\node at (2.505,.1) {$\down$};
\node at (1.005,-.1) {$\up$};
\node at (.505,.1) {$\down$};
\node at (2.005,-.1) {$\up$};
\node at (1.505,.1) {$\down$};
\node at (.005,.1) {$\down$};
\end{scope}

\begin{scope}[xshift=8cm,yshift=-1.5cm]
\node at (0,.5) {$\scriptstyle +$};
\end{scope}
\begin{scope}[xshift=8.5cm,yshift=-1.5cm]
\draw (.5,1) .. controls +(0,.5) and +(0,.5) .. +(.5,0);
\draw (1.5,1) .. controls +(0,-.5) and +(0,-.5) .. +(.5,0);
\draw (2,1) .. controls +(0,.5) and +(0,.5) .. +(.5,0);
\draw (0,1) .. controls +(0,.75) and +(0,.75) .. +(1.5,0);
\draw[thin] (-.25,1) -- +(3,0);
\node at (.005,0.9) {$\up$};
\node at (.505,0.9) {$\up$};
\node at (1.005,1.1) {$\down$};
\node at (1.505,1.1) {$\down$};
\node at (2.005,0.9) {$\up$};
\node at (2.505,0.9) {$\up$};

\fill (2.25,1.365) circle(2pt);
\fill (1.25,-0.75) circle(2pt);

\draw (0,0) -- +(0,1);
\draw (.5,0) -- +(0,1);
\draw (1,0) -- +(0,1);
\draw (2.5,0) -- +(0,1);

\draw (1,0) .. controls +(0,-.5) and +(0,-.5) .. +(.5,0);
\draw (1.5,0) .. controls +(0,.5) and +(0,.5) .. +(.5,0);
\draw (.5,0) .. controls +(0,-.75) and +(0,-.75) .. +(1.5,0);
\draw (0,0) .. controls +(0,-1) and +(0,-1) .. +(2.5,0);
\draw[thin] (-.25,0) -- +(3,0);
\node at (.005,-.1) {$\up$};
\node at (.505,-.1) {$\up$};
\node at (1.005,.1) {$\down$};
\node at (1.505,-.1) {$\up$};
\node at (2.005,.1) {$\down$};
\node at (2.505,-.1) {$\up$};
\end{scope}

\begin{scope}[xshift=11.75cm,yshift=-1cm]
\node at (.5,-.4) {$\scriptstyle {\rm surg}_{E_2,F}$};
\draw[|->] (0,0) -- +(1,0);
\end{scope}

\begin{scope}[xshift=13.3cm,yshift=-1.5cm]
\node at (0,.5) {$\scriptstyle 0$};
\end{scope}
\end{tikzpicture}
\end{center}
In one order the first surgery produces a non-orientable diagram, hence is zero, whereas in the other order the first surgery produces something non-zero, but contained in the kernel of the second surgery.
\end{example}

\subsection{Generating set}
\label{sec:generators}
In addition to the idempotents ${}_{\lambda} \mathbbm{1}_\la$, the following elements related to the notion of $\la$-pairs are of special importance. (For their degrees see \eqref{oriented}).

\begin{definition}\label{def:specialelements}
Let $\lambda,\mu \in \Lambda$ and $\la \rightarrow \mu$, see \eqref{eqn:lambdapairlist}. 
\begin{itemize} 
\item[$\blacktriangleright$] Let ${}_{\mu} \mathbbm{1}_\la = \underline{\mu} \mu \overline{\la}$. Note that ${\rm deg}({}_{\mu} \mathbbm{1}_\la) = {\rm mdeg}(\underline{\mu}\ov{\la})=1$.

\item[$\blacktriangleright$] Let ${}_{\la} \mathbbm{1}_{\mu} = \underline{\la}\mu \overline{\mu}$. Note that ${\rm deg}({}_{\la} \mathbbm{1}_\mu) = {\rm mdeg}(\underline{\la}\ov{\mu})=1$.

\item[$\blacktriangleright$] Let ${}_{\mu} X_{\la}$, respectively ${}_{\la} X_{\mu}$, be the basis vector obtained from ${}_{\mu} \mathbbm{1}_\la$, respectively ${}_{\mu} \mathbbm{1}_{\la}$, by reversing the orientation of the circle containing the cup defining the relation $\la \rightarrow \mu$.

\item[$\blacktriangleright$] Let $X_{i,\la}$ be the basis vector obtained from ${}_{\lambda} \mathbbm{1}_{\lambda}$ by reversing the orientation of the circle containing the vertex $i$ and multiplying with ${\rm sign}_{\underline{\la}\ov{\la}}(i,0)$. If there is no such circle we declare  $X_{i,\la}=0.$
\end{itemize}
\end{definition}

By Proposition \ref{degone} the elements of degree $1$ in $\D$ are precisely the elements ${}_{\mu} \mathbbm{1}_\la$ attached to $\lambda$-pairs. In fact, together with the idempotents ${}_{\la} \mathbbm{1}_\la$ they generate $\D$:

\begin{theorem} \label{prop:generatedindegree1}
The algebra $\D$ is generated by its degree $0$ and $1$ part, i.e. by $\{ {}_\la \mathbbm{1}_\la \mid \la \in \Lambda \} \cup \{ {}_\la \mathbbm{1}_\mu \mid \la,\mu \in \Lambda \text{ and } \lambda \leftrightarrow \mu \}$.
\end{theorem}

The first step of the proof is the following observation.
\begin{lemma} \label{lem:generatingclockwisecircles}
Let $\lambda \in \Lambda$ and $i \in P_\bu(\Lambda)$. Then $X_{i,\lambda}$ is contained in the subalgebra of $\D$ generated by the elements of degree $1$.
\end{lemma}
\begin{proof}
Assume that $X_{i,\lambda} \neq 0$ and let $\gamma$ denote the cup in $\un{\la}$ containing $i$ and let $C$ be the circle in $\un{\la}\ov{\la}$ containing $\gamma$. 

Assume first that the $C$ is an {\it outer circle}, i.e. it is not contained in any other circle and there are no dotted circles to the right of $C$ and no lines to the left of $C$ \footnote{This definition of outer circle makes sense when we work with symmetric diagrams as in \cite{LS} where it would turn just into a circle which is not nested inside any other circle.}. Then there exists a unique $\mu \in \Lambda$ such that $\lambda \rightarrow \mu$ and such that the endpoints of $\gamma$ are precisely the positions where $\lambda$ and $\mu$ differ and this is given by a local move as displayed in the last row of \eqref{eqn:lambdapairlist}. Checking the multiplication one sees that in this case
$$ {}_{\la} \mathbbm{1}_{\mu} \cdot {}_{\mu} \mathbbm{1}_{\la} = \pm X_{i,\lambda}.$$

If on the other hand $C$ is not an outer circle,  that means nested in some other circle $C'$ or to the left of some dotted circle $C'$. Assume that the claim is already shown for all $X_{j,\lambda}$ where $j$ is in such a $C'$. Then there exists $\mu \in \Lambda$ such that $\lambda \rightarrow \mu$ and such that the endpoints of $\gamma$ are precisely the positions where $\lambda$ and $\mu$ differ and the differences is given by a local move as in the first three rows of \eqref{eqn:lambdapairlist}. Going through all the cases one observes that the multiplication gives
$$ {}_{\la} \mathbbm{1}_{\mu} \cdot {}_{\mu} \mathbbm{1}_{\la} = \pm X_{i,\lambda} \pm X_{j,\lambda},$$
for $j$ contained in a component that also contains $C$. But since the second summand is already in the subalgebra, the claim follows by induction.
\end{proof}

\begin{proof}[Proof of Theorem~\ref{prop:generatedindegree1}]
Denote by $A$ the subalgebra generated by the degree $0$ and $1$ elements. It suffices to show that if  $\lambda,\mu,\nu \in \Lambda$ such that $\un{\la}\nu\ov{\mu}$ is oriented, then it is contained in $A$. If $\un{\la}\ov{\mu}$ contains lines we use Section~\ref{annoying} and embed $\D$ into $\mathbb{D}_{\widehat{\Lambda}}$. In this way we can restrict ourselves to consider diagrams involving circles only. To not overburden the notation we will omit the $\widehat{\cdot}$ in all notations throughout the proof.

By Lemma \ref{lem:generatingclockwisecircles} it holds that $X_{i,\lambda} \in A$, which together with the multiplcation rules for merges immediately implies that the claim follows if we show that $\un{\la}\nu_{\rm min}\ov{\mu} \in A$ where $\nu_{\rm min}$ is the orientation such that all circles in $\un{\la}\ov{\mu}$ are oriented anticlockwise, since $\un{\la}\nu\ov{\mu}$ can be obtained by multiplying $\un{\la}\nu_{\rm min}\ov{\mu}$ with various $X_{i,\lambda}$'s from the left. Hence we assume that $\nu$ is such that all circles are oriented anticlockwise.

We prove the claim by induction on the minimal degree of $\un{\la}\ov{\mu}$.\\ \noindent
$\blacktriangleright$ \textit{Case ${\rm mdeg}(\un{\la}\ov{\mu})=0$:} Then $\mu = \la$ and $\un{\la}\nu\ov{\mu}={}_\la \mathbbm{1}_\la \in A$ by definition of $A$.\\ \noindent
$\blacktriangleright$ \textit{Case ${\rm mdeg}(\un{\la}\ov{\mu}) \geq 1$:} Then $\un{\la}\ov{\mu}$ must contain a non-small circle $C$, i.e. a circle containing at least two cups respectively caps. Hence $C$ must contain a pair of cups or a pair of caps $\gamma_1$ and $\gamma_2$ nested in each other. Without loss of generality we assume it is a pair of cups. Choose $\gamma_1$ such that it is not contained in any other cup of $C$ and $\gamma_2$ such that it is only contained in $\gamma_1$. Then there exists an $\eta \in \Lambda$ such that $\lambda \leftrightarrow \eta$ and the cups $\gamma_1$ and $\gamma_2$ are part of the local move as shown in the first two rows of \eqref{eqn:lambdapairlist}. In particular
${\rm deg}({}_\lambda \mathbbm{1}_\eta)=1$, hence it is by induction in $A$. Observe that in $\un{\eta}\ov{\mu}$ the circle $C$ is replaced by two circles and furthermore ${\rm mdeg}(\un{\eta}\ov{\mu}) = {\rm mdeg}(\un{\la}\ov{\mu})-1$. By induction, all oriented diagrams in ${}_\eta(\D)_\mu$ of this degree are contained in $A$. Let $\un{\eta}\nu'\ov{\mu}$ denote the element of minimal degree, i.e. all circles oriented anticlockwise. Then
$$ {}_\lambda \mathbbm{1}_\eta \cdot \un{\eta}\nu'\ov{\mu} = \un{\lambda}\nu\ov{\mu},$$
since the multiplication just involves merges and they are all merging two anticlockwise circles. The claim follows since ${}_\lambda \mathbbm{1}_\eta \in A$ by induction.

In case of strictly positive degree one must be a bit more careful if one or both right endpoints of $\gamma_1$ and $\gamma_2$ are in the ``extended'' part of the diagram, i.e. positions $r,\ldots, r+s-1$ in Section~\ref{annoying}. In this case one must choose $\eta$ in such a way that the orientation at these endpoints does not need to be changed, which is always possible, to make sure that ${}_\lambda \mathbbm{1}_\eta$ and $\underline{\eta}\nu'\overline{\mu}$ are again contained in the subalgebra given in Section~\ref{annoying}.
\end{proof}

\section{Cellularity and projective-injective modules for $\D$}
\label{sec:cell}
The following theorem is an analogue of \cite[Theorem 3.1]{BS1} and the proof follows in principle the proof there. However, because of non-locality and signs the arguments are slightly more involved.
\subsection{Cellularity}
We still fix a block $\Lambda$ and consider the algebra $\D$ with homogeneous basis $\B$ from \eqref{DefB}.
\begin{theorem}\label{cellular}
Let $( a \la b )$ and
$( c \mu d )$ be basis vectors of $\D$.
Then,
$$
( a \la b )
( c \mu d ) =
\left\{
\begin{array}{ll}
0&\text{if $b \neq c^*$,}\\
s_{a \la b}(\mu) ( a \mu d ) + (\dagger)&\text{if $b = c^*$
and $a \mu$ is oriented,}\\
(\dagger)&\text{otherwise,}
\end{array}
\right.
$$
where
\begin{enumerate}[1.)]
\item
$(\dagger)$ denotes a linear combination of basis vectors from $\B$ of the form $( a \nu d)$
for $\nu > \mu$;
\item
the scalar $s_{a \la b}(\mu) \in \{0,1,-1\}$ depends on
$a \la b$ and $\mu$ but not on $d$.
\end{enumerate}
\end{theorem}

We first state an easy fact and deduce a few consequences of the theorem:

\begin{corollary}\label{ideal}
The product
$( a \la b )( c \mu d )$ of two basis vectors of $\D$
is a linear combination of vectors of the form
$( a \nu d)\in\B$ with $\la \leq \nu \geq \mu$.
\end{corollary}

\begin{proof}
By Theorem~\ref{cellular}(i),
$( a \la b ) ( c \mu d )$ is a linear combination
of $( a \nu d)$'s for various $\nu \geq \mu$
and
$( d^* \mu c^* ) ( b^* \la a^* )$
is a linear combination of $( d^* \nu a^*)$'s for various
$\nu \geq \la$.
Applying the anti-automorphism $*$ from Corollary~\ref{antiaut}
to the latter statement gives
that
$( a \la b ) ( c \mu d )$ is a linear combination
of $( a \nu d)$'s for various $\nu \geq \la$ too.
\end{proof}

\begin{corollary}\label{iscell}
The algebra $\D$ is a cellular algebra in the sense of \cite{GL} with cell datum $(\La, M, C, *)$ where
\begin{enumerate}[i)]
\item
$M(\la)$ denotes $\left\{\alpha \in \La\:|\:\alpha
\subset \la\right\}$
for each $\la \in \La$;
\item
$C$ is defined by setting
$C^\la_{\alpha,\beta}=( \underline{\alpha} \la \overline{\beta} )$
for $\la \in \La$ and $\alpha,\beta \in M(\la)$;
\item $*$ is the anti-automorphism from Corollary~\ref{antiaut}.
\end{enumerate}
\end{corollary}

Before we prove the corollary let us first recall the relevant definitions from
\cite{GL}. A {\em cellular algebra} means an associative unital algebra $H$
together with a {\em cell datum} $(\La, M, C, *)$ such that
\begin{enumerate}[(C-1)]
\item
$\La$ is a partially ordered set and $M(\la)$
is a finite set
for each $\la \in \La$;
\item
$C:\dot\bigcup_{\la \in \La} M(\la)
\times M(\la) \rightarrow H,
(\alpha,\beta) \mapsto C^\la_{\alpha,\beta}$ is an injective map
whose image is a basis for $H$;
\item
the map $*:H \rightarrow H$
is an algebra anti-automorphism such that
$(C^\la_{\alpha,\beta})^* = C_{\beta,\alpha}^\la$
for all $\la \in \La$ and $\alpha, \beta \in M(\la)$;
\item
if $\mu \in \La$ and $\gamma, \delta \in M(\la)$
then for any $x \in H$ we have that
$$
x C_{\gamma,\delta}^\mu \equiv \sum_{\gamma' \in M(\mu)} r_x(\gamma',\gamma) C_{\gamma',\delta}^\mu
\pmod{H(> \mu)}$$
where the scalar $r_x(\gamma',\gamma)$ is independent of $\delta$ and
$H(> \mu)$ denotes the subspace of $H$
generated by $\{C_{\gamma'',\delta''}^\nu\:|\:\nu > \mu,
\gamma'',\delta'' \in M(\nu)\}$.
\end{enumerate}

\begin{proof}[Proof of Corollary~\ref{iscell}]
Condition (C-1) is clear as $\La$ itself is a finite set;
Condition (C-2) is a consequence of the definition of $\B$ from \eqref{DefB}; and since $*$ is an anti-automorphism we get (C-3).
Finally to verify (C-4) it suffices to consider the case that $x = C^\la_{\alpha,\beta}$ for some $\la \in \La$ and
$\alpha, \beta \in M(\la)$. If $\beta = \gamma$ and $\alpha \subset \mu$, then
Theorem~\ref{cellular} 1.) and 2.) shows that
$$
C_{\alpha,\beta}^\la C_{\gamma,\delta}^\mu \equiv
s_{\underline{\alpha} \la \overline{\beta}}(\mu) C_{\alpha, \delta}^\mu
\pmod{{\D}{(> \mu)}}
$$
where $s_{\underline{\alpha} \la \overline{\beta}}(\mu)$
is independent of $\delta$;
otherwise, we have that
$$
C_{\alpha, \beta}^\la C_{\gamma, \delta}^\mu \equiv 0
\pmod{\D(> \mu)}.
$$
Taking
$$
r_x(\gamma', \gamma)
=
\left\{
\begin{array}{ll}
s_{\underline{\alpha}\la\overline{\beta}}(\mu)&
\text{if $\gamma' = \alpha, \beta = \gamma$ and $\alpha \subset \mu$,}\\
0&\text{otherwise,}
\end{array}\right.
$$
we deduce that (C-4) holds.
\end{proof}

\begin{remark}
{\rm
Corollary~\ref{iscell} together with the definition of the grading directly implies that our cellular basis is in fact a graded cellular basis in the sense of Hu and Mathas, \cite{MH}.}
\end{remark}

The following is the analogue of the Khovanov arc algebra from \cite{BS1} and appears in the context of Springer fibres and resolutions of singularities, \cite{ES1} as well as in the representation theory of Lie superalgebras, \cite{ES2}.

\begin{corollary}
\label{Khovalg}
Let $\Lambda$ be a block and $e=\sum {}_{\lambda} \mathbbm{1}_\la$, where $\la$ runs over all weights $\la\in\La$ such that the associated cup diagram has $\op{def}(\La)$ cups, see \eqref{defectcount}. Then the algebra $\mathbb{H}_\La=e\D e$ is again cellular. 
\end{corollary}
\begin{proof}
The first part follows from the cellularity of $\D$, since cellularity behaves well under idempotent truncation, see e.g. \cite[Proposition 4.3]{KoenigXi}.
\end{proof}

Corollary~\ref{projinj} will identify the algebra   $\mathbb{H}_\La$ as the endomorphism ring of the sum of all indecomposable projective-injective $\mathbb{D}_\Lambda$-modules.

\begin{remark}
The formulation of Theorem~\ref{cellular} differs from \cite[Theorem~3.1]{BS1} slightly, since we omitted the analogue of \cite[Theorem 3.1(iii)]{BS1}. This is because that statement is in fact erroneous as stated. It holds however in case $\mu=\la$ with the notation there also in our case here.
\end{remark}

For the next Lemma we assume  that we are given an oriented stacked circle diagram $\un{\lambda}(\mathbf{a},\bm{\nu})\ov{\mu}$ of height $1$ which  appears on the way in the multiplication in Definition \ref{def:multiplication} (before applying the last collapsing map) and that $C$ is an oriented circle in  $\un{\lambda}(\mathbf{a},\bm{\nu})\ov{\mu}$. We use the Bruhat order from Lemma~\ref{lem:Bruhat}.

\begin{lemma}
\label{lem:anticlock}
Assume $C$ is anticlockwise and let $\nu$ be the subsequence of $\nu_0$ or of $\nu_1$ attached to the vertices in $C$.  Then swapping the orientation of $C$ makes $\nu$  bigger in the Bruhat order.
\end{lemma}

\begin{proof}
The statement is clear if the circle is small, Definition~\ref{def:componentorientation}. Indeed, $\nu$ moves from $\down\up$ to $\up\down$ in the undotted case and from $\up\up$ to $\down\down$ in the dotted case. Otherwise we first remove all undotted kinks at the cost of an extra small anticlockwise circle according to the rules \eqref{kinktocircle} (i)-(ii) below not changing $\nu$.  As a result we created from $C$ only small circles, or at least one kink of the form \eqref{kinktocircle} (iii)-(iv). Again we remove these kinks using the rules (iii)-(iv) below. Observe that the weight stays the same and the newly created small circles are all anticlockwise. 
 \begin{equation}
 \label{kinktocircle}
 \begin{tikzpicture}[thick, snake=zigzag, line before snake = 2mm, line after snake = 2mm]
\node at (-.5,0) {i)};
\draw (0,0) -- +(0,.3);
\draw[dotted, snake] (0,.3) -- +(0,.8);
\draw (0,0) .. controls +(0,-.5) and +(0,-.5) .. +(.5,0);
\draw (.5,0) .. controls +(0,.5) and +(0,.5) .. +(.5,0);
\draw (1,0) -- +(0,-.3);
\draw[dotted, snake] (1,-.3) -- +(0,-.8);
\node at (.01,-.05) {$\up$};
\node at (1.01,-.05) {$\up$};
\node at (.51,0) {$\down$};

\draw[->] (1.5,-.2) -- +(1,0);

\draw (3,0) -- +(0,.3);
\draw[dotted, snake] (3,.3) -- +(0,.8);
\draw (3.5,0) .. controls +(0,-.5) and +(0,-.5) .. +(.5,0);
\draw (3.5,0) .. controls +(0,.5) and +(0,.5) .. +(.5,0);
\draw (3,0) -- +(0,-.3);
\draw[dotted, snake] (3,-.3) -- +(0,-.8);
\node at (3.01,-.05) {$\up$};
\node at (4.01,-.05) {$\up$};
\node at (3.51,0) {$\down$};

\begin{scope}[xshift=6cm]
\node at (-.5,0) {ii)};
\draw (0,0) -- +(0,.3);
\draw[dotted, snake] (0,.3) -- +(0,.8);
\draw (0,0) .. controls +(0,-.5) and +(0,-.5) .. +(.5,0);
\draw (.5,0) .. controls +(0,.5) and +(0,.5) .. +(.5,0);
\draw (1,0) -- +(0,-.3);
\draw[dotted, snake] (1,-.3) -- +(0,-.8);
\node at (.01,0) {$\down$};
\node at (1.01,0) {$\down$};
\node at (.51,-.05) {$\up$};

\draw[->] (1.5,-.2) -- +(1,0);

\draw (4,0) -- +(0,.3);
\draw[dotted, snake] (4,.3) -- +(0,.8);
\draw (3,0) .. controls +(0,-.5) and +(0,-.5) .. +(.5,0);
\draw (3,0) .. controls +(0,.5) and +(0,.5) .. +(.5,0);
\draw (4,0) -- +(0,-.3);
\draw[dotted, snake] (4,-.3) -- +(0,-.8);
\node at (3.01,0) {$\down$};
\node at (4.01,0) {$\down$};
\node at (3.51,-.05) {$\up$};

\end{scope}

\begin{scope}[yshift=-3cm]
\node at (-.5,0) {iii)};
\draw (0,0) -- +(0,.3);
\draw[dotted, snake] (0,.3) -- +(0,.8);
\draw (0,0) .. controls +(0,-.5) and +(0,-.5) .. +(.5,0);
\draw (.5,0) .. controls +(0,.5) and +(0,.5) .. +(.5,0);
\fill (.75,.365) circle(2.5pt);
\draw (1,0) -- +(0,-.3);
\draw[dotted, snake] (1,-.3) -- +(0,-.8);
\node at (.01,0) {$\down$};
\node at (1.01,-.05) {$\up$};
\node at (.51,-.05) {$\up$};

\draw[->] (1.5,-.2) -- +(1,0);

\draw (4,0) -- +(0,.3);
\draw[dotted, snake] (4,.3) -- +(0,.8);
\fill (4,.365) circle(2.5pt);
\draw (3,0) .. controls +(0,-.5) and +(0,-.5) .. +(.5,0);
\draw (3,0) .. controls +(0,.5) and +(0,.5) .. +(.5,0);
\draw (4,0) -- +(0,-.3);
\draw[dotted, snake] (4,-.3) -- +(0,-.8);
\node at (3.01,0) {$\down$};
\node at (4.01,-.05) {$\up$};
\node at (3.51,-.05) {$\up$};

\end{scope}

\begin{scope}[xshift=6cm, yshift=-3cm]
\node at (-.5,0) {iv)};
\draw (0,0) -- +(0,.3);
\draw[dotted, snake] (0,.3) -- +(0,.8);
\fill (.25,-.365) circle(2.5pt);
\draw (0,0) .. controls +(0,-.5) and +(0,-.5) .. +(.5,0);
\draw (.5,0) .. controls +(0,.5) and +(0,.5) .. +(.5,0);
\draw (1,0) -- +(0,-.3);
\draw[dotted, snake] (1,-.3) -- +(0,-.8);
\node at (.01,0) {$\down$};
\node at (.51,0) {$\down$};
\node at (1.01,-.05) {$\up$};

\draw[->] (1.5,-.2) -- +(1,0);

\draw (3,0) -- +(0,.3);
\draw[dotted, snake] (3,.3) -- +(0,.8);
\fill (3,-.365) circle(2.5pt);
\draw (3.5,0) .. controls +(0,-.5) and +(0,-.5) .. +(.5,0);
\draw (3.5,0) .. controls +(0,.5) and +(0,.5) .. +(.5,0);
\draw (3,0) -- +(0,-.3);
\draw[dotted, snake] (3,-.3) -- +(0,-.8);
\node at (3.01,0) {$\down$};
\node at (3.51,0) {$\down$};
\node at (4.01,-.05) {$\up$};

\end{scope}
\end{tikzpicture}
 \end{equation}
 The result is a collection of small anticlockwise circles and swapping the orientation of $C$ means swapping the orientation of all these new circles. But this obviously increases the weight.
\end{proof}

For the next proposition we assume again that we are given an oriented stacked circle diagram $D=\un{\lambda}(\mathbf{a},\bm{\nu})\ov{\mu}$ of height $1$ which  appears on the way in the multiplication \eqref{eqn:multiplication} of $(a \la b ) ( c \mu d )$. Let $\tau=\nu_1$ be its {top weight} (the weight on the top number line). Assume there is at least one more surgery to apply to $\un{\lambda}(\mathbf{a},\bm{\nu})\ov{\mu}$ and let $\gamma$ be the corresponding cup-cap-pair. Let $\cupg$ and $\capg$ the cup respectively the cap in $\gamma$.  In this setup we claim the following crucial step for the proof of Theorem~\ref{cellular} with the Bruhat order from Lemma~\ref{lem:Bruhat}.

\begin{prop}
\label{topweight}
The top weight of each diagram obtained at the end of
the surgery procedure at $\gamma$ is greater than or equal to $\tau$ in the Bruhat order.
\end{prop}

\begin{proof}[Proof of Proposition~\ref{topweight}]
Let us first assume that all the diagrams involved do not contain any lines or rays. 
\subsubsection*{Case 1: Merge}  We  consider the case where two circles get merged. In this case either the result is zero or the merge produces a unique diagram. \hfill\\
$\blacktriangleright$ {\it Assume $\gamma$ is undotted.} Consider the four possible orientations of $\gamma$:
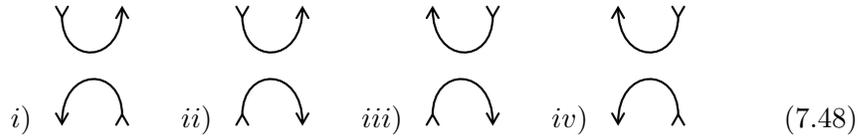
\begin{equation}
\label{4casesundotted}
i)\quad
\usetikzlibrary{arrows}
\begin{tikzpicture}[thick,>=angle 60, scale=0.8]
\draw [>->] (0,0) .. controls +(0,-1) and +(0,-1) .. +(1,0);
\draw [<-<] (0,-2) .. controls +(0,1) and +(0,1) .. +(1,0);
\end{tikzpicture}
\quad\quad
ii)
\quad
\usetikzlibrary{arrows}
\begin{tikzpicture}[thick,>=angle 60,scale=0.8]
\draw [>->] (0,0) .. controls +(0,-1) and +(0,-1) .. +(1,0);
\draw [>->] (0,-2) .. controls +(0,1) and +(0,1) .. +(1,0);
\end{tikzpicture}
\quad\quad
iii)
\quad
\usetikzlibrary{arrows}
\begin{tikzpicture}[thick,>=angle 60,scale=0.8]
\draw [<-<] (0,0) .. controls +(0,-1) and +(0,-1) .. +(1,0);
\draw [>->] (0,-2) .. controls +(0,1) and +(0,1) .. +(1,0);
\end{tikzpicture}
\quad\quad
iv)
\quad
\usetikzlibrary{arrows}
\begin{tikzpicture}[thick,>=angle 60,scale=0.8]
\draw [<-<] (0,0) .. controls +(0,-1) and +(0,-1) .. +(1,0);
\draw [<-<] (0,-2) .. controls +(0,1) and +(0,1) .. +(1,0);
\end{tikzpicture}
\end{equation}
In case both involved circles are clockwise the merge gives zero and there is nothing to do. 

Otherwise start with Cases i) and iii). If the original circles are both anticlockwise, then the merge just joins them together preserving the weights. In case one is anticlockwise and the other is clockwise replacing $\gamma$ by two straight lines produces a circle which could be clockwise (in which case this is the result of the merge and the weights are preserved) or anticlockwise. In the latter case we swap its orientation and are done by Lemma~\ref{lem:anticlock}. 

Consider now Cases ii) and iv). Pick a tag of the merged circle. If both circles are anticlockwise let $C$ be the original circle not containing this tag. Then the merge can be obtained by first swapping the orientation of $C$ and then joining the circles by replacing $\gamma$ by two straight lines. Again by Lemma~\ref{lem:anticlock} the weight can only increase or stay the same. If exactly one circle was anticlockwise we first swap the orientation of this circle and then replace $\gamma$ by two straight lines. We claim that the result agrees with the merge. Indeed, choose a tag of the merged circle. If it lies on the clockwise circle then it determines the orientation of the merged circle. If not, then the orientation of the merged circle is the opposite of this circle. In any case, our claim holds. By Lemma~\ref{lem:anticlock}   the weight increases.  \hfill\\
$\blacktriangleright$ {\it Assume $\gamma$ is dotted.} We now have the following possible orientations of $\gamma$ 
\begin{equation}
\label{4cases}
i)\quad
\usetikzlibrary{arrows}
\begin{tikzpicture}[thick,>=angle 60, scale=0.8]
\draw [>-<] (0,0) .. controls +(0,-1) and +(0,-1) .. +(1,0);
\fill (0.5,-0.77) circle(2.5pt);
\draw [<->] (0,-2) .. controls +(0,1) and +(0,1) .. +(1,0);
\fill (0.5,-1.25) circle(2.5pt);
\end{tikzpicture}
\quad\quad
ii)
\quad
\usetikzlibrary{arrows}
\begin{tikzpicture}[thick,>=angle 60,scale=0.8]
\draw [>-<] (0,0) .. controls +(0,-1) and +(0,-1) .. +(1,0);
\fill (0.5,-0.77) circle(2.5pt);
\draw [>-<] (0,-2) .. controls +(0,1) and +(0,1) .. +(1,0);
\fill (0.5,-1.22) circle(2.5pt);
\end{tikzpicture}
\quad\quad
iii)
\quad
\usetikzlibrary{arrows}
\begin{tikzpicture}[thick,>=angle 60,scale=0.8]
\draw [<->] (0,0) .. controls +(0,-1) and +(0,-1) .. +(1,0);
\fill (0.5,-0.76) circle(2.5pt);
\draw [>-<] (0,-2) .. controls +(0,1) and +(0,1) .. +(1,0);
\fill (0.5,-1.22) circle(2.5pt);
\end{tikzpicture}
\quad\quad
iv)
\quad
\usetikzlibrary{arrows}
\begin{tikzpicture}[thick,>=angle 60,scale=0.8]
\draw [<->] (0,0) .. controls +(0,-1) and +(0,-1) .. +(1,0);
\fill (0.5,-0.76) circle(2.5pt);
\draw [<->] (0,-2) .. controls +(0,1) and +(0,1) .. +(1,0);
\fill (0.5,-1.24) circle(2.5pt);
\end{tikzpicture}
\end{equation}
and we can argue exactly as in the undotted case.
\subsubsection*{Case 2: Split}
Now let us assume that the surgery is a Split. Fix a tag $\op{t}(C)$ of the circle $C$ containing $\gamma$. We assume that the split creates orientable circles, since otherwise there is nothing to do. 
Assume first $\gamma$ is of the form i) or iii) in \eqref{4casesundotted} or  \eqref{4cases}. If $C$ is anticlockwise, then replacing $\gamma$ by two straight lines creates either two anticlockwise circles or one anticlockwise and one clockwise. In the first case changing the orientation of either one of the circles gives the resulting two diagrams of the split and we are done by Lemma~\ref{lem:anticlock}. In the second case the created diagram is one of the diagrams which appear in the split. The second diagram is obtained by changing the orientation of both circles. But this is the same as changing the orientation of $C$ from anticlockwise to clockwise before replacing $\gamma$ by two straight lines. Hence again Lemma~\ref{lem:anticlock} gives the result. If $C$ is clockwise then  replacing $\gamma$ by two straight lines creates either two clockwise circles preserving the weight or one anticlockwise and one clockwise circle in which case we should swap the orientation of the anticlockwise one to obtain the same result as Split. Then we can again apply  Lemma~\ref{lem:anticlock}. 
Assume now $\gamma$ is of the form ii) or iv) in \eqref{4casesundotted} or \eqref{4cases}. Observe that the orientability of $C$ implies that replacing $\gamma$ by two straight lines creates non-orientable circles and the result of the surgery is zero. 
In any case the resulting diagrams have top weights smaller or equal to the original weight $\nu_1$ on the top. In case multiple oriented circle diagrams appear each weight appears at most once by construction (consider the orientation of $\cupg$). The general case involving lines can be reduced to the case of circles only using Section~\ref{annoying}.  The proposition follows.
\end{proof}

To keep the proof of the following proposition  slightly simpler we assume that the order of our multiplication   \eqref{eqn:multiplication} is such that we take  in each step
\begin{equation}
\label{order}
\text{\it the rightmost cup-cap pair where surgery can be applied.} 
\end{equation}

\begin{prop}
\label{tauprop}
In the setup of Proposition ~\ref{topweight} and \eqref{order} we have that the weight $\tau$ appears after a single surgery procedure if and only if we are in one of the following two situations
\begin{enumerate}[($\tau$1)]
\item both $\cupg$ and $\capg$ are oriented anticlockwise, or
\item $\cupg$ is clockwise and $\capg$ is anticlockwise and the component containing $\capg$ does not intersect the top weight line. 
\end{enumerate}
Moreover, $\tau$ appears here exactly once. In all other cases the occurring weights are strictly smaller than $\tau$.
\end{prop}

\begin{corollary}
\label{tau}
In the notation of Proposition ~\ref{topweight} we have that 
the total number of
diagrams produced with top weight equal to $\tau$
is either zero or one, independent of the cap diagram $d$. It appears with a sign depending on whether $\gamma$ is dotted or not.
\end{corollary}

\begin{proof}
This follows from Proposition ~\ref{tauprop} and the multiplication rules, Section~\ref{sec:explicitmult},  noting that the conditions ($\tau$1) and  ($\tau$2) are independent of $d$.
\end{proof}

We first deduce Theorem~\ref{cellular} and then prove Proposition~\ref{tauprop}.

\begin{proof}[Proof of Theorem~\ref{cellular}]
By definition, $( a \la b ) ( c \mu d ) = 0$ if $b \neq c^*$, so assume
$b=c^*$ from now on. By applying Proposition~\ref{topweight} repeatedly, starting with $\tau = \mu$ at the first step,
it follows that
$(a \la b) (c \mu d)$ is a linear combination of
$(a \nu d)$'s for $\nu \geq \mu$.
Assuming $a \mu$ is oriented,
Corollary~\ref{tau} applied repeatedly
implies that the coefficient
$s_{a \la b}(\mu)$ of the basis vector
$(a \mu d)$ in the product is zero or $\pm 1$
independent of the cap diagram $d$. Note also that the result does not depend on the specific choice \eqref{order} for the order of the surgeries by Theorem~\ref{thm:surgeries_commute}.
This proves 1.) and 2.) in Theorem~\ref{cellular}. 
\end{proof}

\begin{proof}[Proof of Proposition~\ref{tauprop}]
We first assume that no involved diagram contains rays. For our surgery on $\gamma$ we are then in one of the 6 basic situations:
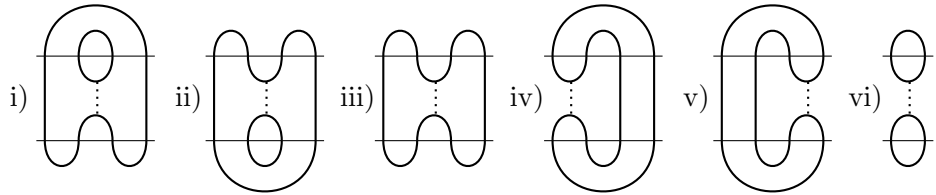
\begin{eqnarray*}
\label{H}
\begin{tikzpicture}[thick,>=angle 60,scale=.45]

\begin{scope}
\node at (1.25,0.75) {i)};
\draw[thin] (1.75,2) -- +(3.5,0);
\draw[thin] (1.75,-.5) -- +(3.5,0);
\draw [-] (2,2) .. controls +(0,2) and +(0,2) .. +(3,0);
\draw [-] (3,2) .. controls +(0,1) and +(0,1) .. +(1,0);

\draw [-] (2,-.5) -- +(0,2.5);
\draw [-] (3,2) .. controls +(0,-1) and +(0,-1) .. +(1,0);
\draw[dotted] (3.55,.3) -- +(0,.95);
\draw [-] (3,-.5) .. controls +(0,1) and +(0,1) .. +(1,0);
\draw [-] (5,-.5) -- +(0,2.5);

\draw [-] (2,-.5) .. controls +(0,-1) and +(0,-1) .. +(1,0);
\draw [-] (4,-.5) .. controls +(0,-1) and +(0,-1) .. +(1,0);
\end{scope}

\begin{scope}[xshift=5cm]
\node at (1.25,0.75) {ii)};
\draw[thin] (1.75,2) -- +(3.5,0);
\draw[thin] (1.75,-.5) -- +(3.5,0);
\draw [-] (2,2) .. controls +(0,1) and +(0,1) .. +(1,0);
\draw [-] (4,2) .. controls +(0,1) and +(0,1) .. +(1,0);

\draw [-] (2,-.5) -- +(0,2.5);
\draw [-] (3,2) .. controls +(0,-1) and +(0,-1) .. +(1,0);
\draw[dotted] (3.55,.3) -- +(0,.95);
\draw [-] (3,-.5) .. controls +(0,1) and +(0,1) .. +(1,0);
\draw [-] (5,-.5) -- +(0,2.5);

\draw [-] (2,-.5) .. controls +(0,-2) and +(0,-2) .. +(3,0);
\draw [-] (3,-.5) .. controls +(0,-1) and +(0,-1) .. +(1,0);
\end{scope}

\begin{scope}[xshift=10cm]
\node at (1.25,0.75) {iii)};
\draw[thin] (1.75,2) -- +(3.5,0);
\draw[thin] (1.75,-.5) -- +(3.5,0);

\draw [-] (2,2) .. controls +(0,1) and +(0,1) .. +(1,0);
\draw [-] (4,2) .. controls +(0,1) and +(0,1) .. +(1,0);

\draw [-] (2,-.5) -- +(0,2.5);
\draw [-] (3,2) .. controls +(0,-1) and +(0,-1) .. +(1,0);
\draw[dotted] (3.55,.3) -- +(0,.95);
\draw [-] (3,-.5) .. controls +(0,1) and +(0,1) .. +(1,0);
\draw [-] (5,-.5) -- +(0,2.5);

\draw [-] (2,-.5) .. controls +(0,-1) and +(0,-1) .. +(1,0);
\draw [-] (4,-.5) .. controls +(0,-1) and +(0,-1) .. +(1,0);
\end{scope}

\begin{scope}[xshift=15cm]
\node at (1.25,0.75) {iv)};
\draw[thin] (1.75,2) -- +(3.5,0);
\draw[thin] (1.75,-.5) -- +(3.5,0);

\draw [-] (2,2) .. controls +(0,2) and +(0,2) .. +(3,0);
\draw [-] (3,2) .. controls +(0,1) and +(0,1) .. +(1,0);

\draw [-] (4,-.5) -- +(0,2.5);
\draw [-] (2,2) .. controls +(0,-1) and +(0,-1) .. +(1,0);
\draw[dotted] (2.55,.3) -- +(0,.95);
\draw [-] (2,-.5) .. controls +(0,1) and +(0,1) .. +(1,0);
\draw [-] (5,-.5) -- +(0,2.5);

\draw [-] (2,-.5) .. controls +(0,-2) and +(0,-2) .. +(3,0);
\draw [-] (3,-.5) .. controls +(0,-1) and +(0,-1) .. +(1,0);
\end{scope}

\begin{scope}[xshift=20cm]
\node at (1.25,0.75) {v)};
\draw[thin] (1.75,2) -- +(3.5,0);
\draw[thin] (1.75,-.5) -- +(3.5,0);

\draw [-] (2,2) .. controls +(0,2) and +(0,2) .. +(3,0);
\draw [-] (3,2) .. controls +(0,1) and +(0,1) .. +(1,0);

\draw [-] (2,-.5) -- +(0,2.5);
\draw [-] (4,2) .. controls +(0,-1) and +(0,-1) .. +(1,0);
\draw[dotted] (4.55,.3) -- +(0,.95);
\draw [-] (4,-.5) .. controls +(0,1) and +(0,1) .. +(1,0);
\draw [-] (3,-.5) -- +(0,2.5);

\draw [-] (2,-.5) .. controls +(0,-2) and +(0,-2) .. +(3,0);
\draw [-] (3,-.5) .. controls +(0,-1) and +(0,-1) .. +(1,0);
\end{scope}

\begin{scope}[xshift=25cm]
\node at (1.25,0.75) {vi)};
\draw[thin] (1.75,2) -- +(1.5,0);
\draw[thin] (1.75,-.5) -- +(1.5,0);

\draw [-] (2,2) .. controls +(0,1) and +(0,1) .. +(1,0);

\draw [-] (2,2) .. controls +(0,-1) and +(0,-1) .. +(1,0);
\draw[dotted] (2.55,.3) -- +(0,.95);
\draw [-] (2,-.5) .. controls +(0,1) and +(0,1) .. +(1,0);

\draw [-] (2,-.5) .. controls +(0,-1) and +(0,-1) .. +(1,0);
\end{scope}
\end{tikzpicture}
\end{eqnarray*}
These diagrams should be interpreted only up to homeomorphism; in particular the circles represented in the pictures may well
cross both weight lines many more times than indicated and moreover might have many dots.

We distinguish three situations depending on the orientation of $\gamma$.\\[0.15cm]
\noindent
$\blacktriangleright$ {\it Case 1: $\capg$ is clockwise:}  
We claim here that the result will never have $\tau$ as its top weight. 
For i) we distinguish the two possible orientations 
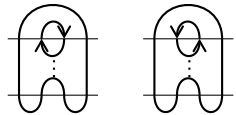
\begin{eqnarray*}
\label{AA}
\begin{tikzpicture}[thick,>=angle 60,scale=.3]
\begin{scope}
\draw[thin] (1.5,2) -- +(4,0);
\draw[thin] (1.5,-.5) -- +(4,0);
\draw [-] (2,2) .. controls +(0,2) and +(0,2) .. +(3,0);
\draw [->] (3,2) .. controls +(0,1) and +(0,1) .. +(1,0);

\draw [-] (2,-.5) -- +(0,2.5);
\draw [<-] (3,2) .. controls +(0,-1) and +(0,-1) .. +(1,0);
\draw[dotted] (3.55,.3) -- +(0,.95);
\draw [-] (3,-.5) .. controls +(0,1) and +(0,1) .. +(1,0);
\draw [-] (5,-.5) -- +(0,2.5);

\draw [-] (2,-.5) .. controls +(0,-1) and +(0,-1) .. +(1,0);
\draw [-] (4,-.5) .. controls +(0,-1) and +(0,-1) .. +(1,0);
\end{scope}

\begin{scope}[xshift=6cm]
\draw[thin] (1.5,2) -- +(4,0);
\draw[thin] (1.5,-.5) -- +(4,0);
\draw [-] (2,2) .. controls +(0,2) and +(0,2) .. +(3,0);
\draw [<-] (3,2) .. controls +(0,1) and +(0,1) .. +(1,0);

\draw [-] (2,-.5) -- +(0,2.5);
\draw [->] (3,2) .. controls +(0,-1) and +(0,-1) .. +(1,0);
\draw[dotted] (3.55,.3) -- +(0,.95);
\draw [-] (3,-.5) .. controls +(0,1) and +(0,1) .. +(1,0);
\draw [-] (5,-.5) -- +(0,2.5);

\draw [-] (2,-.5) .. controls +(0,-1) and +(0,-1) .. +(1,0);
\draw [-] (4,-.5) .. controls +(0,-1) and +(0,-1) .. +(1,0);
\end{scope}
\end{tikzpicture}
\end{eqnarray*}
of the internal circle. In the first diagram the internal circle is clockwise, hence the other circle must be anticlockwise to get something non-zero. But then the labels on the outer circle, and thus also $\tau$, change during the merge, as the outer circle contains all tags of the merged circle. Applying the merge to the second diagram will swap the labels on one circle, hence $\tau$ is not preserved. The same arguments apply also to ii). 
For the third diagram the following is easy to verify using \eqref{order}: Given a tag $\op{t}(C)$ for the circle $C$ and consider the oriented diagram $D$ obtained by replacing $\gamma$ by two straight lines. Then out of the two new circles created from $C$, the circle {\it not containing} $\op{t}(C)$ is anticlockwise. But then the orientation of $D$ cannot occur at the end of the split and thus also $\tau$ must have been changed. Indeed, if 
$C$ was anticlockwise then the split creates two oppositely oriented circles, and so either the one with the tag or the one without the tag must have a different orientation than in $D$. If it was clockwise, then the circle without the tag must have a different orientation than in $D$.
In iv) and v) the split creates an inner and an outer circle. If $\tau$ is preserved, the inner circle must be anticlockwise. But since the outer circle contains any tag of the original circle $C$ we are done by the same argument as in iii).
In vi) let $C_1$ be the circle containing $\capg$ and $C_2$ containing  $\cupg$. If $C_2$ is clockwise, then the result of the merge is either zero  or the orientation of $C_1$ gets swapped and so $\tau$ is not preserved. In case $C_2$ is anticlockwise one easily checks using \eqref{order} that $C_2$ crosses the top weight line and contains all tags. Now if $C_1$ is clockwise, the merge changes the orientation of 
$C_1$ and thereby also $\tau$. If it is anticlockwise, then $\cupg$ must be anticlockwise, hence its orientation gets switched during the merge. 
The claim is proved. \\[0.15cm]
\noindent
$\blacktriangleright$ {\it Case 2: $\capg$ and $\cupg$ are both anticlockwise.} 
Here we claim that the weight $\tau$ appears in exactly one of the resulting diagrams as a top weight. 
In  i) and ii), the orientation of the internal circle is anticlockwise. Hence the result of the merge agrees with  just replacing $\gamma$ by two straight lines and so $\tau$ is preserved. For the diagram iii) the following is easy to verify using \eqref{order}: Given a tag $\op{t}(C)$ for the circle $C$ consider the oriented diagram $D$ obtained by replacing $\gamma$ by two straight lines. Then out of the two new circles produced from $C$, the circle not containing $\op{t}(C)$ is clockwise. Independent of the orientation of  $C$, the diagram $D$ appears as a result of the Split (either as a summand in case $C$ was anticlockwise) or on its own (if $C$ was clockwise). In iv) and v) the split creates an inner and an outer circle. Replacing $\gamma$ by two straight lines creates a diagram $D$ with top weight $\tau$ and such that the internal circle is clockwise. Again independent of the orientation of  $C$, the diagram $D$ appears as a result of the Split (either as a summand in case $C$ was anticlockwise) or on its own (if $C$ was clockwise). In  vi) the result of the merge is either zero or just obtained by replacing $\gamma$ by two straight lines and so $\tau$ is preserved. The claim follows.\\[0.15cm]
\noindent
$\blacktriangleright$ {\it Case 3: $\capg$ is anticlockwise and $\cupg$ is clockwise.} 
We claim that the result will have exactly one diagram with top weight $\tau$ if the circle containing $\capg$ does not intersect the top weight line, and will produce no diagram with top weight $\tau$ otherwise.
In i) the inner circle is clockwise, hence to get non-zero the outer circle must be anticlockwise, and so changes its orientation during the merge. Thus $\tau$ is not preserved as claimed. 
In ii) the inner circle is anticlockwise, hence the merged circle would inherit the orientation of the outer circle and the inner circle swaps orientation. Therefore the top weight $\tau$ is preserved if and only if the inner circle does not cross the top weight line. 
In  iii)-v) the diagram resulting from the split is not orientable, hence the map is zero. In vi) one can check easily that if the circle $C$ containing $\capg$ does not cross the top line then it is anticlockwise and our merge would just replace $\gamma$ by two straight lines, preserving $\tau$. Otherwise $C$ contains any tag of the merged circle and the other circle must be clockwise. So the merge either gives zero or the top weight changes under the merge (that is in case $C$ is anticlockwise).  
The claim follows.

Note that our arguments also work in general for arbitrary circle diagrams using Section~\ref{annoying}. The proposition follows.
\end{proof}

\section{Quasi-hereditarity of $\D$ and graded decomposition numbers}\label{sdecomposition}

We briefly describe the representation theory of the algebras $\D$ for any block $\La$, construct their cell modules and provide explicit closed formulae for $q$-decomposition numbers
which simultaneously describe  composition multiplicities of
cell modules and  cell filtration
multiplicities of projective indecomposable modules.
We will deduce that the category ${\D}\MOD$
of finite dimensional graded $\D$-modules is a graded highest weight
category.

\subsection{Graded modules, projectives and irreducibles}
\label{graded}
If $A$ is a $\mZ$-graded finite dimensional algebra
and $M=\bigoplus_{j \in \mZ} M_j$ is a {\em graded} $A$-module,
i.e. $A_i M_j \subseteq M_{i+j}$,
then we write $M\langle j \rangle$ for the same module but with new grading defined by
$M\langle j \rangle_i = M_{i-j}$. For graded modules $M$ and $N$, we define
\begin{equation}\label{homdef}
\hom_{A}(M, N) = \bigoplus_{j \in \mZ} \hom_{A}(M,N)_j
\end{equation}
where $\hom_{A}(M,N)_j$ denotes all homogeneous $A$-module
homomorphisms of degree $j$, meaning that they map $M_i$ into $N_{i+j}$ for
each $i \in \mZ$.
Later we might also work with (not necessarily unital) $\mZ$-graded algebras with many idempotents (which means they come with a given system
$\{e_\la\:|\:\la\in \La'\}$ of mutually orthogonal idempotents
such that
$
A = \bigoplus_{\la,\mu \in \La'} e_\la  A e_\mu.)
$
By an {\em $A$-module} we mean then a left $A$-module
$M$ such that $
M = \bigoplus_{\la \in \La'} e_\la M;
$
and the notions about graded modules generalize. (Note that if $A$ is finite dimensional, our definitions of $A$-modules agree).

Let $A\gMOD$ be the category of all {\it finite dimensional graded $A$-modules} $M =
\bigoplus_{j \in \mZ} M_j$. Together with degree zero morphisms this is an abelian category.

Now fix an arbitrary block $\La$ and consider the graded algebra $A=\D$. Let ${\D}^{+}$ be the sum of all
components of strictly positive degree, so
\begin{equation}\label{degzero}
\D / {\D}^{+} \cong \bigoplus_{\la \in \La} \mC
\end{equation}
as an algebra, with a basis given by the images of
all the idempotents $e_\la$, $\la\in\La$. The image of $e_\la$
spans a one dimensional graded $\D$-module which we denote by $L(\la)$.
Thus, $L(\la)$ is a copy of the field concentrated in degree $0$, and $(a \mu
b)\in \D$ acts on $L(\la)$ as multiplication by $1$ if $a \mu b =
\underline{\la} \la \overline{\la}$, or as zero otherwise. The modules
\begin{equation}\label{simples}
\{L(\la)\langle j \rangle\:|\:\la \in \La, j\in \mZ\}
\end{equation}
give a complete set of isomorphism classes of simple modules in $\D\gMOD$.
For any graded $\D$-module $M$ we let
${M}^\circledast$ denote its graded dual, which means that
$({M}^\circledast)_{j} = \hom_{\mC}(M_{-j},\mC)$ and $x \in \D$ acts on $f
\in {M}^\circledast$ by $(xf)(m) = f(x^* m)$, where $*$ is the antiautomorphism from Corollary~\ref{antiaut}.
Clearly we have for each $\la \in \La$
\begin{equation}
\label{simplesselfdual}
{L(\la)}^\circledast
 \cong L(\la).
\end{equation}

For $\la \in \La$, let $P(\la) = \D e_\la$. This is a graded $\D$-module
with basis $$\left\{(\underline{\nu} \mu \overline\la)\:\big|\:\text{for all
}\nu,\mu \in \La \text{ such that }\nu \subset \mu \supset \la\right\}$$ and is with the natural surjection a
projective cover of $L(\la)$. The modules
\begin{equation}\label{pims}
\{P(\la)\langle j \rangle\:|\:\la \in \La, j\in \mZ\}
\end{equation}
give a full set of indecomposable projective objects in $\D\gMOD$.

\subsection{Grothendieck groups}
The {\it Grothendieck group} of $\D\gMOD$, denoted $K_0(\D)$, is the free $\mZ$-module on isomorphism classes $[M]$ of objects $M$ in $\D\gMOD$ modulo the relation $[M]=[M']+[M'']$ whenever there is a short exact sequence in $\D\gMOD$ with $M$ as middle and $M'$, $M''$ as outer terms; then $K_0(\D)$ has a basis given by the $\{L(\la)\langle j \rangle\:|\:\la \in \La, j\in \mZ\}$. Let $[\op{proj}({\D})]$ be the subgroup generated by the classes of the projective
indecomposable modules from \eqref{pims}. They both carry a free $\mZ[q,q^{-1}]$-module structure by setting
$$
q^j [M] = [M\langle j
\rangle].
$$
In particular, there are $c_{\la,\mu}(q)\in\mZ[q,q^{-1}]$ such that
\begin{equation}\label{decm}
[P(\mu)] = \sum_{\la \in \La} c_{\la,\mu}(q) [L(\la)].
\end{equation}
The  matrix $C_\La(q) = (c_{\la,\mu}(q))_{\la,\mu \in\La}$ is called the {\em $q$-Cartan matrix} of $\D$.

\begin{lemma}
\label{Cartanmatrix}
The entries in the Cartan matrix of $\D$ are explicitly given as follows:
\begin{equation}\label{exp}
c_{\la,\mu}(q) =
\sum_{\la \subset \nu \supset \mu}
q^{\deg(\underline{\la}\nu\overline{\mu})}.
\end{equation}
\end{lemma}

\begin{proof}
To determine its entries note first that
\begin{equation}\label{cmat}
c_{\la,\mu}(q) = \sum_{j \in \mZ} q^j \dim \hom_{\D}(P(\la),P(\mu))_j.
\end{equation}
Since $\hom_{\D}(P(\la), P(\mu)) = \hom_{\D} (\D e_\la, \D e_\mu)= e_\la
\D e_\mu
$ and $e_\la \D e_\mu$ has basis
$\left\{(\underline{\la}\nu\overline{\mu})\:\big|\:
\nu \in \La \text{ such that }\la \subset\nu\supset\mu\right\}$ the claim follows.
\end{proof}

Note that $c_{\la,\mu}(q)$ is in fact a polynomial in $\mathbb{N}[q]$ with constant coefficient equal to $1$ if $\la = \mu$ and equal to $0$
otherwise, see Example~\ref{exB}.

Now we introduce {\em cell modules} in the sense of \cite{GL}. The construction is totally analogous to \cite{BS1}, hence we just recall the results.

\begin{definition}
\label{cellmods}
For $\mu\in \La$, define $V(\mu)$ to be the vector space on homogeneous basis
\begin{equation}
\left\{
(c \mu | \:\big|\:
\text{for all oriented cup diagrams $c \mu$}\right\},
\end{equation}
where the degree of the vector $(c \mu |$ is the degree $\deg(c\mu)$ of the oriented cup diagram.
We make $V(\mu)$ into a graded $\D$-module by declaring
for any basis vector $(a \la b)$ of $\D$ that
\begin{equation}\label{Actby}
(a \la b) (c \mu| =
\left\{
\begin{array}{ll}
s_{a \la b}(\mu)  (a \mu|
&\text{if $b^* = c$ and $a \mu$ is oriented,}\\
0&\text{otherwise.}
\end{array}\right.
\end{equation}
where $s_{a \la b}(\mu) \in \{0,\pm 1\}$ is the scalar
from Theorem~\ref{cellular}; hence the action is well-defined. We call $V(\mu)$ the {\it cell module of highest weight $\mu$}.
\end{definition}

\begin{theorem}[Cell module filtration of projectives]
\label{qh1}
For $\la \in \La$,
enumerate the $2^{\op{def}(\la)}$ distinct elements of the set
$\{\mu \in \La\:|\:\mu \supset \la\}$
as $\mu_1,\mu_2,\dots,\mu_r= \la$ 
so that $\mu_i > \mu_j$ implies $i < j$.
Let $M(0) = \{0\}$ and for $i=1,\dots,r$ define
$M(i)$ to be the subspace of $P(\la)$
generated by $M(i-1)$ and the vectors
$$
\left\{
(c \mu_i \overline{\la} ) \:\big|\:
\text{for all oriented cup diagrams $c \mu_i$}\right\}.
$$
Then
$
\{0\} = M(0) \subset M(1) \subset\cdots\subset M(r) = P(\la)
$
is a filtration of $P(\la)$ as a $\D$-module such that for each $i=1,\dots, r$:
$$
M(i) / M(i-1) \cong V(\mu_i) \langle \deg(\mu_i
\overline{\la})\rangle.
$$
\end{theorem}

\begin{proof}
Just apply the same arguments as in \cite[Theorem 5.1]{BS1}.
\end{proof}

\begin{theorem}[Composition factors of cell modules]
\label{qh2}
For $\mu \in \La$, let $N(j)$ be the submodule of $V(\mu)$ spanned by all
graded pieces of degree $\geq j$. Then we have a filtration
$$
V(\mu) = N(0) \supseteq N(1) \supseteq N(2) \supseteq \cdots
$$
as a $\D$-module with
$N(j) / N(j+1) \cong \displaystyle\bigoplus_{\substack{\la  \subset \mu\text{\,with}\\
\deg(\underline{\la} \mu) = j}}
  L(\la) \langle j \rangle
$ for each $j \geq 0$.
\end{theorem}
\begin{proof}
Apply the same arguments as in \cite[Theorem 5.2]{BS1}.
\end{proof}

These polynomials and the resulting {\em $q$-decomposition matrix}
\begin{equation}\label{decmat}
M_\La(q) = (d_{\la,\mu}(q))_{\la,\mu \in \La}
\end{equation}
encode by Theorems~\ref{qh1} and~\ref{qh2} the multiplicities of cell modules in projectives and of irreducibles in cell modules; we have
\begin{align}
[V(\mu)] &= \sum_{\la \in \La} d_{\la,\mu}(q) [L(\la)],\label{dmat}&[P(\la)] &= \sum_{\mu \in \La} d_{\la,\mu}(q) [V(\mu)],
\end{align}
in the Grothendieck group $K_0(\D)$.

The $q$-decomposition matrix $M_\La(q)$ is upper
unitriangular when rows and columns are ordered in some way refining the Bruhat order since by Lemma~\ref{lem:Bruhator} $\la \subset \mu$ implies $\la \leq \mu$.
Note that 
\begin{eqnarray}
\label{KLformula}
d_{\la,\mu}(q)&=&
\begin{cases}
q^d&\text{if $\underline{\la}\mu$ is oriented of degree $d$, and} \\
0&\text{if $\underline{\la}\mu$ is not oriented}.
\end{cases}
\end{eqnarray}
\begin{ex}
The decomposition matrix for the principal block of type ${\rm D}_4$ is given as follows (where we again omit the $\circ$'s in the weight diagrams):
\begin{equation}\label{qdec}
\begin{array}{c|cccccccc}
&
\scriptstyle{\down\down\down\down}&
\scriptstyle{\up\up\down\down}&
\scriptstyle{\up\down\up\down}&
\scriptstyle{\down\up\up\down}&
\scriptstyle{\up\down\down\up}&
\scriptstyle{\down\up\down\up}&
\scriptstyle{\down\down\up\up}&
\scriptstyle{\up\up\up\up}\\
\hline
\scriptstyle{\down\down\down\down}&1&q&0&0&0&0&0&q^2\\
\scriptstyle{\up\up\down\down}&0&1&q&0&0&0&q^2&q\\
\scriptstyle{\up\down\up\down}&0&0&1&q&q&q^2&q&0\\
\scriptstyle{\down\up\up\down}&0&0&0&1&0&q&0&0\\
\scriptstyle{\up\down\down\up}&0&0&0&0&1&q&0&0\\
\scriptstyle{\down\up\down\up}&0&0&0&0&0&1&q&0\\
\scriptstyle{\down\down\up\up}&0&0&0&0&0&0&1&q\\
\scriptstyle{\up\up\up\up}&0&0&0&0&0&0&0&1
\end{array}
\end{equation}
\end{ex}

We can restrict ourself to study the principal blocks:
\begin{lemma}
\label{lem:KLpolys}
\begin{enumerate}[1.)]
\item The q-decomposition matrix $M_\La(q)$ depends (up to relabelling rows and columns) only on the atypicality of the block.
\item In case $\La=\Lap$, the entries of ${M}_\La(q)$ are the parabolic Kazhdan-Lusztig polynomials, denoted $n_{x,y}(q)$ in \cite{SoergelKL} for $x,y\in W^{\ov{p}}$ from Section~\ref{typeD}.
\end{enumerate}
\end{lemma}

\begin{proof}
By Corollary~\ref{atypicality} the algebras for two different blocks with the same atypicality are isomorphic and hence the decomposition numbers are of course the same. The second statement is the main result of \cite{LS}.
\end{proof}

As in \cite[Theorem 5.3.]{BS1} we deduce
\begin{theorem}\label{ghw}
The category ${\D}\MOD$ is a positively
graded highest weight category with duality in the sense of Cline, Parshall and
Scott \cite{CPS}. 

For $\la \in \La$ (the weight poset), the irreducible, standard, costandard,
indecomposable projective and indecomposable injective objects are respectively the modules
$L(\la), V(\la), {V(\la)}^\circledast$, $P(\la)$ and ${P(\la)}^\circledast$. In particular, the algebra $\D$ is a
positively graded quasi-hereditary algebra.
\end{theorem}

We also like to state the following observation
\begin{corollary}
\label{lem:endo_commutative}
Let $\La$ be a block and $\la\in\La$.
The endomorphism ring $\op{End}_{\La}(P(\la))$ of $P(\lambda)$ is a non-negatively graded commutative algebra.
\end{corollary}

\begin{proof}
As a summand of a positively graded algebra it is positively graded. By definition of the multiplication, see Definition \ref{def:multiplication}, it is commutative (the multiplication is just a composition of merges).
\end{proof}

\section{The categories $\Perv$ and $\cO_0^\p(\mathfrak{so}(2k))$ diagrammatically} \label{sec:iso_theorem}
Let us finally turn back to the isotropic Grassmannian $Y_k$ from Section ~\ref{section21}. Let $\Perv_k$ be the category of perverse sheaves on $Y_k$ constructible with respect to the Schubert stratification, that is with respect to $B_D$-orbits, see \cite{Braden} for details. By the localization theorem and Riemann-Hilbert correspondence, see \cite{BuchHottaetal}, this category is equivalent to the principal block  $\cO_0^\p(\mg)$ of the parabolic category $\cO^\p(\mg)$ for the semisimple Lie algebra $\mg=\mathfrak{so}(2k)$ of type ${\rm D}_k$, where $\p$ is one of our maximal parabolic subalgebras, say the one corresonding to $W_{\ov{0}}$. For details we refer to \cite{Humphreys}. The simple objects in either category are naturally labelled by $W^{\ov{0}}$. Attached to $w\in W^{\ov 0}$ we have the simple intersection cohomology
complex $\mathcal{I}_w$ corresponding to the Schubert variety labelled $w$ via \eqref{Young} and the simple module $L(w)$ with highest
weight $w\cdot0$. (As usual $W$ acts on weights via the `dot-action' $w\cdot\la=w(\la+\rho)-\rho$). Let $P(w)$ be the projective cover of
$L(w)$. Then $$P=\bigoplus_{w\in W^{\ov{0}}} P(w)$$ is a minimal projective generator of $\cO_0^\p(\mg)$, see \cite{Humphreys} for more details.

\subsection{The isomorphism theorem}
The main theorem is the
following

\begin{theorem}
\label{thm:main}
Let $k\geq 4$. There is an isomorphism of algebras $\End_\mg(P)\cong\mathbb{D}_{\Lambda_k^{\ov{0}}}$, hence there are equivalences of categories
\begin{align}
\cO_0^\p(\mathfrak{so}(2k))\cong\Perv_k\cong\mathbb{D}
_{\Lambda_k^{\ov{0}}}\Mod
\end{align}
which identify the simple objects $L(w)$, $\mathcal{I}_w$ and $\mathcal{L}(w)$ and their projective covers respectively.
\end{theorem}

\begin{remark}
{\rm Note that, although the theorem only deals with the principal block $\Lambda_k^{\ov{0}}$, by Corollary~\ref{atypicality} it gives in fact a geometric description for all blocks $\La$.  Also note that our equivalence ends up naturally in right $\End_\mg(P)$-modules or $\End_\mg(P)^{\op{opp}}$-modules, but using the duality in category $\cO$ we can identify $\End_\mg(P)^{\op{opp}}=\End_\mg(P)$ which we will do from now on.
}
\end{remark}

The following is a type $\rm{D}$ analogue of Khovanov's original arc algebra.
\begin{corollary}
\label{projinj}
Let $\La=\Lambda_k^{\ov{0}}$ then the algebra $\mathbb{H}_\La$ from Lemma~\ref{Khovalg} is the endomorphism algebra of the sum of all indecomposable projective-injective $\D$-modules.
\end{corollary}

\begin{proof}
Let $\la\in\La$ be of maximal defect with indecomposable projective module $P(\la)$. Under the equivalence of Theorem~\ref{thm:main}
it corresponds to an indecomposable projective module $P(\la)$ in $\cO_0^\p(\mg)$. Let $\la=w\cdot\mu$ with $\mu$ the dominant weight in the same block and $w\in W^{\ov{0}}$. By Irving's characterization of projective-injective modules in parabolic category $\cO$, \cite[4.3]{Irvingselfdual} $P(\la)$ is projective-injective, if and only if $w$ is in the same Kazhdan-Lusztig right cell as the longest element $w_0^{\ov{0}}$ in $W^{\ov{0}}$. Hence the indecomposable projective-injective modules $P(x\cdot\mu)$ are precisely those which can be obtained from $P(w_0^{\ov{0}}\cdot\mu)$ by applying the Hecke algebra action given by translation functors and taking summands. 

By
\cite[Theorem 3.10]{LS} the action of the Hecke algebra factors through the type $\rm D$ Temperley-Lieb algebra and we can identify the Kazhdan-Lusztig basis with our cup diagrams and the action of the Temperley-Lieb algebra purely diagrammatically. Note that $P(w_0^1\cdot\mu)$ corresponds to the cup diagram $C$ of maximal defect with all cups dotted and we ask which cup diagrams can be obtained from it by acting with the Temperley-Lieb algebra. It is obvious that the defect cannot decrease. On the other hand one can easily verify that every cup diagram of maximal defect can be obtained, see \cite[Remark 5.24]{LS}. Then the lemma follows.
\end{proof}

\subsection{Braden's algebra $\ABr$}
To prove Theorem~\ref{thm:main} we will identify $\mathbb{D}_k$ with Braden's algebra $\ABr$ defined below and use  \cite[Theorem 1.7.2]{Braden}:
\begin{theorem}
\label{thmBraden}
Let $k\geq 4$. There is an equivalence of categories $$\ABr\Mod\cong\Perv_k$$ which identifies $\ABr$ with the endomorphism ring of a minimal projective generator of $\Perv_k$.
\end{theorem}

We first define the algebra  $\ABr$ and then state the explicit Isomorphism Theorem, see Theorem~\ref{mainthm}, from which Theorem \ref{thm:main} then follows. 

\begin{definition}
Let $\Lambda$ be a block and $\la \in \Lambda$ be a weight. Given a $\la$-pair $(\alpha,\beta)$ corresponding to a cup $C$ (see Definition~\ref{lapair}), we say it {\it has a parent}, if $C$ is nested in another (dotted or undotted) cup or if otherwise there is a dotted cup to the right of $C$. We call then the minimal cup containing $C$ respectively the leftmost dotted cup to the right of $C$ and its associated $\la$-pair the {\it parent} of $C$; we denote them by $C'$ and $(\alpha', \beta')$ respectively.
All the possible examples are shown in \eqref{parent1}-\eqref{parent4}.
\end{definition}

\begin{definition}
\begin{itemize}
\item[$\blacktriangleright$] A {\it diamond} in $\Lambda_k^{\ov{0}}$ is a quadruple $(\lambda^{\scriptscriptstyle(1)},\lambda^{\scriptscriptstyle(2)},\lambda^{\scriptscriptstyle(3)},\lambda^{\scriptscriptstyle(4)})$ of elements in $\Lambda_k^{\ov{0}}$ such that $\lambda^{\scriptscriptstyle(1)} \leftrightarrow \lambda^{\scriptscriptstyle(2)} \leftrightarrow \lambda^{\scriptscriptstyle(3)} \leftrightarrow \lambda^{\scriptscriptstyle(4)} \leftrightarrow \lambda^{\scriptscriptstyle(1)}$ and all elements are pairwise distinct. We depict them as follows
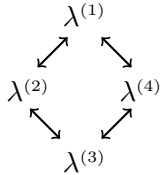
\begin{eqnarray*}
\begin{tikzpicture}[thick,scale=0.5]
\node at (1.25,-.5) {$\lambda^{\scriptscriptstyle(1)}$};
\node at (-.25,-2.5) {$\lambda^{\scriptscriptstyle(2)}$};
\node at (2.75,-2.5) {$\lambda^{\scriptscriptstyle(4)}$};
\node at (1.25,-4.4) {$\lambda^{\scriptscriptstyle(3)}$};
\draw[<->] (.8,-1.1) -- +(-.8,-.8);
\draw[<->] (1.7,-1.1) -- +(.8,-.8);
\draw[<->] (-.2,-3) -- +(.8,-.8);
\draw[<->] (2.5,-3) -- +(-.8,-.8);
\end{tikzpicture}
\end{eqnarray*}
A diamond of the form $(\lambda^{(\sigma(1))},\lambda^{(\sigma(2))},\lambda^{(\sigma(3))},\lambda^{(\sigma(4))})$ for $\sigma \in S_4$ such that $\sigma(1) \equiv \sigma(3)\, {\rm mod}\, 2$ and $\sigma(2) \equiv \sigma(4)\, {\rm mod}\, 2$ is said to be \emph{equivalent} to $(\lambda^{\scriptscriptstyle(1)},\lambda^{\scriptscriptstyle(2)},\lambda^{\scriptscriptstyle(3)},\lambda^{\scriptscriptstyle(4)})$.
\item[$\blacktriangleright$] For fixed $k$, a triple $\la^{\scriptscriptstyle(1)} \leftrightarrow \la^{\scriptscriptstyle(2)} \leftrightarrow \la^{\scriptscriptstyle(3)}$ in $\Lambda_{k}^{\ov{0}}$ {\it cannot be extended to a diamond} if there is no diamond $(\la^{\scriptscriptstyle(1)},\la^{\scriptscriptstyle(2)},\la^{\scriptscriptstyle(3)},\la^{\scriptscriptstyle(4)})$ in $\Lambda_{k}^{\ov{0}}$.
\item[$\blacktriangleright$] For fixed $k$, a triple $\la^{\scriptscriptstyle(1)} \leftrightarrow \la^{\scriptscriptstyle(2)} \leftrightarrow \la^{\scriptscriptstyle(3)}$ in $\Lambda_{k}^{\ov{0}}$ {\it can be enlarged} to a diamond if it cannot be extended but, after extending the weights $\la^{(i)}$ to $\widehat{\la}^{(i)}$ by adding some fixed number of $\down$'s to the left and $\up$'s to the right, the resulting triple can be extended to a diamond $(\widehat{\la}^{\scriptscriptstyle(1)},\widehat{\la}^{\scriptscriptstyle(2)},\widehat{\la}^{\scriptscriptstyle(3)},\widehat{\la}^{\scriptscriptstyle(4)})$ in $\Lambda_{m}^{\ov{0}}$ for some $m>k$.
\end{itemize}
\end{definition}

\begin{ex}
\label{extended}
The triple $(\down\down\down\down,\up\up\down\down,\up\down\up\down)$ can not be extended to a diamond in $\Lambda_{4}^{\ov{0}}$, but it can be enlarged to a diamond $(\down\down\down\down\up\up,\up\up\down\down\up\up,\up\down\up\down\up\up,
\down\down\up\down\up\down)\in\Lambda_{6}^{\ov{0}}$. In other words, the Young diagrams
\begin{eqnarray}
{\tiny\Yvcentermath1\yng(4,4,4,4)\quad\quad\yng(4,4,2,2)\quad\quad\yng(4,3,2,1)
\quad\quad\yng(6,5,4,4,2,1)}
\end{eqnarray}
fit into a $6\times 6$-box, but only three of them into a $4\times 4$-box.
\end{ex}

\begin{ex}
For instance, up to equivalence the diamonds in \eqref{fig:quiv} are 
(\circled{3},\circled{4},\circled{6},\circled{5}), (\circled{3},\circled{4},\circled{6},\circled{7}), (\circled{3},\circled{5},\circled{6},\circled{7}), (\circled{2},\circled{3},\circled{7},\circled{8}), whereas here are the triples that cannot be extended to a diamond, but can be enlarged: 
\begin{equation*}
\begin{array}{cccc}
(\circled{1},\circled{2},\circled{3}),& (\circled{3},\circled{2},\circled{1}),& (\circled{2},\circled{3},\circled{4}),& (\circled{4},\circled{3},\circled{2}),\\[0.1cm] (\circled{2},\circled{3},\circled{5}),& (\circled{5},\circled{3},\circled{2}),& (\circled{6},\circled{7},\circled{8}),& (\circled{8},\circled{7},\circled{6}).
\end{array}
\end{equation*}
\end{ex}

\begin{definition}
\label{Bradenalgebra}
Braden's algebra $\ABr$ is the unitary associative $\mC$-algebra with generators
$$\{e_\la\mid\la\in \Lambda_k^{\ov{0}} \}\cup\{p(\la,\la')\mid \la,\la'\in \Lambda_k^{\ov{0}}, \la\leftrightarrow\la'\}\cup
\{t_{\alpha,\la}\mid \la\in \Lambda_k^{\ov{0}} ,\alpha\in\mZ-\{0\} \}$$
subject to relations (for all $\lambda, \mu, \nu \in \Lambda_k^{\ov{0}}$ and $\alpha, \beta \in \mZ$): 
\begin{enumerate}[(\text{R}-1)]
\item \label{rel1} {\it Idempotent relations:}
\begin{equation*}
\begin{array}{lrclclrcl}
a.)&e_\lambda e_\nu &=& \delta_{\lambda,\nu} e_\lambda, & \; &b.)& \sum_{\la\in \Lambda_k^{\ov{0}}} e_\la&=&1, \\
c.)&e_{\nu}p(\la,\mu)&=&\delta_{\nu,\la}p(\la,\mu), & \; &d.)& p(\la,\mu)e_{\nu}&=&\delta_{\mu,\nu}p(\la,\mu), \\
e.)&e_{\nu}t_{\la,\alpha}&=&\delta_{\nu,\la}t_{\la,\alpha}, & \; &f.)& t_{\la,\alpha}e_{\nu} &=& \delta_{\nu,\la}t_{\la,\alpha};
\end{array}
\end{equation*}
\item {\it The commutative subalgebra:}\label{rel4a}
\begin{equation*}
\begin{array}{crclccrcl}
a.)&t_{\alpha,\la} &=& e_\la \text{ if } | \alpha | > k, & \; &b.)&t_{\lambda,\alpha} t_{\nu,\beta} &=& t_{\nu,\beta} t_{\lambda,\alpha}, \\
c.)&t_{\alpha,\la}t_{-\alpha,\la} &=& e_\lambda, &\;&d.)& t_{\alpha,\la}t_{\beta,\la} &=& e_\la \text{ if } (\alpha,\beta) \text{ is a }\la\text{-pair};
\end{array}
\end{equation*}
\item \label{rel4b} {\it Arrow relations:} $p(\la,\mu)t_{\alpha,\mu} = t_{\alpha,\la}p(\la,\mu)$;
\item \label{rel6} {\it Loop relations:}
Suppose $\la\stackrel{(\alpha,\beta)}\longrightarrow\la'$. Then
  \begin{eqnarray*}
  \op{m}(\la',\la)^{(-1)^\beta}\;=\;t_{\alpha,\la'}t_{\zeta,\la'} \quad \text{and} \quad \op{m}(\la,\la')^{(-1)^\beta}\;=\;t_{\alpha,\la}t_{\zeta,\la}
  \end{eqnarray*}
where $\op{m}(\la,\la')=e_\la+p(\la,\la')p(\la',\la)$ and
$$ \zeta = \left\lbrace \begin{array}{ll} -\alpha', & \text{if } \alpha < -\beta' < \beta < -\alpha' \\ \beta' & \text{otherwise},
\end{array} \right.$$
in case the parent $(\alpha',\beta')$ of $(\alpha,\beta)$ exists, and $t_{\zeta,\la}=e_\lambda$ otherwise.
\item \label{rel7}
{\it Diamond relations:}
\begin{enumerate}[i)]
\item If $(\la^{\scriptscriptstyle(1)},\la^{\scriptscriptstyle(2)},\la^{\scriptscriptstyle(3)},\la^{\scriptscriptstyle(4)})$ is a diamond in $\Lambda_{k}^{\ov{0}}$ then
  \begin{eqnarray*}
    p(\la^{\scriptscriptstyle(3)},\la^{\scriptscriptstyle(2)})p(\la^{\scriptscriptstyle(2)},\la^{\scriptscriptstyle(1)})&=&p(\la^{\scriptscriptstyle(3)},\la^{\scriptscriptstyle(4)})p(\la^{\scriptscriptstyle(4)},\la^{\scriptscriptstyle(1)}).
  \end{eqnarray*}
\item
  Given a triple $\la^{\scriptscriptstyle(1)} \leftrightarrow \la^{\scriptscriptstyle(2)} \leftrightarrow \la^{\scriptscriptstyle(3)}$ in $\Lambda_{k}^{\ov{0}}$ which cannot be extended but can be enlarged to a diamond then
 \begin{eqnarray}
\label{zero}
    p(\la^{\scriptscriptstyle(3)},\la^{\scriptscriptstyle(2)})p(\la^{\scriptscriptstyle(2)},\la^{\scriptscriptstyle(1)})&=&0.
  \end{eqnarray}
\item Given a triple $\la^{\scriptscriptstyle(1)} \stackrel{(\alpha,\beta)}{\rightarrow} \la^{\scriptscriptstyle(2)} \stackrel{(\gamma,\delta)}{\rightarrow} \la^{\scriptscriptstyle(3)}$ in $\Lambda_{k}^{\ov{0}}$ such that
    \begin{itemize}
    \item[$\blacktriangleright$] $\alpha<0$ (hence $(\alpha,\beta)$ corresponds to a dotted cup),
    \item[$\blacktriangleright$] $(\gamma,\delta)$ is not a $\la^{\scriptscriptstyle(1)}$-pair, and
    \item[$\blacktriangleright$] the triple cannot be extended to a diamond, then
    \end{itemize}
      \begin{eqnarray*}
    p(\la^{\scriptscriptstyle(3)},\la^{\scriptscriptstyle(2)})p(\la^{\scriptscriptstyle(2)},\la^{\scriptscriptstyle(1)})&=\,0\,=&p(\la^{\scriptscriptstyle(1)},\la^{\scriptscriptstyle(2)})p(\la^{\scriptscriptstyle(2)},\la^{\scriptscriptstyle(3)}).
  \end{eqnarray*}
\end{enumerate}
\end{enumerate}
\end{definition}

\begin{ex} 
\label{triples}
The triple $(\la^{\scriptscriptstyle(1)},\la^{\scriptscriptstyle(2)},\la^{\scriptscriptstyle(3)})=(\up\up\up\up,\down\down\up\up,\down\up\down\up)$, or equivalently $(\circled{8},\circled{7},\circled{6})$ in \eqref{fig:quiv}), satisfies all conditions from
(R-5) iii). In particular it cannot be extended to a diamond. If we replace however  $\la^{\scriptscriptstyle(3)}$ by $\mu = \up\down\up\down$, that is $\circled{3}$ in \eqref{fig:quiv}, then $\la^{\scriptscriptstyle(1)} \rightarrow \la^{\scriptscriptstyle(2)} \rightarrow \mu$ still satisfies the first two conditions in (R-5) iii), but now can be extended to the diamond $(\la^{\scriptscriptstyle(1)},\la^{\scriptscriptstyle(2)},\mu,\eta)$ with $\eta = \up\up\down\down$ ($\circled{2}$ in \eqref{fig:quiv}).
Moreover we have in the notation from Theorem \ref{prop:generatedindegree1}:
\begin{equation*}
\begin{tikzpicture}[thick, scale=0.5]
\begin{scope}
\node at (-2,0){${}_{\la^{\scriptscriptstyle(1)}}\mathbbm{1}_{\la^{\scriptscriptstyle(2)}}=$};
\node at (0,.1){$\down$};
\node at (1,.1){$\down$};
\node at (2,-.1){$\up$};
\node at (3,-.1){$\up$};
\draw (0,0) .. controls +(0,-1) and +(0,-1) .. +(1,0);
\fill (0.5,-.73) circle(3.5pt);
\draw (2,0) .. controls +(0,-1) and +(0,-1) .. +(1,0);
\fill (2.5,-.73) circle(3.5pt);
\draw (1,0) .. controls +(0,1) and +(0,1) .. +(1,0);
\draw (0,0) .. controls +(0,2.2) and +(0,2.2) .. +(3,0);
\end{scope}
\begin{scope}[xshift=8cm]
\node at (-2,0){${}_{\la^{\scriptscriptstyle(2)}}\mathbbm{1}_{\la^{\scriptscriptstyle(3)}}=$};
\node at (0,.1){$\down$};
\node at (1,-.1){$\up$};
\node at (2,.1){$\down$};
\node at (3,-.1){$\up$};
\draw (0,0) .. controls +(0,1) and +(0,1) .. +(1,0);
\draw (2,0) .. controls +(0,1) and +(0,1) .. +(1,0);
\draw (1,0) .. controls +(0,-1) and +(0,-1) .. +(1,0);
\draw (0,0) .. controls +(0,-2.2) and +(0,-2.2) .. +(3,0);
\end{scope}
\begin{scope}[xshift=16cm]
\node at (-2,0){$\un{\la^{\scriptscriptstyle(1)}}\ov{\la^{\scriptscriptstyle(3)}}=$};
\draw (0,0) .. controls +(0,1) and +(0,1) .. +(1,0);
\fill (0.5,-.73) circle(3.5pt);
\draw (2,0) .. controls +(0,1) and +(0,1) .. +(1,0);
\fill (2.5,-.73) circle(3.5pt);
\draw (0,0) .. controls +(0,-1) and +(0,-1) .. +(1,0);
\draw (2,0) .. controls +(0,-1) and +(0,-1) .. +(1,0);
\end{scope}
\end{tikzpicture}
\end{equation*}
Since $\un{\la^{\scriptscriptstyle(1)}}\ov{\la^{\scriptscriptstyle(3)}}$ is not orientable, it follows ${}_{\la^{\scriptscriptstyle(1)}}\mathbbm{1}_{\la^{\scriptscriptstyle(2)}} \cdot {}_{\la^{\scriptscriptstyle(2)}}\mathbbm{1}_{\la^{\scriptscriptstyle(3)}} = 0$ and since $\la^{\scriptscriptstyle(1)} \neq \la^{\scriptscriptstyle(3)}$ it holds ${}_{\la^{\scriptscriptstyle(2)}}\mathbbm{1}_{\la^{\scriptscriptstyle(3)}} \cdot {}_{\la^{\scriptscriptstyle(1)}}\mathbbm{1}_{\la^{\scriptscriptstyle(2)}} = 0$. While for the diamond we have
\begin{equation*}
\begin{tikzpicture}[thick,scale=0.5]
\begin{scope}
\node at (-2,0){${}_{\la^{\scriptscriptstyle(2)}}\mathbbm{1}_{\mu}=$};
\node at (0,-.1){$\up$};
\node at (1,.1){$\down$};
\node at (2,-.1){$\up$};
\node at (3,.1){$\down$};
\draw (0,0) -- +(0,.7);
\draw (1,0) .. controls +(0,1) and +(0,1) .. +(1,0);
\draw (3,0) -- +(0,.7);
\fill (0,.25) circle(3.5pt);
\draw (1,0) .. controls +(0,-1) and +(0,-1) .. +(1,0);
\draw (0,0) .. controls +(0,-2.2) and +(0,-2.2) .. +(3,0);
\end{scope}

\begin{scope}[xshift=8cm]
\node at (-2,0){${}_{\la^{\scriptscriptstyle(1)}}\mathbbm{1}_{\eta}=$};
\node at (0,-.1){$\up$};
\node at (1,-.1){$\up$};
\node at (2,.1){$\down$};
\node at (3,.1){$\down$};
\draw (0,0) .. controls +(0,1) and +(0,1) .. +(1,0);
\fill (.5,-.73) circle(3.5pt);
\draw (2,0) .. controls +(0,-1) and +(0,-1) .. +(1,0);
\fill (2.5,-.73) circle(3.5pt);
\draw (0,0) .. controls +(0,-1) and +(0,-1) .. +(1,0);
\fill (.5,.73) circle(3.5pt);
\draw (2,0) -- +(0,.7);
\draw (3,0) -- +(0,.7);
\end{scope}

\begin{scope}[xshift=16cm]
\node at (-2,0){${}_{\eta}\mathbbm{1}_{\mu}=$};
\node at (0,-.1){$\up$};
\node at (1,-.1){$\up$};
\node at (2,.1){$\down$};
\node at (3,.1){$\down$};
\draw (0,0) -- +(0,.7);
\draw (1,0) .. controls +(0,1) and +(0,1) .. +(1,0);
\draw (3,0) -- +(0,.7);
\fill (0,.25) circle(3.5pt);
\draw (0,0) .. controls +(0,-1) and +(0,-1) .. +(1,0);
\fill (.5,-.73) circle(3.5pt);
\draw (2,0) -- +(0,-.7);
\draw (3,0) -- +(0,-.7);
\end{scope}
\end{tikzpicture}
\end{equation*}
which gives $ {}_{\la^{\scriptscriptstyle(1)}} \mathbbm{1}_{\la^{\scriptscriptstyle(2)}} \cdot {}_{\la^{\scriptscriptstyle(2)}}\mathbbm{1}_{\mu} = \un{\la^{\scriptscriptstyle(1)}} \eta \ov{\mu} = {}_{\la^{\scriptscriptstyle(1)}} \mathbbm{1}_{\eta} \cdot {}_{\eta}\mathbbm{1}_{\mu}.$ This is an example for the diagrammatic analogue of Relation (R-5) i). 
\end{ex}

\subsection{$\la$-pairs and the quiver}
The formulation in Definition~\ref{Bradenalgebra} is already adapted to our setup and slightly different from the original definition in \cite{Braden}. To match Definition~\ref{Bradenalgebra} with \cite[1.7]{Braden} we need the following key technical lemma which compares our notion of $\la$-pair from Definition~\ref{lapair} with Braden's. 

We start with some notation. Given a weight $\la\in \Lambda_k^{\ov{0}}$ recall the associated antisymmetric sequence $\mathbf{s}(\lambda) = (s_i)_{-k \leq i \leq k}$  from Section \ref{section21}. We extend this sequence by infinitely many $\down$'s to the left and infinitely many $\up$'s to the right to get an infinite sequence $\widetilde{\la}$ indexed by half-integers such that the anti-symmetry line $L$ passes between -$\frac{1}{2}$ and $\frac{1}{2}$, in formulas $\widetilde{\la}=(S_i)_{i \in\mathbb{Z}+\frac{1}{2}}$, where
$$
S_{i-\frac{1}{2}}= \left\lbrace \begin{array}{ll}
- & \text{if } i < -k, \\
- & \text{if } s_i = \down \text{ and } -k \leq i \leq k, \\
+ & \text{if } s_i = \up \text{ and } -k \leq i \leq k, \\
+ & \text{if } i > k.
\end{array} \right.
$$

We switch notations from $\up$ and $\down$ to $+$ and $-$ in the sequence $\widetilde{\lambda}$ to make the comparison to \cite{Braden} more convenient.

For the next lemma and its proof recall the Definitions~\ref{decoratedcups}-\ref{lapairscup} for $\la$-pairs.
\begin{lemma}
\label{lem:comparelapairs}
Let $\la,\la' \in \Lambda_k^{\ov{0}}$ and $-k\leq i,j\leq k$ then
\begin{eqnarray}
\la\stackrel{(i,j)}{\longrightarrow}\la' \,\, \text{ if and only if }\,\,\widetilde{\la}\stackrel{(i',j')}{\longrightarrow}\widetilde{\la}'
\end{eqnarray}
with the right hand side in the terminology of \cite{Braden} where $i'=i-\frac{1}{2}$, $j'=j-\frac{1}{2}$.
\end{lemma}
\begin{proof}
For a weight $\la$ with associated sequence $(S_i)_{i \in\mathbb{Z}+\frac{1}{2}}$ and $a,b\in\mathbb{Z}+\frac{1}{2}$ with $a<b$, we set $\op{b}(\pm,a,b)=|\{r\mid a< r< b, S_r=\pm\}|$.
A $\la$-pair in the sense of \cite{Braden} is a pair $(\alpha,\beta)\in(\mathbb{Z}+\frac{1}{2})^2$ such that either
\begin{enumerate}[i)]
\item $0<\alpha<\beta$ and $S_\alpha=-$ and $S_\beta=+$ and $\op{b}(-,\alpha,\beta)=\op{b}(+,\alpha,\beta)$, or
\item $0<-\alpha<0<\beta$ with $\alpha+\frac{1}{2}$ even and $S_\alpha=S_\beta=+$ and  moreover $\op{b}(+,\alpha,\beta+1)=\op{b}(-,\alpha,\beta+1)-1$ and $\op{b}(-,\gamma,\nu)=\op{b}(\pm,\gamma,\nu)$ for the pairs $(\gamma,\nu)\in\{(-\alpha, \alpha),(-\beta,\beta)\}$.\hfill\\ (In this case we automatically have $S_{-\alpha}=S_{-\beta}=-$)
\end{enumerate}
For  $\la,\la' \in \Lambda_k^{\ov{0}}$, the case i) obviously corresponds precisely to cups of type (Cup 1), hence to our $\la$ pairs $(i,j)$ for $0<i<j$. We claim case ii) corresponds to cups of type (Cup 3), hence to our $\la$ pairs $(-i,j)$ for $0<-i<j$. The conditions on the pairs $(-\alpha, \alpha),(-\beta,\beta)$ means they correspond to two cups crossing the middle line $L$ in the symmetric cup diagrams from \cite{LS}. Since $\alpha+\frac{1}{2}$ is even, these two cups get turned into a dotted cup using the rules from \cite[5.2]{LS}.
\end{proof}
\begin{remark}
{\rm
One should note that Braden has an infinite number of $\la$-pairs for a single non-truncated weight $\lambda$, but only a finite number with $\la'\in \Lambda_k^{\ov{0}}$ as well. Our definition of $\la$-pairs produces only these relevant pairs.
}
\end{remark}
\begin{corollary}
\label{CorCartan}
Let $k\geq 4$ and $\La={\Lambda_k^{\ov{0}}}$. The algebra $\ABr$ agrees with the algebra defined in \cite[1.7]{Braden}. The Cartan matrix of $\ABr$ agrees with the Cartan matrix of $\mathbb{D}_{\Lambda}.$
\end{corollary}
\begin{proof}
The first part follows from Lemma~\ref{lem:comparelapairs} and the definitions.\footnote{We tried to clarify the misleading formulation of the analogue of (R-5) iii) in \cite{Braden}.} The second statement follows directly from the fact that both Cartan matrices are of the form $C_\La=M_\La^tM_\La$, where the entries of the decomposition matrix $M_\La$ are parabolic Kazhdan-Lusztig polynomials of type $({\rm D}_k,{\rm A}_{k-1})$, see \cite[Theorem 3.11.4 (i)]{BGS} and Lemma~\ref{lem:KLpolys} respectively.
\end{proof}
We directly deduce that the underlying graph of the quiver describing the category $\Perv_k$ is precisely determined by the $\la$=pairs. ( For the definition of $\mathcal{Q}(\Lambda)$ see Definition~\ref{Extquivdef}.)
\begin{corollary}
\label{samegraph}
The graph $\mathcal{Q}_k$ underlying the quiver of  $\ABr$ is $\mathcal{Q}(\Lambda_k^{\ov{0}})$.
\end{corollary}
In case $k=4$ the quiver is explicitly given in \eqref{exB}.

\subsection{Diamonds}
We now classify the possible diamonds diagrammatically.

\begin{prop}
\label{diamondsclass}
\begin{enumerate}[1.)]
\item Up to equivalence and any possible decorations with dots the following local configurations are the only possible diamonds where the relevant parts do not contain any rays. 
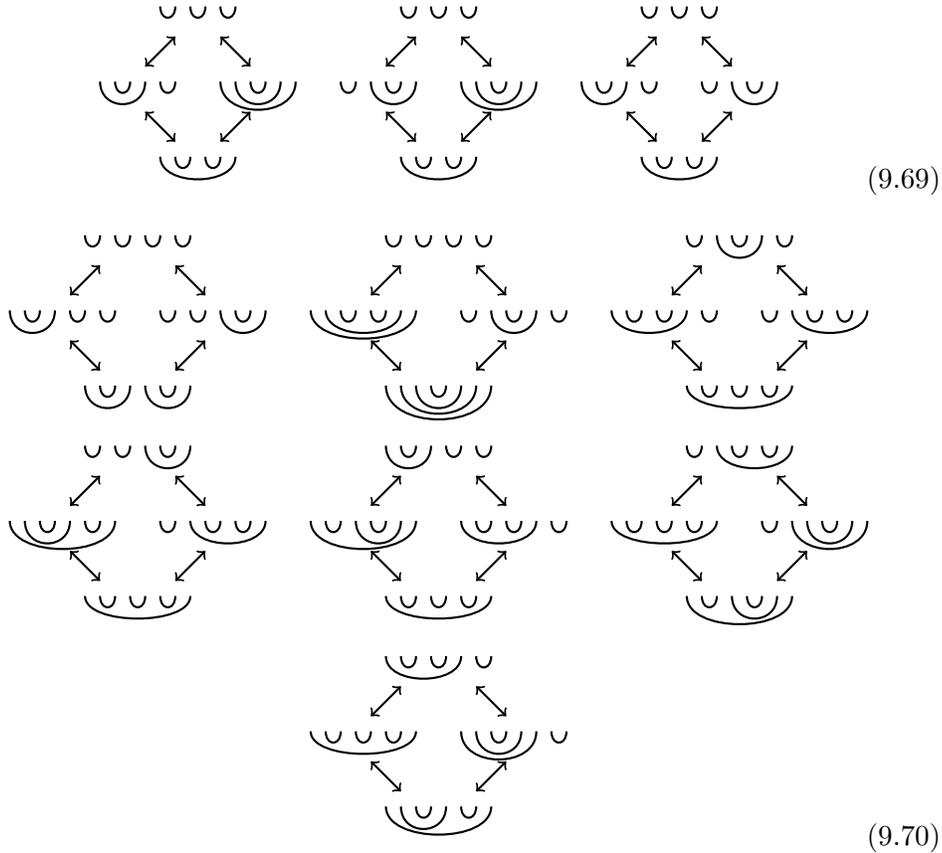
\begin{eqnarray}
\label{diamonds2}
\begin{tikzpicture}[thick,scale=0.4]
\draw (2,0) .. controls +(0,-.5) and +(0,-.5) .. +(.5,0);
\draw (3,0) .. controls +(0,-.5) and +(0,-.5) .. +(.5,0);
\draw (4,0) .. controls +(0,-.5) and +(0,-.5) .. +(.5,0);

\draw[<->] (2.5,-1) -- +(-1,-1);
\draw[<->] (4,-1) -- +(1,-1);

\draw (0,-2.5) .. controls +(0,-1) and +(0,-1) .. +(1.5,0);
\draw (.5,-2.5) .. controls +(0,-.5) and +(0,-.5) .. +(.5,0);
\draw (2,-2.5) .. controls +(0,-.5) and +(0,-.5) .. +(.5,0);

\draw (4,-2.5) .. controls +(0,-1.25) and +(0,-1.25) .. +(2.5,0);
\draw (4.5,-2.5) .. controls +(0,-1) and +(0,-1) .. +(1.5,0);
\draw (5,-2.5) .. controls +(0,-.5) and +(0,-.5) .. +(.5,0);

\draw[<->] (2.5,-4.5) -- +(-1,1);
\draw[<->] (4,-4.5) -- +(1,1);

\draw (2,-5) .. controls +(0,-1) and +(0,-1) .. +(2.5,0);
\draw (2.5,-5) .. controls +(0,-.5) and +(0,-.5) .. +(.5,0);
\draw (3.5,-5) .. controls +(0,-.5) and +(0,-.5) .. +(.5,0);

\begin{scope}[xshift=8cm]
\draw (2,0) .. controls +(0,-.5) and +(0,-.5) .. +(.5,0);
\draw (3,0) .. controls +(0,-.5) and +(0,-.5) .. +(.5,0);
\draw (4,0) .. controls +(0,-.5) and +(0,-.5) .. +(.5,0);

\draw[<->] (2.5,-1) -- +(-1,-1);
\draw[<->] (4,-1) -- +(1,-1);

\draw (1,-2.5) .. controls +(0,-1) and +(0,-1) .. +(1.5,0);
\draw (0,-2.5) .. controls +(0,-.5) and +(0,-.5) .. +(.5,0);
\draw (1.5,-2.5) .. controls +(0,-.5) and +(0,-.5) .. +(.5,0);

\draw (4,-2.5) .. controls +(0,-1.25) and +(0,-1.25) .. +(2.5,0);
\draw (4.5,-2.5) .. controls +(0,-1) and +(0,-1) .. +(1.5,0);
\draw (5,-2.5) .. controls +(0,-.5) and +(0,-.5) .. +(.5,0);

\draw[<->] (2.5,-4.5) -- +(-1,1);
\draw[<->] (4,-4.5) -- +(1,1);

\draw (2,-5) .. controls +(0,-1) and +(0,-1) .. +(2.5,0);
\draw (2.5,-5) .. controls +(0,-.5) and +(0,-.5) .. +(.5,0);
\draw (3.5,-5) .. controls +(0,-.5) and +(0,-.5) .. +(.5,0);
\end{scope}

\begin{scope}[xshift=16cm]
\draw (2,0) .. controls +(0,-.5) and +(0,-.5) .. +(.5,0);
\draw (3,0) .. controls +(0,-.5) and +(0,-.5) .. +(.5,0);
\draw (4,0) .. controls +(0,-.5) and +(0,-.5) .. +(.5,0);

\draw[<->] (2.5,-1) -- +(-1,-1);
\draw[<->] (4,-1) -- +(1,-1);

\draw (0,-2.5) .. controls +(0,-1) and +(0,-1) .. +(1.5,0);
\draw (.5,-2.5) .. controls +(0,-.5) and +(0,-.5) .. +(.5,0);
\draw (2,-2.5) .. controls +(0,-.5) and +(0,-.5) .. +(.5,0);

\draw (4,-2.5) .. controls +(0,-.5) and +(0,-.5) .. +(.5,0);
\draw (5,-2.5) .. controls +(0,-1) and +(0,-1) .. +(1.5,0);
\draw (5.5,-2.5) .. controls +(0,-.5) and +(0,-.5) .. +(.5,0);

\draw[<->] (2.5,-4.5) -- +(-1,1);
\draw[<->] (4,-4.5) -- +(1,1);

\draw (2,-5) .. controls +(0,-1) and +(0,-1) .. +(2.5,0);
\draw (2.5,-5) .. controls +(0,-.5) and +(0,-.5) .. +(.5,0);
\draw (3.5,-5) .. controls +(0,-.5) and +(0,-.5) .. +(.5,0);
\end{scope}
\end{tikzpicture}
\end{eqnarray}
\begin{eqnarray}
\label{diamonds1}
\begin{tikzpicture}[thick,scale=0.4]
\draw (2.5,0) .. controls +(0,-.5) and +(0,-.5) .. +(.5,0);
\draw (3.5,0) .. controls +(0,-.5) and +(0,-.5) .. +(.5,0);
\draw (4.5,0) .. controls +(0,-.5) and +(0,-.5) .. +(.5,0);
\draw (5.5,0) .. controls +(0,-.5) and +(0,-.5) .. +(.5,0);

\draw[<->] (3,-1) -- +(-1,-1);
\draw[<->] (5.5,-1) -- +(1,-1);

\draw (0,-2.5) .. controls +(0,-1) and +(0,-1) .. +(1.5,0);
\draw (.5,-2.5) .. controls +(0,-.5) and +(0,-.5) .. +(.5,0);
\draw (2,-2.5) .. controls +(0,-.5) and +(0,-.5) .. +(.5,0);
\draw (3,-2.5) .. controls +(0,-.5) and +(0,-.5) .. +(.5,0);

\draw (5,-2.5) .. controls +(0,-.5) and +(0,-.5) .. +(.5,0);
\draw (6,-2.5) .. controls +(0,-.5) and +(0,-.5) .. +(.5,0);
\draw (7,-2.5) .. controls +(0,-1) and +(0,-1) .. +(1.5,0);
\draw (7.5,-2.5) .. controls +(0,-.5) and +(0,-.5) .. +(.5,0);

\draw[<->] (3,-4.5) -- +(-1,1);
\draw[<->] (5.5,-4.5) -- +(1,1);

\draw (2.5,-5) .. controls +(0,-1) and +(0,-1) .. +(1.5,0);
\draw (3,-5) .. controls +(0,-.5) and +(0,-.5) .. +(.5,0);
\draw (4.5,-5) .. controls +(0,-1) and +(0,-1) .. +(1.5,0);
\draw (5,-5) .. controls +(0,-.5) and +(0,-.5) .. +(.5,0);

\begin{scope}[xshift=10cm]
\draw (2.5,0) .. controls +(0,-.5) and +(0,-.5) .. +(.5,0);
\draw (3.5,0) .. controls +(0,-.5) and +(0,-.5) .. +(.5,0);
\draw (4.5,0) .. controls +(0,-.5) and +(0,-.5) .. +(.5,0);
\draw (5.5,0) .. controls +(0,-.5) and +(0,-.5) .. +(.5,0);

\draw[<->] (3,-1) -- +(-1,-1);
\draw[<->] (5.5,-1) -- +(1,-1);

\draw (0,-2.5) .. controls +(0,-1.25) and +(0,-1.25) .. +(3.5,0);
\draw (0.5,-2.5) .. controls +(0,-1) and +(0,-1) .. +(2.5,0);
\draw (1,-2.5) .. controls +(0,-.5) and +(0,-.5) .. +(.5,0);
\draw (2,-2.5) .. controls +(0,-.5) and +(0,-.5) .. +(.5,0);

\draw (5,-2.5) .. controls +(0,-.5) and +(0,-.5) .. +(.5,0);
\draw (6,-2.5) .. controls +(0,-1) and +(0,-1) .. +(1.5,0);
\draw (6.5,-2.5) .. controls +(0,-.5) and +(0,-.5) .. +(.5,0);
\draw (8,-2.5) .. controls +(0,-.5) and +(0,-.5) .. +(.5,0);

\draw[<->] (3,-4.5) -- +(-1,1);
\draw[<->] (5.5,-4.5) -- +(1,1);

\draw (2.5,-5) .. controls +(0,-1.5) and +(0,-1.5) .. +(3.5,0);
\draw (3,-5) .. controls +(0,-1.25) and +(0,-1.25) .. +(2.5,0);
\draw (3.5,-5) .. controls +(0,-1) and +(0,-1) .. +(1.5,0);
\draw (4,-5) .. controls +(0,-.5) and +(0,-.5) .. +(.5,0);
\end{scope}

\begin{scope}[xshift=20cm]
\draw (2.5,0) .. controls +(0,-.5) and +(0,-.5) .. +(.5,0);
\draw (3.5,0) .. controls +(0,-1) and +(0,-1) .. +(1.5,0);
\draw (4,0) .. controls +(0,-.5) and +(0,-.5) .. +(.5,0);
\draw (5.5,0) .. controls +(0,-.5) and +(0,-.5) .. +(.5,0);

\draw[<->] (3,-1) -- +(-1,-1);
\draw[<->] (5.5,-1) -- +(1,-1);

\draw (0,-2.5) .. controls +(0,-1) and +(0,-1) .. +(2.5,0);
\draw (0.5,-2.5) .. controls +(0,-.5) and +(0,-.5) .. +(.5,0);
\draw (1.5,-2.5) .. controls +(0,-.5) and +(0,-.5) .. +(.5,0);
\draw (3,-2.5) .. controls +(0,-.5) and +(0,-.5) .. +(.5,0);

\draw (6,-2.5) .. controls +(0,-1) and +(0,-1) .. +(2.5,0);
\draw (5,-2.5) .. controls +(0,-.5) and +(0,-.5) .. +(.5,0);
\draw (6.5,-2.5) .. controls +(0,-.5) and +(0,-.5) .. +(.5,0);
\draw (7.5,-2.5) .. controls +(0,-.5) and +(0,-.5) .. +(.5,0);

\draw[<->] (3,-4.5) -- +(-1,1);
\draw[<->] (5.5,-4.5) -- +(1,1);

\draw (2.5,-5) .. controls +(0,-1) and +(0,-1) .. +(3.5,0);
\draw (3,-5) .. controls +(0,-.5) and +(0,-.5) .. +(.5,0);
\draw (4,-5) .. controls +(0,-.5) and +(0,-.5) .. +(.5,0);
\draw (5,-5) .. controls +(0,-.5) and +(0,-.5) .. +(.5,0);
\end{scope}

\begin{scope}[yshift=-7cm]
\draw (2.5,0) .. controls +(0,-.5) and +(0,-.5) .. +(.5,0);
\draw (3.5,0) .. controls +(0,-.5) and +(0,-.5) .. +(.5,0);
\draw (4.5,0) .. controls +(0,-1) and +(0,-1) .. +(1.5,0);
\draw (5,0) .. controls +(0,-.5) and +(0,-.5) .. +(.5,0);

\draw[<->] (3,-1) -- +(-1,-1);
\draw[<->] (5.5,-1) -- +(1,-1);

\draw (0,-2.5) .. controls +(0,-1.25) and +(0,-1.25) .. +(3.5,0);
\draw (0.5,-2.5) .. controls +(0,-1) and +(0,-1) .. +(1.5,0);
\draw (1,-2.5) .. controls +(0,-.5) and +(0,-.5) .. +(.5,0);
\draw (2.5,-2.5) .. controls +(0,-.5) and +(0,-.5) .. +(.5,0);

\draw (5,-2.5) .. controls +(0,-.5) and +(0,-.5) .. +(.5,0);
\draw (6,-2.5) .. controls +(0,-1) and +(0,-1) .. +(2.5,0);
\draw (6.5,-2.5) .. controls +(0,-.5) and +(0,-.5) .. +(.5,0);
\draw (7.5,-2.5) .. controls +(0,-.5) and +(0,-.5) .. +(.5,0);

\draw[<->] (3,-4.5) -- +(-1,1);
\draw[<->] (5.5,-4.5) -- +(1,1);

\draw (2.5,-5) .. controls +(0,-1) and +(0,-1) .. +(3.5,0);
\draw (3,-5) .. controls +(0,-.5) and +(0,-.5) .. +(.5,0);
\draw (4,-5) .. controls +(0,-.5) and +(0,-.5) .. +(.5,0);
\draw (5,-5) .. controls +(0,-.5) and +(0,-.5) .. +(.5,0);
\end{scope}

\begin{scope}[xshift=10cm,yshift=-7cm]
\draw (4.5,0) .. controls +(0,-.5) and +(0,-.5) .. +(.5,0);
\draw (5.5,0) .. controls +(0,-.5) and +(0,-.5) .. +(.5,0);
\draw (2.5,0) .. controls +(0,-1) and +(0,-1) .. +(1.5,0);
\draw (3,0) .. controls +(0,-.5) and +(0,-.5) .. +(.5,0);

\draw[<->] (3,-1) -- +(-1,-1);
\draw[<->] (5.5,-1) -- +(1,-1);

\draw (0,-2.5) .. controls +(0,-1.25) and +(0,-1.25) .. +(3.5,0);
\draw (1.5,-2.5) .. controls +(0,-1) and +(0,-1) .. +(1.5,0);
\draw (2,-2.5) .. controls +(0,-.5) and +(0,-.5) .. +(.5,0);
\draw (0.5,-2.5) .. controls +(0,-.5) and +(0,-.5) .. +(.5,0);

\draw (5.5,-2.5) .. controls +(0,-.5) and +(0,-.5) .. +(.5,0);
\draw (5,-2.5) .. controls +(0,-1) and +(0,-1) .. +(2.5,0);
\draw (6.5,-2.5) .. controls +(0,-.5) and +(0,-.5) .. +(.5,0);
\draw (8,-2.5) .. controls +(0,-.5) and +(0,-.5) .. +(.5,0);

\draw[<->] (3,-4.5) -- +(-1,1);
\draw[<->] (5.5,-4.5) -- +(1,1);

\draw (2.5,-5) .. controls +(0,-1) and +(0,-1) .. +(3.5,0);
\draw (3,-5) .. controls +(0,-.5) and +(0,-.5) .. +(.5,0);
\draw (4,-5) .. controls +(0,-.5) and +(0,-.5) .. +(.5,0);
\draw (5,-5) .. controls +(0,-.5) and +(0,-.5) .. +(.5,0);
\end{scope}

\begin{scope}[xshift=20cm,yshift=-7cm]
\draw (2.5,0) .. controls +(0,-.5) and +(0,-.5) .. +(.5,0);
\draw (3.5,0) .. controls +(0,-1) and +(0,-1) .. +(2.5,0);
\draw (4,0) .. controls +(0,-.5) and +(0,-.5) .. +(.5,0);
\draw (5,0) .. controls +(0,-.5) and +(0,-.5) .. +(.5,0);

\draw[<->] (3,-1) -- +(-1,-1);
\draw[<->] (5.5,-1) -- +(1,-1);

\draw (0,-2.5) .. controls +(0,-1) and +(0,-1) .. +(3.5,0);
\draw (0.5,-2.5) .. controls +(0,-.5) and +(0,-.5) .. +(.5,0);
\draw (1.5,-2.5) .. controls +(0,-.5) and +(0,-.5) .. +(.5,0);
\draw (2.5,-2.5) .. controls +(0,-.5) and +(0,-.5) .. +(.5,0);

\draw (5,-2.5) .. controls +(0,-.5) and +(0,-.5) .. +(.5,0);
\draw (6,-2.5) .. controls +(0,-1.25) and +(0,-1.25) .. +(2.5,0);
\draw (6.5,-2.5) .. controls +(0,-1) and +(0,-1) .. +(1.5,0);
\draw (7,-2.5) .. controls +(0,-.5) and +(0,-.5) .. +(.5,0);

\draw[<->] (3,-4.5) -- +(-1,1);
\draw[<->] (5.5,-4.5) -- +(1,1);

\draw (2.5,-5) .. controls +(0,-1.25) and +(0,-1.25) .. +(3.5,0);
\draw (3,-5) .. controls +(0,-.5) and +(0,-.5) .. +(.5,0);
\draw (4,-5) .. controls +(0,-1) and +(0,-1) .. +(1.5,0);
\draw (4.5,-5) .. controls +(0,-.5) and +(0,-.5) .. +(.5,0);
\end{scope}

\begin{scope}[xshift=10cm,yshift=-14cm]
\draw (3,0) .. controls +(0,-.5) and +(0,-.5) .. +(.5,0);
\draw (2.5,0) .. controls +(0,-1) and +(0,-1) .. +(2.5,0);
\draw (4,0) .. controls +(0,-.5) and +(0,-.5) .. +(.5,0);
\draw (5.5,0) .. controls +(0,-.5) and +(0,-.5) .. +(.5,0);

\draw[<->] (3,-1) -- +(-1,-1);

\draw[<->] (5.5,-1) -- +(1,-1);

\draw (0,-2.5) .. controls +(0,-1) and +(0,-1) .. +(3.5,0);
\draw (0.5,-2.5) .. controls +(0,-.5) and +(0,-.5) .. +(.5,0);
\draw (1.5,-2.5) .. controls +(0,-.5) and +(0,-.5) .. +(.5,0);
\draw (2.5,-2.5) .. controls +(0,-.5) and +(0,-.5) .. +(.5,0);

\draw (8,-2.5) .. controls +(0,-.5) and +(0,-.5) .. +(.5,0);
\draw (5,-2.5) .. controls +(0,-1.25) and +(0,-1.25) .. +(2.5,0);
\draw (5.5,-2.5) .. controls +(0,-1) and +(0,-1) .. +(1.5,0);
\draw (6,-2.5) .. controls +(0,-.5) and +(0,-.5) .. +(.5,0);

\draw[<->] (3,-4.5) -- +(-1,1);
\draw[<->] (5.5,-4.5) -- +(1,1);

\draw (2.5,-5) .. controls +(0,-1.25) and +(0,-1.25) .. +(3.5,0);
\draw (3.5,-5) .. controls +(0,-.5) and +(0,-.5) .. +(.5,0);
\draw (3,-5) .. controls +(0,-1) and +(0,-1) .. +(1.5,0);
\draw (5,-5) .. controls +(0,-.5) and +(0,-.5) .. +(.5,0);
\end{scope}
\end{tikzpicture}
\end{eqnarray}
\item The possible diamonds with rays are obtained from these by allowing only vertices in a fixed smaller interval and forgetting cups connecting not allowed vertices only and turn cups which connect an allowed with a not allowed vertex into dotted or undotted rays depending if the vertex which gets removed is to the left or to the right of the allowed interval.
\end{enumerate}
\end{prop}

\begin{proof}
One easily verifies that \eqref{diamonds1} and \eqref{diamonds2} give indeed diamonds and also when we restrict the vertices to an interval. To see that these are all let us first
consider the case where the relevant pieces are cups only, hence the arrow correspond to $\la$-pairs of the form
\begin{eqnarray}
\begin{tikzpicture}[thick,scale=0.7]
\draw (0,0) node[above]{} .. controls +(0,-.5) and +(0,-.5) .. +(.5,0) node[above]{};
\draw (1,0) node[above]{} .. controls +(0,-.5) and +(0,-.5) .. +(.5,0) node[above]{};
\draw[<->] (1.75,-.2) -- +(1,0);
\draw (3,0) node[above]{} .. controls +(0,-1) and +(0,-1) .. +(1.5,0) node[above]{};
\draw (3.5,0) node[above]{} .. controls +(0,-.5) and +(0,-.5) .. +(.5,0) node[above]{};
\end{tikzpicture}
\end{eqnarray}
with all possible configurations of dots. To simplify arguments we first ignore all dots and consider the possible configurations
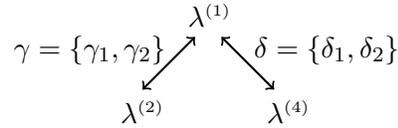
\begin{eqnarray*}
\begin{tikzpicture}[thick,scale=0.7]
\node (A) at (0,-1.25) {$\gamma = \{ \gamma_1,\gamma_2\}$};
\node (B) at (4.5,-1.25) {$\delta = \{ \delta_1,\delta_2\}$};
\draw[<->] (2,-1) node[above]{\quad $\lambda^{\scriptscriptstyle(1)}$} -- +(-1,-1) node[below]{$\lambda^{\scriptscriptstyle(2)}$};
\draw[<->] (2.5,-1) -- +(1,-1) node[below]{\quad $\lambda^{\scriptscriptstyle(4)}$};
\end{tikzpicture}
\end{eqnarray*}
where we number, from left to right according to their endpoints, the cups in $\underline{\lambda^{\scriptscriptstyle(1)}}$ involved in either of the two moves and indicate by the pairs $\gamma$ and $\delta$ the two pairs of vertices where the labels get changed. \hfill\\

Then there are the following cases:

\begin{enumerate}[i)]
\item $\gamma=\delta$: This would imply $\lambda^{\scriptscriptstyle(2)} = \lambda^{\scriptscriptstyle(4)}$ which is not allowed.
\item $|\gamma \cap \delta | = 1$: Then one can verify directly that the only possible configurations are the following\hfill\\


\begin{table}[!h]
\caption{Configurations for $|\gamma \cap \delta | = 1$}
\label{table_conf_1}
\begin{tabular}{|c|c|c|c|c|c|c|}
\hline $\lambda^{\scriptscriptstyle(1)}$ &
\begin{tikzpicture}[thick,scale=0.4]
\node (B) at (0,.1) {$\quad$};

\draw (0,0) .. controls +(0,-.5) and +(0,-.5) .. +(.5,0);
\draw (1,0) .. controls +(0,-.5) and +(0,-.5) .. +(.5,0);
\draw (2,0) .. controls +(0,-.5) and +(0,-.5) .. +(.5,0);

\node (A) at (0,-1.15) {$\quad$};
\end{tikzpicture}
&
\begin{tikzpicture}[thick,scale=0.4]
\node (B) at (0,.1) {$\quad$};

\draw (0,0) .. controls +(0,-.5) and +(0,-.5) .. +(.5,0);
\draw (1,0) .. controls +(0,-.5) and +(0,-.5) .. +(.5,0);
\draw (2,0) .. controls +(0,-.5) and +(0,-.5) .. +(.5,0);

\node (A) at (0,-1.15) {$\quad$};
\end{tikzpicture}
&
\begin{tikzpicture}[thick,scale=0.4]
\node (B) at (0,.1) {$\quad$};

\draw (0,0) .. controls +(0,-.5) and +(0,-.5) .. +(.5,0);
\draw (1,0) .. controls +(0,-.5) and +(0,-.5) .. +(.5,0);
\draw (2,0) .. controls +(0,-.5) and +(0,-.5) .. +(.5,0);

\node (A) at (0,-1.15) {$\quad$};
\end{tikzpicture}
&
\begin{tikzpicture}[thick,scale=0.4]
\node (B) at (0,.1) {$\quad$};

\draw (0,0) .. controls +(0,-1) and +(0,-1) .. +(2.5,0);
\draw (0.5,0) .. controls +(0,-.5) and +(0,-.5) .. +(.5,0);
\draw (1.5,0) .. controls +(0,-.5) and +(0,-.5) .. +(.5,0);

\node (A) at (0,-1.15) {$\quad$};
\end{tikzpicture}
&
\begin{tikzpicture}[thick,scale=0.4]
\node (B) at (0,.1) {$\quad$};

\draw (0,0) .. controls +(0,-1) and +(0,-1) .. +(2.5,0);
\draw (0.5,0) .. controls +(0,-.5) and +(0,-.5) .. +(.5,0);
\draw (1.5,0) .. controls +(0,-.5) and +(0,-.5) .. +(.5,0);

\node (A) at (0,-1.15) {$\quad$};
\end{tikzpicture}
&
\begin{tikzpicture}[thick,scale=0.4]
\node (B) at (0,.1) {$\quad$};

\draw (0,0) .. controls +(0,-1) and +(0,-1) .. +(2.5,0);
\draw (0.5,0) .. controls +(0,-.5) and +(0,-.5) .. +(.5,0);
\draw (1.5,0) .. controls +(0,-.5) and +(0,-.5) .. +(.5,0);

\node (A) at (0,-1.15) {$\quad$};
\end{tikzpicture}\\
\hline $\gamma$ & $\{1,2 \}$ & $\{1,3 \}$ & $\{1,2 \}$ & $\{1,2 \}$ & $\{1,3 \}$ & $\{1,2 \}$\\
\hline $\delta$ & $\{2,3 \}$ & $\{2,3 \}$ & $\{1,3 \}$ & $\{2,3 \}$ & $\{2,3 \}$ & $\{1,3 \}$\\
\hline $\lambda^{\scriptscriptstyle(3)}$ &
\begin{tikzpicture}[thick,scale=0.4]
\node (B) at (0,.1) {$\quad$};

\draw (0,0) .. controls +(0,-1) and +(0,-1) .. +(2.5,0);
\draw (0.5,0) .. controls +(0,-.5) and +(0,-.5) .. +(.5,0);
\draw (1.5,0) .. controls +(0,-.5) and +(0,-.5) .. +(.5,0);

\node (A) at (0,-1.15) {$\quad$};
\end{tikzpicture}
&
\begin{tikzpicture}[thick,scale=0.4]
\node (B) at (0,.1) {$\quad$};

\draw (0,0) .. controls +(0,-1) and +(0,-1) .. +(2.5,0);
\draw (0.5,0) .. controls +(0,-.5) and +(0,-.5) .. +(.5,0);
\draw (1.5,0) .. controls +(0,-.5) and +(0,-.5) .. +(.5,0);

\node (A) at (0,-1.15) {$\quad$};
\end{tikzpicture}
&
\begin{tikzpicture}[thick,scale=0.4]
\node (B) at (0,.1) {$\quad$};

\draw (0,0) .. controls +(0,-1) and +(0,-1) .. +(2.5,0);
\draw (0.5,0) .. controls +(0,-.5) and +(0,-.5) .. +(.5,0);
\draw (1.5,0) .. controls +(0,-.5) and +(0,-.5) .. +(.5,0);

\node (A) at (0,-1.15) {$\quad$};
\end{tikzpicture}
&
\begin{tikzpicture}[thick,scale=0.4]
\node (B) at (0,.1) {$\quad$};

\draw (0,0) .. controls +(0,-.5) and +(0,-.5) .. +(.5,0);
\draw (1,0) .. controls +(0,-.5) and +(0,-.5) .. +(.5,0);
\draw (2,0) .. controls +(0,-.5) and +(0,-.5) .. +(.5,0);

\node (A) at (0,-1.15) {$\quad$};
\end{tikzpicture}
&
\begin{tikzpicture}[thick,scale=0.4]
\node (B) at (0,.1) {$\quad$};

\draw (0,0) .. controls +(0,-.5) and +(0,-.5) .. +(.5,0);
\draw (1,0) .. controls +(0,-.5) and +(0,-.5) .. +(.5,0);
\draw (2,0) .. controls +(0,-.5) and +(0,-.5) .. +(.5,0);

\node (A) at (0,-1.15) {$\quad$};
\end{tikzpicture}
&
\begin{tikzpicture}[thick,scale=0.4]
\node (B) at (0,.1) {$\quad$};

\draw (0,0) .. controls +(0,-.5) and +(0,-.5) .. +(.5,0);
\draw (1,0) .. controls +(0,-.5) and +(0,-.5) .. +(.5,0);
\draw (2,0) .. controls +(0,-.5) and +(0,-.5) .. +(.5,0);

\node (A) at (0,-1.15) {$\quad$};
\end{tikzpicture} \\
\hline
\cline{1-4} $\lambda^{\scriptscriptstyle(1)}$ &
\begin{tikzpicture}[thick,scale=0.4]
\node (B) at (0,.1) {$\quad$};

\draw (1,0) .. controls +(0,-1) and +(0,-1) .. +(1.5,0);
\draw (0,0) .. controls +(0,-.5) and +(0,-.5) .. +(.5,0);
\draw (1.5,0) .. controls +(0,-.5) and +(0,-.5) .. +(.5,0);

\node (A) at (0,-1.15) {$\quad$};
\end{tikzpicture}
&
\begin{tikzpicture}[thick,scale=0.4]
\node (B) at (0,.1) {$\quad$};

\draw (0,0) .. controls +(0,-1.25) and +(0,-1.25) .. +(2.5,0);
\draw (0.5,0) .. controls +(0,-1) and +(0,-1) .. +(1.5,0);
\draw (1,0) .. controls +(0,-.5) and +(0,-.5) .. +(.5,0);

\node (A) at (0,-1.15) {$\quad$};
\end{tikzpicture}
&
\begin{tikzpicture}[thick,scale=0.4]
\node (B) at (0,.1) {$\quad$};

\draw (0,0) .. controls +(0,-1) and +(0,-1) .. +(1.5,0);
\draw (.5,0) .. controls +(0,-.5) and +(0,-.5) .. +(.5,0);
\draw (2,0) .. controls +(0,-.5) and +(0,-.5) .. +(.5,0);

\node (A) at (0,-1.15) {$\quad$};
\end{tikzpicture}
\\
\cline{1-4} $\gamma$ & $\{1,2 \}$ & $\{1,2 \}$ & $\{1,2 \}$\\
\cline{1-4} $\delta$ & $\{2,3 \}$ & $\{2,3 \}$ & $\{2,3 \}$\\
\cline{1-4} $\lambda^{\scriptscriptstyle(3)}$ &
\begin{tikzpicture}[thick,scale=0.4]
\node (B) at (0,.1) {$\quad$};

\draw (0,0) .. controls +(0,-1) and +(0,-1) .. +(1.5,0);
\draw (.5,0) .. controls +(0,-.5) and +(0,-.5) .. +(.5,0);
\draw (2,0) .. controls +(0,-.5) and +(0,-.5) .. +(.5,0);

\node (A) at (0,-1.15) {$\quad$};
\end{tikzpicture}
&
\begin{tikzpicture}[thick,scale=0.4]
\node (B) at (0,.1) {$\quad$};

\draw (0,0) .. controls +(0,-.5) and +(0,-.5) .. +(.5,0);
\draw (1,0) .. controls +(0,-.5) and +(0,-.5) .. +(.5,0);
\draw (2,0) .. controls +(0,-.5) and +(0,-.5) .. +(.5,0);

\node (A) at (0,-1.15) {$\quad$};
\end{tikzpicture}
&
\begin{tikzpicture}[thick,scale=0.4]
\node (B) at (0,.1) {$\quad$};

\draw (1,0) .. controls +(0,-1) and +(0,-1) .. +(1.5,0);
\draw (0,0) .. controls +(0,-.5) and +(0,-.5) .. +(.5,0);
\draw (1.5,0) .. controls +(0,-.5) and +(0,-.5) .. +(.5,0);

\node (A) at (0,-1.15) {$\quad$};
\end{tikzpicture}\\
\cline{1-4}
\end{tabular}
\end{table}
\hfill\\
\item $\gamma \cap \delta  = \emptyset$: then all possible configurations are listed in the following table.

\end{enumerate}

\hfill\\

\begin{table}[!h]
\caption{Configurations for $\gamma \cap \delta  = \emptyset$}
\label{table_conf_2}
\begin{tabular}{|c|c|c|c|c|c|c|}
\hline $\lambda^{\scriptscriptstyle(1)}$ &
\begin{tikzpicture}[thick,scale=0.4]
\node (B) at (0,.1) {$\quad$};

\draw (0,0) .. controls +(0,-.5) and +(0,-.5) .. +(.5,0);
\draw (1,0) .. controls +(0,-.5) and +(0,-.5) .. +(.5,0);
\draw (2,0) .. controls +(0,-1) and +(0,-1) .. +(1.5,0);
\draw (2.5,0) .. controls +(0,-.5) and +(0,-.5) .. +(.5,0);

\node (A) at (0,-1.15) {$\quad$};
\end{tikzpicture}
&
\begin{tikzpicture}[thick,scale=0.4]
\draw (0,0) .. controls +(0,-.5) and +(0,-.5) .. +(.5,0);
\draw (1,0) .. controls +(0,-.5) and +(0,-.5) .. +(.5,0);
\draw (2,0) .. controls +(0,-.5) and +(0,-.5) .. +(.5,0);
\draw (3,0) .. controls +(0,-.5) and +(0,-.5) .. +(.5,0);

\node (A) at (0,-1.15) {$\quad$};
\end{tikzpicture}
&
\begin{tikzpicture}[thick,scale=0.4]
\draw (0,0) .. controls +(0,-.5) and +(0,-.5) .. +(.5,0);
\draw (1,0) .. controls +(0,-.5) and +(0,-.5) .. +(.5,0);
\draw (2,0) .. controls +(0,-.5) and +(0,-.5) .. +(.5,0);
\draw (3,0) .. controls +(0,-.5) and +(0,-.5) .. +(.5,0);

\node (A) at (0,-1.15) {$\quad$};
\end{tikzpicture}
&
\begin{tikzpicture}[thick,scale=0.4]
\draw (0,0) .. controls +(0,-1) and +(0,-1) .. +(1.5,0);
\draw (0.5,0) .. controls +(0,-.5) and +(0,-.5) .. +(.5,0);
\draw (2,0) .. controls +(0,-1) and +(0,-1) .. +(1.5,0);
\draw (2.5,0) .. controls +(0,-.5) and +(0,-.5) .. +(.5,0);

\node (A) at (0,-1.15) {$\quad$};
\end{tikzpicture}
&
\begin{tikzpicture}[thick,scale=0.4]
\draw (0,0) .. controls +(0,-1.25) and +(0,-1.25) .. +(3.5,0);
\draw (0.5,0) .. controls +(0,-1) and +(0,-1) .. +(1.5,0);
\draw (1,0) .. controls +(0,-.5) and +(0,-.5) .. +(.5,0);
\draw (2.5,0) .. controls +(0,-.5) and +(0,-.5) .. +(.5,0);

\node (A) at (0,-1.15) {$\quad$};
\end{tikzpicture}
&
\begin{tikzpicture}[thick,scale=0.4]
\draw (0,0) .. controls +(0,-1.5) and +(0,-1.5) .. +(3.5,0);
\draw (0.5,0) .. controls +(0,-1.25) and +(0,-1.25) .. +(2.5,0);
\draw (1,0) .. controls +(0,-1) and +(0,-1) .. +(1.5,0);
\draw (1.5,0) .. controls +(0,-.5) and +(0,-.5) .. +(.5,0);

\node (A) at (0,-1.15) {$\quad$};
\end{tikzpicture}

\\
\hline $\gamma$ & $\{1,2 \}$ & $\{1,2 \}$ & $\{1,4 \}$ & $\{1,2 \}$ & $\{2,3 \}$ & $\{1,2 \}$\\
\hline $\delta$ & $\{3,4 \}$ & $\{1,2 \}$ & $\{2,3 \}$ & $\{3,4 \}$ & $\{1,4 \}$ & $\{3,4 \}$\\
\hline $\lambda^{\scriptscriptstyle(3)}$ &
\begin{tikzpicture}[thick,scale=0.4]
\draw (0,0) .. controls +(0,-1) and +(0,-1) .. +(1.5,0);
\draw (.5,0) .. controls +(0,-.5) and +(0,-.5) .. +(.5,0);
\draw (2,0) .. controls +(0,-.5) and +(0,-.5) .. +(.5,0);
\draw (3,0) .. controls +(0,-.5) and +(0,-.5) .. +(.5,0);

\node (A) at (0,-1.15) {$\quad$};
\end{tikzpicture}
&
\begin{tikzpicture}[thick,scale=0.4]
\draw (0,0) .. controls +(0,-1) and +(0,-1) .. +(1.5,0);
\draw (0.5,0) .. controls +(0,-.5) and +(0,-.5) .. +(.5,0);
\draw (2,0) .. controls +(0,-1) and +(0,-1) .. +(1.5,0);
\draw (2.5,0) .. controls +(0,-.5) and +(0,-.5) .. +(.5,0);

\node (A) at (0,-1.15) {$\quad$};
\end{tikzpicture}
&
\begin{tikzpicture}[thick,scale=0.4]
\draw (0,0) .. controls +(0,-1.5) and +(0,-1.5) .. +(3.5,0);
\draw (0.5,0) .. controls +(0,-1.25) and +(0,-1.25) .. +(2.5,0);
\draw (1,0) .. controls +(0,-1) and +(0,-1) .. +(1.5,0);
\draw (1.5,0) .. controls +(0,-.5) and +(0,-.5) .. +(.5,0);

\node (A) at (0,-1.15) {$\quad$};
\end{tikzpicture}
&
\begin{tikzpicture}[thick,scale=0.4]
\draw (0,0) .. controls +(0,-.5) and +(0,-.5) .. +(.5,0);
\draw (1,0) .. controls +(0,-.5) and +(0,-.5) .. +(.5,0);
\draw (2,0) .. controls +(0,-.5) and +(0,-.5) .. +(.5,0);
\draw (3,0) .. controls +(0,-.5) and +(0,-.5) .. +(.5,0);

\node (A) at (0,-1.15) {$\quad$};
\end{tikzpicture}
&
\begin{tikzpicture}[thick,scale=0.4]
\draw (0,0) .. controls +(0,-1) and +(0,-1) .. +(2.5,0);
\draw (0.5,0) .. controls +(0,-.5) and +(0,-.5) .. +(.5,0);
\draw (1.5,0) .. controls +(0,-.5) and +(0,-.5) .. +(.5,0);
\draw (3,0) .. controls +(0,-.5) and +(0,-.5) .. +(.5,0);

\node (A) at (0,-1.15) {$\quad$};
\end{tikzpicture}
&
\begin{tikzpicture}[thick,scale=0.4]
\node (B) at (0,.1) {$\quad$};

\draw (0,0) .. controls +(0,-.5) and +(0,-.5) .. +(.5,0);
\draw (1,0) .. controls +(0,-.5) and +(0,-.5) .. +(.5,0);
\draw (2,0) .. controls +(0,-.5) and +(0,-.5) .. +(.5,0);
\draw (3,0) .. controls +(0,-.5) and +(0,-.5) .. +(.5,0);

\node (A) at (0,-1.15) {$\quad$};
\end{tikzpicture}
\\
\hline
\cline{1-4}
$\lambda^{\scriptscriptstyle(1)}$ &
\begin{tikzpicture}[thick,scale=0.4]
\node (B) at (0,.1) {$\quad$};

\draw (0,0) .. controls +(0,-.5) and +(0,-.5) .. +(.5,0);
\draw (1,0) .. controls +(0,-1) and +(0,-1) .. +(1.5,0);
\draw (1.5,0) .. controls +(0,-.5) and +(0,-.5) .. +(.5,0);
\draw (3,0) .. controls +(0,-.5) and +(0,-.5) .. +(.5,0);

\node (A) at (0,-1.15) {$\quad$};
\end{tikzpicture}
&
\begin{tikzpicture}[thick,scale=0.4]
\draw (0,0) .. controls +(0,-.5) and +(0,-.5) .. +(.5,0);
\draw (1,0) .. controls +(0,-1) and +(0,-1) .. +(2.5,0);
\draw (1.5,0) .. controls +(0,-.5) and +(0,-.5) .. +(.5,0);
\draw (2.5,0) .. controls +(0,-.5) and +(0,-.5) .. +(.5,0);

\node (A) at (0,-1.15) {$\quad$};
\end{tikzpicture}
&
\begin{tikzpicture}[thick,scale=0.4]
\draw (0,0) .. controls +(0,-1.25) and +(0,-1.25) .. +(3.5,0);
\draw (0.5,0) .. controls +(0,-1) and +(0,-1) .. +(2.5,0);
\draw (1,0) .. controls +(0,-.5) and +(0,-.5) .. +(.5,0);
\draw (2,0) .. controls +(0,-.5) and +(0,-.5) .. +(.5,0);

\node (A) at (0,-1.15) {$\quad$};
\end{tikzpicture}

\\
\cline{1-4} $\alpha$ & $\{1,4 \}$ & $\{1,2 \}$ & $\{3,4 \}$ \\
\cline{1-4} $\beta$ & $\{2,3 \}$ & $\{3,4 \}$ & $\{1,4 \}$ \\
\cline{1-4} $\lambda^{\scriptscriptstyle(3)}$ &
\begin{tikzpicture}[thick,scale=0.4]
\node (B) at (0,.1) {$\quad$};

\draw (0,0) .. controls +(0,-1.25) and +(0,-1.25) .. +(3.5,0);
\draw (0.5,0) .. controls +(0,-1) and +(0,-1) .. +(2.5,0);
\draw (1,0) .. controls +(0,-.5) and +(0,-.5) .. +(.5,0);
\draw (2,0) .. controls +(0,-.5) and +(0,-.5) .. +(.5,0);

\node (A) at (0,-1.15) {$\quad$};
\end{tikzpicture}
&
\begin{tikzpicture}[thick,scale=0.4]
\draw (0,0) .. controls +(0,-1.25) and +(0,-1.25) .. +(3.5,0);
\draw (1.5,0) .. controls +(0,-1) and +(0,-1) .. +(1.5,0);
\draw (2,0) .. controls +(0,-.5) and +(0,-.5) .. +(.5,0);
\draw (0.5,0) .. controls +(0,-.5) and +(0,-.5) .. +(.5,0);

\node (A) at (0,-1.15) {$\quad$};
\end{tikzpicture}
&
\begin{tikzpicture}[thick,scale=0.4]
\draw (0,0) .. controls +(0,-.5) and +(0,-.5) .. +(.5,0);
\draw (1,0) .. controls +(0,-1) and +(0,-1) .. +(1.5,0);
\draw (1.5,0) .. controls +(0,-.5) and +(0,-.5) .. +(.5,0);
\draw (3,0) .. controls +(0,-.5) and +(0,-.5) .. +(.5,0);

\node (A) at (0,-1.15) {$\quad$};
\end{tikzpicture}
\\
\cline{1-4}
\end{tabular}
\end{table}
\hfill\\
Then part 1.) of the proposition follows by adding all possible configurations of dots to the diamonds; while 2.) follows from the definitions and an easy case-by-case argument which is left to the reader.
\end{proof}

\subsection{The main Isomorphism Theorem}
To prove Theorem~\ref{thm:main} it is enough to establish an isomorphism of algebras $\ABr\cong \mathbb{D}_k = \mathbb{D}_{\Lambda_k^{\ov{0}}}$. 

Recalling the special elements from Definition~\ref{def:specialelements}, our Theorem \ref{thm:main} can be refined as follows:
\begin{theorem}
\label{mainthm}
Let $k\geq 4$. The map $\Phi=\Phi_k$, defined on generators by
\begin{eqnarray} \label{isops}
\Phi_k \,: \,\ABr & \longrightarrow & \mathbb{D}_k \\ 
e_\lambda & \longmapsto & {}_{\lambda} \mathbbm{1}_{\lambda}, \nonumber \\
t_{\alpha,\lambda} & \longmapsto &
\begin{cases}
{}_{\lambda} \mathbbm{1}_{\lambda} + X_{\alpha,\lambda} & \qquad \,\, \text{ if } 1 \leq \alpha \leq k, \\
{}_{\lambda} \mathbbm{1}_{\lambda} - X_{-\alpha,\lambda} & \qquad \,\, \text{ if } -k \leq \alpha \leq -1, \\
{}_{\lambda} \mathbbm{1}_{\lambda} & \qquad \,\,  \text{ otherwise,}
\end{cases}\nonumber\\
p(\lambda,\mu) & \longmapsto & {}_{\lambda} \mathbbm{1}_{\mu} + \frac{1}{2} (-1)^j {}_{\lambda} X_{\mu} \qquad  \text{ if } \lambda \stackrel{(\pm i,j)}\longleftrightarrow \mu, \nonumber
\end{eqnarray}
is a (well-defined) isomorphism of algebras, equipping Braden's algebra $\ABr$ with a non-negative $\mathbb{Z}$-grading.
\end{theorem}

\begin{proof}
The most involved part of the proof is the well-definedness of $\Phi$ which follows from Proposition~\ref{welldefined} below. By definition the image of $\Phi$ contains all idempotents ${}_{\lambda} \mathbbm{1}_{\lambda}$ and by looking at $\Phi(t_{\alpha,\la}^2)$ also all $X_{\alpha,\lambda}$. Since  ${}_{\lambda} \mathbbm{1}_{\mu}$ is  contained in the vector space span of $\Phi(p(\lambda,\mu))$ and $X_{\alpha,\lambda}\Phi(p(\lambda,\mu))$ it is in the image of $\Phi$ as well. Therefore, by Theorem~\ref{prop:generatedindegree1} the map $\Phi$ is surjective. But then it is already an isomorphism, since by Corollary~\ref{CorCartan} the two algebras have the same dimension.
\end{proof}
\begin{example}
We illustrate here Relation (R-\ref{rel7}) i) by an explicit example  in $\mathbb{D}_6$ keeping carefully track of signs. Consider the diamond
\begin{eqnarray*}
\begin{tikzpicture}[thick,scale=0.7]
\draw (2,0) .. controls +(0,-.5) and +(0,-.5) .. +(.5,0);
\draw (3,0) .. controls +(0,-.5) and +(0,-.5) .. +(.5,0);
\draw (4,0) .. controls +(0,-.5) and +(0,-.5) .. +(.5,0);
\fill (4.25,-.35) circle(2.5pt);
\node at (0.8,-.2) {$\underline{\lambda^{\scriptscriptstyle (1)}}\;=$};

\draw[<-] (2.5,-1) -- +(-1,-1);
\draw[->] (4,-1) -- +(1,-1);
\node at (1.2,-1.3) {$(2,3)$};
\node at (5.3,-1.3) {$(1,2)$};

\draw (0,-2.5) .. controls +(0,-1) and +(0,-1) .. +(1.5,0);
\draw (.5,-2.5) .. controls +(0,-.5) and +(0,-.5) .. +(.5,0);
\draw (2,-2.5) .. controls +(0,-.5) and +(0,-.5) .. +(.5,0);
\fill (2.25,-2.85) circle(2.5pt);
\node at (-1.3,-2.7) {$\underline{\lambda^{\scriptscriptstyle (2)}}\;=$};

\draw (4,-2.5) .. controls +(0,-1.25) and +(0,-1.25) .. +(2.5,0);
\fill (5.25,-3.45) circle(2.5pt);
\draw (4.5,-2.5) .. controls +(0,-1) and +(0,-1) .. +(1.5,0);
\draw (5,-2.5) .. controls +(0,-.5) and +(0,-.5) .. +(.5,0);
\node at (7.5,-2.7) {$=\;\underline{\lambda^{\scriptscriptstyle (4)}}$};

\draw[<-] (2.5,-4.7) -- +(-1,1);
\draw[<-] (4,-4.7) -- +(1,1);
\node at (1.2,-4.3) {$(1,4)$};
\node at (5.3,-4.3) {$(3,4)$};

\draw (2,-5) .. controls +(0,-1) and +(0,-1) .. +(2.5,0);
\fill (3.25,-5.75) circle(2.5pt);
\draw (2.5,-5) .. controls +(0,-.5) and +(0,-.5) .. +(.5,0);
\draw (3.5,-5) .. controls +(0,-.5) and +(0,-.5) .. +(.5,0);
\node at (1.1,-5.2) {$\underline{\lambda^{\scriptscriptstyle (3)}}\;=$};
\end{tikzpicture}
\end{eqnarray*}
We now verify that the following diamond relation holds in $\mathbb{D}_6$.
\begin{eqnarray*}
    \Phi(p(\la^{\scriptscriptstyle(3)},\la^{\scriptscriptstyle(2)}))\Phi(p(\la^{\scriptscriptstyle(2)},\la^{\scriptscriptstyle(1)}))&=&\Phi(p(\la^{\scriptscriptstyle(3)},\la^{\scriptscriptstyle(4)}))\Phi(p(\la^{\scriptscriptstyle(4)},\la^{\scriptscriptstyle(1)})).
\end{eqnarray*}
For the left hand side we calculate
\begin{eqnarray*}
\left(
\hskip-.3cm \begin{array}{c}\begin{tikzpicture}[thick,scale=0.5]
\draw (0,0) .. controls +(0,1) and +(0,1) .. +(1.5,0);
\draw (.5,0) .. controls +(0,.5) and +(0,.5) .. +(.5,0);
\draw (2,0) .. controls +(0,.5) and +(0,.5) .. +(.5,0);
\fill (2.25,.35) circle(2.5pt);
\node at (0,0) {$\scriptstyle \up$};
\node at (.5,0) {$\scriptstyle \down$};
\node at (1,0) {$\scriptstyle \up$};
\node at (1.5,0) {$\scriptstyle \down$};
\node at (2,0) {$\scriptstyle \up$};
\node at (2.5,0) {$\scriptstyle \up$};
\draw (0,0) .. controls +(0,-1) and +(0,-1) .. +(2.5,0);
\fill (1.25,-.75) circle(2.5pt);
\draw (0.5,0) .. controls +(0,-.5) and +(0,-.5) .. +(.5,0);
\draw (1.5,0) .. controls +(0,-.5) and +(0,-.5) .. +(.5,0);
\end{tikzpicture}\end{array}\hskip-.3cm
+\frac{1}{2}
\hskip-.3cm \begin{array}{c}\begin{tikzpicture}[thick,scale=0.5]
\draw (0,0) .. controls +(0,1) and +(0,1) .. +(1.5,0);
\draw (.5,0) .. controls +(0,.5) and +(0,.5) .. +(.5,0);
\draw (2,0) .. controls +(0,.5) and +(0,.5) .. +(.5,0);
\fill (2.25,.35) circle(2.5pt);
\node at (0,0) {$\scriptstyle \down$};
\node at (.5,0) {$\scriptstyle \down$};
\node at (1,0) {$\scriptstyle \up$};
\node at (1.5,0) {$\scriptstyle \up$};
\node at (2,0) {$\scriptstyle \down$};
\node at (2.5,0) {$\scriptstyle \down$};
\draw (0,0) .. controls +(0,-1) and +(0,-1) .. +(2.5,0);
\fill (1.25,-.75) circle(2.5pt);
\draw (0.5,0) .. controls +(0,-.5) and +(0,-.5) .. +(.5,0);
\draw (1.5,0) .. controls +(0,-.5) and +(0,-.5) .. +(.5,0);
\end{tikzpicture}\end{array}\hskip-.3cm
\right)\left(
\hskip-.3cm \begin{array}{c}\begin{tikzpicture}[thick,scale=0.5]
\draw (0,0) .. controls +(0,.5) and +(0,.5) .. +(.5,0);
\draw (1,0) .. controls +(0,.5) and +(0,.5) .. +(.5,0);
\draw (2,0) .. controls +(0,.5) and +(0,.5) .. +(.5,0);
\fill (2.25,.35) circle(2.5pt);
\node at (0,0) {$\scriptstyle \down$};
\node at (.5,0) {$\scriptstyle \up$};
\node at (1,0) {$\scriptstyle \down$};
\node at (1.5,0) {$\scriptstyle \up$};
\node at (2,0) {$\scriptstyle \up$};
\node at (2.5,0) {$\scriptstyle \up$};
\draw (0,0) .. controls +(0,-1) and +(0,-1) .. +(1.5,0);
\draw (.5,0) .. controls +(0,-.5) and +(0,-.5) .. +(.5,0);
\draw (2,0) .. controls +(0,-.5) and +(0,-.5) .. +(.5,0);
\fill (2.25,-.35) circle(2.5pt);
\end{tikzpicture}\end{array}\hskip-.3cm
-\frac{1}{2}
\hskip-.3cm \begin{array}{c}\begin{tikzpicture}[thick,scale=0.5]
\draw (0,0) .. controls +(0,.5) and +(0,.5) .. +(.5,0);
\draw (1,0) .. controls +(0,.5) and +(0,.5) .. +(.5,0);
\draw (2,0) .. controls +(0,.5) and +(0,.5) .. +(.5,0);
\fill (2.25,.35) circle(2.5pt);
\node at (0,0) {$\scriptstyle \up$};
\node at (.5,0) {$\scriptstyle \down$};
\node at (1,0) {$\scriptstyle \up$};
\node at (1.5,0) {$\scriptstyle \down$};
\node at (2,0) {$\scriptstyle \down$};
\node at (2.5,0) {$\scriptstyle \down$};
\draw (0,0) .. controls +(0,-1) and +(0,-1) .. +(1.5,0);
\draw (.5,0) .. controls +(0,-.5) and +(0,-.5) .. +(.5,0);
\draw (2,0) .. controls +(0,-.5) and +(0,-.5) .. +(.5,0);
\fill (2.25,-.35) circle(2.5pt);
\end{tikzpicture}\end{array}\hskip-.3cm
\right)=
\hskip-.3cm \begin{array}{c}\begin{tikzpicture}[thick,scale=0.5]
\draw (0,0) .. controls +(0,.5) and +(0,.5) .. +(.5,0);
\draw (1,0) .. controls +(0,.5) and +(0,.5) .. +(.5,0);
\draw (2,0) .. controls +(0,.5) and +(0,.5) .. +(.5,0);
\fill (2.25,.35) circle(2.5pt);
\node at (0,0) {$\scriptstyle \up$};
\node at (.5,0) {$\scriptstyle \down$};
\node at (1,0) {$\scriptstyle \up$};
\node at (1.5,0) {$\scriptstyle \down$};
\node at (2,0) {$\scriptstyle \up$};
\node at (2.5,0) {$\scriptstyle \up$};
\draw (0,0) .. controls +(0,-1) and +(0,-1) .. +(2.5,0);
\fill (1.25,-.75) circle(2.5pt);
\draw (0.5,0) .. controls +(0,-.5) and +(0,-.5) .. +(.5,0);
\draw (1.5,0) .. controls +(0,-.5) and +(0,-.5) .. +(.5,0);
\end{tikzpicture}\end{array}\hskip-.3cm
+
\hskip-.3cm \begin{array}{c}\begin{tikzpicture}[thick,scale=0.5]
\draw (0,0) .. controls +(0,.5) and +(0,.5) .. +(.5,0);
\draw (1,0) .. controls +(0,.5) and +(0,.5) .. +(.5,0);
\draw (2,0) .. controls +(0,.5) and +(0,.5) .. +(.5,0);
\fill (2.25,.35) circle(2.5pt);
\node at (0,0) {$\scriptstyle \down$};
\node at (.5,0) {$\scriptstyle \up$};
\node at (1,0) {$\scriptstyle \down$};
\node at (1.5,0) {$\scriptstyle \up$};
\node at (2,0) {$\scriptstyle \down$};
\node at (2.5,0) {$\scriptstyle \down$};
\draw (0,0) .. controls +(0,-1) and +(0,-1) .. +(2.5,0);
\fill (1.25,-.75) circle(2.5pt);
\draw (0.5,0) .. controls +(0,-.5) and +(0,-.5) .. +(.5,0);
\draw (1.5,0) .. controls +(0,-.5) and +(0,-.5) .. +(.5,0);
\end{tikzpicture}\end{array}\hskip-.3cm
\end{eqnarray*}
\noindent where the signs are determined by the second component in the $\la$-pairs. For the right hand side we calculate
\begin{eqnarray*}
\left(
\hskip-.3cm \begin{array}{c}\begin{tikzpicture}[thick,scale=0.5]
\draw (0,0) .. controls +(0,1.25) and +(0,1.25) .. +(2.5,0);
\fill (1.25,.95) circle(2.5pt);
\draw (0.5,0) .. controls +(0,1) and +(0,1) .. +(1.5,0);
\draw (1,0) .. controls +(0,.5) and +(0,.5) .. +(.5,0);
\node at (0,0) {$\scriptstyle \up$};
\node at (.5,0) {$\scriptstyle \down$};
\node at (1,0) {$\scriptstyle \up$};
\node at (1.5,0) {$\scriptstyle \down$};
\node at (2,0) {$\scriptstyle \up$};
\node at (2.5,0) {$\scriptstyle \up$};
\draw (0,0) .. controls +(0,-1) and +(0,-1) .. +(2.5,0);
\fill (1.25,-.75) circle(2.5pt);
\draw (0.5,0) .. controls +(0,-.5) and +(0,-.5) .. +(.5,0);
\draw (1.5,0) .. controls +(0,-.5) and +(0,-.5) .. +(.5,0);
\end{tikzpicture}\end{array}\hskip-.3cm
+\frac{1}{2}
\hskip-.3cm \begin{array}{c}\begin{tikzpicture}[thick,scale=0.5]
\draw (0,0) .. controls +(0,1.25) and +(0,1.25) .. +(2.5,0);
\fill (1.25,.95) circle(2.5pt);
\draw (0.5,0) .. controls +(0,1) and +(0,1) .. +(1.5,0);
\draw (1,0) .. controls +(0,.5) and +(0,.5) .. +(.5,0);
\node at (0,0) {$\scriptstyle \up$};
\node at (.5,0) {$\scriptstyle \up$};
\node at (1,0) {$\scriptstyle \down$};
\node at (1.5,0) {$\scriptstyle \up$};
\node at (2,0) {$\scriptstyle \down$};
\node at (2.5,0) {$\scriptstyle \up$};
\draw (0,0) .. controls +(0,-1) and +(0,-1) .. +(2.5,0);
\fill (1.25,-.75) circle(2.5pt);
\draw (0.5,0) .. controls +(0,-.5) and +(0,-.5) .. +(.5,0);
\draw (1.5,0) .. controls +(0,-.5) and +(0,-.5) .. +(.5,0);
\end{tikzpicture}\end{array}\hskip-.3cm
\right)\left(
\hskip-.3cm \begin{array}{c}\begin{tikzpicture}[thick,scale=0.5]
\draw (0,0) .. controls +(0,.5) and +(0,.5) .. +(.5,0);
\draw (1,0) .. controls +(0,.5) and +(0,.5) .. +(.5,0);
\draw (2,0) .. controls +(0,.5) and +(0,.5) .. +(.5,0);
\fill (2.25,.35) circle(2.5pt);
\node at (0,0) {$\scriptstyle \up$};
\node at (.5,0) {$\scriptstyle \down$};
\node at (1,0) {$\scriptstyle \down$};
\node at (1.5,0) {$\scriptstyle \up$};
\node at (2,0) {$\scriptstyle \up$};
\node at (2.5,0) {$\scriptstyle \up$};
\draw (0,0) .. controls +(0,-1.25) and +(0,-1.25) .. +(2.5,0);
\fill (1.25,-.95) circle(2.5pt);
\draw (0.5,0) .. controls +(0,-1) and +(0,-1) .. +(1.5,0);
\draw (1,0) .. controls +(0,-.5) and +(0,-.5) .. +(.5,0);
\end{tikzpicture}\end{array}\hskip-.3cm
+\frac{1}{2}
\hskip-.3cm \begin{array}{c}\begin{tikzpicture}[thick,scale=0.5]
\draw (0,0) .. controls +(0,.5) and +(0,.5) .. +(.5,0);
\draw (1,0) .. controls +(0,.5) and +(0,.5) .. +(.5,0);
\draw (2,0) .. controls +(0,.5) and +(0,.5) .. +(.5,0);
\fill (2.25,.35) circle(2.5pt);
\node at (0,0) {$\scriptstyle \down$};
\node at (.5,0) {$\scriptstyle \up$};
\node at (1,0) {$\scriptstyle \down$};
\node at (1.5,0) {$\scriptstyle \up$};
\node at (2,0) {$\scriptstyle \down$};
\node at (2.5,0) {$\scriptstyle \down$};
\draw (0,0) .. controls +(0,-1.25) and +(0,-1.25) .. +(2.5,0);
\fill (1.25,-.95) circle(2.5pt);
\draw (0.5,0) .. controls +(0,-1) and +(0,-1) .. +(1.5,0);
\draw (1,0) .. controls +(0,-.5) and +(0,-.5) .. +(.5,0);
\end{tikzpicture}\end{array}\hskip-.3cm
\right)=
\hskip-.3cm \begin{array}{c}\begin{tikzpicture}[thick,scale=0.5]
\draw (0,0) .. controls +(0,.5) and +(0,.5) .. +(.5,0);
\draw (1,0) .. controls +(0,.5) and +(0,.5) .. +(.5,0);
\draw (2,0) .. controls +(0,.5) and +(0,.5) .. +(.5,0);
\fill (2.25,.35) circle(2.5pt);
\node at (0,0) {$\scriptstyle \up$};
\node at (.5,0) {$\scriptstyle \down$};
\node at (1,0) {$\scriptstyle \up$};
\node at (1.5,0) {$\scriptstyle \down$};
\node at (2,0) {$\scriptstyle \up$};
\node at (2.5,0) {$\scriptstyle \up$};
\draw (0,0) .. controls +(0,-1) and +(0,-1) .. +(2.5,0);
\fill (1.25,-.75) circle(2.5pt);
\draw (0.5,0) .. controls +(0,-.5) and +(0,-.5) .. +(.5,0);
\draw (1.5,0) .. controls +(0,-.5) and +(0,-.5) .. +(.5,0);
\end{tikzpicture}\end{array}\hskip-.3cm
+
\hskip-.3cm \begin{array}{c}\begin{tikzpicture}[thick,scale=0.5]
\draw (0,0) .. controls +(0,.5) and +(0,.5) .. +(.5,0);
\draw (1,0) .. controls +(0,.5) and +(0,.5) .. +(.5,0);
\draw (2,0) .. controls +(0,.5) and +(0,.5) .. +(.5,0);
\fill (2.25,.35) circle(2.5pt);
\node at (0,0) {$\scriptstyle \down$};
\node at (.5,0) {$\scriptstyle \up$};
\node at (1,0) {$\scriptstyle \down$};
\node at (1.5,0) {$\scriptstyle \up$};
\node at (2,0) {$\scriptstyle \down$};
\node at (2.5,0) {$\scriptstyle \down$};
\draw (0,0) .. controls +(0,-1) and +(0,-1) .. +(2.5,0);
\fill (1.25,-.75) circle(2.5pt);
\draw (0.5,0) .. controls +(0,-.5) and +(0,-.5) .. +(.5,0);
\draw (1.5,0) .. controls +(0,-.5) and +(0,-.5) .. +(.5,0);
\end{tikzpicture}\end{array}\hskip-.3cm
\end{eqnarray*}
where now all signs are positive. 
Note for instance that the different signs in the second factor for both sides are taken care of by the different signs in the surgeries occurring on both sides.
\end{example}

\subsection{Well-definedness of the map $\Phi$}

The final step is to show that $\Phi$ is well-defined.
\begin{prop}
\label{welldefined}
The map $\Phi$ from Theorem \ref{mainthm} is well-defined.
\end{prop}
\begin{proof} We check that the map $\Phi$ respects all the relations of $\ABr$.\\ \noindent
$\blacktriangleright$ ({\text R}-1): Clearly the ${}_{\lambda} \mathbbm{1}_{\lambda}$'s form a set of pairwise orthogonal idempotents with $\sum_\lambda {}_{\lambda} \mathbbm{1}_{\lambda} = 1$ by Theorem~\ref{algebra_structure}. Hence part a.) and b.) in (\text{R}-1) hold. The others from (\text{R}-1) hold by definition.\\ \noindent
$\blacktriangleright$ ({\text R}-2): By definition $\Phi(t_{\alpha,\la}) = {}_{\lambda} \mathbbm{1}_{\lambda}$ if $| \alpha | > k$. That $\Phi(t_{\alpha,\la}t_{-\alpha,\la})= {}_{\lambda} \mathbbm{1}_{\lambda}$ follows immediately since  $\Phi(t_{\alpha,\la}t_{-\alpha,\la})={}_{\lambda} \mathbbm{1}_{\lambda}-X_{\alpha,\lambda}^2$ and  $X_{\alpha,\lambda}^2=0$ by definition. Hence parts a.) and c.) hold. We have (with the appropriate sign choice)
\begin{eqnarray*}
\Phi(t_{\alpha,\lambda} t_{\beta,\mu}) &=& \left({}_{\lambda} \mathbbm{1}_{\lambda} \pm X_{\alpha,\lambda}\right)
\left({}_{\mu} \mathbbm{1}_{\mu}\pm X_{\beta,\mu}\right)\\
&=& \begin{cases}
\quad 0 & \text{if } \lambda \neq \mu\\
{}_{\lambda} \mathbbm{1}_{\lambda} \pm\alpha X_{\alpha,\lambda} \pm X_{\beta,\lambda} +
(-1)^{\alpha+\beta}X_{\alpha,\lambda}X_{\beta,\lambda} & \text{otherwise.\quad}
\end{cases}
\end{eqnarray*}
and similarly for $\Phi(t_{\beta,\lambda} t_{\alpha,\mu})$ with analogous signs. (We do not care which signs we need, but know from Section~\ref{sec:explicitmult} that we get the same in both cases.) Hence to verify b.) it is enough to show $X_{\alpha,\lambda}X_{\beta,\lambda} = X_{\beta,\lambda}X_{\alpha,\lambda}.$
But this is obviously true by Lemma~\ref{lem:endo_commutative}, thus $\Phi(t_{\alpha,\lambda} t_{\beta,\mu}) = \Phi(t_{\beta,\mu} t_{\alpha,\lambda})$.

Assume now that $(\alpha,\beta)$ is a $\la$-pair. We start with the first case in \eqref{lapair2}. Then the vertices $\alpha$ and $\beta$ are connected by an undotted cup in $\underline\la\ov{\la}$ and we have by Proposition~\ref{coho} $X_{\alpha,\lambda} =-X_{\beta,\lambda}$ and thus
$$\Phi(t_{\alpha,\lambda} t_{\beta,\lambda})={}_{\lambda} \mathbbm{1}_{\lambda}-X_{\alpha,\lambda}X_{\alpha,\lambda}={}_{\lambda} \mathbbm{1}_{\lambda}.$$
If instead we are in the second case of \eqref{lapair2} then the vertices $-\alpha$ and $\beta$ are connected by a dotted cup and we obtain 
$$\Phi(t_{-\alpha,\lambda} t_{\beta,\lambda}) = \left({}_{\lambda} \mathbbm{1}_{\lambda} - X_{\alpha,\lambda}\right)\left({}_{\lambda} \mathbbm{1}_{\lambda} + X_{\beta,\lambda}\right).
$$
But since $\alpha$ and $\beta$ are connected by a dotted cup in $\underline\la\ov{\lambda}$ we have by Proposition~\ref{coho} $X_{\alpha,\lambda} =X_{\beta,\lambda}$ and thus again
$$\Phi(t_{\alpha,\lambda} t_{\beta,\lambda})={}_{\lambda} \mathbbm{1}_{\lambda}-X_{\alpha,\lambda}X_{\alpha,\lambda}={}_{\lambda} \mathbbm{1}_{\lambda}.$$
and so the last equality of (R-2) holds.\\ \noindent
$\blacktriangleright$ ({\text R}-3):
For $\lambda \stackrel{(\pm i,j)}\longleftrightarrow \mu$ and using the definition of $\Phi$ from Theorem \ref{mainthm}, we obtain
$$ \Phi(p(\lambda,\mu)t_{\alpha,\mu}) = {}_{\lambda} \mathbbm{1}_{\mu} \Phi(t_{\alpha,\mu}) + \frac{1}{2}(-1)^j{}_{\lambda} X_{\mu}
\Phi(t_{\alpha,\mu})$$
and
$$ \Phi(t_{\alpha,\lambda}p(\lambda,\mu)) = \Phi(t_{\alpha,\lambda}){}_{\lambda} \mathbbm{1}_{\mu} +
\frac{1}{2}(-1)^j\Phi(t_{\alpha,\lambda}) {}_{\lambda} X_{\mu}.$$
Since ${}_{\lambda} \mathbbm{1}_{\mu}$ and ${}_{\lambda} X_{\mu}$ are basis vectors whose underlying cup diagram looks like ${}_{\mu} \mathbbm{1}_{\mu}$
except for a local change given by the $\lambda$-pair, the circle diagram $\underline{\la}\ov{\mu}$ contains only vertical lines obtained by gluing two rays, small circles and one extra component $C$ containing the defining cup for $\la\stackrel{(\alpha,\beta)}\leftrightarrow\mu$. Assume first that $C$ is a circle. Then it involves 4 vertices, let us assume on positions $a_1<a_2<a_3<a_4$. We now compare multiplication of ${}_{\lambda} \mathbbm{1}_{\mu}$ and ${}_{\lambda}
    X_{\mu}$ by the image of $t_{\alpha,\lambda}$ from the left or $t_{\alpha,\mu}$ from the right. There are two cases:
\begin{enumerate}[i)]
\item
$\alpha \not\in \{a_1,a_2,a_3,a_4\}$: then both multiplications change the
    orientation of one small anticlockwise circle and multiply with the same overall sign or they both annihilate the diagram.
\item
$\alpha \in \{a_1,a_2,a_3,a_4\}$: then both multiplications change the orientation of $C$ to clockwise with the same overall sign or both annihilate  the diagram in case $C$ was already clockwise.
\end{enumerate}
If $C$ is not a circle, both multiplications annihilate the diagram. Hence  ({\text R}-3) holds.\\ \noindent
$\blacktriangleright$ ({\text R}-4):
Consider first the cases where the relevant changes from $\la$ to $\mu$ involve no rays. The possible ${}_{\lambda} \mathbbm{1}_{\mu}{}_{\lambda}, X_{\mu}$ and  ${}_{\mu} \mathbbm{1}_{\lambda}$, ${}_{\mu} X_{\lambda}$ are displayed in Example~\ref{ps} when putting anticlockwise respectively clockwise orientation and possibly take the mirror image of the shape.
We go now through all cases $\mu\leftarrow\la$ as displayed in Example~\ref{ps} and describe the maps algebraically, see \eqref{parent1}-\eqref{parent4},  via the identification from Proposition~\ref{coho}. Note that $\pos(i)=i$ holds since we are in the principal block. We start with the case

\begin{equation}
\label{parent1}
\begin{tikzpicture}[thick,scale=0.7]
\draw (0,0) node[above]{$$} .. controls +(0,-.5) and +(0,-.5) ..
+(.5,0) node[above]{};
\draw (1,0) node[above]{$$} .. controls +(0,-.5) and +(0,-.5) ..
+(.5,0) node[above]{};
\draw[<-] (1.75,-.2) -- +(1,0);
\draw (3,0) node[above]{$\alpha'$} .. controls +(0,-1) and +(0,-1) .. +(1.5,0)
node[above]{$\beta'$};
\draw (3.5,0) node[above]{$\alpha$} .. controls +(0,-.5) and +(0,-.5) ..
+(.5,0) node[above]{$\beta$};
\end{tikzpicture}
\end{equation}
Then the image of $\op{m}(\mu,\la)=1+p(\mu,\la)p(\la,\mu)$ equals
$$\left({}_{\mu} \mathbbm{1}_{\lambda}+\frac{1}{2}(-1)^\beta {}_{\mu} X_{\la}\right)
\left({}_{\mu} \mathbbm{1}_{\mu}+\frac{1}{2}(-1)^\beta {}_{\la} X_{\mu}\right)$$ which equals, using the surgery rules, the following expression in $\cM(\underline{\la}\ov{\la})$:
\begin{eqnarray*}
&&1+(-1)^\alpha(1+(-1)^\beta X_{\beta'})(X_{\beta}-X_\alpha)\\
&=&1+(-1)^\alpha(X_{\beta}-X_\alpha)-(-1)^{\alpha+\beta}X_\alpha X_{\beta'}\\
&=&1-(-1)^\alpha(X_{\beta'}+X_\alpha)+X_\alpha X_{\beta'}\\
&=&1+(-1)^\beta(X_{\beta'}+X_\alpha)+X_\alpha X_{\beta'}
\end{eqnarray*}
since $\alpha+\beta$ is odd.
On the other hand $\Phi\left((t_{\alpha,\lambda} t_{\beta',\lambda})^{({-1})^\beta}\right)$
corresponds to $(1+(-1)^\beta X_\alpha)(1+(-1)^\beta X_{\beta'})=1+(-1)^\beta(X_{\beta'}+X_\alpha)+X_\alpha X_{\beta'}.$ Hence the required relation for $\op{m}(\mu,\la)$ holds.\\
Similar calculations give $$\Phi(e_\la+p(\la,\mu)p(\mu,\la))=\Phi\left(\left(t_{\alpha,\la} t_{\beta',\la}\right)^{({-1})^\beta}\right)$$ corresponding to
\begin{eqnarray*}
1+(-1)^\beta(X_{\beta'}+X_\alpha)+X_\alpha X_{\beta'}
&=&1+(-1)^\beta(X_{\beta'}-X_{\beta})-X_\beta X_{\beta'}
\end{eqnarray*}
and so the claim follows in this case as well.
For the remaining cases we just list the corresponding polynomials in $\cM(\underline{\mu}\ov{\mu})$ and $\cM(\underline{\la}\ov{\la})$ respectively, since the calculations are completely analogous:

\begin{equation}
\label{parent2}
\begin{tikzpicture}[thick,scale=0.7]
\draw (0,0) node[above]{} .. controls +(0,-.5) and +(0,-.5) .. +(.5,0)
node[above]{};
\fill (.25,-.36) circle(2.5pt);
\draw (1,0) node[above]{$$} .. controls +(0,-.5) and +(0,-.5) ..
+(.5,0) node[above]{$$};
\draw[<-] (1.75,-.2) -- +(1,0);
\draw (3,0) node[above]{$\alpha'$} .. controls +(0,-1) and +(0,-1) .. +(1.5,0)
node[above]{$\beta'$};
\fill (3.75,-.74) circle(2.5pt);
\draw (3.5,0) node[above]{$\alpha$} .. controls +(0,-.5) and +(0,-.5) ..
+(.5,0) node[above]{$\beta$};
\node at (11,0)
{$\begin{cases}
\; 1+(-1)^\beta(X_{\beta'}-X_\beta)-X_{\beta'}X_\beta\in\cM(\underline{\mu}\ov{\mu})\\
\; 1+(-1)^\beta(X_\alpha+X_{\beta'}) +X_\alpha X_{\beta'}\in\cM(\underline{\la}\ov{\la})
\end{cases}$};
\end{tikzpicture}
\end{equation}

\begin{equation}
\label{parent3}
\begin{tikzpicture}[thick,scale=0.7]
\draw (0,0) node[above]{$$} .. controls +(0,-1) and +(0,-1) .. +(1.5,0)
node[above]{};
\draw (.5,0) node[above]{$$} .. controls +(0,-.5) and +(0,-.5) ..
+(.5,0) node[above]{};
\fill (0.75,-.74) circle(2.5pt);
\draw[<-] (1.75,-.2) -- +(1,0);
\draw (3,0) node[above]{$\alpha$} .. controls +(0,-.5) and +(0,-.5) ..
+(.5,0) node[above]{$\beta$};
\draw (4,0) node[above]{$\alpha'$} .. controls +(0,-.5) and +(0,-.5) .. +(.5,0)
node[above]{$\beta'$};
\fill (4.25,-.36) circle(2.5pt);
\node at (11,0)
{$\begin{cases}
\;1+(-1)^\beta(X_{\beta'}+X_{\alpha'})+X_{\beta'}X_{\alpha'}\in\cM(\underline{\mu}\ov{\mu})\\
\;1+(-1)^\beta (X_{\beta'}-X_\beta)-X_\beta X_{\beta'}\in\cM(\underline{\la}\ov{\la})
\end{cases}$};
\end{tikzpicture}
\end{equation}

\begin{eqnarray}
\label{parent4}
\begin{tikzpicture}[thick,scale=0.7]
\draw (0,0) node[above]{$$} .. controls +(0,-1) and +(0,-1) .. +(1.5,0)
node[above]{};
\draw (.5,0) node[above]{$$} .. controls +(0,-.5) and +(0,-.5) ..
+(.5,0) node[above]{};
\draw[<-] (1.75,-.2) -- +(1,0);
\draw (3,0) node[above]{$\alpha$} .. controls +(0,-.5) and +(0,-.5) .. +(.5,0)
node[above]{$\beta$};
\fill (3.25,-.36) circle(2.5pt);
\draw (4,0) node[above]{$\alpha'$} .. controls +(0,-.5) and +(0,-.5) .. +(.5,0)
node[above]{$\beta'$};
\fill (4.25,-.36) circle(2.5pt);
\node at (11,0)
{$\begin{cases}
1+(-1)^\beta(X_{\alpha}+X_{\beta'})+X_{\alpha'}X_{\beta'}\in\cM(\un\mu\ov\mu)\\
1+(-1)^\beta(X_{\alpha'}-X_\beta)-X_{\beta}X_{\beta'}\in\cM(\un\la\ov\la)
\end{cases}$};
\end{tikzpicture}
\\ \nonumber
\end{eqnarray}

The relations for $\la$-pairs involving also rays as the relevant pieces are the same, except that the corresponding variables $X$ should be set to zero and so some of the generators of type $t_\gamma$ act by an idempotent, see Section~\ref{annoying}.\\ \noindent
$\blacktriangleright$ ({\text R}-5): These relations follow from Lemmas~\ref{stupidcalcc}-\ref{stupidcalca} below.
\noindent\\
Hence the assignment \eqref{isops} defines an algebra homomorphism.
\end{proof}

\begin{lemma}
\label{stupidcalcc}
Let $k\geq 4$ and assume that $(\la^{\scriptscriptstyle(1)},\la^{\scriptscriptstyle(2)},\la^{\scriptscriptstyle(3)})$ is a triple as in Definition~\ref{Bradenalgebra} (R-\ref{rel7}) iii). With $\Phi$ as in \eqref{isops} we have
  \begin{eqnarray*}
    &\Phi(p(\la^{\scriptscriptstyle(3)},\la^{\scriptscriptstyle(2)}))\Phi(p(\la^{\scriptscriptstyle(2)},\la^{\scriptscriptstyle(1)}))=0=\Phi(p(\la^{\scriptscriptstyle(1)},\la^{\scriptscriptstyle(2)}))\Phi(p(\la^{\scriptscriptstyle(2)},\la^{\scriptscriptstyle(3)})).&
  \end{eqnarray*}
\end{lemma}

\begin{proof}
Let us first assume that the corresponding cup diagrams $\underline{\la^{(i)}}$ have no rays. The restriction $\alpha<0$ forces $\la^{\scriptscriptstyle(2)}\leftarrow \la^{\scriptscriptstyle(1)}$ to be locally of the form \eqref{parent4}. Let $C_1, C_2$ be the two cups (the outer and inner respectively) in the left picture of \eqref{parent4}. Since $(\gamma,\delta)$ is a $\la^{\scriptscriptstyle(2)}$-pair, but not a $\la^{\scriptscriptstyle(1)}$-pair, it must be given by one of the cups $C_1$ or $C_2$. The triple $(\la^{\scriptscriptstyle(1)},\la^{\scriptscriptstyle(2)},\la^{\scriptscriptstyle(3)})$ look then locally like $(\la^{\scriptscriptstyle(1)},\la^{\scriptscriptstyle(2)},\la^{\scriptscriptstyle(3)})$ respectively $(\mu^{\scriptscriptstyle(1)},\mu^{\scriptscriptstyle(2)},\mu^{\scriptscriptstyle(3)})$ in Example~\ref{triples}. Hence in the case of $C_1$ the compositions in the lemma are obviously zero, whereas in case of $C_2$ the triple is not of the required form. Obviously, (by adding $\up$'s to the left and $\down$'s to the right of the diagram), the general case can be deduced from the case where no rays occur.
\end{proof}

\begin{lemma}
\label{stupidcalb}
Let $k\geq 4$ and assume that $(\la^{\scriptscriptstyle(1)},\la^{\scriptscriptstyle(2)},\la^{\scriptscriptstyle(3)})$ is a triple as in Definition~\ref{Bradenalgebra} (\text{R}-\ref{rel7}) ii). With $\Phi$ as in \eqref{isops} we have
  \begin{eqnarray*}
    \Phi(p(\la^{\scriptscriptstyle(3)},\la^{\scriptscriptstyle(2)}))\Phi(p(\la^{\scriptscriptstyle(2)},\la^{\scriptscriptstyle(1)}))&=0=&\Phi(p(\la^{\scriptscriptstyle(1)},\la^{\scriptscriptstyle(2)}))\Phi(p(\la^{\scriptscriptstyle(2)},\la^{\scriptscriptstyle(3)})).
  \end{eqnarray*}
\end{lemma}

\begin{proof}
By assumption the triple can not be extended to a diamond, but can be enlarged to a diamond $(\tilde{\la}^{\scriptscriptstyle(1)},\tilde{\la}^{\scriptscriptstyle(2)},\tilde{\la}^{\scriptscriptstyle(3)},\tilde{\la}^{\scriptscriptstyle(4)})$. In particular, $\tilde{\la}^{\scriptscriptstyle(4)}$ must contains at least one $\down$ at position larger then $k$ or at least one $\up$ at position smaller than $1$ and only one of these two cases can occur.
Assume we have such an $\up$, then there is only one, since assuming there are two or more we have to find
\begin{eqnarray}
\label{leftright}
\tilde{\la}^{\scriptscriptstyle(1)}\leftrightarrow \tilde{\la}^{\scriptscriptstyle(4)}\leftrightarrow \tilde{\la}^{\scriptscriptstyle(3)}
\end{eqnarray}
with both $\leftrightarrow$ converting these $\up$'s into $\down$'s. But at most two $\up$'s can be switched and there is only one way to do so, which implies $\tilde{\la}^{\scriptscriptstyle(1)}=\tilde{\la}^{\scriptscriptstyle(3)}$, a contradiction.
Hence we need
two $\leftrightarrow$'s as in \eqref{leftright} which either remove our special (hence leftmost) $\up$ or at least move it to the right. Without loss of generality we may assume there are no rays. If we take one of the form \eqref{parent1}, then there is no second possibility. Hence they must be of the form \eqref{parent3} or \eqref{parent4} respectively. Then locally the enlarged diamond has to be of the form displayed on the left
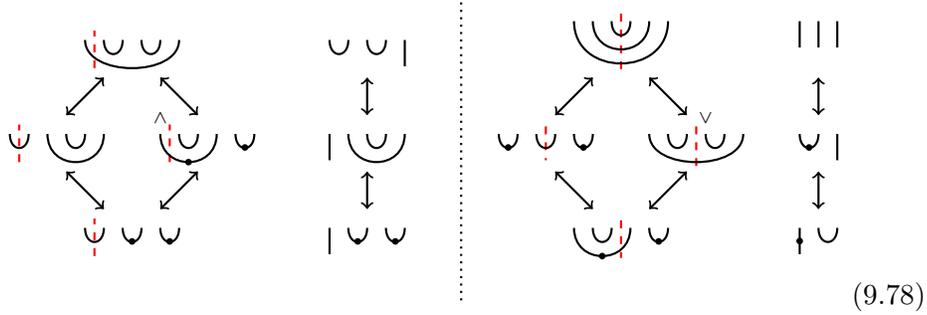
\begin{eqnarray}
\label{nodiamonds}
\begin{tikzpicture}[thick,,scale=.5]
\draw[<->] (2.5,-1) -- +(-1,-1);
\draw[<->] (4,-1) -- +(1,-1);
\draw[<->] (2.5,-4.5) -- +(-1,1);
\draw[<->] (4,-4.5) -- +(1,1);

\begin{scope}[shift={(2,0)}]
\draw [dashed,red] (.25,-.75) -- +(0,1);
\draw (0,0) .. controls +(0,-1) and +(0,-1) .. +(2.5,0);
\draw (.5,0) .. controls +(0,-.5) and +(0,-.5) .. +(.5,0);
\draw (1.5,0) .. controls +(0,-.5) and +(0,-.5) .. +(.5,0);
\end{scope}

\begin{scope}[shift={(0,-2.5)}]
\draw [dashed,red] (.25,-.75) -- +(0,1);
\draw (1,0) .. controls +(0,-1) and +(0,-1) .. +(1.5,0);
\draw (0,0) .. controls +(0,-.5) and +(0,-.5) .. +(.5,0);
\draw (1.5,0) .. controls +(0,-.5) and +(0,-.5) .. +(.5,0);
\end{scope}

\begin{scope}[shift={(2,-5)}]
\draw [dashed,red] (.25,-.75) -- +(0,1);
\draw (0,0) .. controls +(0,-.5) and +(0,-.5) .. +(.5,0);
\draw (1,0) .. controls +(0,-.5) and +(0,-.5) .. +(.5,0);
\draw (2,0) .. controls +(0,-.5) and +(0,-.5) .. +(.5,0);
\fill (1.25,-.36) circle(2.5pt);
\fill (2.25,-.36) circle(2.5pt);
\end{scope}

\begin{scope}[shift={(4,-2.5)}]
\draw [dashed,red] (.25,-.75) -- +(0,1);
\draw (0,0) node[above]{\tiny $\wedge$} .. controls +(0,-1) and +(0,-1) .. +(1.5,0);
\draw (.5,0) .. controls +(0,-.5) and +(0,-.5) .. +(.5,0);
\draw (2,0) .. controls +(0,-.5) and +(0,-.5) .. +(.5,0);
\fill (.75,-.75) circle(2.5pt);
\fill (2.25,-.36) circle(2.5pt);
\end{scope}

\begin{scope}[xshift=8cm]
\draw[<->] (1.5,-1) -- +(0,-1);
\draw[<->] (1.5,-4.5) -- +(0,1);

\begin{scope}[shift={(0,0)}]
\draw (.5,0) .. controls +(0,-.5) and +(0,-.5) .. +(.5,0);
\draw (1.5,0) .. controls +(0,-.5) and +(0,-.5) .. +(.5,0);
\draw (2.5,0) -- +(0,-.7);
\end{scope}

\begin{scope}[shift={(0,-2.5)}]
\draw (.5,0) -- +(0,-.7);
\draw (1,0) .. controls +(0,-1) and +(0,-1) .. +(1.5,0);
\draw (1.5,0) .. controls +(0,-.5) and +(0,-.5) .. +(.5,0);
\end{scope}

\begin{scope}[shift={(0,-5)}]
\draw (.5,0) -- +(0,-.7);
\draw (1,0) .. controls +(0,-.5) and +(0,-.5) .. +(.5,0);
\draw (2,0) .. controls +(0,-.5) and +(0,-.5) .. +(.5,0);
\fill (1.25,-.36) circle(2.5pt);
\fill (2.25,-.36) circle(2.5pt);
\end{scope}
\end{scope}

\draw[thick, dotted] (12,1) -- +(0,-8);

\begin{scope}[xshift=13cm]
\draw[<->] (2.5,-1) -- +(-1,-1);
\draw[<->] (4,-1) -- +(1,-1);
\draw[<->] (2.5,-4.5) -- +(-1,1);
\draw[<->] (4,-4.5) -- +(1,1);

\begin{scope}[shift={(2,.5)}]
\draw [dashed,red] (1.25,.2) -- +(0,-1.5);
\draw (0,0) .. controls +(0,-1.5) and +(0,-1.5) .. +(2.5,0);
\draw (.5,0) .. controls +(0,-1) and +(0,-1) .. +(1.5,0);
\draw (1,0) .. controls +(0,-.5) and +(0,-.5) .. +(.5,0);
\end{scope}

\begin{scope}[shift={(0,-2.5)}]
\draw [dashed,red] (1.25,.2) -- +(0,-.9);
\draw (0,0) .. controls +(0,-.5) and +(0,-.5) .. +(.5,0);
\draw (1,0) .. controls +(0,-.5) and +(0,-.5) .. +(.5,0);
\draw (2,0) .. controls +(0,-.5) and +(0,-.5) .. +(.5,0);
\fill (.25,-.36) circle(2.5pt);
\fill (2.25,-.36) circle(2.5pt);
\end{scope}

\begin{scope}[shift={(4,-2.5)}]
\draw [dashed,red] (1.25,.2) -- +(0,-1.1);
\draw (0,0) .. controls +(0,-1) and +(0,-1) .. +(2.5,0);
\draw (.5,0) .. controls +(0,-.5) and +(0,-.5) .. +(.5,0);
\draw (1.5,0) node[above]{\tiny $\vee$} .. controls +(0,-.5) and +(0,-.5) .. +(.5,0);
\end{scope}

\begin{scope}[shift={(2,-5)}]
\draw [dashed,red] (1.25,.2) -- +(0,-1);
\draw (0,0) .. controls +(0,-1) and +(0,-1) .. +(1.5,0);
\draw (.5,0) .. controls +(0,-.5) and +(0,-.5) .. +(.5,0);
\draw (2,0) .. controls +(0,-.5) and +(0,-.5) .. +(.5,0);
\fill (.75,-.75) circle(2.5pt);
\fill (2.25,-.36) circle(2.5pt);
\end{scope}

\begin{scope}[xshift=8cm]
\draw[<->] (.5,-1) -- +(0,-1);
\draw[<->] (.5,-4.5) -- +(0,1);

\begin{scope}[shift={(0,.5)}]
\draw (0,0) -- +(0,-.7);
\draw (.5,0) -- +(0,-.7);
\draw (1,0) -- +(0,-.7);
\end{scope}

\begin{scope}[shift={(0,-2.5)}]
\draw (0,0) .. controls +(0,-.5) and +(0,-.5) .. +(.5,0);
\draw (1,0) -- +(0,-.7);
\fill (.25,-.36) circle(2.5pt);
\end{scope}

\begin{scope}[shift={(0,-5)}]
\draw (0,0) -- +(0,-.7);
\draw (.5,0) .. controls +(0,-.5) and +(0,-.5) .. +(.5,0);
\fill (0,-.35) circle(2.5pt);
\end{scope}
\end{scope}
\end{scope}

\end{tikzpicture}
\end{eqnarray}
with the triple $(\underline{\tilde{\la}^{\scriptscriptstyle(1)}},\underline{\tilde{\la}^{\scriptscriptstyle(2)}},
\underline{\tilde{\la}^{\scriptscriptstyle(3)}})$ displayed next to it (the additional points are to the left of the dashed line). In particular, the circle diagrams $\underline{\tilde{\la}^{\scriptscriptstyle(1)}}\overline{\tilde{\la}^{\scriptscriptstyle(3)}}$ and $\underline{\tilde{\la}^{\scriptscriptstyle(3)}}\overline{\tilde{\la}^{\scriptscriptstyle(1)}}$ cannot be oriented, hence the claim of the lemma follows in this case.

Now assume we have a special $\down$ (automatically the rightmost one) which either has to be removed or be moved to the left. Then the possible moves (with the position of our special $\down$ indicated) are of the form
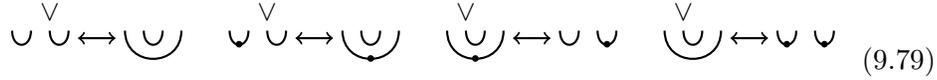
\begin{equation}
\label{downout}
\begin{tikzpicture}[thick,scale=0.5]
\draw (0,0) node[above]{$$} .. controls +(0,-.5) and +(0,-.5) ..
+(.5,0) node[above]{};
\draw (1,0) node[above]{$\down$} .. controls +(0,-.5) and +(0,-.5) ..
+(.5,0) node[above]{};
\draw[<->] (1.75,-.2) -- +(1,0);
\draw (3,0) node[above]{} .. controls +(0,-1) and +(0,-1) .. +(1.5,0)
node[above]{};
\draw (3.5,0) node[above]{} .. controls +(0,-.5) and +(0,-.5) ..
+(.5,0) node[above]{};
\end{tikzpicture}
\quad
\begin{tikzpicture}[thick,scale=0.5]
\draw (0,0) node[above]{} .. controls +(0,-.5) and +(0,-.5) .. +(.5,0)
node[above]{};
\fill (.25,-.36) circle(2.5pt);
\draw (1,0) node[above]{$\down$} .. controls +(0,-.5) and +(0,-.5) ..
+(.5,0) node[above]{};
\draw[<->] (1.75,-.2) -- +(1,0);
\draw (3,0) node[above]{$$} .. controls +(0,-1) and +(0,-1) .. +(1.5,0)
node[above]{};
\fill (3.75,-.74) circle(2.5pt);
\draw (3.5,0) node[above]{} .. controls +(0,-.5) and +(0,-.5) ..
+(.5,0) node[above]{};
\end{tikzpicture}
\quad
\begin{tikzpicture}[thick,scale=0.5]
\draw (0,0) node[above]{$$} .. controls +(0,-1) and +(0,-1) .. +(1.5,0)
node[above]{};
\draw (.5,0) node[above]{$\down$} .. controls +(0,-.5) and +(0,-.5) ..
+(.5,0) node[above]{};
\fill (0.75,-.74) circle(2.5pt);
\draw[<->] (1.75,-.2) -- +(1,0);
\draw (3,0) node[above]{} .. controls +(0,-.5) and +(0,-.5) ..
+(.5,0) node[above]{};
\draw (4,0) node[above]{} .. controls +(0,-.5) and +(0,-.5) .. +(.5,0)
node[above]{};
\fill (4.25,-.36) circle(2.5pt);
\end{tikzpicture}
\quad
\begin{tikzpicture}[thick,scale=0.5]
\draw (0,0) node[above]{} .. controls +(0,-1) and +(0,-1) .. +(1.5,0)
node[above]{};
\draw (.5,0) node[above]{$\down$} .. controls +(0,-.5) and +(0,-.5) ..
+(.5,0) node[above]{};
\draw[<->] (1.75,-.2) -- +(1,0);
\draw (3,0) node[above]{} .. controls +(0,-.5) and +(0,-.5) .. +(.5,0)
node[above]{};
\fill (3.25,-.36) circle(2.5pt);
\draw (4,0) node[above]{} .. controls +(0,-.5) and +(0,-.5) .. +(.5,0)
node[above]{};
\fill (4.25,-.36) circle(2.5pt);
\end{tikzpicture}
\end{equation}
One can easily verify that the possible extended diamonds must involve the first move. (For instance the third move in \eqref{downout} is only possible if the special cup is nested inside a dotted cup, but then the second move is impossible).
We are left with the following diamonds and the second diamond in \eqref{nodiamonds} (the additional points are now to the right of the dashed line).
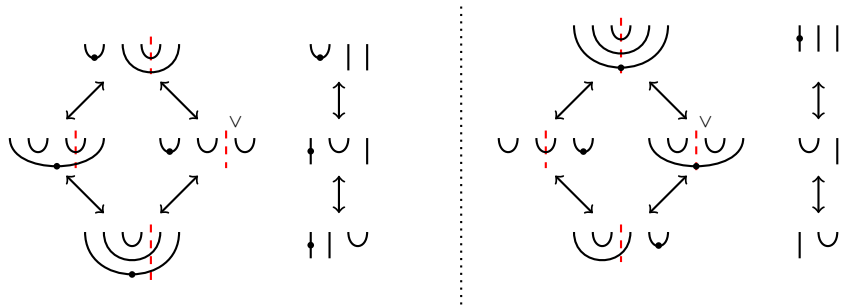
\begin{eqnarray*}
\begin{tikzpicture}[thick,,scale=.5]
\draw[<->] (2.5,-1) -- +(-1,-1);
\draw[<->] (4,-1) -- +(1,-1);
\draw[<->] (2.5,-4.5) -- +(-1,1);
\draw[<->] (4,-4.5) -- +(1,1);

\begin{scope}[shift={(2,0)}]
\draw [dashed,red] (1.75,.2) -- +(0,-1);
\draw (1,0) .. controls +(0,-1) and +(0,-1) .. +(1.5,0);
\draw (0,0) .. controls +(0,-.5) and +(0,-.5) .. +(.5,0);
\draw (1.5,0) .. controls +(0,-.5) and +(0,-.5) .. +(.5,0);
\fill (.25,-.36) circle(2.5pt);
\end{scope}

\begin{scope}[shift={(0,-2.5)}]
\draw [dashed,red] (1.75,.2) -- +(0,-1);
\draw (0,0) .. controls +(0,-1) and +(0,-1) .. +(2.5,0);
\draw (.5,0) .. controls +(0,-.5) and +(0,-.5) .. +(.5,0);
\draw (1.5,0) .. controls +(0,-.5) and +(0,-.5) .. +(.5,0);
\fill (1.25,-.75) circle(2.5pt);
\end{scope}

\begin{scope}[shift={(4,-2.5)}]
\draw [dashed,red] (1.75,.2) -- +(0,-1);
\draw (0,0) .. controls +(0,-.5) and +(0,-.5) .. +(.5,0);
\draw (1,0) .. controls +(0,-.5) and +(0,-.5) .. +(.5,0);
\draw (2,0) node[above]{\tiny $\vee$} .. controls +(0,-.5) and +(0,-.5) .. +(.5,0);
\fill (.25,-.36) circle(2.5pt);
\end{scope}

\begin{scope}[shift={(2,-5)}]
\draw [dashed,red] (1.75,.2) -- +(0,-1.5);
\draw (0,0) .. controls +(0,-1.5) and +(0,-1.5) .. +(2.5,0);
\draw (.5,0) .. controls +(0,-1) and +(0,-1) .. +(1.5,0);
\draw (1,0) .. controls +(0,-.5) and +(0,-.5) .. +(.5,0);
\fill (1.25,-1.13) circle(2.5pt);
\end{scope}

\begin{scope}[xshift=8cm]
\draw[<->] (.75,-1) -- +(0,-1);
\draw[<->] (.75,-4.5) -- +(0,1);

\begin{scope}[shift={(0,0)}]
\draw (0,0) .. controls +(0,-.5) and +(0,-.5) .. +(.5,0);
\draw (1,0) -- +(0,-.7);
\draw (1.5,0) -- +(0,-.7);
\fill (.25,-.36) circle(2.5pt);
\end{scope}

\begin{scope}[shift={(0,-2.5)}]
\draw (0,0) -- +(0,-.7);
\draw (.5,0) .. controls +(0,-.5) and +(0,-.5) .. +(.5,0);
\draw (1.5,0) -- +(0,-.7);
\fill (0,-.35) circle(2.5pt);
\end{scope}

\begin{scope}[shift={(0,-5)}]
\draw (0,0) -- +(0,-.7);
\draw (.5,0) -- +(0,-.7);
\draw (1,0) .. controls +(0,-.5) and +(0,-.5) .. +(.5,0);
\fill (0,-.35) circle(2.5pt);
\end{scope}
\end{scope}

\draw[thick, dotted] (12,1) -- +(0,-8);

\begin{scope}[xshift=13cm]
\draw[<->] (2.5,-1) -- +(-1,-1);
\draw[<->] (4,-1) -- +(1,-1);
\draw[<->] (2.5,-4.5) -- +(-1,1);
\draw[<->] (4,-4.5) -- +(1,1);

\begin{scope}[shift={(2,.5)}]
\draw [dashed,red] (1.25,.2) -- +(0,-1.5);
\draw (0,0) .. controls +(0,-1.5) and +(0,-1.5) .. +(2.5,0);
\draw (.5,0) .. controls +(0,-1) and +(0,-1) .. +(1.5,0);
\draw (1,0) .. controls +(0,-.5) and +(0,-.5) .. +(.5,0);
\fill (1.25,-1.13) circle(2.5pt);
\end{scope}

\begin{scope}[shift={(0,-2.5)}]
\draw [dashed,red] (1.25,.2) -- +(0,-1);
\draw (0,0) .. controls +(0,-.5) and +(0,-.5) .. +(.5,0);
\draw (1,0) .. controls +(0,-.5) and +(0,-.5) .. +(.5,0);
\draw (2,0) .. controls +(0,-.5) and +(0,-.5) .. +(.5,0);
\fill (2.25,-.36) circle(2.5pt);
\end{scope}

\begin{scope}[shift={(4,-2.5)}]
\draw [dashed,red] (1.25,.2) -- +(0,-1.1);
\draw (0,0) .. controls +(0,-1) and +(0,-1) .. +(2.5,0);
\draw (.5,0) .. controls +(0,-.5) and +(0,-.5) .. +(.5,0);
\draw (1.5,0) node[above]{\tiny $\vee$} .. controls +(0,-.5) and +(0,-.5) .. +(.5,0);
\fill (1.25,-.75) circle(2.5pt);
\end{scope}

\begin{scope}[shift={(2,-5)}]
\draw [dashed,red] (1.25,.2) -- +(0,-1);
\draw (0,0) .. controls +(0,-1) and +(0,-1) .. +(1.5,0);
\draw (.5,0) .. controls +(0,-.5) and +(0,-.5) .. +(.5,0);
\draw (2,0) .. controls +(0,-.5) and +(0,-.5) .. +(.5,0);
\fill (2.25,-.36) circle(2.5pt);
\end{scope}

\begin{scope}[xshift=8cm]
\draw[<->] (.5,-1) -- +(0,-1);
\draw[<->] (.5,-4.5) -- +(0,1);

\begin{scope}[shift={(0,.5)}]
\draw (0,0) -- +(0,-.7);
\draw (.5,0) -- +(0,-.7);
\draw (1,0) -- +(0,-.7);
\fill (0,-.35) circle(2.5pt);
\end{scope}

\begin{scope}[shift={(0,-2.5)}]
\draw (0,0) .. controls +(0,-.5) and +(0,-.5) .. +(.5,0);
\draw (1,0) -- +(0,-.7);
\end{scope}

\begin{scope}[shift={(0,-5)}]
\draw (0,0) -- +(0,-.7);
\draw (.5,0) .. controls +(0,-.5) and +(0,-.5) .. +(.5,0);
\end{scope}
\end{scope}
\end{scope}

\end{tikzpicture}
\end{eqnarray*}
In all cases the circle diagrams $\underline{\tilde{\la}^{\scriptscriptstyle(1)}}\overline{\tilde{\la}^{\scriptscriptstyle(3)}}$ and $\underline{\tilde{\la}^{\scriptscriptstyle(3)}}\overline{\tilde{\la}^{\scriptscriptstyle(1)}}$ cannot be oriented. Since the compositions are zero if no rays are involved they are obviously also zero in the general case by Section~\ref{annoying}. The lemma follows.
\end{proof}

The following will ensure the compatibility of $\Phi_k$ with Relation (\text{R}-\ref{rel7}) i).

\begin{lemma}
\label{stupidcalca}
Let $k\geq 4$ and assume that $(\la^{\scriptscriptstyle(1)},\la^{\scriptscriptstyle(2)},\la^{\scriptscriptstyle(3)},\la^{\scriptscriptstyle(4)})$ is a diamond in $\Lambda_{k}^{\ov{0}}$. With $\Phi$ as in \eqref{isops} we have
  \begin{eqnarray*}
    \Phi(p(\la^{\scriptscriptstyle(3)},\la^{\scriptscriptstyle(2)}))\Phi(p(\la^{\scriptscriptstyle(2)},\la^{\scriptscriptstyle(1)}))&=&\Phi(p(\la^{\scriptscriptstyle(3)},\la^{\scriptscriptstyle(4)}))\Phi(p(\la^{\scriptscriptstyle(4)},\la^{\scriptscriptstyle(1)})).
  \end{eqnarray*}
\end{lemma}

\begin{proof}
 Given a diamond from \eqref{diamonds2} or \eqref{diamonds1}, say 
\begin{eqnarray}
\label{diamondwrong}
\begin{tikzpicture}[thick,scale=0.7]
\node at (1.25,-.5) {$\lambda^{\scriptscriptstyle(1)}$};
\node at (-.25,-2.5) {$\lambda^{\scriptscriptstyle(2)}$};
\node at (2.75,-2.5) {$\lambda^{\scriptscriptstyle(4)}$};
\node at (1.25,-4.4) {$\lambda^{\scriptscriptstyle(3)}$};
\draw[<->] (.8,-1.1) -- +(-.8,-.8);
\draw[<->] (1.7,-1.1) -- +(.8,-.8);
\draw[<->] (-.2,-3) -- +(.8,-.8);
\draw[<->] (2.5,-3) -- +(-.8,-.8);
\node at (-.7,-1.3) {$(\alpha_1,\beta_1)$};
\node at (3.2,-1.3) {$(\alpha_2,\beta_2)$};
\node at (-.7,-3.7) {$(\alpha_3,\beta_3)$};
\node at (3.2,-3.7) {$(\alpha_4,\beta_4)$};
\end{tikzpicture}
\end{eqnarray}

we consider the four cup diagrams  $\underline{\la^{\scriptscriptstyle(1)}},\underline{\la^{\scriptscriptstyle(2)}},\underline{\la^{\scriptscriptstyle(3)}},\underline{\la^{\scriptscriptstyle(4)}}$ and number the involved  vertices from left to right by $1$ to $6$ respectively $1$ to $8$. 
Number in each of them the cups from left to right according to their left endpoint. Encode the dots via the four sets $D^{\scriptscriptstyle(1)},D^{\scriptscriptstyle(2)},D^{\scriptscriptstyle(3)},D^{\scriptscriptstyle(4)}$ of dotted cups for the top, middle left, middle right and bottom cup diagram. For a $\la$-pair $(\alpha,\beta)$ we denote by $m$ the position of the rightmost vertex of the component which contains the cup-cap-pair involved in the surgery. For instance, in the notation of \eqref{lapair2}, the first diagram in \eqref{diamonds2} without dots has the following triples $(\alpha_i,\beta_i,m_i)$ for $i=1,\ldots, 4$ of $\la$-pairs and maximal vertices  :
$$(2,3,4),(2,5,6),(4,5,6),(3,4,5).$$
Then the composition of the  maps on the left of \eqref{diamondwrong} equals $$1+\frac{1}{2}(-1)^3X_4+\frac{1}{2}(-1)^5X_6=1-X_6$$ as elements of $\cM(\underline{\la^{\scriptscriptstyle(2)}}\overline{\la^{\scriptscriptstyle(4)}})$ from Proposition \ref{coho}
by the relations given by $\underline{\la^{\scriptscriptstyle(4)}}$, whereas the composition of the  maps on the right equals $1+\frac{1}{2}(-1)^5X_6+\frac{1}{2}(-1)^4X_5=1-X_6$. The allowed composition of any other two maps in the diamond is also equal to $1-X_6$. Note that all the involved surgery moves in the diamonds \eqref{diamonds1} and \eqref{diamonds2} are merges, hence no additional signs appear and we do not have to care about the actual positions, but only the relative position of the cups.

The following three tables list all the possible decorations with the corresponding resulting maps for the three cases in \eqref{diamonds2} where we abbreviate the triples $P_i=(\alpha_i,\beta_i,m_i)$.

\begin{eqnarray*}
\begin{array}[t]{ccccccccc}
D^{\scriptscriptstyle(1)}&D^{\scriptscriptstyle(2)}&D^{\scriptscriptstyle(3)}&D^{\scriptscriptstyle(4)}&P_1&P_2&P_3&P_4&\text{result}\\
\hline
\{\}&\{\}&\{\}&\{\}&(2,3,4)&(2,5,6)&(4,5,6)&(3,4,5)&1-X_6\\
\{1,3\}&\{1,3\}&\{\}&\{\}&(2,3,4)&(1,2,6)&(1,4,6)&(3,4,5)&1+X_6\\
\{1\}&\{1\}&\{1\}&\{1\}&(2,3,4)&(2,5,6)&(4,5,6)&(3,4,5)&1-X_6\\
\{3\}&\{3\}&\{1\}&\{1\}&(2,3,4)&(1,2,6)&(1,4,6)&(3,4,5)&1+X_6
\end{array}
\end{eqnarray*}

\begin{eqnarray*}
\begin{array}[t]{ccccccccc}
D^{\scriptscriptstyle(1)}&D^{\scriptscriptstyle(2)}&D^{\scriptscriptstyle(3)}&D^{\scriptscriptstyle(4)}&P_1&P_2&P_3&P_4&\text{result}\\
\hline
\{\}&\{\}&\{\}&\{\}&(4,5,6)&(2,5,6)&(2,3,6)&(3,4,5)&1-X_6\\
\{3\}&\{2\}&\{1\}&\{1\}&(3,4,6)&(1,2,6)&(1,2,6)&(3,4,5)&1+X_6\\
\{1\}&\{1\}&\{1\}&\{1\}&(4,5,6)&(2,5,6)&(2,3,6)&(3,4,5)&1-X_6\\
\{1,3\}&\{1,2\}&\{\}&\{\}&(3,4,6)&(1,2,6)&(1,2,6)&(3,4,5)&1+X_6\\
\end{array}
\end{eqnarray*}

\begin{eqnarray*}
\begin{array}[t]{ccccccccc}
D^{\scriptscriptstyle(1)}&D^{\scriptscriptstyle(2)}&D^{\scriptscriptstyle(3)}&D^{\scriptscriptstyle(4)}&P_1&P_2&P_3&P_4&\text{result}\\
\hline
\{\}&\{\}&\{\}&\{\}&(2,3,4)&(4,5,6)&(4,5,6)&(2,3,6)&1-X_6\\
\{1\}&\{1\}&\{1\}&\{1\}&(2,3,4)&(4,5,6)&(4,5,6)&(2,3,6)&1-X_6\\
\{2\}&\{1\}&\{1\}&\{2\}&(1,2,4)&(4,5,6)&(4,5,6)&(1,2,6)&1\\
\{1,2,3\}&\{3\}&\{1\}&\{1\}&(1,2,4)&(3,4,6)&(1,4,6)&(2,3,6)&1\\
\{3\}&\{3\}&\{1\}&\{2\}&(2,3,4)&(3,4,6)&(1,4,6)&(1,2,6)&1+X_6\\
\{2,3\}&\{1,3\}&\{\}&\{\}&(1,2,4)&(3,4,6)&(1,4,6)&(2,3,6)&1\\
\{1,3\}&\{1,3\}&\{\}&\{1,2\}&(2,3,4)&(3,4,6)&(1,4,6)&(1,2,6)&1+X_6\\
\{1,2\}&\{\}&\{\}&\{1,2\}&(1,2,4)&(4,5,6)&(4,5,6)&(1,2,6)&1\\
\end{array}
\end{eqnarray*}

The claim is obviously true for the first two diamonds in \eqref{diamonds1} with all possible decorations.
The following five tables list all possible decorations for the remaining five diamonds with the corresponding resulting maps:

\begin{eqnarray*}
\begin{array}[t]{ccccc|cccccc}
D^{\scriptscriptstyle(1)}&D^{\scriptscriptstyle(2)}&D^{\scriptscriptstyle(3)}&D^{\scriptscriptstyle(4)}&\text{result}&D^{\scriptscriptstyle(1)}&D^{\scriptscriptstyle(2)}&D^{\scriptscriptstyle(3)}&D^{\scriptscriptstyle(4)}&\text{result}\\
\hline
\{\}&\{\}&\{\}&\{\}&1-X_8&
\{1,2,4\}&\{4\}&\{1\}&\{1\}&1\\
\{1,2\}&\{\}&\{\}&\{1,2\}&1&
\{2,4\}&\{1,4\}&\{\}&\{\}&1\\
\{1,4\}&\{1,4\}&\{\}&\{1,2\}&1+X_8&
\{1\}&\{1\}&\{1\}&\{1\}&1-X_8\\
\{2\}&\{1\}&\{1\}&\{2\}&1\\
\end{array}
\end{eqnarray*}

\begin{eqnarray*}
\begin{array}[t]{ccccc|cccccc}
D^{\scriptscriptstyle(1)}&D^{\scriptscriptstyle(2)}&D^{\scriptscriptstyle(3)}&D^{\scriptscriptstyle(4)}&\text{result}&D^{\scriptscriptstyle(1)}&D^{\scriptscriptstyle(2)}&D^{\scriptscriptstyle(3)}&D^{\scriptscriptstyle(4)}&\text{result}\\
\hline
\{\}&\{\}&\{\}&\{\}&1-X_8&
\{1\}&\{1\}&\{1\}&\{1\}&1-X_8\\
\{1,3\}&\{\}&\{\}&\{1,2\}&1+X_8&
\{3\}&\{1\}&\{1\}&\{2\}&1+X_8\\
\end{array}
\end{eqnarray*}

\begin{eqnarray*}
\begin{array}[t]{ccccc|cccccc}
D^{\scriptscriptstyle(1)}&D^{\scriptscriptstyle(2)}&D^{\scriptscriptstyle(3)}&D^{\scriptscriptstyle(4)}&\text{result}&D^{\scriptscriptstyle(1)}&D^{\scriptscriptstyle(2)}&D^{\scriptscriptstyle(3)}&D^{\scriptscriptstyle(4)}&\text{result}\\
\hline
\{\}&\{\}&\{\}&\{\}&1-X_8&
\{1,4\}&\{\}&\{\}&\{1,4\}&1+X_8\\
\{1\}&\{1\}&\{1\}&\{1\}&1-X_8&
\{4\}&\{1\}&\{1\}&\{4\}&1+X_8\\
\end{array}
\end{eqnarray*}

\begin{eqnarray*}
\begin{array}[t]{ccccc|cccccc}
D^{\scriptscriptstyle(1)}&D^{\scriptscriptstyle(2)}&D^{\scriptscriptstyle(3)}&D^{\scriptscriptstyle(4)}&\text{result}&D^{\scriptscriptstyle(1)}&D^{\scriptscriptstyle(2)}&D^{\scriptscriptstyle(3)}&D^{\scriptscriptstyle(4)}&\text{result}\\
\hline
\{\}&\{\}&\{\}&\{\}&1+\frac{X_7-X_8}{2}&
\{1,2\}&\{\}&\{\}&\{1,2\}&1+\frac{X_7+X_8}{2}\\
\{1\}&\{1\}&\{1\}&\{1\}&1+\frac{X_7-X_8}{2}&
\{2\}&\{1\}&\{1\}&\{2\}&1+\frac{X_7+X_8}{2}\\
\end{array}
\end{eqnarray*}

\begin{eqnarray*}
\begin{array}[t]{ccccc|cccccc}
D^{\scriptscriptstyle(1)}&D^{\scriptscriptstyle(2)}&D^{\scriptscriptstyle(3)}&D^{\scriptscriptstyle(4)}&\text{result}&D^{\scriptscriptstyle(1)}&D^{\scriptscriptstyle(2)}&D^{\scriptscriptstyle(3)}&D^{\scriptscriptstyle(4)}&\text{result}\\
\hline
\{\}&\{\}&\{\}&\{\}&1+\frac{X_5-X_8}{2}&
\{1,4\}&\{\}&\{\}&\{1,4\}&1+\frac{X_5+X_8}{2}\\
\{1\}&\{1\}&\{1\}&\{1\}&1+\frac{X_5-X_8}{2}&
\{4\}&\{1\}&\{1\}&\{4\}&1+\frac{X_5+X_8}{2}\\
\end{array}
\end{eqnarray*}
Hence the claim follows. 
\end{proof}
We finally showed that $\Phi$ is a well-defined map of algebras. This establishes Theorem \ref{mainthm},  showing that $\ABr$ and $\mathbb{D}_k$ are isomorphic as algebras.

\subsection{Passage to type ${\rm B}_{k}$}
\label{typeB}
Our result for Grassmannians of type ${\rm D}_{k+1}$ transfers directly to ${\rm B}_k$ via the following observation. The natural inclusions of algebraic groups $\op{SO}(2k+1,\mC)\hookrightarrow \op{SO}(2k+2,\mC)$, for $k \geq 2$, induce isomorphisms,  compatible with the Schubert stratification, of the partial flag varieties for the pairs $(G,P)$ of type $({\rm B}_k, {\rm A}_{k-1})$ and $({\rm D}_{k+1}, {\rm A}_{k})$, see  \cite[3.1]{BrionPolo}. As a  consequence, the corresponding categories, $\op{Perv}_k(\mathrm{B})$ and $\op{Perv}_{k+1}(\mathrm{D})$, of perverse sheaves constructible with respect to the Schubert stratification are equivalent, and hence both equivalent to the category $\mathbb{D}_{k+1}\Mod$.
To match our algebra directly with the combinatorics of the type ${\rm B}_k$ Grassmannian note that its Schubert varieties are canonically labelled by the set $W^\p$ of shortest length coset representatives for $W_\p \backslash W$, where $W$ denotes the Weyl group of type ${\rm B}_k$ and $W_\p$ its parabolic subroup of type ${\rm A}_{k-1}$. We choose the standard generators $s_0, s_1,\dots, s_{k-1}$ of $W$ such that the $s_i$ for $i>0$ generate $W_\p$ and $s_0, s_1$ form a Weyl group of type ${\rm B}_2$. Then $W$ acts naturally (from the right) on the set $S$ of $\{\up.\down\}$ sequences  of length $k$, where $s_i$ for $i>0$ acts by swapping the $i$-th and $(i+1)$-st symbol and $s_0$ changes the first symbol from $\up$ to $\down$ and from $\down$ to $\up$. Sending the identity element to $\down\down\ldots\down$ induces a bijection between $W^\p$ and $S$.  Now $S$ can be identified with the principal block $\Lambda_{k+1}^{\ov 0}$ via the assignment  
\begin{eqnarray*}
\mathbf{s}&\longmapsto&\mathbf{s}^\dagger=
\begin{cases}
\down\mathbf{s}&\text{if the parity of the number of $\up$'s in $\mathbf{s}$ is even,}\\
\up\mathbf{s}&\text{if the parity of the number of $\up$'s in $\mathbf{s}$ is odd,}
\end{cases}
\end{eqnarray*}
One easily checks that it corresponds to a bijection $W^\p\cong W^{\overline{0}}$ sending a reduced expression $w=s_0s_{i_1}s_{i_2}\cdots s_{i_r}$ to $s_0s_{j_1}s_{j_2}\cdots s_{j_r}$, where $j_a=i_a$ if $i_a\not=0$ and otherwise $j_a=0$ respectively $j_a=1$ depending whether  there are an even respectively odd number of $s_0$'s strictly to the left of $s_{j_a}$.  Altogether there is an equivalence of categories 
\begin{eqnarray*}
\op{Perv}_k(\mathrm{B})&\cong &\mathbb{D}_{k+1}\Mod
\end{eqnarray*}
sending the simple module labelled by $\mathbf{s}$ to the  the simple module labelled by $\mathbf{s}^\dagger.$ Moreover it provides a direct closed formula for type  $({\rm B}_k, {\rm A}_{k-1})$ parabolic Kazhdan-Lusztig polynomials $d^{{\rm B}_k}_{\mathbf{s},\mathbf{t}}$ via Lemma~\ref{KLformula}, namely 
\begin{eqnarray}
\label{B}
d^{{\rm B}_k}_{\mathbf{s},\mathbf{t}}(q)&=&
\begin{cases}
q^d&\text{if $\underline{\la}\mu$ is oriented of degree $d$,} \\
0&\text{if $\underline{\la}\mu$ is not oriented}.
\end{cases}
\quad=\quad d_{\la,\mu}(q)
\end{eqnarray}
where $\la$ and $\mu$ are the principal weights for ${\rm D}_{k+1}$ given by $\mathbf{s}^\dagger$ and $\mathbf{t}^\dagger$.
This simplifies the approach from  \cite{Shigechi} substantially, by reducing it to \cite{LS}.

\bibliographystyle{siam}
\bibliography{references}

\def\cprime{$'$}
\begin{thebibliography}{10}

\bibitem{BGS}
{\sc A.~Beilinson, V.~Ginzburg, and W.~Soergel}, {\em Koszul duality patterns
  in representation theory}, J. Amer. Math. Soc., 9 (1996), pp.~473--527.

\bibitem{BoeCollingwood}
{\sc B.~D. Boe and D.~H. Collingwood}, {\em Multiplicity free categories of
  highest weight representations. {I}}, Comm. Algebra, 18 (1990),
  pp.~947--1032.

\bibitem{BHKostant}
{\sc B.~D. Boe and M.~Hunziker}, {\em Kostant modules in blocks of category
  {${\mathcal O}_S$}}, Comm. Algebra, 37 (2009), pp.~323--356.

\bibitem{Braden}
{\sc T.~Braden}, {\em Perverse sheaves on {G}rassmannians}, Canad. J. Math., 54
  (2002), pp.~493--532.

\bibitem{BrionPolo}
{\sc M.~Brion and P.~Polo}, {\em Generic singularities of certain {S}chubert
  varieties}, Math. Z., 231 (1999), pp.~301--324.

\bibitem{BrundanSpringer}
{\sc J.~Brundan}, {\em Symmetric functions, parabolic category $\mathcal{O}$,
  and the {S}pringer fiber}, Duke Math. J., 143 (2008), pp.~41--79.

\bibitem{BS1}
{\sc J.~Brundan and C.~Stroppel}, {\em Highest weight categories arising from
  {K}hovanov's diagram algebra {I}: cellularity}, Mosc. Math. J., 11 (2011),
  pp.~685--722.

\bibitem{BS3}
\leavevmode\vrule height 2pt depth -1.6pt width 23pt, {\em Highest weight
  categories arising from {K}hovanov's diagram algebra {III}: Category
  $\mathcal{O}$}, Repr. Theory, 15 (2011), pp.~170--243.

\bibitem{CPS}
{\sc E.~Cline, B.~Parshall, and L.~Scott}, {\em {The homological dual of a
  highest weight category}}, Proc. London Math. Soc., 68 (1994), pp.~296--316.

\bibitem{CoxdVi}
{\sc A.~Cox and M.~De~Visscher}, {\em Diagrammatic {K}azhdan-{L}usztig theory
  for the (walled) {B}rauer algebra}, J. Algebra, 340 (2011), pp.~151--181.

\bibitem{CMD1}
{\sc A.~Cox, M.~De~Visscher, and P.~Martin}, {\em The blocks of the {B}rauer
  algebra in characteristic zero}, Represent. Theory, 13 (2009), pp.~272--308.

\bibitem{CMD2}
\leavevmode\vrule height 2pt depth -1.6pt width 23pt, {\em A geometric
  characterisation of the blocks of the {B}rauer algebra}, J. Lond. Math. Soc.
  (2), 80 (2009), pp.~471--494.

\bibitem{ES2}
{\sc M.~Ehrig and C.~Stroppel}, {\em Nazarov-{W}enzl algebras, coideal
  subalgebras and categorified skew {H}owe duality}, arXiv:1310.1972,  (2013).

\bibitem{ES1}
\leavevmode\vrule height 2pt depth -1.6pt width 23pt, {\em 2-row {S}pringer
  fibres and {K}hovanov diagram algebras for type {$D$}}.
\newblock arXiv:1209.4998 , to appear in CJM, 2015.

\bibitem{ESKoszul}
\leavevmode\vrule height 2pt depth -1.6pt width 23pt, {\em Koszul {G}radings on
  {B}rauer algebras}.
\newblock arXiv:1504.03924, to appear in IMRN, 2015.

\bibitem{Li}
{\sc Z.~Fan and Y.~Li}, {\em Geometric {S}chur {D}uality of {C}lassical {T}ype,
  {II}}, Trans. Amer. Math. Soc. Ser. B, 2 (2015), pp.~51--92.

\bibitem{FKS}
{\sc I.~Frenkel, M.~Khovanov, and C.~Stroppel}, {\em A categorification of
  finite-dimensional irreducible representations of quantum {$\mathfrak{sl}_2$}
  and their tensor products}, Selecta Math. (N.S.), 12 (2006), pp.~379--431.

\bibitem{FKK}
{\sc I.~B. Frenkel, M.~G. Khovanov, and A.~A. Kirillov, Jr.}, {\em
  Kazhdan-{L}usztig polynomials and canonical basis}, Transform. Groups, 3
  (1998), pp.~321--336.

\bibitem{Fulton}
{\sc W.~Fulton}, {\em Young tableaux}, vol.~35 of LMS Student Texts, Cambridge
  University Press, 1997.

\bibitem{GL}
{\sc J.~J. Graham and G.~I. Lehrer}, {\em Cellular algebras}, Invent. Math.,
  123 (1996), pp.~1--34.

\bibitem{Green}
{\sc R.~M. Green}, {\em Generalized {T}emperley-{L}ieb algebras and decorated
  tangles}, J. Knot Theory Ramifications, 7 (1998), pp.~155--171.

\bibitem{BuchHottaetal}
{\sc R.~Hotta, K.~Takeuchi, and T.~Tanisaki}, {\em {$D$}-modules, perverse
  sheaves, and representation theory}, vol.~236 of Progress in Mathematics,
  Birkh\"auser Boston Inc., Boston, MA, 2008.
\newblock Translated from the 1995 Japanese edition by Takeuchi.

\bibitem{MH}
{\sc J.~Hu and A.~Mathas}, {\em Graded cellular bases for the cyclotomic
  {K}hovanov-{L}auda-{R}ouquier algebras of type {$A$}}, Adv. Math., 225
  (2010), pp.~598--642.

\bibitem{Humphreys}
{\sc J.~E. Humphreys}, {\em Representations of semisimple {L}ie algebras in the
  {BGG} category {$\mathcal{O}$}}, vol.~94 of Graduate Studies in Mathematics,
  American Mathematical Society, Providence, RI, 2008.

\bibitem{Irvingselfdual}
{\sc R.~S. Irving}, {\em Projective modules in the category {${\mathcal O}_S$}:
  self-duality}, Trans. Amer. Math. Soc., 291 (1985), pp.~701--732.

\bibitem{Khovtangles}
{\sc M.~Khovanov}, {\em A functor-valued invariant of tangles}, Alg. and
  {G}eom. {T}op., 2 (2002), pp.~665--741.

\bibitem{KL}
{\sc M.~Khovanov and A.~D. Lauda}, {\em A diagrammatic approach to
  categorification of quantum groups. {I}}, Represent. Theory, 13 (2009),
  pp.~309--347.

\bibitem{Kolb}
{\sc S.~Kolb}, {\em Quantum symmetric {K}ac-{M}oody pairs}, Adv. Math., 267
  (2014), pp.~395--469.

\bibitem{KoenigXi}
{\sc S.~K{\"o}nig and C.~Xi}, {\em On the structure of cellular algebras}, in
  Algebras and modules, {II} ({G}eiranger, 1996), vol.~24 of CMS Conf. Proc.,
  Amer. Math. Soc., Providence, RI, 1998, pp.~365--386.

\bibitem{LS}
{\sc T.~Lejczyk and C.~Stroppel}, {\em A graphical description of
  $({D}_n,{A}_{n-1})$ {K}azhdan-{L}usztig polynomials}, GMJ, 55 (2012),
  pp.~313--340.

\bibitem{Letzter}
{\sc G.~Letzter}, {\em Coideal subalgebras and quantum symmetric pairs}, in New
  directions in {H}opf algebras, vol.~43 of Math. Sci. Res. Inst. Publ.,
  Cambridge Univ. Press, Cambridge, 2002, pp.~117--165.

\bibitem{Martinblob}
{\sc P.~Martin and H.~Saleur}, {\em The blob algebra and the periodic
  {T}emperley-{L}ieb algebra}, Lett. Math. Phys., 30 (1994), pp.~189--206.

\bibitem{Mazcat}
{\sc V.~Mazorchuk}, {\em Lectures on algebraic categorification}, QGM Master
  Class Series, European Mathematical Society (EMS), Z\"urich, 2012.

\bibitem{Rou2KM}
{\sc R.~Rouquier}, {\em 2-{K}ac-{M}oody algebras}, 2008.
\newblock arXiv:0812.5023.

\bibitem{Antonio}
{\sc A.~{Sartori}}, {\em Categorification of tensor powers of the vector
  representation of ${\mathcal{u}}_q(\mathfrak{gl}(1|1))$}.
\newblock arXiv:1305.6162, to appear in Selecta Math., May 2013.

\bibitem{Shigechi}
{\sc K.~Shigechi}, {\em Kazhdan-{L}usztig polynomials for the {H}ermitian
  symmetric pair $({B}_{N},{A}_{{N}-1})$}.
\newblock arXiv:1412.6740, 2014.

\bibitem{Sperv}
{\sc W.~Soergel}, {\em Kategorie {$\mathcal O$}, perverse {G}arben und {M}oduln
  \"uber den {K}oinvarianten zur {W}eylgruppe}, J. Amer. Math. Soc., 3 (1990),
  pp.~421--445.

\bibitem{SoergelKL}
\leavevmode\vrule height 2pt depth -1.6pt width 23pt, {\em Kazhdan-{L}usztig
  polynomials and a combinatoric[s] for tilting modules}, Represent. Theory, 1
  (1997), pp.~83--114.

\bibitem{Str09}
{\sc C.~Stroppel}, {\em {Parabolic category {$\mathcal{O}$}, perverse sheaves
  on Grassmanninans, Springer fibres and Khovanov homology}}, Comp. Math., 145
  (2009), pp.~954--992.

\bibitem{Wilbert}
{\sc A.~Wilbert}, {\em Topology of two-row {S}pringer fibers for the even
  orthogonal and symplectic group}.
\newblock arXiv:1511.01961, 2015.

\end{thebibliography}

\end{document}